\documentclass[opre,nonblindrev]{informs3} 

\OneAndAHalfSpacedXI

\usepackage{booktabs}
\usepackage{multirow}
\usepackage{caption}
\usepackage{enumitem}
\usepackage{bm}
\usepackage{todo, soul}
\usepackage{algorithm,algpseudocode}
\usepackage{tikz}
\usetikzlibrary{arrows.meta,shapes,decorations,automata,backgrounds,petri,topaths,calc,mindmap,trees,positioning,chains,arrows}

\usepackage{graphicx}
\usepackage{subcaption}
\graphicspath{{../Figures/}}
\usepackage{color, bm}
\usepackage{enumitem} 
\usepackage{lineno}
\usepackage{mwe}
\usepackage{url}
\usepackage{textcomp}

\usepackage{xr-hyper}
\usepackage[hypertexnames=false]{hyperref}
\usepackage[capitalize]{cleveref}

\newtheorem{thm}{Theorem}[section]
\newtheorem{cor}[thm]{Corollary}
\newtheorem{lem}[thm]{Lemma}

\newtheorem{assm}[thm]{Assumption}
\newtheorem{defn}[thm]{Definition}
\newtheorem{rem}[thm]{Remark}
\newtheorem{examp}[thm]{Example}

\crefname{assm}{Assumption}{Assumptions}
\crefname{lem}{Lemma}{Lemmas}
\crefname{app}{Appendix}{Appendices}
\crefname{problem}{Problem}{Problems}
\creflabelformat{problem}{#2(#1)#3}
\crefname{part}{part}{parts}
\crefname{examp}{Example}{Examples}
\crefname{thm}{Theorem}{Theorems}

\newcommand{\R}{{\mathbb{R}}}
\newcommand{\E}{{\mathbb{E}}}

\renewcommand{\L}{\mathcal L}

\newcommand{\abs}[1]{\left|#1\right|}
\newcommand{\Eb}[1]{\mathbb{E}\left[ #1 \right]}
\newcommand{\Pb}[1]{\mathbb{P}\left( #1 \right)}
\newcommand{\Ib}[1]{\mathbb{I}\left\{ #1 \right\}}

\newcommand{\bx}{\bm{x}}
\newcommand{\by}{\bm{y}}

\newcommand{\Y}{\mathcal X^0}

\providecommand{\be}{}
\renewcommand{\be}{\bm{e}}

\newcommand{\bA}{\bm  A}
\newcommand{\br}{\bm{r}}
\newcommand{\bz}{\bm{z}}
\newcommand{\bzbar}{\overline{\bm{z}}}

\newcommand{\ba}{\bm{a}}
\newcommand{\bZ}{\bm{Z}}
\newcommand{\bW}{\bm{W}}
\newcommand{\blambda}{\bm{\lambda}}

\newcommand{\bmu}{\boldsymbol \mu}
\newcommand{\bxi}{\boldsymbol \xi}

\newcommand{\btheta}{\boldsymbol \theta}
\newcommand{\bthetabar}{\bar{\boldsymbol \theta}}

\newcommand{\Obj}{V}

\newcommand{\vmin}{\nu_{\rm min}}
\newcommand{\vmax}{\nu_{\rm max}}

\newcommand{\blockedit}{\color{black}}

\newcommand{\citeGR}{\mbox{GR 2021}}
\newcommand{\DanskinTitle}{\mbox{Variance Gradient Correction}}
\newcommand{\Danskin}{\mbox{VGC}}

\newcommand{\polylog}{{\rm polylog}}

\newcommand{\DR}{D^{\sf R}}
\newcommand{\Var}{\mathbb V\text{ar}}

\newcommand{\blue}{\textcolor{black}}
\newcommand{\minoredit}{\textcolor{black}}

\usepackage{natbib}
 \bibpunct[, ]{(}{)}{,}{a}{}{,}%
 \def\bibfont{\small}%
 \def\bibsep{\smallskipamount}%
\TheoremsNumberedBySection  
\ECRepeatTheorems

\EquationsNumberedBySection 

\MANUSCRIPTNO{}

\begin{document}


 \RUNAUTHOR{Gupta, Huang, and Rusmevichientong}

\RUNTITLE{Debiasing In-Sample Policy Performance}

\TITLE{Debiasing In-Sample Policy Performance for Small-Data,  Large-Scale Optimization}

\ARTICLEAUTHORS{%
\AUTHOR{Vishal Gupta, Michael Huang, and Paat Rusmevichientong}
\AFF{Data Science and Operations, USC Marshall School of Business, Los Angles, CA 90089,\\ \EMAIL{guptavis@usc.edu, huan076@usc.edu, rusmevic@marshall.usc.edu}} 
} 

\ABSTRACT{
Motivated by the poor performance of cross-validation in settings where data are scarce, 
we propose a novel estimator of the out-of-sample performance of a policy in data-driven optimization.
Our approach exploits the optimization problem's sensitivity analysis to estimate the gradient of the optimal objective value with respect to the amount of noise in the data and uses the estimated gradient to 
debias the policy's in-sample performance.  Unlike cross-validation techniques, our approach avoids sacrificing data for a test set, utilizes all data when training and, hence, is well-suited to settings where data are scarce.  
We prove bounds on the bias and variance of our estimator for optimization problems with uncertain linear objectives but known, potentially non-convex, feasible regions.  For more specialized optimization problems where the feasible region is ``weakly-coupled" in a certain sense, we prove stronger results.  Specifically, we provide explicit high-probability bounds on the error of our estimator that hold uniformly over a policy class and depends on the problem's dimension and policy class's complexity.  Our bounds show that under mild conditions, the error of our estimator vanishes as the dimension of the optimization problem grows, even if the amount of available data remains small and constant.  Said differently, we prove our estimator performs well in the small-data, large-scale regime.  Finally, we numerically compare our proposed method to state-of-the-art approaches
through a case-study on dispatching emergency medical response services using real data.  Our method provides  more accurate estimates of out-of-sample performance and learns better-performing policies.

}


\maketitle
%

\section{Introduction}

The crux of data-driven decision-making is using past data to identify decisions that will have  good out-of-sample performance on future, unseen data.  Indeed, estimating out-of-sample performance is key to both policy evaluation (assessing the quality of a given policy), and to  policy learning (identifying the best policy from a potentially large set of candidates).  Estimating out-of-sample performance, however, is non-trivial.  Naive estimates that leverage the same data to train a policy and to evaluate its performance often suffer a systematic, optimistic bias, referred to as ``in-sample bias" in machine learning and the ``optimizer's curse" in optimization \citep{smith2006optimizer}.  

Consequently, cross-validation and sample-splitting techniques have emerged as the gold-standard approach to estimating out-of-sample performance.  Despite the multitude of cross-validation methods, at a high-level, these methods all proceed by setting aside a portion of the data as ``testing" data \emph{not} to be used when training the policy, and then evaluating the policy on these testing data. The policy's performance on testing data then serves as an estimate of its performance on future, unseen data, thereby circumventing the aforementioned in-sample bias. Cross-validation is ubiquitous in machine learning and statistics 
with provably good performance in large sample settings \citep{bousquet2001algorithmic,kearns1999algorithmic}. 

%
%

Unfortunately, when data are scarce, cross-validation can perform poorly.  \cite{gupta2017small} prove that for the small-data, large-scale regime ---  when the number of uncertain parameters in an optimization problem is large but the amount of relevant data per parameter is small --- each of hold-out, $5$-fold, $10$-fold and leave-one-out cross validation can have poor performance when used for policy learning, even for very simple optimization problems.  \cite{shao1993linear} observes a similar failure for leave-one-out cross-validation in a high-dimensional linear regression setting.  The key issue in both cases is that when relevant data are {\it scarce}, estimates of uncertain parameters are necessarily imprecise, and omitting even a small amount of data when training a policy dramatically degrades its performance. Hence, the performance of a policy trained \emph{with a portion of the data} on the test set is not indicative of the performance of the policy trained \emph{with all the data} on future unseen data.  We elucidate this phenomenon with a stylized example in \cref{sec:MotivatingExample} below.

Worse, this phenomenon is not merely an intellectual curiosity.  Optimization problems plagued by numerous low-precision estimates are quite common in modern, large-scale operations.  For example, optimization models for personalized pricing necessarily include parameters for each distinct customer type, and these parameters can be estimated only imprecisely since relevant data for each type are limited.  Similar issues appear in large-scale supply-chain design, promotion optimization, and dispatching emergency response services; see \cref{sec:Formulation} for further discussion.

In this paper, we propose a new method for estimating out-of-sample performance without sacrificing data for a test set.  The key idea is to debias the in-sample performance of the policy trained on all the data.  Specifically, we focus on the optimization problem
\begin{equation} 
\label[problem]{eq:obj-problem}
    \bx^{*}\in\argmin_{\bx\in\mathcal{X} \subseteq [0,1]^n}
    \quad
\bmu^\top \bx
\end{equation}
where $\mathcal X$ is a known, potentially non-convex feasible region contained within $[0, 1]^n$, and $\bmu \in \R^n$ is an unknown vector of parameters.  We assume access to a vector $\bm Z$ of noisy, unbiased predictions of $\bmu$ (based on historical data) and are interested in constructing a policy $\bx(\bm Z)$ with good out-of-sample performance $ \bmu^\top\bx(\bm Z)$.  (For clarity, the in-sample performance of $\bx(\bZ)$ is $ \bZ^\top \bx(\bZ)$.)
Note that for many applications of interest, $\bmu^\top \bx^* = O(n)$ as $n\rightarrow\infty$; i.e., the full-information solution grows at least linearly as the dimension grows.  Hence, the unknown out-of-sample performance $\bmu^\top \bx(\bZ)$ must also be at least $O_p(n)$ as $n\rightarrow\infty$.\footnote{
    \blue{\label{note:bigO_note} Following \cite{van2000asymptotic}, we say a sequence of random variables $X_n = O_p(a_n)$ if the sequence $X_n/a_n$ is stochastically bounded, i.e., for every $\epsilon > 0$, there exists finite $M>0$ and finite $N>0$ such that $\mathbb{P} \left\{ X_n/a_n \ge M \right\} < \epsilon, \ \text{for all } n > N$.} }  See \cref{sec:Formulation} for examples.


Despite its simplicity, \cref{eq:obj-problem} subsumes a wide class of optimization problems because $\mathcal X$ can be non-convex and/or discrete. This class includes mixed-binary linear optimization problems such as facility location, network design, and promotion maximization.  By transforming decision variables, even some non-linear optimization problems such as personalized pricing can be rewritten as \cref{eq:obj-problem}; see \cref{sec:Formulation}.  In this sense, \cref{eq:obj-problem} is fairly general.

Our estimator applies to classes of affine plug-in policies which are formally defined in \cref{sec:Formulation}. 
Loosely, affine plug-in policies are those obtained by solving \cref{eq:obj-problem} after ``plugging-in" some estimator \minoredit{$r_j(Z_j)$} in place of \minoredit{$\mu_j$}, and $\minoredit{r_j(Z_j)}$ depends affinely on \minoredit{$Z_j$}.  
Many policies used in practice and previously studied in the literature can be viewed as elements of an affine plug-in policy class including Sample Average Approximation (SAA),
estimate-then-optimize policies based on regression,
 the Bayes-Inspired policies of \cite{gupta2017small}, and the SPO+ policy of \cite{AdamSPO}.  
Thus, our estimator provides a theoretically rigorous approach to assessing the quality of optimization policies based on many modern machine learning techniques.

We debias $\bm Z^\top\bx(\bZ)$ by exploiting the structure of \cref{eq:obj-problem} with the plug-in $\br(\bZ)$. Specifically, by leveraging this problem's sensitivity analysis, we approximately compute the gradient of its objective value with respect to the variance of $\bZ$, and use the estimated gradient to debias the in-sample performance.  We term this correction the \emph{Variance Gradient Correction} (VGC).  Because our method strongly exploits optimization structure, the VGC is Lipschitz continuous in the plug-in values $\br(\bZ)$.  This continuity is not enjoyed by other techniques such as those in \cite{gupta2017small}.  

Although the VGC's continuity may seem like mere a mathematical nicety, empirical evidence suggests it improves  empirical performance.   Similar empirical phenomena -- where an estimator that varies smoothly in the data often outperforms similar estimators that change discontinuously -- are rife in machine learning.  Compare $k$-nn regression with Gaussian kernel smoothing \citep{friedman2001elements}, CART trees with bagged trees \citep{breiman1996bagging}, or best subset-regression with lasso regression \citep{hastie2017extended}.  Theoretically, we exploit this smoothness heavily to establish bounds that hold uniformly over the policy class.

Specifically, we show that, when $\bZ$ is approximately Gaussian, the bias of our estimator for out-of-sample performance is $\tilde O(h)$ as $h \rightarrow 0$, where $h$ is a user-defined parameter that controls the accuracy of our gradient estimate (\cref{thm:equiv-in-sample}).  Characterizing the variance is more delicate.  We introduce the concept of \emph{Average Solution Instability}, and prove that if the instability of the
policy vanishes at rate $O(n^{-\alpha})$ for $\alpha \geq 0$, then the variance of our estimator is roughly $O(n^{3-\alpha}/h)$.  Collectively, these results suggest interpreting $h$  as a parameter controlling the bias-variance tradeoff of our estimator.  Moreover, when $\alpha > 1$, the variance of our estimator is $o(n^2)$.  Since, as mentioned, the unknown out-of-sample performance often grows at least linearly in $n$, i.e., $\bmu^\top \bx(\bZ) = O_p(n)$, our variance bound shows that when $\alpha > 1$ and $n$ is large, the stochastic fluctuations of our estimator are negligible relative to the out-of-sample performance.  In other words, our estimator is quite accurate in these settings.

Our notion of Average Solution Instability is formally defined in \cref{sec:Variance}.  Loosely, it measures 
 the expected change in the $j^{\text{th}}$ component of the policy after replacing the $k^{\text{th}}$ data point with an i.i.d. copy, where $j$ and $k$ are chosen uniformly at random from $\{ 1, \ldots, n\}$. This notion of stability is similar to hypothesis stability \citep{bousquet2001algorithmic}, but, to the best of our knowledge, is distinct.  Moreover, insofar as we expect that a small perturbation of the data is unlikely to have a large change on the solution for most real-world, large-scale optimization problems, we expect Average Solution Instability to be small and our estimator to have low variance.  

We then prove stronger high-probability tail bounds on the error of our estimator for two special classes of ``weakly-coupled" instances of \cref{eq:obj-problem}: \mbox{weakly-coupled-by-variables} and \mbox{weakly-coupled-by-constraints}.
In \cref{sec:Weakly-Coupled-Variables}, we consider problems that are weakly-coupled-by-variables, i.e., problems that decouple into many, disjoint subproblems once a small number of decision variables are fixed.  In \cref{sec:Weakly-Coupled-Constraints} we consider problems that are weakly-coupled-by-constraints, i.e., problems that decouple into many, disjoint subproblems once a small number of constraints are removed.  For each problem class, we go beyond bounding the variance to provide an explicit tail bound on the relative error of our estimator that holds uniformly over the policy class. We show that for problems weakly-coupled-by-variables the relative error scales like $\tilde{O}(C_{\sf PI} \frac{\polylog(1/\epsilon)}{\sqrt[3]{n}})$ where $C_{\sf PI}$ is a constant measuring the complexity of the policy class; see \cref{thm:unif-oos-est-wc-v}. Similarly, we show the relative error for problems weakly-coupled-by-constraints scales like $\tilde{O}(C_{\sf PI} \frac{\polylog(1/\epsilon)}{\sqrt[4]{n}})$, where $C_{\sf PI}$ measures both the complexity of the policy class and number of constraints of the problem; see \cref{thm:WC-Constraints-bound}.  Importantly, since these bounds hold uniformly, our debiased in-sample performance can be used both for policy evaluation and policy learning, even when data are scarce, so long as $n$ (the dimension of the problem) is sufficiently large.  Said differently, our estimator of out-of-sample performance is particularly well-suited to small-data, large-scale optimization. 

Admittedly, weakly-coupled problems as described above do not cover all instances of \cref{eq:obj-problem} \blue{and the appropriateness of modeling $\bZ$ as approximately Gaussian is application specific. \label{comment:Gaussian2}}
Nonetheless, our results and their proofs highly suggest our estimator will have strong performance whenever the underlying optimization problem is well-behaved enough for certain uniform laws of large numbers to pertain.  

Finally, to complement these theoretical results, we perform a numerical case study of dispatching emergency medical services with real data from cardiac arrest incidents in Ontario, Canada.  With respect to policy evaluation and learning, we show that our debiased in-sample performance outperforms both traditional cross-validation methods and the Stein correction of \cite{gupta2017small}.  
In particular, while the bias of cross-validation is non-vanishing as the problem size grows for a fixed amount of data, the bias of our VGC converges to zero. Similarly, while both the Stein correction and our VGC have similar asymptotic performance, the smoothness of VGC empirically leads to lower bias and variance for moderate and sized instances.


\subsection{Our Contributions} We summarize our contributions as follows:

\begin{enumerate}
\item We propose an estimator of out-of-sample performance for \cref{eq:obj-problem} by debiasing \mbox{in-sample} performance through a novel \emph{Variance Gradient Correction} (VGC).  Our VGC applies to a general class of affine plug-in policies that subsumes many policies used in practice.  
Most importantly, unlike cross-validation, VGC does \emph{not}  sacrifice data when training, and, hence, is particularly well-suited to settings where data are scarce. 

\item We prove that under some assumptions on the \mbox{data-generating} process, for general instances of \cref{eq:obj-problem},  the bias of our estimator is at most $\tilde{O}(h)$ as $h \rightarrow 0$, where $h$ is a user-defined parameter.  For policy classes that satisfy a certain Average Solution Instability condition, we also prove that that its variance scales like $o\left(\frac{n^2}{h}\right)$ as $n \rightarrow \infty$.

%

\item We prove stronger results for instances of \cref{eq:obj-problem} in which the feasible region is only weakly-coupled.  When the feasible region is weakly-coupled by variables, we prove that, with probability at least $1-\epsilon$, debiasing in-sample performance with our VGC recovers the true out-of-sample performance up to relative error that is at most  $\tilde{O}\left(C_{\sf PI}\frac{ \log(1/\epsilon)}{n^{1/3}}\right)$ as $n \rightarrow \infty$, uniformly over the policy class, where $C_{\sf PI}$ is a constant that measures the complexity of the plug-in policy class (\cref{thm:unif-oos-est-wc-v}). 
Similarly, for certain linear optimization problems that are weakly coupled by constraints, we prove that, with probability at least $1-\epsilon$, debiasing in-sample performance with VGC estimates the true out-of-sample performance uniformly over the policy class with relative error that is at most $\tilde{O}\left(C_{\sf PI}\frac{ \sqrt{\log(1/\epsilon)}}{n^{1/4}} \right)$ where $C_{\sf PI}$ is a constant measuring the complexity of the plug-in policy class and the number of of constraints of the problem. (\cref{thm:WC-Constraints-bound}).  \blue{We stress that since both these bounds hold uniformly, our debiased in-sample performance can not only be used for policy evaluation, but also policy learning, even when data are scarce, so long as $n$ (the size of the problem) is sufficiently large. }

\item Finally, we present a numerical case study based on real data from dispatching emergency response services to highlight the strengths and weaknesses of our approach relative to cross-validation and the Stein correction of \cite{gupta2017small}.  Overall, we find that since our VGC exploits the optimization structure of \cref{eq:obj-problem}, it outperforms the benchmarks when the number of uncertain parameters is sufficiently large. \blue{\label{note:cv-vs-vgc-adv}Additionally, in settings where the signal to noise ratio is low, VGC more effectively balances the bias-variance trade-off than cross-validation which can be quite sensitive to the number of folds used.}
\end{enumerate}




\subsection{A Motivating Example: Poor Performance of Cross-Validation with Limited Data}
\label{sec:MotivatingExample}
Before proceeding, we present an example that highlights the shortcomings of \mbox{cross-validation} and the benefits of our method when data are limited.   Consider a special case of \cref{eq:obj-problem}
\begin{equation} \label[problem]{eq:ToyProblem}
	\max_{\bx \in \left\{0,1\right\}^n} \ \sum_{j=1}^n \mu_j x_j
\end{equation}
where the true parameters $\bmu \in \{-1,1\}^n$ are unknown, but we observe $S$ samples $\bm Y_1,\dots,\bm Y_S$ where $\bm Y_i \in {\R}^{n}$ and $\bm Y_i \sim \mathcal{N}(\bmu, 2\bm{I})$ for all $i$ and $\bm{I}$ is the identity matrix.  
A standard data-driven policy in this setting is Sample Average Approximation (SAA), also called empirical risk minimization, which prescribes the policy
\begin{align*}
	\bm x^{\sf SAA}(\bm Z) \in \argmax_{\bx \in \left\{0,1\right\}^n} \ \sum_{j=1}^n Z_j x_j
\qquad 
\text{ where } 
\qquad 
Z_j = \frac{1}{S}\sum_{i=1}^S Y_{ij}.
\end{align*}
The key question, of course, is ``What is SAA's out-of-sample behavior $\bmu^\top \bx^{\sf SAA}(\bm Z)$?"

\blue{\label{note:respons-toy-example}To study this question, the left panel of \cref{tab:MotivatingExample} shows different estimators for the out-of-sample performance of SAA
$\Eb{ \bmu^\top \bx^{\sf SAA}(\bm Z)}$ when $S=3$ in \cref{eq:ToyProblem}. The right panel shows the expected relative error (with respect to the oracle) of these estimators as the number of samples $S$ grows. To account for the noise level of the samples, we plot the estimation error with respect to the signal-to-noise ratio (SNR) of $Z_j$\footnote{Following \cite{hastie2020best}, we define 
 $SNR 
 = 
 Var(\mu_{\pi})/Var(Z_{\pi}) 
 = 
  \frac{S}{2n} \sum_{j=1}^{n} \left(\mu_j - \frac{1}{n}\sum_{i=1}^{n} \mu_i \right)^2$ where 
$\pi$ is an index drawn uniformly random from 1 to $n$.
 }.  For reference, \cite{hastie2020best} argues that SNR greater than $1$ is ``rare" when working with ``noisy, observational data," and an SNR of $0.25$ is more ``typical."}


\blue{The first row of \cref{tab:MotivatingExample} presents the in-sample performance, i.e, the objective of the SAA problem. As expected, we see in-sample performance significantly over-estimates the out-of-sample performance. The right panel of \cref{tab:MotivatingExample} suggests this effect persists across SNRs and the relative error is at least 23\% for SNRs less than 2.} 

The second row of the left panel of \cref{tab:MotivatingExample} shows the leave-one-out cross-validation error, which aims to correct the over-optimistic bias and computes
\blue{\label{note:typo-3}$\frac{1}{S}\sum_{i=1}^{S} \bm Y_i^\top  \bx^{\sf SAA}(\bZ^{-i})$}, where \mbox{$\bZ^{-i}  = \frac{1}{S-1}\sum_{j \neq i} \bm Y_{j}$}. Cross-validation is also fairly inaccurate, suggesting SAA performs worse than the trivial, non-data-driven policy $\bm x  = \bm 0$, which has an out-of-sample performance of $0$.  \blue{In the right panel, this incorrect implication occurs for SNRs less than about 0.875.}

Why does cross-validation perform so poorly?  \emph{By construction} cross-validation omits some data in training, and hence, does \emph{not} estimate the out-of-sample performance of $\bx^{\sf SAA}(\bm Z)$, but rather, that of $\bx^{\sf SAA}(\bm Z^{-1})$.\footnote{For clarity, this is the same as the performance as $\bx^{\sf SAA}(\bm Z^{-2})$ because the data are i.i.d.}  From the second \blue{column} of \cref{tab:MotivatingExample}, we see the cross-validation estimate \emph{does} nearly match the true (oracle) performance of $\bx^{\sf SAA}(\bm Z^{-1})$.  When data are scarce, sacrificing even a small amount of data in training can dramatically degrade a policy. \blue{As seen in the right panel of \cref{tab:MotivatingExample}, this phenomenon is non-negligible (at least 10\% relative error) for signal-to-noise ratios less than or equal to $1.75$.  Thus, the performance of $\bx^{\sf SAA}(\bm Z^{-1})$ may not always be a good proxy of the performance of $\bx^{\sf SAA}(\bm Z)$.}  

\begin{table}
    
	\begin{minipage}{0.45\linewidth}
		\centering
		\begin{tabular}{c r r} 
        \toprule
        {} & \multicolumn{1}{c}{$\bx^{\sf SAA}(\bm Z)$} & \multicolumn{1}{c}{$\bx^{\sf SAA}(\bm Z^{-i})$}  \\
        \midrule
        In-Sample & $18.36$ & $22.33$ \\
        Cross-Val & $-1.86$ & $-9.98$ \\
        Our Method & $2.95$ & $-1.89$ \\
        Oracle & $2.97$ & $-1.87$ \\
        \bottomrule
        \end{tabular}
	\end{minipage}\hfill
	\begin{minipage}{0.5\linewidth}
		\centering
		\scalebox{.47}{\input{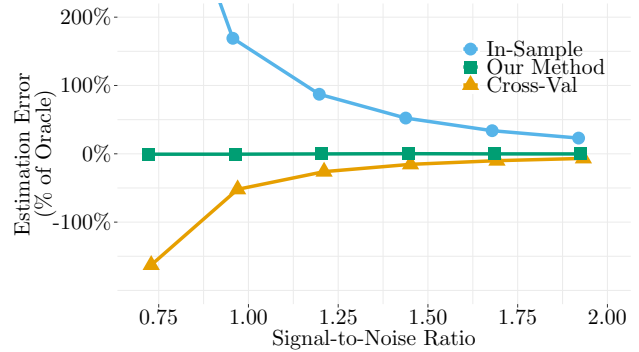}}
	\end{minipage}
    \captionof{figure}{\label{tab:MotivatingExample}\textbf{Expected Estimates of Out-of-Sample Performance by Policy for \cref{eq:ToyProblem}}.  In the left table, we take $n=100$, $S=3$ and $\mu_j = 1$ if $j \leq 14$ and  $\mu_j = -1$ otherwise. We estimate the expected out-of-sample perf. across 1,000,000 simulations.  Std. errors are less than $0.005$. In the right graph, we plot the bias of the estimates with respect to the expected out-of-sample performance as we increase the signal-to-noise ratio. The In-Sample error not shown at 0.75 SNR exceeds $500\%$.}
\end{table}

How then might we resolve the issue?  The third row of the left panel of \cref{tab:MotivatingExample} presents our estimator based on debiasing the in-sample performance of $\bx^{\sf SAA}(\bm Z)$ with our VGC. Our estimate is essentially unbiased (see also \cref{thm:equiv-in-sample} below).  \blue{The right panel of \cref{tab:MotivatingExample} confirms this excellent behavior across a range of signal-to-noise ratios.}  Finally, although this example focuses on the bias of our estimator, our results in \cref{sec:Weakly-Coupled-Problems} are stronger and bound the (random) error of our estimator directly, rather than its expectation.
%

\subsection{Relationship to Prior Work} Cross-validation is the gold-standard  for estimating out-of-sample performance in the large-sample regime with i.i.d. data; see \cite{bousquet2001algorithmic,kearns1999algorithmic} for some fundamental results.  
As discussed above, when estimating the performance of a fixed-policy, these approaches entail sacrificing some data in the training step to set aside for validation, and, hence, may be ill-suited to data-scarce settings.  Similar issues arise in a variety of other sample-splitting methods, including ``honest-trees"  \cite{wager2018estimation} and most forms of doubly-robust estimation \cite{dudik2011doubly}.  
By contrast, our VGC based approach to debiasing the in-sample performance effectively uses all the data when training, making it somewhat better suited to data-scarce settings and small-data, large-scale optimization.  


Our work also contributes to the growing literature on ``optimization-aware" estimation.  These works employ a variety of names including operational statistics \citep{liyanage2005practical}, learning-enabled optimization \citep{deng2018learning}, decision-focused learning \citep{wilder2019melding}, end-to-end learning \citep{wilder2019end} and task-based learning \citep{donti2017task}.  Fundamentally, this area of research seeks estimators that optimize the out-of-sample performance of a policy in a downstream optimization problem rather than the prediction error of the estimate.  Closest to our work is the ``Smart `Predict then Optimize'" framework studied in \cite{AdamSPO} and \cite{elmachtoub2020decision}.  These works also study \cref{eq:obj-problem}, but in a slightly different data setting, and propose policy selection methods for affine and tree-based policies, respectively. Also related is \cite{ito18a} which develops an unbiased estimate of the sample average approximation (SAA) policy for \cref{eq:obj-problem}, but does not consider higher level moments, policy evaluation for other policies, or policy learning.

\minoredit{\label{note:generalization} Recently, \cite{el2019generalization,hu2022fast} have sought to establish generalization guarantees for data-driven policies for plug-in policies for \cref{eq:obj-problem}, i.e., bounds on the difference between in-sample performance and out-of-sample performance that hold uniformly over the policy class.  Both works prove generalization guarantees that vanish in the large-sample regime (when the amount of data grows large).  We similarly bound the difference between our \emph{debiased} in-sample performance and out-of-sample performance, uniformly over a policy class. However, unlike the previous works, our bounds are specifically constructed to vanish relative to the unknown out-of-sample performance in the small-data, large-scale regime.  When applied to the types of policies studied in \cite{el2019generalization,hu2022fast}, our debiased in-sample performance equals the ordinary in-sample performance (see discussion of ``Policy Classes that Do Not Depend on $\bZ$" in \cref{sec:EstimatingInSampleBias}). Hence, our results can be reinterpreted as generalization guarantees for these classes, showing that generalization error vanishes relative to the out-of-sample performance as the problem size grows.  In this sense, our work complements the large-sample analysis of \cite{el2019generalization,hu2022fast} with new, small-data, large-scale analysis.}  



More generally, there has been somewhat less work in the small-data, large-scale regime, most notably \cite{gupta2019shrunk} and \cite{gupta2017small}.  \blue{\label{note:typo-4} Of these, \cite{gupta2017small}, henceforth \citeGR,} is most closely related to our work. 
Loosely,~\citeGR~ study a class of weakly coupled linear optimization problems and propose an estimator of the out-of-sample performance based on Stein's Lemma.  By leveraging a careful duality argument, the authors prove that the estimation error of their procedure vanishes in both the large-sample and small-data, large-scale~regime.  

Our work differs in two important respects:  First, our estimator applies to a more general class of problems and more general policy classes.  Indeed, we focus on \cref{eq:obj-problem} with specialized results for weakly-coupled instances.  Our weakly-coupled by constraints variant in \cref{sec:Weakly-Coupled-Constraints} mirrors the setting of \citeGR, and our weakly-coupled by variables variant in \cref{sec:Weakly-Coupled-Variables} is more general,  allowing us to model, for example, discrete optimization problems.  Moreover, our affine plug-in policy class significantly generalizes the ``Bayes-Inspired" policy class of \citeGR~by incorporating covariate information.  

The second important difference from \citeGR ~relates to exploiting optimization structure in \cref{eq:obj-problem}.
\citeGR ~fundamentally relies on Stein's lemma, a result which applies to general functions and does not specifically leverage optimization structure.  By contrast, our method directly leverages the structure of \cref{eq:obj-problem} through its sensitivity analysis and Danskin's theorem.  By leveraging optimization structure, our VGC is, by construction, continuous in the policy class. Evidence from \cref{sec:experiments} suggests this smoothness yields an empirical advantage of our method.

Finally, our work also contributes to a growing literature on debiasing estimates in high-dimensional statistics, most notably for LASSO regression \citep{javanmard2018debiasing,zhang2014confidence} and $M$-estimation \citep{javanmard2014confidence}.  Like these works, VGC involves estimating a gradient of the underlying system and using this gradient information to form a correction.  Unlike these works, however, our gradient estimation strongly leverages ideas from sensitivity analysis in optimization.  Moreover, the proofs of our performance guarantees involve substantively different mathematical techniques.

\section{Model}
\label{sec:Formulation}

As mentioned, our focus is on data-driven instances of \cref{eq:obj-problem} where the feasible region $\mathcal X$ is known, but the parameters $\bmu$ are unknown.  
Despite its simplicity, several applications can be modeled in this form after a suitable transformation of variables.  
\begin{examp}[Promotion Optimization]
{\rm
Promotion optimization is an increasingly well-studied application area \citep{cohen2017impact,baardman2019scheduling}.  Our formulation mirrors a formulation from the ride-sharing company Lyft around incentive allocation \citep{schmoys_wang_2019}, but also resembles the online advertising portfolio optimization problem  (\cite{rusmevichientong2006adaptive,panilarge},\citeGR).

The decision-maker (platform) seeks to allocate $J$ different types of coupons (promotions) to $K$ different customer (passenger) types.  Coupons are costly, and there is a finite budget $C$ available.  Let $\mu_{jk}$ be the reward (induced spending) and $c_{jk}$ be the cost of assigning coupon type $j$ to customer type $k$.  Using $x_{jk}$ to denote the fraction of customers of type $k$ who receive coupons of type $j$, we can formulate the following linear optimization problem of the form of \cref{eq:obj-problem}.
$$
    \max_{\bx \geq 0} \left\{ 	\sum_{k=1}^{K}\sum_{j = 1}^J \mu_{jk}x_{jk}  ~:~ \sum_{j = 1}^J x_{jk} \leq 1 \quad \text{for each } k=1,\dots,K,~~
	\sum_{k=1}^{K}\sum_{j = 1}^J c_{jk}x_{jk}\le C \right\}~.
$$
In typical instances, the cost $c_{jk}$ are likely known (a ``\$10 off" coupon costs $\$10$), whereas the reward $\mu_{jk}$ must be estimated from historical data.  In settings with many types of coupons and customers, we might further expect that the reward estimates may be imprecise.  
}
\end{examp}
Some reflection suggests many linear optimization problems including shortest-path with uncertain edge costs, or even binary linear optimization problems like multi-choice knapsack with uncertain rewards can be cast as above.  

We next observe that some two-stage linear optimization problems can also be framed as \cref{eq:obj-problem}.

\begin{examp}[Drone-Assisted Emergency Medical Response] \label[examp]{ex:drone-aed}
{\rm
In recent years, emergency response systems have begun utilizing drones as part of their operations, specifically for rapid delivery of automatic electronic defibrillators (AEDs) for out-of-hospital cardiac arrests (OHCA)  \citep{sanfridsson2019drone,cheskes2020improving}.  The intention is that a drone might reach a patient in a remote region before a dispatched ambulance, and (untrained) bystanders can use the AED to assist the patient until the ambulance arrives.  Consequently, researchers have begun studying both how to design a drone-assisted emergency response network (where to locate depots) 
\citep{boutilier2019response} and how to create optimal dispatch rules (to which locations should we allocate a drone and from which depot) \citep{chu2021machine}.  Combining these two problems yields a two-stage 
optimization problem, similar to facility location, aimed at minimizing the response time.  

Namely, let $\mu_{kl}$ be the response time of drone routed from a source $l$ for to a patient location~$k$, $l =1, \ldots, L$ and $k =1, \ldots, K$.  Let $y_l \in \{0, 1\}$ be a binary decision variable encoding if we build a drone depot at location $l$, and let $x_{kl}$ be     a binary decision variable encoding if, after building the network, we should dispatch a drone from location $l$ to patient requests at location $k$.  
We let $x_{k0}$ be the choice not to route a drone (sending only an ambulance) to location $k$ and $\mu_{k0}$ be  the corresponding ambulance travel time.
Suppose we can build at most $B$ depots. Then, we have the following optimization problem.
\begin{align*}
	\min_{\by \in \{0, 1\}^L, \ \bx \in \{0, 1\}^{K \times L} }\quad &	
	\sum_{k=1}^{K}\sum_{l=0}^{L}\mu_{kl}x_{kl},  \\
\text{s.t. }\quad &		\sum_{l=1}^{L}y_{l}\le B, \quad
x_{kl}\le y_{l}, \quad 
\sum_{l=0}^{L}x_{kl} =1, \qquad &&\forall k = 1, \ldots, K,  \ \ l = 1, \ldots, L.
\end{align*}
Insofar as some drone response times are difficult to predict (depending on the weather, local environment, ability of bystanders to locate and use the drone's payload), we expect in typical instances that estimates $\mu_{kl}$ may be imprecise.  
}
\end{examp}

Interestingly, some non-linear problems can be transformed into the form of \cref{eq:obj-problem}. 
\begin{examp}[Personalized Pricing] \label[examp]{ex:PersonalizedPricing}
{\rm Personalized pricing strategies seek to assign a tailored price to each of many customer types reflecting their heterogeneous willingness-to-pay \citep{cohen2018promotion, javanmard2020multiproduct,aouad2019market}.  One simple formulation posits distinct demand models $D_j(p) = m_j \phi_j(p) + b_j$ in each customer segment $j$, for some decreasing function~$\phi_j(p)$.  This  yields the revenue maximization problem 
\begin{align*}
	\max_{\bm{p} \geq \bm 0}\quad & \sum_{j=1}^{n} m_j p_j \phi_j(p_j) + b_j p, 
\end{align*}
where $p_j$ is the price for the $j^\text{th}$ segment.  We can cast this nonlinear objective in the form \cref{eq:obj-problem} by transforming variables,
$$
	\max_{\bm{p} \geq \bm 0, \bx} \left\{  \sum_{j=1}^{n}m_{j}x_{j}+b_{j}p_{j} ~:~
	 x_{j} = p_{j} \phi_j(p_j) \quad\text{for each }j=1,\dots,n  \right\}
$$
where the resulting feasible region is now non-convex.  In typical settings, we expect the parameters $m_j, b_j$ are unknown, and estimated via machine learning methods \citep{aouad2019market}. When there are many customer types, these estimates may be imprecise for rarely occurring types.  
}
\end{examp}

Finally, we mention as an aside that some dynamic programs like the economic lot-sizing problem in inventory management can be cast in the above form through a careful representation of the dynamic program; see \cite{AdamSPO} for details.  

\subsection{Data} \label{sec:data-setup}
Following \citeGR, we do not assume an explicit data generation procedure.  Rather, we assume that as a result of analyzing whatever raw data are available, we obtain noisy, unbiased predictions $\bm Z$ such that $\Eb{ \bm Z} = \bmu$ with known precision $\Eb{ (Z_j - \mu_j)^2} = 1/\nu_j$ for $j=1, \ldots, n$.  These predictions might arise as sample averages as in \cref{sec:MotivatingExample}, or as the outputs of some pre-processing regression procedure.  We further assume that for each $j =1, \ldots, n$, we observe a non-random covariate of feature data $\bW_j \in \R^p$, which may (or may not) be informative for the unknown $\mu_j$.  

We believe this set-up reasonably reflects many applications.  In the case of drone-assisted emergency response (\cref{ex:drone-aed}),  $\bW_j$ encodes features that are predictive of EMS response times such as physical road distance between the patient and the responding ambulance,  time of day,  day of week, and  weather conditions \citep{chu2021machine}, while $Z_{k0}$ may be an average of historical response times to location $k$.

An advantage of modeling $\bZ$ in lieu of the data generation process is that the precisions $\nu_j$ 
implicitly describe the amount of relevant data available for each $\mu_j$.  Let \mbox{$\vmin \equiv \min_j \nu_j$} and \mbox{$\vmax \equiv \max_j \nu_j.$}  Then, loosely speaking, the  large-sample 
regime describes instances where $\nu_{\min}$ is large, i.e., where data are plentiful and we can estimate $\bmu$ easily.  By contrast, the small-data, large-scale regime describes instances in which $n$ is large (large-scale), but there are limited relevant data, and, hence, $\vmax$ is small.  

To simplify our exposition, we will also assume:
\begin{assm}[Independent Gaussian Corruptions] \label[assm]{asn:Gaussian}
For each $j=1, \ldots, n$, $Z_j$ has Gaussian distribution with \mbox{$Z_j \sim \mathcal N(\mu_j, 1/\nu_j)$} where $\nu_j$ is the \emph{known} precision of $Z_j$.  Moreover, $Z_1, \dots, Z_n$ are independent.
\end{assm}
\Cref{asn:Gaussian} is common.  \citeGR~ employ a similar assumption; \cite{javanmard2018debiasing} strongly leverages a Gaussian design assumption when debiasing lasso estimates, and \cite{ito18a} also assumes Gaussian errors in their debiasing technique.  In each case, the idea is if a technique enjoys \emph{provably} good performance under Gaussian corruptions, it will likely have good \emph{practical} performance when data are approximately Gaussian.  
Indeed, if $\bZ$ is obtained by maximum likelihood estimation, ordinary linear regression, simple averaging as in \cref{sec:MotivatingExample}, or Gaussian process regression, then the resulting estimates will be approximately Gaussian.
We adopt the same perspective in our work.  

Note, the independence assumption in \cref{asn:Gaussian} is without loss of generality as illustrated in the following example.
\begin{examp}[Correlated Predictions]
{\rm 
Suppose we are given an instance of \cref{eq:obj-problem} and predictions $\bm Z \sim \mathcal N(\bmu, \bm \Sigma)$ where $\bm \Sigma$ is a \blue{\label{note:typo-7}\emph{known}}, positive semidefinite matrix.  Consider a Cholesky decomposition $\bm \Sigma = \bm L \bm L^\top$ and the transformed predictions $\overline{\bZ} \equiv \bm L^{-1} \bZ$.  
Notice $\overline{\bZ} \sim \mathcal N(\overline{\bmu}, \bm I)$ where $\overline{\bmu} \equiv \bm L^{-1} \bmu$.  We then recast \cref{eq:obj-problem} as the equivalent problem
\begin{align*}
\min_{\overline{\bx} \in \overline{\mathcal X}} \quad &\overline{\bmu}^\top \overline{\bx},
\quad \text{ where } \quad \overline{\mathcal X} \equiv \{ \bm L^\top \bx \; : \; \bx \in \mathcal X \}.
\end{align*}
Our new problem is of the required form with transformed predictions independent across $j$.
}
\end{examp}

Most importantly, \cref{asn:Gaussian} is \emph{not} crucial to many of our results. \blue{\label{note:gaussian-note} Violating the Gaussian assumption only affects the bias of our estimator (see \cref{thm:equiv-in-sample}).  \blue{Our analysis bounding the variance and tails of the stochastic errors utilize empirical process theory, and can easily be adapted for non-Gaussian corruptions.}  Moreover, although the bias of our estimator is non-negligible when $\bZ$ is non-Gaussian, we can bound this bias in terms of the Wasserstein distance between $\bZ$ and a multivariate Gaussian, suggesting our method has good performance as long as corruptions are \emph{approximately} Gaussian and this distance is small (see \cref{lem:NonGaussian} in Appendix for the bound).}  

Finally, similar results also hold when $\bm \nu$ is, itself, estimated noisily with the addition of a small bias term related to its estimate's accuracy (see \cref{lem:noisy-precision} in Appendix).

\subsection{Affine Plug-in Policy Classes}\label{sec:plug-in-policy}
A data-driven policy for \cref{eq:obj-problem} is a mapping  $\bm Z \mapsto \bx(\bZ) \in \mathcal X$ that determines a feasible decision $\bx(\bZ)$  from the observed data $\bZ$.  We focus on classes of affine plug-in policies.  Intuitively, 
a plug-in policy first proxies the unknown $\bmu$ by some estimate, $\br(\bZ)$, and then solves \cref{eq:obj-problem} after ``plugging-in" this estimate for $\bmu$.  

\begin{defn}[Affine Plug-in Policy Classes]\label{eq:plug-in-policy} \label[defn]{def:AffinePlugInPolicy}
For $j =1, \ldots, n$, let \mbox{$r_j(z, \btheta) = a_j(\btheta) z + b_j(\btheta)$}
be an affine function of $z$ where $a_j(\btheta)$ and $b_j(\btheta)$ are arbitrary functions of the parameter $\bm \theta \in \Theta$.  Let $\br(\bZ, \btheta) = \left(r_1(Z_1,\btheta), r_2(Z_2,\btheta), \dots, r_n(Z_n,\btheta) \right)^\top \in \R^n$.  The plug-in policy with respect to $\br(\cdot, \btheta)$ is given by 
\begin{equation}\label[problem]{eq:DefXPolicy}
        \bx \left(\bZ, \btheta\right)
        \ \in \ 
        \arg\min_{\bx \in \mathcal{X}}
        \  
        \br(\bZ, \btheta)^\top \bx,
\end{equation}
where ties are broken arbitrarily.  Furthermore, we let 
\(
\mathcal{X}_{\Theta}(\bZ) \equiv
 \left\{  \bx(\bZ, \btheta) \in \mathcal{X} : \btheta \in \Theta \right\} \subseteq \mathcal X
\)
denote the corresponding class of plug-in policies over $\Theta$.  
\end{defn}

When $\btheta$ is fixed and clear from context, we suppress its dependence, writing $\bx(\bZ)$ and $\br(\bZ)$. \blue{Moreover, for a fixed $\btheta$, \blue{\label{note:typo-5} $r_j(Z_j, \btheta)$} only depends on the data (linearly) through the $j^\text{th}$ component. 
} 

\blue{Plug-in policies are attractive in data-driven optimization because computing $\bx(\bZ, \btheta)$ involves solving a problem of the same form as \cref{eq:obj-problem} \label{computational}.}  Thus, if a specialized algorithm exists for solving \cref{eq:obj-problem} -- \blue{e.g., as with many network optimization problems} -- the same algorithm can be used to compute the policy.  This property does not necessarily hold for other classes of policies such as regularization based policies (\citeGR). 

\blue{Moreover, many policies used in practice are of the form $\bx(\bZ, \hat{\btheta}(\bZ))$ for some $\hat{\btheta}(\bZ)$.  (See examples below.)  
Such policies are \emph{not} affine plug-in policies; $r_j(Z_j, \hat{\btheta}(\bZ))$ may depend nonlinearly on all the data $\bZ$.  Nonetheless, our analysis will bound the error of our estimator applied to such policies.  Namely, in \cref{sec:Weakly-Coupled-Problems}, we provide error bounds on our estimator that hold uniformly over $\mathcal X_\Theta(\bZ)$.  Since these bounds hold uniformly, such bounds also hold for all policies of the form $\bx(\bZ, \hat{\btheta}(\bZ)).$} 

\blue{\label{note:opt-affine-policy} 
For clarity, we make no claim about the optimality of affine plug-in policies for \cref{eq:obj-problem}; for a particular application, there may exist non-affine policies with superior performance. Our focus on affine plug-ins is motivated by their ubiquity and computational tractability.}

\blue{We next present examples:}
\begin{itemize}[leftmargin=*]
\item \textbf{Sample Average Approximation (SAA).}
The Sample Average Approximation (SAA) is a canonical data-driven policy for Problem~\eqref{eq:obj-problem}.  It is defined by 
\begin{equation} \label{ex:SAAPolicy}
    \bx^{\sf SAA}(\bZ)  
    \ \in \  
    \arg\min_{\bx \in \mathcal{X}} 
	\ 
	\bZ^\top \bx.
\end{equation}
SAA is thus an affine plug-in policy where the function $r_j(z, \btheta) = z$.

\item \textbf{Plug-ins for Regression Models.}
Consider the linear model \blue{$r_j\left(\bZ, \btheta \right) = \btheta^\top \bW_j$}, \blue{which does not depend on $\bZ$}, and the affine plug-in~policy 
\begin{equation} \label{ex:LMPolicy}
\blue{\bx^{\sf LM}(\bZ, \btheta) \ \in \ \arg \min_{\bx \in \mathcal X} \  \sum_{j=1}^n \bW_j^\top \btheta \cdot x_j.}
\end{equation}

\blue{\label{note:SPO}
As mentioned, many policies in the literature are
of the form 
$\bx^{\sf LM}(\bZ, \hat{\btheta}(\bZ))$ for a particular $\hat{\btheta}(\bZ)$.   
For example, letting $\btheta^{\sf OLS}(\bZ) \in \argmin_{\btheta} \sum_{j=1}^n (Z_j - \btheta^\top \bW_j)^2$ be the ordinary least-squares fit yields the estimate-then-optimize policy $\bx^{\sf LM}(\bZ, \btheta^{\sf OLS}(\bZ))$.}  
Similarly, \blue{by appropriately padding the covariate with zeros}, 
we can write the ``optimization-aware" SPO and SPO+ methods of \cite{AdamSPO} over linear hypothesis classes in the form $\bx^{\sf LM}(\bZ, \btheta^{\sf SPO}(\bZ))$ and  $\btheta^{\sf SPO+}(\bZ))$ where $\btheta^{\sf SPO}(\bZ)$ and $\btheta^{\sf SPO+}(\bZ)$ are obtained by minimizing the so-called SPO and SPO+ losses, respectively.  
Other methods, e.g., \citep{wilder2019melding}, can be rewritten similarly. As mentioned, our analysis will bound the error when debiasing these polices as well.  

{\blockedit Of course, we are not limited to a linear model for $r_j(z, \btheta)$.  We could alternatively use a nonlinear specification $r_j(z, \btheta) = f(\bW_j, \btheta)$ for some, given, nonlinear regression $f$ with parameters $\btheta$.  This specification of $\br(\bZ, \btheta)$ \emph{still} gives rise to a class of affine plug-in policies.  Again, many policies in the literature, including estimate-then-optimize policies and SPO+ over non-linear hypothesis classes can be written in the form $\bx(\bZ, \btheta(\bZ))$ for some particular mapping $\btheta(\bZ) \in \Theta$.}



\item \textbf{Mixed-Effects Policies.} 
When $\bW_j$ is not informative for $\mu_j$, plug-ins for \blue{regression} models perform poorly \blue{because no choice of $\btheta$ yields a good estimate of $\bmu$}.  By contrast, if $\nu_{\min}$ is large, SAA performs quite well.  Mixed-effects policies interpolate between these choices. Define
    \begin{equation} \label{eq:DefMixedEffectPolicy}
        \bx^{\sf ME}\big (\bZ, (\tau,\bm{\beta}) \big)
        \in
        \arg\min_{\bx\in\mathcal{X}} 
        \
        \sum_{j=1}^{n}
        \left( 
            \frac{\nu_j}{\nu_j + \tau} Z_j ~+~ 
            \frac{\tau}{\nu_j + \tau}
                \bW_{j}^{\top}\bm{\beta}
        \right)
        x_j,
    \end{equation}
\blue{where we have focused on a linear model for simplicity} and made the dependence on $\btheta = (\tau, \bm \beta)$ explicit for clarity.  Mixed-effects policies are strongly motivated by Bayesian analysis \citep{gelman2014bayesian}.  These policies generalize the Bayes-Inspired policy class considered in \citeGR.  
Again, we observe that $\bx^{\sf ME}\big( \bZ, (\tau, \bm \beta) \big)$ is an affine-plug in policy.  \blue{Moreover, we can also consider shrinking towards a nonlinear regression model as in \citep{NEURIPS2019_48f7d304}.}  
\end{itemize}

\minoredit{ Note in \cref{def:AffinePlugInPolicy}, we require that $r_j(Z_j)$ depends only on $Z_j$, not on $Z_k$ for $k \neq j$. We exploit this structure in the design and analysis of our debiasing technique.  However, this requirement precludes certain types of plug-ins, e.g., those based on linear smoothers \citep{buja1989linear} including local polynomial regression and k-nearest neighbors.  Extension of our method to these settings remains an interesting open research question.}




\vspace{10pt}

\blue{\label{note:choice-affine-policy} 
The choice of which affine plug-in policy class to use is largely application dependent.  Our bounds in \cref{sec:Weakly-Coupled-Problems} provide some preliminary guidance, suggesting a tradeoff between the expressiveness of the policy class and the error of our estimator.}  
\section{\DanskinTitle} 
\label{sec:EstimatingInSampleBias}
We make the following assumption on problem parameters for the remainder of the paper: 

\begin{assm}[Assumptions on Parameters] \label[assm]{asn:Parameters}
There exists a constant $C_\mu > 1$ such that $\| \bmu\|_\infty \leq C_\mu$, and constants $0 < \vmin < 1 < \vmax < \infty$ such that $\vmin \leq \nu_j \leq \vmax$ for all $j$.  Moreover, we assume that $n \geq 3$.
\end{assm}

\blue{\label{note:assum_3_1} The assumptions for $C_\mu$ and $\vmin, \vmax$ are without loss of generality.   
These assumptions  and the assumption on $n$ 
allow us to simplify the presentation of some results by absorbing lower order terms into leading constants.}

The in-sample performance of a policy $\bx(\bZ)$ is $\bZ^\top \bx(\bZ)$.  Let $\bxi = \bZ - \bmu$.  We call the difference between in-sample and out-of-sample performance, corresponding to \mbox{$(\bZ - \bmu)^\top \bx(\bZ) \,=\, \bxi^\top\bx(\bZ)$}, the \emph{in-sample optimism}.  The expected in-sample optimism $\Eb{\bxi^\top \bx(\bZ)}$ is the in-sample bias.  

Our method  estimates the in-sample optimism of an affine, plug-in policy 
$\bx(\bZ, \btheta)$.  To this end, denote the plug-in objective value by
\begin{equation}
\Obj(\bZ, \btheta)
 \ \equiv \  
r(\bZ, \btheta)^\top \bx(\bZ, \btheta)
\ = \ 
\min_{\bx \in \mathcal X} \ r(\bZ, \btheta)^\top \bx 
.
\end{equation}\label{eq:plug-in-obj-val}
Because it is the minimum of linear functions, $\bZ \mapsto V(\bZ, \btheta)$ is always concave.   We then estimate the in-sample optimism by the \emph{Variance Gradient Correction} (VGC) defined by
\blue{
\begin{align} \label{eq:Danskin-corr}
D(\bZ, (\btheta, h)) & \ \equiv  \ \sum_{j=1}^n D_j(\bZ, (\btheta, h)), 
\end{align}
where for $j = 1, 2, \ldots, n$,
\begin{align} \label{eq:rand-finite-diff-D}
D_j(\bZ, (\btheta, h)) &\ \equiv  \ 
\begin{cases} \Eb{ \left. \frac{1}{ h \sqrt{\nu_j} a_j(\btheta) }  
\bigg( \Obj(\bZ + \delta_j \be_j) - \Obj(\bZ) \bigg) \right | \bZ }, & \text{ if } a_j \neq 0, 
\\
	0, & \text{ otherwise,}
\end{cases}
\end{align}
}
and
$\delta_1, \ldots, \delta_n$ are independent Gaussian random variables such that  
$\delta_j \sim \mathcal N \left(0,h^{2}+\frac{2h}{\sqrt{\nu_j}} \right)$ for all $j$. 
\blue{To reduce notation, we define $\bar{\Theta} \equiv \Theta \times [h_{\min}, h_{\max}]$ for $0 < h_{\min} \le h_{\max}$ and write $D(\bZ, \btheta)$ with $\btheta \in \bar{\Theta}$. We utilize the defined notation for results that require separating $h$ and $\btheta$. 
}

The VGC is defined as a (conditional) expectation over the auxiliary random variables $\delta_j$.  In practice, we can approximate this expectation to arbitrary precision by simulating $\delta_j$ and averaging; see \cref{app:ImplementationDetails} for more efficient implementations.

Given the VGC, we estimate the out-of-sample performance by 
\begin{equation}\label{eq:OOSEstimator}
\bmu^\top \bx(\bZ, \btheta) \approx \bZ^\top \bx(\btheta) - D(\bZ, \btheta).
\end{equation}

In \cref{sec:IntuitionDanskin}, we motivate the VGC. We then establish some of its key properties, namely that it is almost an unbiased estimator for the in-sample optimism, its variance is often vanishing as $n\rightarrow\infty$, and it is smooth in the policy class.

\vspace{10pt}
\minoredit{\paragraph{Policy Classes that Do Not Depend on $\bZ$: } \label{note:NonDataDrivenPolicies}
Recall our plug-ins for linear regression models example from \cref{sec:plug-in-policy}.  From \cref{ex:LMPolicy}, we can see that $D(\bZ, (\btheta, h)) = 0$ uniformly over the class.  Said differently, the in-sample performance is already an unbiased estimator of out-of-sample performance.  This happy coincidence occurs whenever the plug-in function $r_j(z, \btheta)$ does not depend on $z$ for each $j$ and $\btheta$.  That said, we stress that although in-sample performance is an unbiased estimator, it is not immediately clear what the variance of this estimator is.  We discuss this further in \cref{sec:Variance,sec:Weakly-Coupled-Problems}.  Moreover, when $r_j(z, \btheta)$ \emph{does} depend on $z$, the VGC is typically non-zero.}

\subsection{Motivating the \DanskinTitle~(VGC)}
\label{sec:IntuitionDanskin}
Throughout this section, $\btheta$ is fixed so we drop it from the notation.  Our heuristic derivation of $D(\bZ)$ proceeds in three steps.  

\vskip 5pt \noindent \textbf{Step 1:  Re-expressing the In-Sample Optimism via Danskin's Theorem.}  
Fix some $j$.  If $a_j = 0$, then from the plug-in policy problem (\cref{eq:DefXPolicy}) we see that $\bx(\bZ)$ is independent of $Z_j$ and the corresponding term in the in-sample bias is mean-zero, i.e., $\Eb{ \xi_j \bx(\bZ) } = 0 = D_j(\bZ)$.  In other words, we do not correct such terms.

When $a_j \neq 0$, consider the function 
\begin{equation}
	\lambda \mapsto \Obj(\bZ + \lambda \xi_j \be_j).
\end{equation}\label{eq:aug-obj}
This function is an example of a parametric optimization problem.  Danskin's Theorem \citep[Section B.5]{bertsekas1997nonlinear} characterizes its derivative with respect to $\lambda$.\footnote{\blue{See \cref{thm:DanskinsTheorem} in \cref{sec:DanskinStatement} for a statement of Danskin's Theorem. \label{PointerToDanskin}}} 
Specifically, for any $\lambda \in \R$ such that $\bx(\bZ + \lambda \xi_j \be_j)$ is the unique optimizer to \cref{eq:DefXPolicy}, we have
\[
\frac{\partial}{\partial \lambda} \Obj(\bZ + \lambda \xi_j \be_j) \ = \ a_j  \xi_j x_j(\bZ + \lambda \xi_j \be_j).
\]
When $\bx(\bZ + \lambda \xi_j \be_j)$ is not the unique optimizer, $a_j \xi_j x_j(\bZ + \lambda \xi_j \be_j)$ is a subgradient, see \cref{fig:DanskinIntuition} for~intuition.

Notice that $\frac{\partial}{\partial \lambda} \Obj(\bZ + \lambda \xi_j \be_j)$ is the derivative of the plug-in value when we make the $j^\text{th}$ component of $\bZ$ more variable, i.e., variance increases by a factor \blue{\label{note:typo-6}$(1 + \partial \lambda)^2$, where $\partial \lambda$ represents an infinitely small perturbation to $\lambda$}. This observation motivates our nomenclature ``Variance Gradient Correction."

\begin{figure} 
\begin{minipage}[c]{0.5\textwidth}
\begin{center}
\includegraphics[width=.7\textwidth]{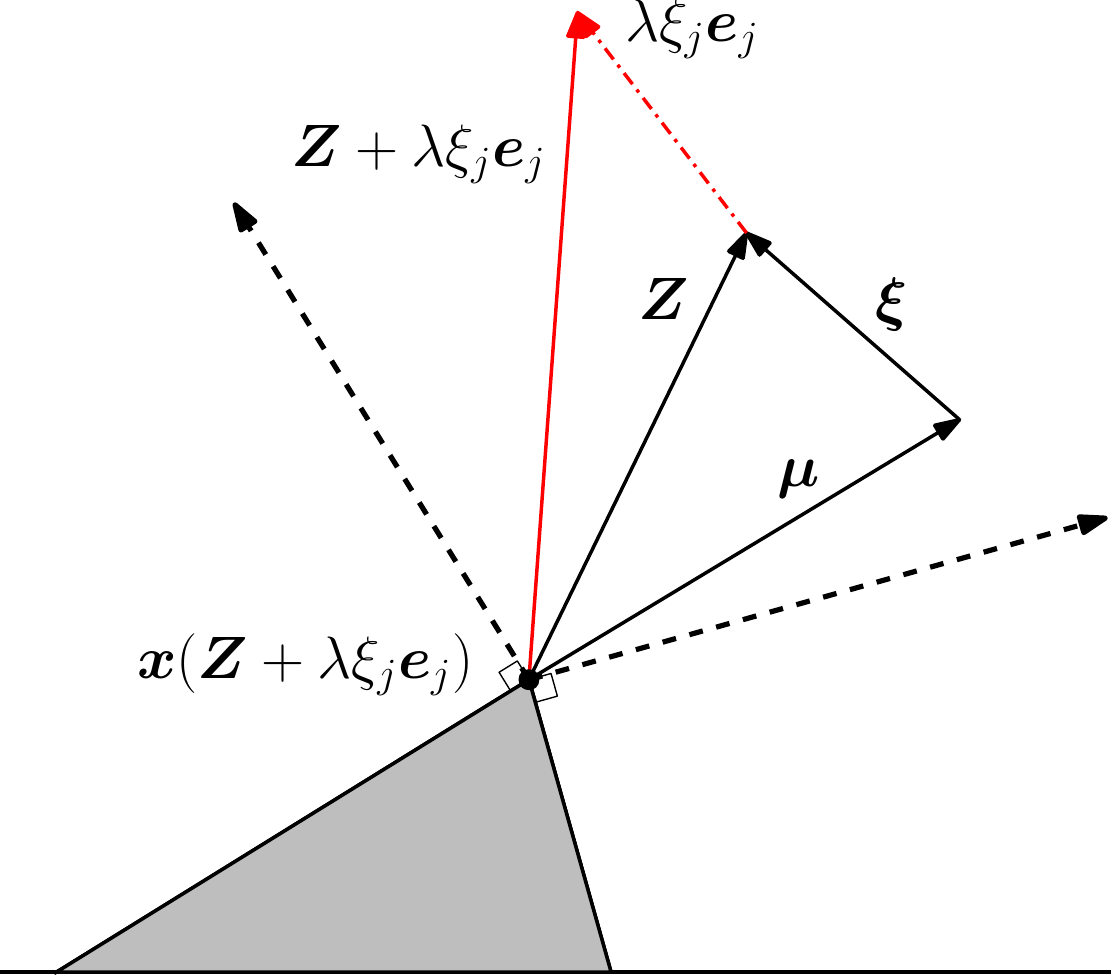}
\end{center}
\end{minipage}\hfill
\begin{minipage}[c]{0.5\textwidth}
\caption{   When $\mathcal X$ is polyhedral, $\bx_j(\bZ + \lambda \xi_j \be_j)$ must occur at a vertex if it is unique.  Hence, small perturbations to $\lambda$ do not change the solution (see figure), and the derivative of \mbox{$\Obj(\bZ  + \lambda \xi_j \be_j)$} is entirely determined by the derivative of $\br (\bZ + \lambda \xi_j \be_j)$. Similar intuition holds for non-polyhedral $\mathcal X$. 
\label{fig:DanskinIntuition} 
}
\end{minipage}        
\end{figure}

Evaluating the above derivative at $\lambda = 0$, dividing by $a_j$, and summing over $j$ such that $a_j \neq 0$,
allows us to re-express the in-sample bias whenever $\bx(\bZ)$ is unique as
\[
\sum_{j=1}^n \xi_j x_j(\bZ)  \ = \ \sum_{j : a_j \neq 0} \frac{1}{ a_j } \cdot  \left. \frac{\partial}{\partial \lambda} \Obj(\bZ + \lambda \xi_j \be_j)\right|_{\lambda = 0}.
\]

Unfortunately, it is not clear how to evaluate these derivatives from the data.  This leads to the second step in our derivation.

\vskip 5pt\noindent \textbf{Step 2: Approximating the Derivative via Randomized Finite Differencing.}
As a first attempt, we approximate the above derivatives with first-order, forward finite-differences \blue{\label{note:fin-diff-cite}\citep[Chapter 1]{leveque2007finite}}.   Intuitively, we expect that for a sufficiently small step size $ h >0$,
\begin{equation}\label{eq:FiniteDiffFirstOrderRemainder}
\left. \frac{\partial}{\partial \lambda} \Obj(\bZ + \lambda \xi_j \be_j)\right|_{\lambda = 0} 
\ \ = \ \ 
\frac{1}{h \sqrt{\nu}_j} \biggr( \Obj(\bZ + h \sqrt{\nu_j} \xi_j \be_j) - \Obj(\bZ) \biggr) \  +  \ o_p(1) \ \ \text{ as } h\rightarrow 0,
\end{equation}
which suggests that 
\begin{equation} \label{eq:StraightFiniteDiff}
\sum_{j=1}^n \xi_j x_j(\bZ) \  \ = \  \   \sum_{j : a_j \neq 0 } \left[ \frac{1}{ h \sqrt{\nu_j} a_j }  
\biggr( \Obj(\bZ + h \sqrt{\nu_j} \xi_j \be_j) - \Obj(\bZ) \biggr) \right]  \ + \ o_p(n) \ \ \text{ as } h\rightarrow 0.
\end{equation}

Unfortunately, the right side of \cref{eq:StraightFiniteDiff} is not computable from the data, because we do not observe $\mu_j$, and, hence, do not observe $\xi_j = Z_j - \mu_j$.  

To circumvent this challenge, recall that $\xi_j$ is Gaussian and independent across $j$, and let $\delta_j$ be the independent Gaussian random variables defined in the definition of the \Danskin~(\cref{eq:rand-finite-diff-D}).
A direct computation shows,  $\bZ + h \sqrt{\nu_j} \xi_j \be_j \;\sim_d\; \bZ + \delta_j \be_j$, 
because both $Z+h\sqrt{\nu_j} \xi_j \be_j$ and $Z_j +\delta_j\be_j$ are Gaussians with matching mean and covariances.\footnote{Here, $\sim_d$ denotes  equality in distribution.}
Hence, $\Obj(\bZ+ h \sqrt{\nu_j}  \xi_j \be_j) \sim_d \Obj(\bZ + \delta_j \be_j)$.


Inspired by this relation, we replace the unknown $ \Obj(\bZ + h \sqrt{\nu_j}  \xi_j \be_j)$ by $\Obj(\bZ + \delta_j \be_j)$ in our first-order, finite difference approximation, yielding a \emph{randomized} finite difference:
\begin{equation} \label{eq:RandomizedFiniteDiff}
\underbrace{\sum_{j=1}^n \xi_j x_j(\bxi + \bmu; \btheta)}_{\text{In-Sample Optimism}}  \ \  \approx  \ \ \underbrace{\sum_{j: a_j \neq 0} \left[ \frac{1}{ h \sqrt{\nu_j} a_j(\btheta) }  
\biggr( \Obj(\bmu  +  \bxi + \delta_j \be_j) - \Obj(\bmu  +  \bxi) \biggr) \right]}_{\text{Randomized Finite Difference}}. 
\end{equation}

\vskip 5pt\noindent \textbf{Step 3: De-Randomizing the Correction.}
Finally, in the spirit of Rao-Blackwellization, we then de-randomize this correction by taking conditional expectations over $\bm \delta$. This de-randomization reduces the variability of our estimator and yields the \Danskin~(\cref{eq:rand-finite-diff-D}).

\vspace{10pt}
\paragraph{Higher Order Finite Difference Approximations:} Our heuristic motivation above employs a first-order finite difference approximation, and our theoretical analysis below focuses on this setting for simplicity.  However, it is possible to use higher order approximations, which in turn reduce the bias.  Theoretical analysis of such higher order approximations is tedious, but not substantively different from the first-order case.  Hence, it is omitted.  In our experiments, we use a particular second-order approximation described in \cref{app:ImplementationDetails}.  

\subsection{Bias of \DanskinTitle}
\label{sec:prop-D-correction}
Our first main result shows that one can make the heuristic derivation of the previous section rigorous when quantifying the bias of the VGC.  

\begin{thm}[Bias of the \DanskinTitle]\label{thm:equiv-in-sample} Under \cref{asn:Gaussian,asn:Parameters},  for any $0 < h < 1/e$ and any affine, plug-in policy $\bx(\bZ, \btheta)$,  there exists a constant C (depending on $\nu_{\min}$) such that
\[
	0 \le \mathbb{E}\left[ \sum_{j=1}^n \xi_j x_j(\bmu + \bxi, \btheta) - \sum_{j=1}^n D_j(\bmu + \bxi , \btheta) \right]  \ \le \ C\cdot h n
		\log \left( \frac{1}{h} \right)
\] 
\end{thm}
Recall we expect that, in typical instances,  the full-information performance of \cref{eq:obj-problem} is $O(n)$ as $n\rightarrow \infty$.  Thus, this theorem asserts that as long as $h$ is small, say $h = h_n = o(1)$ as $n \rightarrow \infty$, the bias of VGC is negligible relative to the true out-of-sample performance.  In this sense, VGC is asymptotically unbiased for large $n$.  

The proof of the theorem is in \cref{sec:appendix_FindingBestInClassPolicy} and proceeds similarly to our heuristic derivation but uses the following monotonicity property to precisely quantify the ``little oh" terms.  
\begin{lem}[Monotonicity of Affine Plug-in Policies]\label[lem]{lem:Monotonicity}
For any $\bz$ and $j$, the function \mbox{$t \mapsto \bx(\bz + t \be_j)$} is {non-increasing} if $a_j \geq 0$, and the function is {non-decreasing if $a_j < 0$.} 
\end{lem}
Intuitively, the lemma holds because $\bz \mapsto V(\bz)$ is a concave function (it is the minimum of affine functions).  
In particular, $t \mapsto V(\bz + t)$ is also concave, and by Danskin's Theorem, \mbox{$\frac{d}{dt} V(\bz + t) = a_j x_j(\bz + t)$} whenever $\bx(\bz + t)$ is unique.  Informally, the lemma then follows since the derivative of a concave function is non-increasing.  \Cref{sec:appendix_FindingBestInClassPolicy} provides a formal proof accounting for points of non-differentiability.  

Before proceeding, we remark that \cref{thm:equiv-in-sample} holds with small modifications under mild violations of the independent Gaussian assumption (\cref{asn:Gaussian}).  Specifically, in cases where $\nu_j$ are not known but are estimated, the bias of the VGC constructed with the estimated $\nu_j$ increases by a small term depending on the accuracy of the precisions.  
See \cref{lem:noisy-precision} in the appendix for formal statements and proof.

\subsection{The Variance of the VGC} \label{sec:Variance}
As mentioned in the contributions, the parameter $h$ controls the trade offs between bias and variance in our estimator.  Unfortunately, while \cref{thm:equiv-in-sample} gives a direct analysis of the bias under mild assumptions, a precise analysis of the variance (or tail behavior) of the VGC is more delicate.  In this section we provide a loose, but intuitive bound on the variance of VGC that illustrates the types of problems for which our estimator should perform well.  
The main message of this section is that the VGC concentrates at its expectation so long as the policy $\bx(\bZ)$ is ``stable" in the sense that perturbing one element of $\bZ$ does not cause $\bx(\bZ)$ solution to change too much.  

The main challenge in showing $D(\bZ)$ concentrates at its expectation is that $D(\bZ)$ is a sum of \emph{dependent} random variables $D_j(\bZ)$.  Worse, this dependence subtly hinges on the structure of \cref{eq:obj-problem} and the plug-in policy problem (\cref{eq:DefXPolicy}) and hence is not amenable to techniques based on mixing or bounding the correlations between terms.  We require a different approach. 

As a first step towards analyzing  $D(\bZ)$, we upper bound its variance by a related, fully-randomized estimator.
\begin{lem}[Fully-Randomized VGC] \label[lem]{lem:FullyRandomizedVGC}
Suppose that the solution $\bx(\bZ)$ to \cref{eq:DefXPolicy} is almost surely unique.  
For each $j$ such that $a_j \neq 0$, let 
\begin{equation} \label{eq:DefFullyRandomized}
\DR_j(\bZ) \equiv \frac{\delta_j}{h \sqrt{\nu_j} a_j} x_j(\bZ + \delta_j \tilde U_j \be_j )
\end{equation}
where $\tilde U_j \sim \text{Uniform}[0, 1]$ and $\delta_j$ is defined in \cref{eq:rand-finite-diff-D}.  
Let $\DR(\bZ) = \sum_{j: a_j \neq 0} \DR_j(\bZ)$ denote the \emph{fully-randomized VGC}.  
Then, for any $j$ such that $a_j \neq 0$, 
\[
D_j(\bZ) = \Eb{ \DR_j(\bZ) \mid Z }
\quad \text{ and } 
\quad \Var(D(\bZ)) \leq  \Var(\DR(\bZ)).
\]
\end{lem}
\begin{proof}{\hspace{-12pt}Proof:}
Again, by Danskin's Theorem and the fundamental theorem of calculus, 
\[ 
V(\bZ + \delta_j \be_j) - V(\bZ) \ = \ 
\int_0^{\delta_j} a_j x_j(\bZ + t\be_j) dt
\ = \ \int_0^1 a_j \delta_j \bx_j(\bZ + t \delta_j \be_j ) dt 
\ = \ \Eb{ a_j \delta_j \bx(\bZ + \delta_j \tilde U \be_j ) \mid \bZ, \delta_j}.
\]
Scaling both sides by $\frac{1}{a_j h \sqrt{\nu_j}}$ and taking expectations over $\delta_j$ proves the first statement.  The second then follows from Jensen's inequality.  
\hfill \Halmos \end{proof}

We next propose upper bounding $\Var(\DR(\bZ))$ with the Efron-Stein Inequality.  In particular, let $\overline \bZ$, $\overline{\bm \delta}$ and $\overline { \bm U}$ be i.i.d. copies of $\bZ$, $\bm \delta$  and $\tilde{\bm U}$ respectively, and let $\bZ^k$ denote the vector $\bZ$ but with the $k^\text{th}$ component replaced by $\overline Z_k$.  Define $\bm \delta^k$ and $\tilde{\bm U}^k$ similarly.  Let $\DR(\bZ, \bm \delta, \tilde{\bm U} )$ be the fully-randomized VGC with  dependence on all constituent random variables made explicit.  Then, by the Efron-Stein Inequality 
\begin{subequations} \label{eq:EfronSteinOverall}
\begin{align} \label{eq:EfronStein}
\Var(\DR(\bZ)) &\ \leq \ 
\frac{1}{2} \sum_{k=1}^n \Eb{ (\DR(\bZ, \bm \delta, \tilde{\bm U} ) - \DR(\bZ^k,  \bm \delta, \tilde{ \bm U}) )^2 } 
\\  \label{eq:EfronStein_delta_terms}
&\qquad + 
	\frac{1}{2} \sum_{k=1}^n \Eb{ (\DR(\bZ, \bm \delta, \tilde{ \bm U} ) - \DR(\bZ, \bm \delta^k, \tilde {\bm U} ) )^2 } 
\\ \label{eq:EfronStein_U_terms}
&\qquad + 
	\frac{1}{2}\sum_{k=1}^n \Eb{ (\DR(\bZ, \bm \delta, \tilde{\bm U} ) - \DR(\bZ, \bm \delta, \tilde{\bm U^k}) )^2 }.
\end{align}
\end{subequations}

Recall that in the typical case, $\bmu^\top \bx(\bZ) = O_p(n)$.  Hence in what follows, we will focus on developing conditions for which the upper bound in \cref{eq:EfronStein,eq:EfronStein_delta_terms,eq:EfronStein_U_terms} is $o(n^2)$.  Indeed, such a bound would suggest $D(\bZ) - \Eb{D(\bZ)} = o_p(n)$, i.e., the stochastic fluctuations in the VGC are negligible relative to the magnitude of the out-of-sample error for $n$ sufficiently large.  With this perspective, it is not difficult to argue that 
both \cref{eq:EfronStein_delta_terms,eq:EfronStein_U_terms} both contribute at most $ O\left (\frac{n}{h} \right)$ (see proof of \cref{thm:EfronStein}).  Thus, we focus on \cref{eq:EfronStein}.  

Consider the $k^\text{th}$ element of the sum. Write, 
\begin{align} \notag
\abs{ \DR(\bZ) - \DR(\bZ^k) } 
&\ \leq \ 
\frac{1}{h a_{k} \vmin} \abs{ \sum_{j=1}^n \delta_j (x_j(\bZ + \delta_j \tilde U_j \be_j) - x_j(\bZ^k + \delta_j \tilde U_j \be_j ) ) }
\\ \label{eq:WeakCauchySchwarzStep}
&  \ \leq \ 
\frac{\| \bm \delta \|_2 \sqrt n }{h a_{k} \vmin}  \cdot  \left( \frac{1}{n} \sum_{j=1}^n \abs{ x_j(\bZ + \delta_j \tilde U_j \be_j) - x_j(\bZ^k + \delta_j \tilde U_j \be_j )}^2 \right)^{1/2}~,
\end{align}
where the last inequality follows from Cauchy-Schwarz's inequality. Since each $\delta_j$ is Gaussian, we expect $\| \bm \delta\|_2^2$ to concentrate sharply at its mean, i.e., $\| \bm \delta \|_2^2 = O_p( hn )$.  Thus by squaring \cref{eq:WeakCauchySchwarzStep}, taking expectations and substituting into \cref{eq:EfronStein}, we roughly~expect
\begin{align*}
\Var(\DR(\bZ)) &\ \leq \ 
\underbrace{O\left(\frac{n}{h} \right)}_{\text{\cref{eq:EfronStein_delta_terms,eq:EfronStein_U_terms}}}
 \ +  \ \ \Eb{ \frac{ \|\bm \delta \|_2^2 n^2}{h^2 \vmin^2 a_{\min}^2}  \cdot \frac{1}{n^2} \sum_{k=1}^n \sum_{j=1}^n \abs{ x_j(\bZ + \delta_j \tilde U \be_j) - x_j(\bZ^k + \delta_j \tilde U \be_j )}^2}
 \\
 & \ \approx \ 
 O\left(\frac{n}{h} \right) + O\left( \frac{ n^3}{h} \right)  \cdot 
\underbrace{\frac{1}{n^2} \sum_{k=1}^n \sum_{j=1}^n  \Eb{ \abs{ x_j(\bZ + \delta_j \tilde U_j \be_j) - x_j(\bZ^k + \delta_j \tilde U_j \be_j )}^2}}_{\text{Avg. Solution Instability}},
\end{align*}
where $a_{\min} \equiv \min_{j: a_j \neq 0} \abs{a_j}$.  

We call the indicated term the \emph{Average Solution Instability}. In the worst case, it is at most $1$ since $\mathcal X \subseteq [0, 1]^n$.  If, however, it were $O(n^{-\alpha})$ for some $\alpha > 1$, then the $\Var(\DR(\bZ)) = o(n^2)$ as~desired.  

How do we intuitively interpret Average Solution Instability?  Roughly, in the limit as $h \rightarrow 0$, we might expect that $\bx(\bZ + \delta_j \tilde U_j\be_j) \approx \bx(\bZ)$ because $\delta_j = O_p(\sqrt{h})$.  Then the Average Solution Instability is essentially the expected change in the solution in a randomly chosen component $j$ when we replace the data for a randomly chosen component $k$ with an i.i.d. copy.  This interpretation suggests Average Solution Instability should be small so long as a small perturbation to the $k^\text{th}$ component doesn't change the entire solution vector $\bx(\bZ)$ by a large amount, i.e., if small perturbations lead to small, local changes in the solutions.  Intuitively, many large-scale optimization problems exhibit such phenomenon (see, e.g., \cite{gamarnik2013correlation2}), so we broadly expect the VGC to have low variance.  

The above heuristic argument can be made formal as in the following theorem.  
\begin{thm}[Variance of the VGC] \label{thm:EfronStein} \blue{Suppose that the solution $\bx(\bZ)$ to \cref{eq:DefXPolicy} is almost surely unique,} that there exists a constant $C_1$ (not depending on $n$) such that \mbox{$\Eb{ \frac{1}{n^2} \sum_{k=1}^n \sum_{j=1}^n \left( x_j(\bZ + \delta_j \tilde U_j \be_j) - x_j(\bZ^k + \delta_j \tilde U_j \be_j) \right)^2 }  \ \leq C_1 n^{-\alpha}$} and that \cref{asn:Parameters} holds.  Then, there exists a constant $C_2$ (depending on $\vmin$ and $a_{\min} \equiv \min_j a_j$) such that for any $0 < h < 1/e$
\[
\Var(D(\bZ) ) = \frac{C_2 }{h}  \max( n^{3-\alpha}, n).
\]
In particular, in the typical case where the full-information solution to \cref{eq:obj-problem} is $O(n)$, the stochastic fluctuations in the VGC are negligible relative to the out-of-sample performance if $\alpha > 1$.  
\end{thm}
The proof of \cref{thm:EfronStein} is in \cref{sec:ProofOfEfronStein}. 

\blue{We remark that \cref{thm:EfronStein} provides a \emph{sufficient} condition for the variance of the VGC to be negligible asymptotically and to show that $h$ controls the bias-variance tradeoff, however, the bound is not tight. In \cref{sec:Weakly-Coupled-Problems} we provide a tighter analysis given more stringent assumptions on \cref{eq:obj-problem,eq:DefXPolicy}, which then also provides us guidance on how to select $h$ to approximately balance the bias-variance tradeoff. }


\subsection{Smoothness and Boundedness of the VGC}
One of the key advantages of our VGC is that it is smooth in the policy class, provided $\btheta \mapsto r(\cdot, \btheta)$ is ``well-behaved."  Other corrections, like the Stein Correction of \citeGR, do not enjoy such smoothness.  In \cref{sec:experiments}, we argue this smoothness improves the empirical performance of our method.  We formalize ``well-behaved" in the next assumption:
\begin{assm}[Plug-in Function is Smooth] \label[assm]{asn:SmoothPlugIn}
We assume the functions $a_j(\btheta)$, $b_j(\btheta)$ are each $L$-Lipschitz continuous for all $j=1, \ldots, n$.  Moreover, we assume there exists $a_{\max}, b_{\max} < \infty$ such that 
\[
\sup_{\btheta \in \Theta} \abs{ a_j(\btheta)} \ \leq \ a_{\max} \ \ \text{ and } \ \ 
\sup_{\btheta \in \Theta} \abs{ b_j(\btheta)} \ \leq \  b_{\max} \quad \forall j =1, \ldots, n.
\]
Finally, we assume there exists $a_{\min}$ such that 
\[
0 \ < \ a_{\min} \ \leq \ \inf\{ \abs{a_j(\btheta)} : a_j(\btheta) \neq 0 , \ j = 1, \ldots, n, \ \btheta \in \Theta \}
\]
\end{assm}
In words, \cref{asn:SmoothPlugIn} requires the functions $a_j(\btheta)$ and $b_j(\btheta)$ to be Lipschitz smooth, bounded, and that the non-zero components of $a_j(\btheta)$ be bounded away from 0.

\enlargethispage{5pt}
{\blockedit 
\vspace{10pt}
\paragraph{Bias and Variance of VGC for Plug-In Linear Regression Models.}
Recall our Plug-in Linear Model class from \cref{sec:plug-in-policy}.  Since $a_j(\btheta) = 0$ for all $j$,  $D(\bZ, \btheta) = 0$ for all $\bZ$ and (non data-driven) $\btheta$ for this class.  Said differently, the in-sample performance of a policy is, itself, our estimate of the out-of-sample performance, and, both \cref{thm:equiv-in-sample,thm:EfronStein} can both be strengthened; the bias of our estimator and variance of the correction are both zero.  More generally, $D(\bZ, \btheta) = 0$ whenever the plug-in functionals $r_j(z, \btheta)$ do not depend on $z$ for all $j$. 

We stress however that this analysis does not immediately guarantee that the in-sample performance of policies of the form $\bx^{\sf LM}(\bZ, \btheta(\bZ))$ is a good estimate of out-of-sample performance, because $\btheta(\bZ)$ depends on $\bZ$.  In \cref{sec:Weakly-Coupled-Problems} we provide sufficient conditions to ensure that in-sample performance is, indeed, a good estimate of out-of-sample performance.  Moreover, when $\br(\bZ, \btheta)$ does depend on $\bZ$, e.g., as with our Mixed Effects Regression class, $D(\bZ, \btheta)$ is generally non-zero.  \label{eq:NonDataDrivenPolicyclasses}
}

\begin{lem}[Smoothness of \DanskinTitle]\label[lem]{lem:D-is-Lipschitz}
Under \cref{asn:SmoothPlugIn,asn:Parameters}, the following hold: \\
i) There exists a constant $C_1$ (depending on $a_{\min}, a_{\max}, b_{\max}$, and $\vmin$)
such that for any $\bz \in \R^n$
and  any $0 < h < 1/e$, the function $\btheta \mapsto D(z, (\btheta,h))$ is Lipschitz continuous with parameter 
\blue{
\(
\frac{C_1 n^{2} L}{h \sqrt \vmin} \left( \| \bz \|_\infty + 1 \right). 
\)
}
Moreover, there exists a constant $C_2$ (depending on $C_\mu$ and $C_1$) such that for any $R > 1$, 
with probability at least $1- e^{-R}$, the (random) function $\btheta \mapsto D(\bZ, (\btheta, h))$ is Lipschitz continuous with parameter 
\blue{
\(
	\frac{C_2 L}{h} \sqrt{\frac{R}{\vmin}}\cdot n^2 \sqrt{\log n}  .
\)}\\
\blue{ ii) Consider $D(z,(\btheta,h))$ where $h \in [h_{\min},h_{\max}]$ and $0 <h_{\max} - h_{\min} < 1$. There exists an absolute constant $C_3$ such that for any $\bz \in \R^n$ and $\btheta \in \Theta$, the following holds,
\[
 	\abs{D(z,(\btheta,h)) - D(z,(\btheta,\overline{h}))} \le \frac{C_3 n}{h_{\min} \nu_{\min}^{3/4}} \sqrt{\abs{h - \overline{h}}} . 
\]
}
\end{lem}

See \cref{app:properties-VGC} for a proof.  Intuitively, the result follows because $\btheta \mapsto V(\bz, \btheta)$ is Lipschitz by Danskin's theorem and $D(\bz, \btheta)$ is a linear combination of such functions.  The second part follows from a high-probability bound on $\| \bZ\|_\infty$.

In addition to being smooth, the VGC is also bounded as a direct result of taking the conditional expectation over the perturbation parameters $\delta_j$. 
\begin{lem}[VGC is Bounded] \label[lem]{lem:VGCBounded}
Suppose that \cref{asn:SmoothPlugIn,asn:Parameters} hold. 
For any $\bz$, and any $j =1, \ldots, n$, 
\blue{
\[
\abs{ D_j(\bz)} \ \leq \ \frac{\sqrt 3}{\vmin^{3/4} \sqrt h }.
\] 
}
\end{lem}
The proof can be found in \cref{app:properties-VGC}. The result follows from observing that the $j^{\rm th}$ component of the VGC is the difference of the optimal objectives values of two optimization problems whose cost vector differs by $O(|\delta_j|)$ in one component. Thus, the two optimal objective values can only differ by $O(|\delta_j|)$ which is at most a constant once we take the conditional expectation. 
\enlargethispage{0pt}
\section{Estimating Out-of-sample Performance for Weakly-Coupled Problems}
\label{sec:Weakly-Coupled-Problems}
In this section we provide high-probability tail bounds on the error of our estimator for \mbox{out-of-sample} performance that hold uniformly over a given policy class. Such bounds justify using estimator for policy learning, i.e., identifying the best policy within the class.  They are also substantively stronger than the variance analysis of \cref{thm:EfronStein} as they provide exponential bounds on the tail behavior, rather than bounding the second moment. \blue{\label{note:uniform-bound-h} Additionally, we show the uniform results hold even when $\btheta \in \bar{\Theta}$ is chosen in a data-driven manner (which recall also includes $h$).}

From the definition of the \Danskin~out-of-sample estimator (\cref{eq:OOSEstimator}), the error of our estimator of out-of-sample performance for $\bx(\bZ, \btheta(\bZ))$ is 

\blue{
\begin{subequations} \label{eq:ErrorExpansion}
\begin{align} \notag
&\underbrace{\abs{ \bmu^\top \bx(\bZ, \btheta(\bZ)) - \left(\bZ^\top \bx(\bZ, \btheta(\bZ)) - D(\bZ, \btheta(\bZ)) \right) }}_{\text{Error Estimating Out of Sample Perf.}}
 \ \leq \ \sup_{ \btheta \in \bar{\Theta} } \underbrace{ \abs{ \bxi^\top \bx(\bZ, \btheta) - D(\bZ, \btheta)}}_{\text{Error Estimating In-Sample Optimism}}
\\ \label{eq:ErrorExpansion_Optimism}
& \qquad \ \leq \ 
\sup_{ \btheta \in \bar{\Theta} } \abs{ \bxi^\top \bx(\bZ, \btheta) - \Eb{\bxi^\top \bx(\bZ, \btheta) } }
\\ \label{eq:ErrorExpanion_VGC}
& \qquad \qquad \ + \ 
\sup_{ \btheta \in \bar{\Theta} }  \abs{ D(\bZ, \btheta)  - \Eb{D(\bZ, \btheta) }} 
\\ \label{eq:Bias}
& \qquad \qquad \ + \ 
\sup_{ \btheta \in \bar{\Theta} }  \abs{ \Eb{ \bmu^\top \bx(\bZ, \btheta) - \bZ^\top \bx(\bZ, \btheta) - D(\bZ, \btheta)  } }.
\end{align}
\end{subequations}
}

We bounded \cref{eq:Bias} in \cref{thm:equiv-in-sample}.  Our goal will be to find sufficient conditions to show the remaining terms are also $o_p(n)$ uniformly over $\btheta \in \Theta$.  Then, in the typical case where the out-of-sample performance is $O_p(n)$, the error of our estimator will be negligible relative to the true out-of-sample performance.  
Our strategy will be to leverage empirical process theory since the argument of each suprema is a sum of random variables. Importantly, this empirical process analysis does \emph{not} strictly require the independent Gaussian assumption (\cref{asn:Gaussian}). The challenge of course is that the constraints of \cref{eq:DefXPolicy} introduce a complicated dependence between the terms.  

Inspired by the average stability condition of \cref{thm:EfronStein}, we focus on classes of ``weakly-coupled" optimization problems.  
We consider two such classes of problems, those weakly-coupled by variables in \cref{sec:Weakly-Coupled-Variables} and those weakly-coupled by constraints \cref{sec:Weakly-Coupled-Constraints}.  
We provide formal definitions below.  

\subsection{Problems Weakly-Coupled by Variables} 
\label{sec:Weakly-Coupled-Variables}
We say an instance of \cref{eq:obj-problem} is weakly-coupled by variables if fixing a small number of variables causes the problem to separate into many, decoupled subproblems.  Generically,
such problems can be written as
\begin{align}\label{eq:wc-v-problem}
\min_{\bx}\quad &	\left( \bmu^0\right)^\top \bx^0 +\sum_{k=1}^{K}\left( \bm{\mu}^{k} \right) ^{\top}\bx^{k} \\ \notag
\text{s.t.}\quad & \bx^0 \in \Y, \quad 
\bx^{k}\in\mathcal{X}^{k}(\bx^0),\quad\forall k=1,\dots,K.
\end{align}
Here, $\bx^0$ represents the coupling variables and $k =1, \ldots, K$ represent distinct subproblems.  Notice that once $\bx^0$ is fixed, each subproblem can be solved separately.  Intuitively, if $\text{dim}(\bx^0)$ is small relative to $n$, the subproblems of \cref{eq:wc-v-problem} are only ``weakly" coupled.  Some reflection shows both \cref{ex:drone-aed,ex:PersonalizedPricing} from \cref{sec:Formulation} are weakly-coupled by variables.  

Let $S_k \subseteq \{ 1, \ldots, n\}$ be the indices corresponding to $\bx^k$ for $k = 0, \ldots, K$,  and $S_{\max} = \max_{k\geq 0} \abs{ S_k}$.  The sets $S_0, \ldots, S_K$ form a disjoint partition of $\{1, \ldots, n \}$.  Without loss of generality, reorder the indices so that the $S_k$ occur ``in order,"; i.e., $(j  \,:\, j \in S_0), \ldots, (j \,:\, j \in S_K)$ is a consecutive sequence.  

Given the weakly-coupled  structure of \cref{eq:wc-v-problem}, we define a generalization of $\bx(\bZ, \btheta)$:  For each $\bx^0 \in \Y$ and $\btheta \in \bar{\Theta}$, let 
\begin{equation} \label{eq:DefXPolicy_Variables}
\bx^k(\bZ, \btheta, \bx^0) \ \in \ \argmin_{\bx^k \in \mathcal X^k(\bx^0) } \br^k(\bZ, \btheta)^\top \bx^k, \quad k =1, \ldots, K,
\end{equation}
where $\br^k(\bZ, \btheta) = \left( r_j(\bZ, \btheta) \,:\, j \in S_k \right) $.  Intuitively, the vector 
\[
\bx(\bZ, \btheta, \bx^0) \equiv \left( (\bx^0)^\top, \bx^1(\bZ, \btheta, \bx^0)^\top, \ldots, \bx^K(\bZ, \btheta, \bx^0)^\top \right)^\top
\]
satisfies the Average Instability Condition of \cref{thm:EfronStein} so long as $S_{\max}$ is not too large since the $j^{\text{th}}$ component of the solution changes when perturbing the $k^{\text{th}}$ data point if and only if $j$ and~$k$ belong to the same subproblem.  This event happens with probability at most $S_{\max}/ n^{2}$.  

The key to making this intuition formal and obtaining exponential tails for the error of the out-of-sample estimator (\cref{eq:ErrorExpansion}) is that
\begin{align*}
\sup_{\btheta \in \bar{\Theta}} \abs{ \bxi^\top \bx(\bZ, \btheta) - \Eb{\bxi^\top \bx(\bZ, \btheta)}  }
&\ \leq \
\sup_{\btheta \in \bar{\Theta}, \bx^0 \in \Y} \abs{ \bxi^\top \bx(\bZ, \btheta, \bx^0) - \Eb{\bxi^\top \bx(\bZ, \btheta, \bx^0)} }
\quad \text{ and,} 
\\
\sup_{\btheta \in \bar{\Theta}} \abs{ D(\bZ, \btheta) - \Eb{ D(\bZ, \btheta)} } 
&\ \leq \
\sup_{\btheta \in \bar{\Theta}, \bx^0 \in \Y} \abs{ D(\bZ, \btheta, \bx^0) - \Eb{D(\bZ, \btheta, \bx^0)} },
\end{align*}
where both $\bxi^\top \bx(\bZ, \btheta, \bx^0)$ and $D(\bZ, \btheta, \bx^0)$ can, for a fixed $\btheta, \bx^0$, be seen as sums of $K$ \emph{independent} random variables.  To obtain uniform bounds, we then need to control only the metric entropy of the resulting (lifted) stochastic processes indexed by $(\btheta, \bx^0)$.  

\blue{\label{WhyLinearClass} \label{note:assum_4_1} We propose a simple assumption on the policy class to control this metric entropy.}  We believe this assumption is \blue{easier to verify} than other assumptions used in the literature (e.g., bounded linear subgraph dimension or bounded Natarajan dimension), \blue{but admittedly slightly more stringent}. 
\begin{assm}[Lifted Affine Plug-in Policy]\label[assm]{asn:hyperplane} 
Given an affine plug-in policy class defined by $\br(\cdot, \btheta)$ for $\btheta \in \Theta$, we say this class satisfies the lifted affine plug-in policy assumption for problems weakly-coupled by variables (\cref{eq:wc-v-problem}) if there exists mapping 
$\phi(\cdot)$ 
and mappings 
$g_{k}(\cdot)$ for $k=1, \ldots, K$ such that  
\[
	\bx^{k}(\bZ, \btheta, \bx^0) \in \argmin_{\bx^k \in \mathcal{X}^k(\bx^0)} \phi(\btheta)^{\top}
	g_k(\bZ^k, \bx^k, \bx^0)
\quad k =1, \ldots, K, \quad \forall \bx^0 \in \Y.
\]
\end{assm}
We stress that the mapping $\phi(\cdot)$ is common to all $K$ subproblems and all $\bx^0 \in \Y$, and both $\phi(\cdot)$ and $g_k(\cdot)$ can be arbitrarily nonlinear.  Moreover, $g_k(\cdot)$ may implicitly depend on the precisions $\bm{\nu}$ and covariates $\bm W$ as these are fixed constants.  
With the exception of policies from linear smoothers, each of our examples from \cref{sec:plug-in-policy}, satisfies \cref{asn:hyperplane}.   For example, for plug-ins for linear regression models, we can simply take $\phi(\btheta) = \btheta$ and $g_k(\bZ^k, \bx^k, \bx^0) = \sum_{j \in S_k} \bW_j^\top x_j$.  

When $\bx^k(\bZ, \btheta, \bx^0)$ is not the unique minimizer to the problem weakly-coupled variables defined in \cref{eq:DefXPolicy_Variables}, we require that ties are broken consistently.  Let $\text{Ext}(\text{Conv}(\mathcal{X}^{k}(\bx^0)))$ denote the set of extreme points of $\text{Conv}(\mathcal X^k(\bx^0))$ and let $\mathcal{X}_{\max} = \max_{k \ge 0} \text{Ext}(\text{Conv}(\mathcal{X}^{k}(\bx^0)))$. Note, if $\bx^k(\bZ, \btheta, \bx^0)$ is unique, it is an extreme point.
\begin{assm}[Consistent Tie-Breaking] \label[assm]{asn:TieBreaking}
We assume there exists functions $\sigma_{k\bx^0}: 2^{\mathcal X^k(\bx^0) } \mapsto \text{\rm Ext}(\text{\rm Conv}(\mathcal{X}^{k}(\bx^0)))$ such that 
\[
	\bx^{k}(\bZ, \btheta, \bx^0) =  \sigma_{k\bx^0}\left( \argmin_{\bx^k \in \mathcal{X}^k(\bx^0)} \phi(\btheta)^{\top}
	g_k(\bZ^k, \bx^k, \bx^0) \right)
\quad k =1, \ldots, K, \quad \forall \bx^0 \in \Y.
\]
\end{assm}
Consistent tie-breaking requires that if ($\btheta_1, \bx^0_1$) and $(\btheta_2, \bx^0_2)$ induce the same minimizers in \cref{eq:DefXPolicy_Variables} for some $\bZ$, then $\bx^k(\bZ, \btheta_1, \bx^0_1) = \bx^k(\bZ, \btheta_2 ,\bx^0_2)$, and this point is an extreme point of $\mathcal X^k(\bx^0)$.

\Cref{asn:hyperplane,asn:TieBreaking} allow us to bound the cardinality of the set
\(
\left\{ \left( \bx^0, \bx^1(\bZ, \btheta, \bx^0), \ldots,\bx^K(\bZ, \btheta, \bx^0) \right) \ : \ \bx^0 \in \Y, \btheta \in \Theta \right\}
\)
by adapting a geometric argument counting regions in a hyperplane arrangements from \cite{gupta2019shrunk} (Lemma C.7). The cardinality of the set characterizes the metric entropy of the policy class. 

Finally, for this section, we say a constant $C$ is \emph{dimension-independent} if $C$ does \emph{not} depend on $\{K,S_{\max},h,\mathcal{X}^0,\mathcal{X}_{\max},\text{dim}(\phi)\}$, but may depend on $\{\nu_{\min},  C_{\mu},  L\}$. We now present the main result of this section:
{ \blockedit
\begin{thm}[Policy Learning for Problems Weakly-Coupled by Variables]\label{thm:unif-oos-est-wc-v} 
Suppose \cref{asn:Gaussian,asn:hyperplane,asn:Parameters,asn:TieBreaking,asn:SmoothPlugIn} all hold. 
Let $\mathcal X_{\max} \geq \abs{ \text{\rm Ext}(\text{\rm Conv}(\mathcal{X}^{k}(\bx^0))) }$ for all $k=1, \dots, K$ and $\bx^0 \in \Y$, and assume $\mathcal X_{\max} < \infty$.  Then, for $0 < h_{\min} \le h_{\max} \le 1$, there exists a dimension-independent constant $C$ such that, for any $R > 1$, with probability at least $1-2\exp(-R)$,
\begin{align*}
	& \sup_{\theta \in \bar{\Theta}} 
	\abs{ \bxi^\top \bx(\bZ, \btheta) - D(\bZ, \btheta, h) }
 	\ \le \ 
	C K  S_{\max}  \cdot h_{\max} \log \left( \frac{1}{h_{\min}} \right)	
	\\ & \quad ~+~ 	
C S_{\max} R \sqrt{\frac{K \log(1 + \abs{\Y})}{h_{\min}}} 
\Biggr( \log(K) \sqrt{\log(S_{\max}) \log \left(h^{-1}_{\min} \cdot N\left(\sqrt{\frac{h_{\min}}{Kn^2}}, \Theta\right) \right) } 
\\ & \qquad 
	 ~+~ 
	\sqrt{ \log(K) \text{\rm dim}(\phi) \log(1 +\mathcal X_{\max}) } \Biggr).
\end{align*}
where $N(\epsilon, \Theta)$ is the $\epsilon$-covering number of the set $\Theta$.
\end{thm}
}
{\blockedit
\label{note:selecting-h-WCV} In the typical case 
where $\Theta$ does not depend on $n$ or $K$ and 
\begin{align} \label{eq:wc-var-intuition-eq-1}
	\max(S_{\max},\abs{\mathcal{X}_0}, \text{dim}(\phi)) = \tilde{O}(1) \text{ as } K\rightarrow \infty, 
\end{align}
we can approximately minimize the above bound
by selecting $h \equiv h_k = O(K^{-1/3})$ and noting the relevant covering number grows at most logarithmically in $K$. This choice of $h$ approximately balances the deterministic and stochastic contributions, and the bound reduces to $\tilde{O}(C_{\sf PI}K^{2/3})$ for some $C_{\sf PI}$ (depending on $\abs{\mathcal{X}_0}$, $\text{dim}(\phi)$, $\text{dim}(\btheta)$) that measures the complexity of the policy class. Many applications satisfy the conditions in \cref{eq:wc-var-intuition-eq-1}, including the drone-assisted emergency medical response application (\cref{ex:drone-aed}). }

To illustrate, recall \cref{ex:drone-aed}.  Here, $\by$ represents the binding variables ``$\bx^0$", and we see $\abs{\mathcal X^0} = \binom{L}{B}$.  Moreover, $S_{\max}= L$, since $\bx$ decouples across $k$.  Inspecting the constraints, $\mathcal X_{\max} \leq B$, since for each $k$, we choose exactly one  depot from which to serve location $k$, and there are most $B$ available depots. Finally, for a fixed policy class, $\text{dim}(\phi)$ is constant and the log covering number above grows at most logarithmically in $K$. 
 Most importantly, we expect $L$ (the number of possible depots) and $B$ (the budget) to be fairly small relative to $K$ since regulations and infrastructure limit placement of depots, but there are many possible locations for cardiac events. Here typical instances of \cref{ex:drone-aed} satisfy \cref{eq:wc-var-intuition-eq-1}. We return to \cref{ex:drone-aed} in \cref{sec:experiments} where we study the performance of our method numerically.

\subsection{Problems with Weakly-Coupled by Constraints}
\label{sec:Weakly-Coupled-Constraints}
An instance of \cref{eq:obj-problem} is weakly-coupled by constraints if, after removing a small number of binding constraints, the problem decouples into many separate subproblems.  Data-driven linear optimization problems of this form have been well studied by \cite{li2019online} and \citeGR.  In order to facilitate comparisons to the existing literature, we study the specific problem
\begin{align}\label[problem]{eq:wc-c-problem}
\max_{\bx \in \mathcal X} \ \ \sum_{j=1}^{n} \mu_j x_j, 
\qquad \mathcal{X} = \left\{ \bx \in [0,1]^n : \frac{1}{n}\sum_{j=1}^n \bA_j x_j \le \bm{b} \right\},
\end{align}
and corresponding plug-in policies
\begin{align}\label[problem]{eq:wc-c-plug-in-problem}
\bx(\bZ, \btheta) \in \argmax_{\bx \in \mathcal X}\quad & \sum_{j=1}^{n} r_j(Z_j, \btheta) x_j.
\end{align}
Here, $\bA_j \in \mathbb{R}^m$ with $m \geq 1$.  In particular, we consider a maximization instead of a minimization.

We next introduce a dual representation of \cref{eq:wc-c-plug-in-problem}.
Specifically, scaling the objective of \cref{eq:wc-c-plug-in-problem} by $\frac{1}{n}$ and dualizing the binding constraints yields 
\begin{align}\label{eq:wc-c-problem-dualize-i}
 \blambda(\bZ, \btheta) \in \argmin_{\blambda \in \R^m_+} \ \left\{ \bm{b}^\top \blambda  + \max_{\bx \in [0, 1]^n} \ \ \frac{1}{n} \sum_{j=1}^{n} (r_j(Z_j, \blue{\btheta}) - \bA_j^{\top}\blambda)  x_j \right\}
\end{align}
For a fixed $\blambda$, the inner maximization of \cref{eq:wc-c-problem-dualize-i} can be solved explicitly, yielding
\begin{align*}
	\blambda(\bZ, \btheta) \in \argmin_{\blambda \in \mathbb{R}^m_{+}} \mathcal{L}(\blambda, \bZ, \btheta), \quad \text{where} \quad 
	\mathcal{L}(\blambda, \bZ, \btheta) \equiv
	\bm{b}^\top \blambda + \frac{1}{n} \sum_{j=1}^n \left( r_j(Z_j,\btheta) - \bA_j ^\top \blambda \right)^+.
\end{align*}
By strong duality, $V(\bZ, \btheta) = \br(\bZ, \btheta)^\top \bx(\bZ, \btheta) = n \mathcal L(\blambda(\bZ, \btheta), \bZ, \btheta)$. 

This dual representation highlights the weakly-coupled structure.  Indeed, the dependence across terms in the sum in $\L(\blambda(\bZ, \btheta), \bZ, \btheta)$ arises because $\blambda(\bZ, \btheta)$ depends on the entire vector $\bZ$.  However, this dependence has to be ``channeled" through the $m$ dimensional vector $\blambda(\bZ, \btheta)$, and, hence, when $m$ is small relative to $n$, cannot create too much dependence between the summands.  Indeed, we will show that if $m$ is small relative to $n$, then $\blambda(\bZ,\btheta)$ concentrates at its expectation, i.e., a constant, as $n \rightarrow \infty$, and, hence, the summands become independent asymptotically.  This insight is key to the analysis.




To formalize these ideas, we make assumptions similar to those in \citeGR~and \cite{li2019online}: 
\begin{assm}[$s_0$-Strict Feasibility] \label[assm]{asn:strict-feasibility}
There exists an $s_0 > 0$ and $\bx_0 \in \mathcal{X}$ such that $\frac{1}{n}\sum_{j=1}^n \bA_j x^0_j + s_0 \be \le \bm{b}$.
\end{assm}
\begin{assm}[Regularity of Matrix $\bA$] \label[assm]{asn:MatrixAConditions}
There exists a constant $C_A$ such that $\| \bA_j \|_{\infty} \leq C_A$ for all $1 \leq j \leq n$.  Moreover, there exists a constant $\beta > 0$ such that the minimal eigenvalue of $\frac{1}{n} \sum_{j=1}^n \bA_j \bA_j^\top$ is at least $\beta$.
\end{assm}
The strict feasibility assumption can often be satisfied by perturbing $\bA$~or~$\bm{b}$ and ensures the dual optimal values $\blambda(\bZ, \btheta)$ are bounded with high probability. The regularity assumptions on $\bm A$ ensure the function $\blambda \mapsto \Eb{\L(\blambda, \bZ, \btheta)}$ is strongly convex, a key property in our proof (see below). \blue{Such a property holds, e.g., if the columns $\bA_j$ are drawn randomly from some distribution. \label{explainingStrictFeasibility}} 

Like \cref{sec:Weakly-Coupled-Variables}, in order to obtain uniform bounds we must also control the metric entropy of the different stochastic error terms in out-of-sample error \cref{eq:ErrorExpansion}. Generalizing \citeGR, we make the following assumption:
\begin{assm}[VC Policy Class]
\label[assm]{asn:LiftedAffinePlugInII}
There exists a function $\rho(\cdot)$ such that 
\[
	r_j(z_j, \btheta) = \rho((z_j, \nu_j, \bW_j, \bA_j), \btheta), \quad j=1, \ldots, n
\]
and a constant $V$ such that the class of functions 
\[
	\left\{(z, \nu, \bW, \bA) \mapsto
			\rho((z, \nu, \bW, \bA), \btheta) - \bA^\top \blambda \ : \ \btheta \in \Theta, \; \blambda \in \R^m_+
			\right\}
\]
has a pseudo-dimension at most $V$. Without loss of generality, we further assume $V \geq \max(m, 2)$.
\end{assm}

\blue{The size of the constant $V$ captures the complexity of the policy class which typically depends upon the dimension of $\btheta$ as well as the number of binding constraints $m$. \label{explainingVCAssumption}} As an illustration, recall the plug-in for linear regression models policy class from \cref{sec:Formulation}.  By \cite[Lemma 4.4]{pollard1990empirical}, this policy class satisfies \cref{asn:LiftedAffinePlugInII} with $V = \text{dim}(\btheta) + m$.

Finally, we say a constant $C$ is \emph{dimension-independent} if $C$ does \emph{not} depend on $\{n,h,m,V,\text{dim}(\btheta)\}$, but may depend on $\{\nu_{\min}, C_A,  C_{\mu},  \beta,s_0, a_{\min}, a_{\max}, b_{\max}, L\}$.

The main result of this section is then:
\blue{
\begin{thm}[Estimation Error for Problems Weakly Coupled Constraints]\label{thm:WC-Constraints-bound}
Under \cref{asn:Gaussian,asn:Parameters,asn:SmoothPlugIn,asn:strict-feasibility,asn:MatrixAConditions,asn:LiftedAffinePlugInII}, for $0 < h_{\min} \le h_{\max} \le 1$
there exists dimension-independent constants $C$ and $n_0$, such that for any $ R > 1$ and all $n\ge n_0 e^R$, we have that with probability $1-C\exp(-R)$,
\begin{align*}
	\sup_{\btheta \in \bar{\Theta}} &
 	\left| \sum_{j=1}^n \xi_j x_j (\bZ, \btheta) - D_j(\bZ, \btheta) \right| 
 	\ \le \ 
	C n \cdot h_{\max} \log \left( \frac{1}{h_{\min}} \right) \\
	& \qquad \qquad \qquad + C \cdot V^3 \log^3 V  \cdot R \cdot \sqrt{n \log \left(n \cdot N(n^{-3/2}, \Theta) \right)} \cdot \frac{\log^5 n }{h_{\min}} 
\end{align*}
\end{thm}
}
\blue{ \label{note:select-h-WCC} To build intuition, consider instances where $\Theta$ does not depend on $n$. Then, $V= O(1)$ and the covering number grows at most logarithmically as $n \rightarrow \infty$. We can then minimize the above bound (up to logarithmic terms) by taking 
taking $h\equiv h_n = O(n^{-1/4})$, yielding a bound of order $\tilde{O}(n^{3/4})$. In particular, in the typical instance where the full-information optimum (c.f. \cref{eq:wc-c-problem}) is $O(n)$, the relative error of our estimate is $\tilde O(n^{-1/4})$ which is vanishing as $n \rightarrow \infty$. }

\begin{rem} \rm
The rate above ($\tilde{O}(n^{-1/4})$) is slightly slower than the rate of convergence of the Stein correction in \citeGR~($\tilde{O}(n^{-1/3})$). We attribute this difference to our choice of a first order finite difference when constructing the \Danskin. A heuristic argument strongly suggests that had we instead used a second order forward finite difference scheme as in \cref{app:ImplementationDetails}, we would recover the rate $\tilde{O}(n^{-1/3})$. Moreover, our numerical experiments (with the second order scheme) in \cref{sec:experiments} shows our second-order \Danskin~to be very competitive.
\end{rem}




\subsection*{Proof Intuition:  Approximate Strong-Convexity and Dual Stability}
To build intuition, recall that to show the convergence of VGC, it suffices to bound the Average Solution Instability defined in \cref{thm:EfronStein}.  By complementary slackness, $x_j(\bZ) = \Ib{r_j(Z_j) > \bA_j^\top \blambda(\bZ)}$ except possibly for $m$ fractional components.  Hence, by rounding the fractional components, we have, for $j \neq k$,
\begin{align*}
&x_j(\bZ + \delta_j \tilde U_j \be_j) - x_j(\bZ^k + \delta_j \tilde U_j \be_j)
\\
& \qquad \leq \ 
\Ib{ r_j(Z_j + \delta_j \tilde U_j \be_j) \geq \bA_j^\top \blambda(\bZ + \delta_j \tilde U_j \be_j) } - 
\Ib{ r_j(Z_j + \delta_j \tilde U_j \be_j) > \bA_j^\top \blambda(\bZ^k + \delta_j \tilde U_j \be_j) } 
\\ 
& \qquad \leq \ 
\Ib{ r_j(Z_j + \delta_j \tilde U_j \be_j) \in 
	\left\langle \bA_j^\top \blambda(\bZ + \delta_j \tilde U_j \be_j), \ 
	 	 	\bA_j^\top \blambda(\bZ^k + \delta_j \tilde U_j \be_j)\right\rangle 
},
\end{align*}
where we use $\langle l, u\rangle$ to denote the interval $[ \min(l, u),  \ \max(l, u)]$.  By symmetry, the same bound holds for $x_j(\bZ^k + \delta_j \tilde U_j \be_j) - x_j(\bZ + \delta_j \tilde U_j \be_j)$.  Since summands where $j = k$ each contribute at most $1$ to the Average Solution Instability, we thus have that 
\begin{align*}
& \frac{1}{n^2} \sum_{k=1}^n \sum_{j=1}^n \left( x_j(\bZ + \delta_j \tilde U_j \be_j) - \bx_j(\bZ^k + \delta_j \tilde U_j \be_j) \right)^2
\\	& \quad \leq \ 
	\frac{1}{n^2} \sum_{k=1}^n \sum_{j=1}^n \Ib{ r_j(Z_j + \delta_j \tilde U_j \be_j) \in 
	\left\langle \bA_j^\top \blambda(\bZ + \delta_j \tilde U_j \be_j), \ 
	 	 	\bA_j^\top \blambda(\bZ^k + \delta_j \tilde U_j \be_j)\right\rangle }
	 	 	+ O_p\left(\frac{1}{n}\right).
\end{align*}

The principal driver of the Solution Instability is the first double sum; in a worst-case, it might be $O_p(1)$.  It will be small if $\| \blambda(\bZ + \delta_j \tilde U_j \be_j) - \blambda(\bZ^k + \delta_j \tilde U_j \be_j)\|$ is small for most $j$ and $k$.  
Said differently, problems like \cref{eq:wc-c-plug-in-problem} that are weakly-coupled by constraints will have small Solution Instability if the dual solutions $\blambda(\cdot)$ are, themselves, stable, i.e., if the dual solution does not change very much when we perturb one data point.  Our analysis thus focuses on establishing this dual solution stability.  

Unfortunately, solutions of linear optimization problems need not be stable -- a small perturbation to the cost vector might cause the optimal solution to switch between extreme points, inducing a large change.   By contrast, solutions of convex optimization problems with strongly-convex objectives \emph{are} stable (see, e.g., \cite{shalev2010learnability}). 
The next key step of our proof is to show that although $\blambda \mapsto \L(\blambda, \bZ)$ is not strongly-convex, it is still ``approximately" strongly-convex with high probability in a certain sense.

\begin{figure} 
\begin{center}
\includegraphics[width=.8\textwidth]{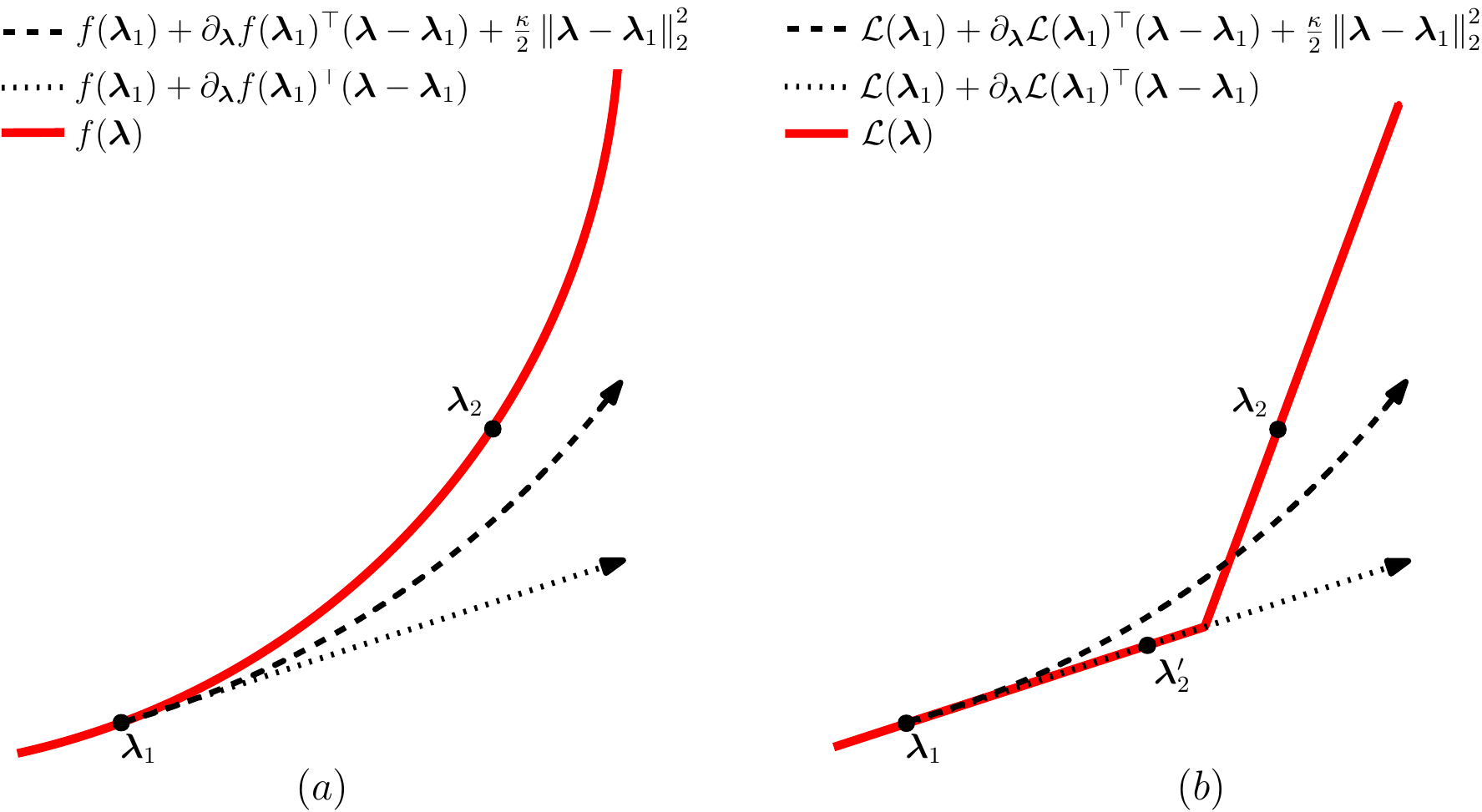}
\end{center}
\caption{\label{fig:ApxStrongCvxIntuition}  \textbf{Approximate Strong Convexity of $\L(\blambda)$.} Figure (a) shows a strongly convex function $f(\blambda)$ and visualizes the strong convexity condition \cref{eq:dual-strongly-convex-i}. Figure (b) shows that because $\L(\blambda)$ is piecewise linear, it does not satisfy \cref{eq:dual-strongly-convex-i} for points on same line segment ($\blambda_1$ and $\blambda_2^\prime$).  However, when $\blambda_1,\blambda_2$ are sufficiently far apart, they are on different line segments and \cref{eq:dual-strongly-convex-i} is satisfied. 
}
\end{figure}

To be more precise, recall that a function $f(\blambda)$ is $\kappa$-strongly-convex if 
\begin{equation}\label{eq:dual-strongly-convex-i}
	f(\blambda_2) - f(\blambda_1)  \ge 
	\nabla_{\blambda} f(\blambda_1)^{\top} (\blambda_2 - \blambda_1) 
	+ \frac{\kappa}{2}\|\blambda_2 -\blambda_1\|_2^2,\qquad \forall \blambda_1,\blambda_2 \in \text{Dom}(f).
\end{equation}
where $\kappa > 0$ and $\nabla_{\blambda}$ denotes the subgradient.  The left panel of \cref{fig:ApxStrongCvxIntuition} depicts this condition graphically.  For any two points $\blambda_2$ and $\blambda_1$, the first-order Taylor series underestimates the function value, and one can ``squeeze in" a quadratic correction $\frac{\kappa}{2}\|\blambda_2 -\blambda_1\|_2^2$ in the gap.

The function $\blambda \mapsto \L(\blambda, \bZ)$ does not satisfy this condition, as seen in the right panel for points $\blambda_1$ and $\blambda_2^\prime$.  This function is piecewise-linear, and, for two points on the same line-segment, the first-order Taylor series is exact.  However, for points on different line segments, such as $\blambda_1$ and $\blambda_2$, the first-order Taylor series is a strict underestimation, and we can squeeze in a quadratic correction.  Said differently, \cref{eq:dual-strongly-convex-i} does not hold for all $\blambda_1$, $\blambda_2$, but holds for most $\blambda_1$, $\blambda_2$.  
In this sense, $\blambda \mapsto \L(\blambda, \bZ)$ is ``approximately" strongly-convex.  

To make a formal statement, it is more convenient to use a different, equivalent definition of strong-convexity.  \Cref{eq:dual-strongly-convex-i} is equivalent to the condition 
\begin{equation} \label{eq:StrongConvexityDefinition}
\big( \nabla_{\blambda} f(\blambda_1) - \nabla_{\blambda} f(\blambda_2) \big)^\top (\blambda_1 - \blambda_2) \geq \kappa \| \blambda_1 - \blambda_2\|_2^2   \quad \forall \blambda_1, \blambda_2 \in \text{Dom}(f).
\end{equation} 
\Cref{lem:ApproxStrongConvexCondition} then shows that  $\blambda \mapsto \L(\blambda, \bZ)$ is approximately strongly convex in the sense that, with high probability, there exists a $C >0$ such that 
\begin{align}\label{eq:strong-cvx-main-body}
	\big( \nabla_{\blambda} \L(\blambda_1, \bZ) - \nabla_{\blambda} \L(\blambda_2, \bZ) \big)^\top (\blambda_1 - \blambda_2) \ \geq \ &
C \| \blambda_1 - \blambda_2\|_2^2 
-  \frac{\| \blambda_1 - \blambda_2\|_2^{3/2}}{C\sqrt{n}}, \\ \notag
	\text{ for all } \| \blambda_1 \|_1 \le \lambda_{\max}, \ & \| \blambda_2 \|_1 \le \lambda_{\max}, \ \text{and} \
	\| \blambda_1 - \blambda_2 \|_2 \geq \frac{4}{n}
\end{align}
where $\lambda_{\max}$ is a dimension independent constant satisfying $\mathbb{E} \| \blambda(\bZ,\btheta) \|_1 \le \lambda_{\max}$.

\Cref{eq:strong-cvx-main-body} mirrors \cref{eq:StrongConvexityDefinition}.  As $n\rightarrow \infty$, \cref{eq:strong-cvx-main-body} reduces to the analogue of \cref{eq:StrongConvexityDefinition}  for $\| \blambda_1 \|_1, \| \blambda_2\|_1 \le \lambda_{\max}$.  Moreover, for $\| \blambda_1 - \blambda_2 \|_2  \gg \frac{1}{n}$, the first term on the left of \cref{eq:strong-cvx-main-body} above dominates the second, so that the right hand side is essentially a quadratic in $\| \blambda_1 - \blambda_2\|_2$.  These relations motivate our terminology ``approximately strongly-convex."

Using this notion of approximate strong-convexity, we show in \cref{lem:DualStability} that there exists a set $\mathcal E_n \subseteq \R^n$ such that $\bZ \in \mathcal E_n$ with high probability, and, more importantly, for any $\bz \in \mathcal E_n$, the dual solutions are stable, i.e., 
\begin{equation} \label{eq:HammingDual_mainbody}
	\| \blambda(\bz, \btheta) - \blambda(\overline \bz, \btheta)\|_2 \ \leq \frac{C}{n} \sum_{j=1}^n \Ib{ z_j \neq \overline z_j} \quad \forall \overline{\bz} \text{ s.t. } \| \blambda(\overline{\bz}) \|_1 \leq \lambda_{\max}.
\end{equation}
Equipped with this dual-stability condition, we can bound the average solution instability and the variance of $D(\bZ)$  as in \cref{thm:EfronStein}.  

However, since the above stability condition holds with high probability instead of in expectation, we can actually use a modification of McDiarmid's inequality (see \cref{thm:Combes-McDiarmid}) to prove the following, stronger tail bound: 
 \begin{lem}[Pointwise Convergence of VGC]\label[lem]{lem:VGC-pw-conv} Fix some $\btheta \in \Theta$.
 Under the assumptions of \cref{thm:WC-Constraints-bound}, there exists a dimension independent constants $C,n_0$
such that, for any $R> 1$ and $n\ge n_0 e^R$, we have with probability at least $1-4\exp(-R)$,
\[
\left|D(\bZ,\btheta)-\mathbb{E}\left[D(\bZ,\btheta)\right]\right|
\ \le \ 
	C V^3 \log^2(V) \frac{\log^4(n)\sqrt n}{h} \sqrt{R}. 
\]
\end{lem}

We then complete the proof of \cref{thm:WC-Constraints-bound} by i) invoking a covering argument over $\Theta$ to extend this tail bound to a uniform concentration for the VGC and ii) again leveraging dual stability to show the in-sample optimism (\cref{eq:ErrorExpansion_Optimism}) concentrates similarly.  See \cref{sec:appendix_pointwise_results} for the details. 

\subsection*{Comparison to Proof Technique to \citeGR~and \cite{li2019online}}
\citeGR~and \cite{li2019online} also analyze the behavior of $\blambda(\bZ)$ in the limit as $n\rightarrow \infty$.  The key to their analysis is observing that the function $\blambda \mapsto \Eb{\L(\blambda, \bZ)}$ is strongly-convex.    
Using this property, they prove that 
\begin{equation}\label{eq:exp-dual-sol-bnd-i}
	\| \blambda(\bZ) - \blambda^{*} \|_2 = \tilde{O}_p\left(\frac{1}{\sqrt{n}} \right)
\end{equation}
for some constant $\blambda^*$ that does not depend on the realization of $\bZ$. 

Our analysis via approximate strong-convexity takes a different perspective.  Specifically, instead of studying the function $\blambda \mapsto \Eb{\L(\blambda, \bZ)}$, we study the (random) function $\blambda \mapsto \L(\blambda, \bZ)$.  While more complex, this analysis permits a tighter characterization of the behavior of the dual variables. In particular, leveraging \cref{eq:HammingDual_mainbody}, one can prove a statement similar to \cref{eq:exp-dual-sol-bnd-i} (see \cref{lem:Uniform-Dual-sol-conv}), however, to the best of our knowledge, one cannot easily prove \cref{eq:HammingDual_mainbody} given the strong-convexity of $\blambda \mapsto \Eb{\L(\blambda, \bZ)}$ or  \cref{eq:exp-dual-sol-bnd-i}.  A simple triangle inequality from \cref{eq:exp-dual-sol-bnd-i} would suggest the much slower rate $ \|\blambda(\bZ) - \blambda(\bar{\bZ}) \| = O_p\left(\frac{1}{\sqrt n}\right)$.

It is an open question whether this tighter analysis might yield improved results for the online linear programming setting studied in \cite{li2019online}.

\section{Numerical Case Study: Drone-Assisted Emergency Medical Response}
\label{sec:experiments}
We reconsider \cref{ex:drone-aed} using real data from Ontario, Canada.  
Our data analysis and set-up largely mirror \cite{boutilier2019response}, however, unlike that work, our optimization formulation explicitly models response time uncertainty.  

\vspace{5pt}
\noindent \textbf{Data and Setup.}
Recall, our formulation decides the location of drone depots $(y_l)$ and dispatch rules ($x_{kl}$) where a dispatch rule  determines whether to send a drone from depot $l$ to location $k$ when requested.  Our objective is to minimize the expected response time
to out-of-hospital cardiac arrest (OHCA) events.  We consider $L = 31$ potential drone depot locations at existing paramedic, fire, and police stations in the Ontario area.  

To study the effect of problem dimension on our estimator, we vary the number of OHCA events via simulation similarly to \cite{boutilier2019response}.  
Specifically, we estimate a (spatial) probability density over Ontario for the occurrence of OHCA events using a kernel density estimator trained on $8$ years of historical OHCA events.  We then simulate $K$ (ranging from $50$ to $3{,}200$) events according to this density giving the locations $k$ used in our formulation.   

In our case-study, $\mu_{kl}$ represents the excess time a drone-assisted response takes over an ambulance-only response.  (This objective is typically negative).  We learn these constants by first training a $k$-nearest neighbors regression (kNN) for the historical ambulance response times to nearby OHCAs.  (For a sense of scale, the maximum ambulance response time is less than $1500s  = 25$ min.)
We estimate a drone response time based on the (straight-line) distance between $k$ and $l$ assuming an average speed of 27.8 m/s and 20s for take-off and landing (assumptions used in \cite{boutilier2019response}).  We then set $\mu_{kl}$ to the difference of the drone time minus the ambulance time.  These values are fixed constants throughout our simulations and range from $-3100$ seconds to $1200$ seconds.

We take $Z_{kl}$ be normally corrupted predictions of $\mu_{kl}$ where the precisions $\nu_{kl}$ are determined by bootstrapping.  Specifically, we take many bootstrap samples of our original historical dataset and refit the $k$-nearest neighbor regression and recompute an estimate of ambulance and drone response times.  The precision $\nu_{kl}$ is taken to be the reciprocal of the variance of these bootstrap replicates. Precisions range from $\nu_{\min} = 4 \times 10^{-7}$ to $\nu_{\max} = 2 \times 10^{-4}$. 



\vspace{5pt} \noindent \textbf{Policy Class}.
To determine dispatch rules for our case study, we consider the following policies:
\begin{align*}
	\bx(\bZ,\bm{W},(\tau,\mu_{0}))\in\arg \min_{\bx\in\mathcal{X}}\sum_{k=1}^{K}\sum_{j=1}^{L} \left(\frac{\nu_{jk}}{\nu_{jk}+\tau}Z_{jk}+\frac{\tau}{\nu_{jk}+\tau}\left(W_{jk} - \mu_{0}\right)\right)x_{jk},
\end{align*}
where $W_{jk}$ is the computed drone travel time between facility $j$ and OHCA $k$.  Our policy class consists of policies where $\tau \in [0,100\nu_{\min}]$ and $\mu_0 \in [0,1500]$.  Similar to the Mixed-Effects Policies from \cref{sec:plug-in-policy}, each policy is a weighted average between the SAA policy and a deterministic policy that dispatches to any location whose drone travel-time is at most $\mu_0$. 

For the first three experiments, we generate out-of-sample estimates using our VGC, the Stein-correction of \citeGR, and cross-validation using hold-out ($2$-fold) cross-validation. 
We assume that we are given two samples $\bZ^1, \bZ^2$ so that $Z_{jk} = \frac{1}{2} (Z_{jk}^1 + Z_{jk}^2)$.  
We set $h = n^{-1/6}$ for both the VGC and the Stein-Correction based on the recommendation of \citeGR.  
\blue{ In the last experiment, we are given one hundred samples $\bZ^i$ for $i=1,\dots,100$ where $Z_{jk} = \frac{1}{100} \sum_{i=1}^{100} Z_{jk}^i$ and generate out-of-sample estimates for $2$, $5$, $10$, $20$, $50$, and $100$ fold cross-validation. }
For ease of comparison, we present all results as a percentage relative to full-information optimal performance. 

\subsection{Results}
\begin{figure} 
\begin{center}
\scalebox{.45}{
\begin{tikzpicture}[x=1pt,y=1pt]
\definecolor{fillColor}{RGB}{255,255,255}
\path[use as bounding box,fill=fillColor,fill opacity=0.00] (0,0) rectangle (505.89,361.35);
\begin{scope}
\path[clip] (  0.00,  0.00) rectangle (505.89,361.35);
\definecolor{drawColor}{RGB}{255,255,255}
\definecolor{fillColor}{RGB}{255,255,255}

\path[draw=drawColor,line width= 0.6pt,line join=round,line cap=round,fill=fillColor] (  0.00, -0.00) rectangle (505.89,361.35);
\end{scope}
\begin{scope}
\path[clip] ( 47.09, 44.99) rectangle (500.39,355.85);
\definecolor{drawColor}{gray}{0.92}

\path[draw=drawColor,line width= 0.3pt,line join=round] ( 47.09, 94.45) --
	(500.39, 94.45);

\path[draw=drawColor,line width= 0.3pt,line join=round] ( 47.09,165.10) --
	(500.39,165.10);

\path[draw=drawColor,line width= 0.3pt,line join=round] ( 47.09,235.75) --
	(500.39,235.75);

\path[draw=drawColor,line width= 0.3pt,line join=round] ( 47.09,306.40) --
	(500.39,306.40);

\path[draw=drawColor,line width= 0.3pt,line join=round] ( 91.61, 44.99) --
	( 91.61,355.85);

\path[draw=drawColor,line width= 0.3pt,line join=round] (197.01, 44.99) --
	(197.01,355.85);

\path[draw=drawColor,line width= 0.3pt,line join=round] (307.48, 44.99) --
	(307.48,355.85);

\path[draw=drawColor,line width= 0.3pt,line join=round] (417.94, 44.99) --
	(417.94,355.85);

\path[draw=drawColor,line width= 0.6pt,line join=round] ( 47.09, 59.12) --
	(500.39, 59.12);

\path[draw=drawColor,line width= 0.6pt,line join=round] ( 47.09,129.77) --
	(500.39,129.77);

\path[draw=drawColor,line width= 0.6pt,line join=round] ( 47.09,200.42) --
	(500.39,200.42);

\path[draw=drawColor,line width= 0.6pt,line join=round] ( 47.09,271.07) --
	(500.39,271.07);

\path[draw=drawColor,line width= 0.6pt,line join=round] ( 47.09,341.72) --
	(500.39,341.72);

\path[draw=drawColor,line width= 0.6pt,line join=round] (144.31, 44.99) --
	(144.31,355.85);

\path[draw=drawColor,line width= 0.6pt,line join=round] (249.72, 44.99) --
	(249.72,355.85);

\path[draw=drawColor,line width= 0.6pt,line join=round] (365.23, 44.99) --
	(365.23,355.85);

\path[draw=drawColor,line width= 0.6pt,line join=round] (470.64, 44.99) --
	(470.64,355.85);
\definecolor{drawColor}{RGB}{230,159,0}

\path[draw=drawColor,line width= 1.7pt,dash pattern=on 1pt off 3pt on 4pt off 3pt ,line join=round] ( 77.07,250.73) --
	(143.57,151.63) --
	(210.08,143.42) --
	(276.58,125.03) --
	(343.09,117.00) --
	(409.59,110.62) --
	(476.10,112.74);
\definecolor{drawColor}{RGB}{86,180,233}

\path[draw=drawColor,line width= 1.7pt,dash pattern=on 4pt off 4pt ,line join=round] ( 77.81,192.89) --
	(144.31,124.05) --
	(210.81,113.84) --
	(277.32,100.25) --
	(343.82, 93.63) --
	(410.33, 83.56) --
	(476.83, 78.08);
\definecolor{drawColor}{RGB}{0,158,115}

\path[draw=drawColor,line width= 1.7pt,line join=round] ( 78.54,188.47) --
	(145.05,119.47) --
	(211.55,106.29) --
	(278.06, 94.30) --
	(344.56, 87.41) --
	(411.06, 80.13) --
	(477.57, 74.87);
\definecolor{fillColor}{RGB}{0,158,115}

\path[fill=fillColor] ( 73.90,183.83) --
	( 83.18,183.83) --
	( 83.18,193.11) --
	( 73.90,193.11) --
	cycle;
\definecolor{fillColor}{RGB}{86,180,233}

\path[fill=fillColor] ( 77.81,200.11) --
	( 84.05,189.28) --
	( 71.56,189.28) --
	cycle;
\definecolor{fillColor}{RGB}{230,159,0}

\path[fill=fillColor] ( 77.07,250.73) circle (  4.64);
\definecolor{fillColor}{RGB}{0,158,115}

\path[fill=fillColor] (140.41,114.83) --
	(149.69,114.83) --
	(149.69,124.11) --
	(140.41,124.11) --
	cycle;
\definecolor{fillColor}{RGB}{86,180,233}

\path[fill=fillColor] (144.31,131.26) --
	(150.56,120.44) --
	(138.06,120.44) --
	cycle;
\definecolor{fillColor}{RGB}{230,159,0}

\path[fill=fillColor] (143.57,151.63) circle (  4.64);
\definecolor{fillColor}{RGB}{0,158,115}

\path[fill=fillColor] (206.91,101.65) --
	(216.19,101.65) --
	(216.19,110.93) --
	(206.91,110.93) --
	cycle;
\definecolor{fillColor}{RGB}{86,180,233}

\path[fill=fillColor] (210.81,121.05) --
	(217.06,110.23) --
	(204.57,110.23) --
	cycle;
\definecolor{fillColor}{RGB}{230,159,0}

\path[fill=fillColor] (210.08,143.42) circle (  4.64);
\definecolor{fillColor}{RGB}{0,158,115}

\path[fill=fillColor] (273.42, 89.66) --
	(282.70, 89.66) --
	(282.70, 98.94) --
	(273.42, 98.94) --
	cycle;
\definecolor{fillColor}{RGB}{86,180,233}

\path[fill=fillColor] (277.32,107.46) --
	(283.57, 96.64) --
	(271.07, 96.64) --
	cycle;
\definecolor{fillColor}{RGB}{230,159,0}

\path[fill=fillColor] (276.58,125.03) circle (  4.64);
\definecolor{fillColor}{RGB}{0,158,115}

\path[fill=fillColor] (339.92, 82.77) --
	(349.20, 82.77) --
	(349.20, 92.05) --
	(339.92, 92.05) --
	cycle;
\definecolor{fillColor}{RGB}{86,180,233}

\path[fill=fillColor] (343.82,100.85) --
	(350.07, 90.03) --
	(337.58, 90.03) --
	cycle;
\definecolor{fillColor}{RGB}{230,159,0}

\path[fill=fillColor] (343.09,117.00) circle (  4.64);
\definecolor{fillColor}{RGB}{0,158,115}

\path[fill=fillColor] (406.43, 75.49) --
	(415.70, 75.49) --
	(415.70, 84.77) --
	(406.43, 84.77) --
	cycle;
\definecolor{fillColor}{RGB}{86,180,233}

\path[fill=fillColor] (410.33, 90.78) --
	(416.58, 79.95) --
	(404.08, 79.95) --
	cycle;
\definecolor{fillColor}{RGB}{230,159,0}

\path[fill=fillColor] (409.59,110.62) circle (  4.64);
\definecolor{fillColor}{RGB}{0,158,115}

\path[fill=fillColor] (472.93, 70.23) --
	(482.21, 70.23) --
	(482.21, 79.51) --
	(472.93, 79.51) --
	cycle;
\definecolor{fillColor}{RGB}{86,180,233}

\path[fill=fillColor] (476.83, 85.29) --
	(483.08, 74.47) --
	(470.58, 74.47) --
	cycle;
\definecolor{fillColor}{RGB}{230,159,0}

\path[fill=fillColor] (476.10,112.74) circle (  4.64);
\end{scope}
\begin{scope}
\path[clip] (  0.00,  0.00) rectangle (505.89,361.35);
\definecolor{drawColor}{RGB}{0,0,0}

\node[text=drawColor,anchor=base east,inner sep=0pt, outer sep=0pt, scale=  1.80] at ( 42.14, 52.92) {0};

\node[text=drawColor,anchor=base east,inner sep=0pt, outer sep=0pt, scale=  1.80] at ( 42.14,123.57) {20};

\node[text=drawColor,anchor=base east,inner sep=0pt, outer sep=0pt, scale=  1.80] at ( 42.14,194.22) {40};

\node[text=drawColor,anchor=base east,inner sep=0pt, outer sep=0pt, scale=  1.80] at ( 42.14,264.87) {60};

\node[text=drawColor,anchor=base east,inner sep=0pt, outer sep=0pt, scale=  1.80] at ( 42.14,335.52) {80};
\end{scope}
\begin{scope}
\path[clip] (  0.00,  0.00) rectangle (505.89,361.35);
\definecolor{drawColor}{gray}{0.20}

\path[draw=drawColor,line width= 0.6pt,line join=round] ( 44.34, 59.12) --
	( 47.09, 59.12);

\path[draw=drawColor,line width= 0.6pt,line join=round] ( 44.34,129.77) --
	( 47.09,129.77);

\path[draw=drawColor,line width= 0.6pt,line join=round] ( 44.34,200.42) --
	( 47.09,200.42);

\path[draw=drawColor,line width= 0.6pt,line join=round] ( 44.34,271.07) --
	( 47.09,271.07);

\path[draw=drawColor,line width= 0.6pt,line join=round] ( 44.34,341.72) --
	( 47.09,341.72);
\end{scope}
\begin{scope}
\path[clip] (  0.00,  0.00) rectangle (505.89,361.35);
\definecolor{drawColor}{gray}{0.20}

\path[draw=drawColor,line width= 0.6pt,line join=round] (144.31, 42.24) --
	(144.31, 44.99);

\path[draw=drawColor,line width= 0.6pt,line join=round] (249.72, 42.24) --
	(249.72, 44.99);

\path[draw=drawColor,line width= 0.6pt,line join=round] (365.23, 42.24) --
	(365.23, 44.99);

\path[draw=drawColor,line width= 0.6pt,line join=round] (470.64, 42.24) --
	(470.64, 44.99);
\end{scope}
\begin{scope}
\path[clip] (  0.00,  0.00) rectangle (505.89,361.35);
\definecolor{drawColor}{RGB}{0,0,0}

\node[text=drawColor,anchor=base,inner sep=0pt, outer sep=0pt, scale=  1.80] at (144.31, 27.65) {100};

\node[text=drawColor,anchor=base,inner sep=0pt, outer sep=0pt, scale=  1.80] at (249.72, 27.65) {300};

\node[text=drawColor,anchor=base,inner sep=0pt, outer sep=0pt, scale=  1.80] at (365.23, 27.65) {1,000};

\node[text=drawColor,anchor=base,inner sep=0pt, outer sep=0pt, scale=  1.80] at (470.64, 27.65) {3,000};
\end{scope}
\begin{scope}
\path[clip] (  0.00,  0.00) rectangle (505.89,361.35);
\definecolor{drawColor}{RGB}{0,0,0}

\node[text=drawColor,anchor=base,inner sep=0pt, outer sep=0pt, scale=  1.80] at (273.74,  9.00) {OHCA Events ($K$)};
\end{scope}
\begin{scope}
\path[clip] (  0.00,  0.00) rectangle (505.89,361.35);
\definecolor{drawColor}{RGB}{0,0,0}

\node[text=drawColor,rotate= 90.00,anchor=base,inner sep=0pt, outer sep=0pt, scale=  1.80] at ( 17.90,200.42) {Abs. Bias (\% Full-Info)};
\end{scope}
\end{tikzpicture}}
\scalebox{.45}{\input{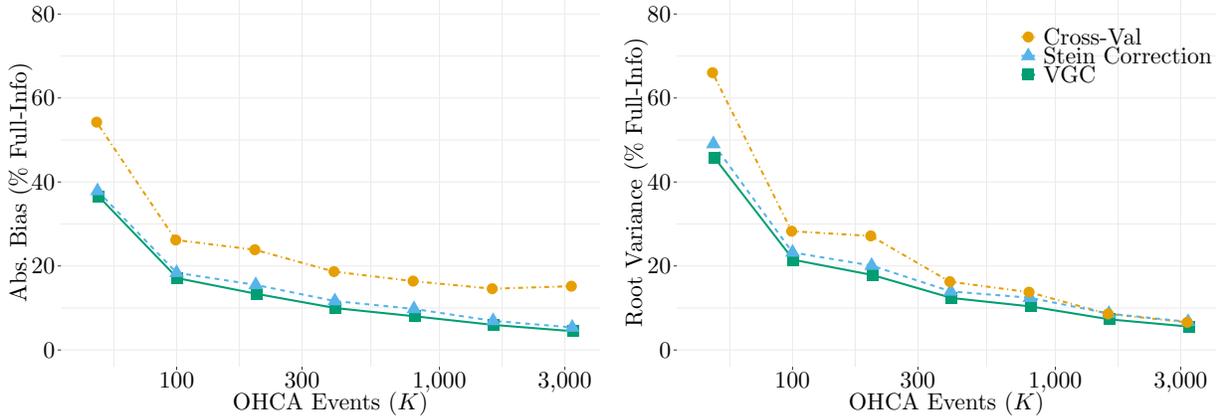}}
\end{center}
\caption{\label{fig:estimator-performance}  \textbf{Bias and variance as $K \rightarrow \infty$.} The two graphs plot the bias and variance of the different out-of-sample performance estimators for the Sample Average Approximation (SAA) policy.  The bias and variance were estimated across 100 simulations for each $K$. Although variance vanishes for all methods as $K$ increases, cross-validation exhibits a non-vanishing bias and is uniformly worse for all $K$.  
}
\end{figure}

\begin{figure}
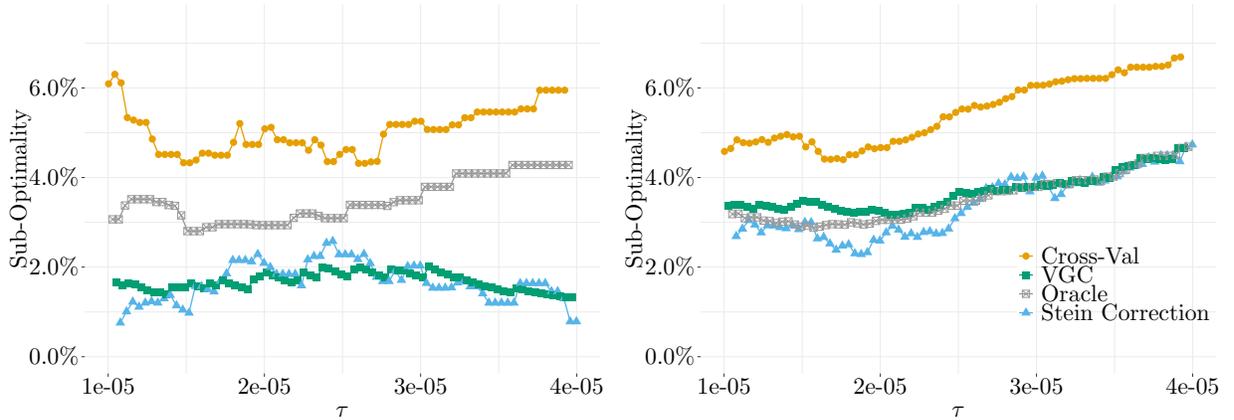
 
\begin{center}
\scalebox{.45}{\input{Figures/Size_25_estimates.tex}}
\scalebox{.45}{\input{Figures/Size_100_estimates.tex}}
\end{center}
\caption{\label{fig:policy-size}  \textbf{Estimating Performance across Policy Class.}
The first graph shows the estimates of out-of-sample performance across the policy class for the parameter $\tau \in [25\nu_{\min},100\nu_{\min}]$ and $\mu_0 = 1000$ for one sample path when $K = 400$.  The second graph is similar, but for $K = 3200$.  Both plots  highlight the smoothness of VGC relative to the Stein-Correction.
}
\end{figure}

In our first experiment, we evaluate the bias and square root variance of each method for the out-of-sample performance of the SAA policy ($\tau = 0$)  as $K$, the number of OHCA events grows (see \cref{fig:estimator-performance}).  
As predicted by our theoretical analysis, the quality of the out-of-sample estimates improve as we increase the problem size for both the VGC and the Stein Correction. However, cross-validation incurs a non-vanishing bias because it only leverages half the data in training.  


As a second experiment, in \cref{fig:policy-size}, we can observe the quality of the estimators over multiple policies in the policy class. We highlight the smoothness of the VGC as  $\tau$ varies.  Since, for large $K$, the true performance is quite smooth, the worst-case error of VGC is generally smaller than that of the Stein Correction.  
We also note that while it appears both Stein and VGC systematically over-estimate performance, this is an artifact of the particular sample path chosen.  By contrast, cross-validation does systematically underestimate performance, because it estimates the performance of a policy trained with half the data, which is necessarily worse.  

\begin{figure} 
\begin{center}
\scalebox{.48}{\includegraphics[width=1\textwidth]{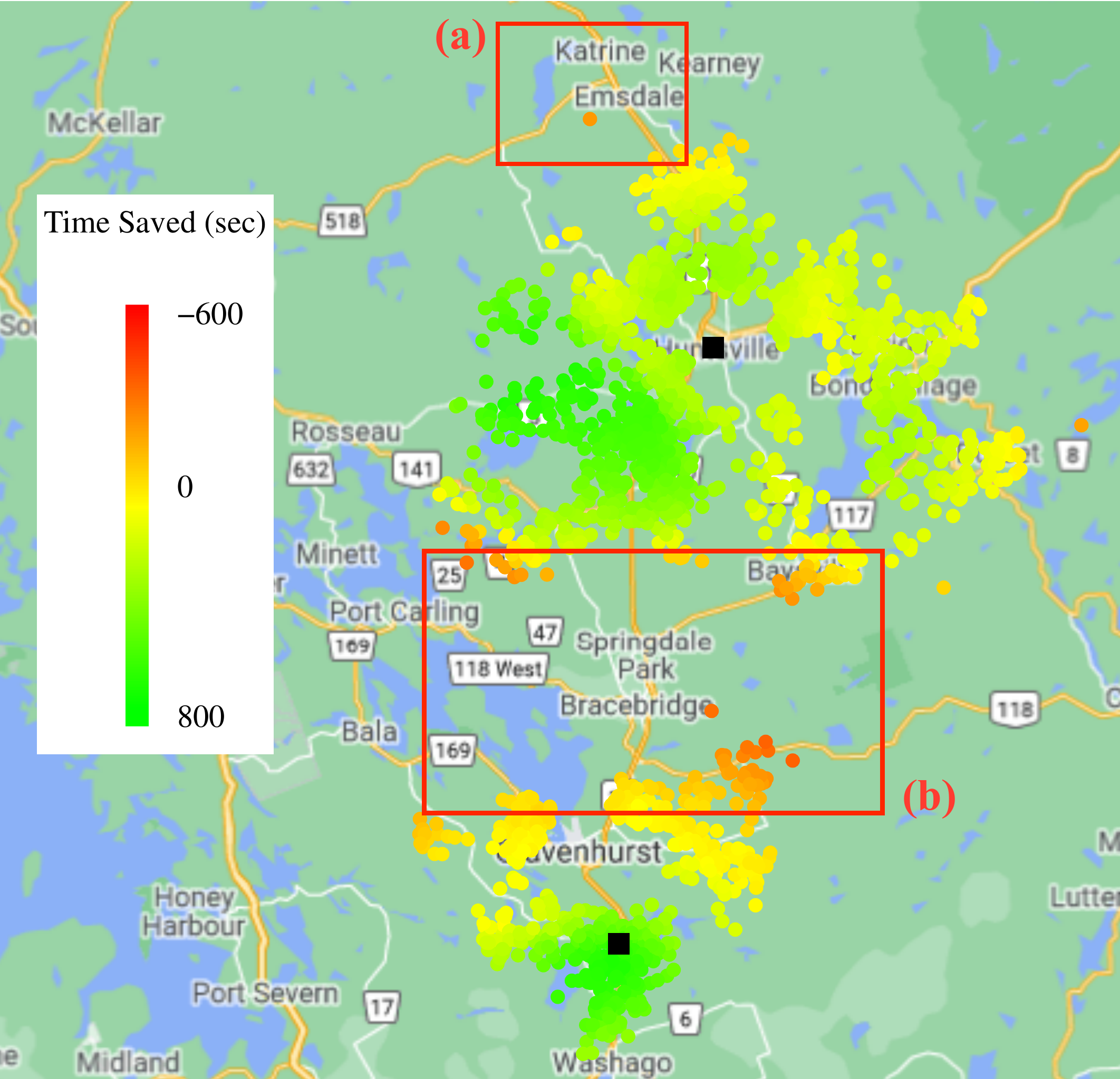}}
\scalebox{.48}{\includegraphics[width=1\textwidth]{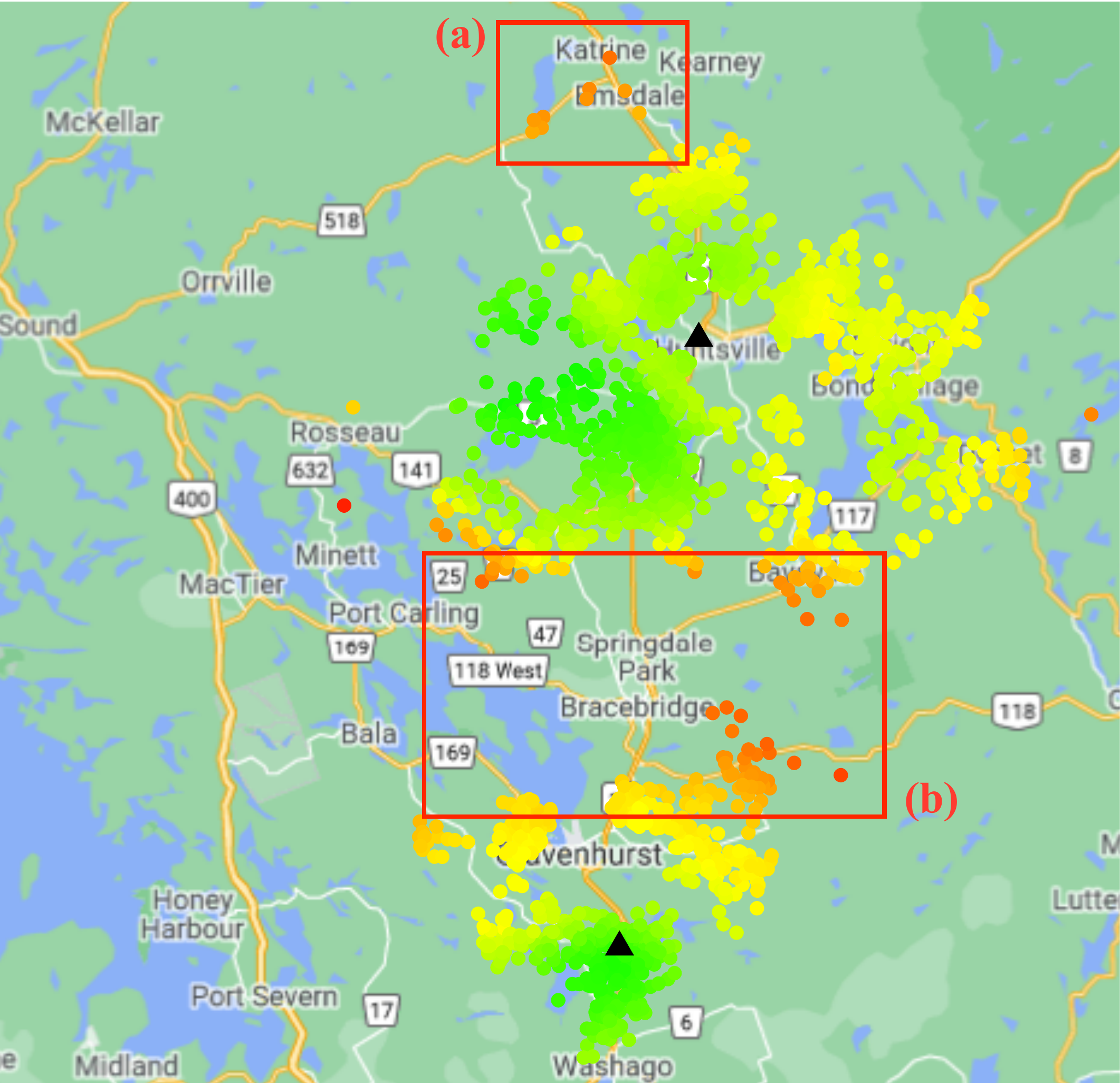}}
\end{center}
\caption{\label{fig:aed-policy-comp}  \textbf{Comparing Policy Decisions.}  Left (resp. right) panel shows routing decisions for the policy selected by VGC (resp. cross-validation).  Color indicates time-saved relative to an ambulance-only policy (green is good, red is bad) computed relative to the ground truth.  
Although routing is largely similar, Regions $(a)$ and $(b)$ highlight some differences where the cross-validation policy makes poorer routing decisions (more orange dots). The larger black points are drone depots.
}
\end{figure}

\begin{figure}
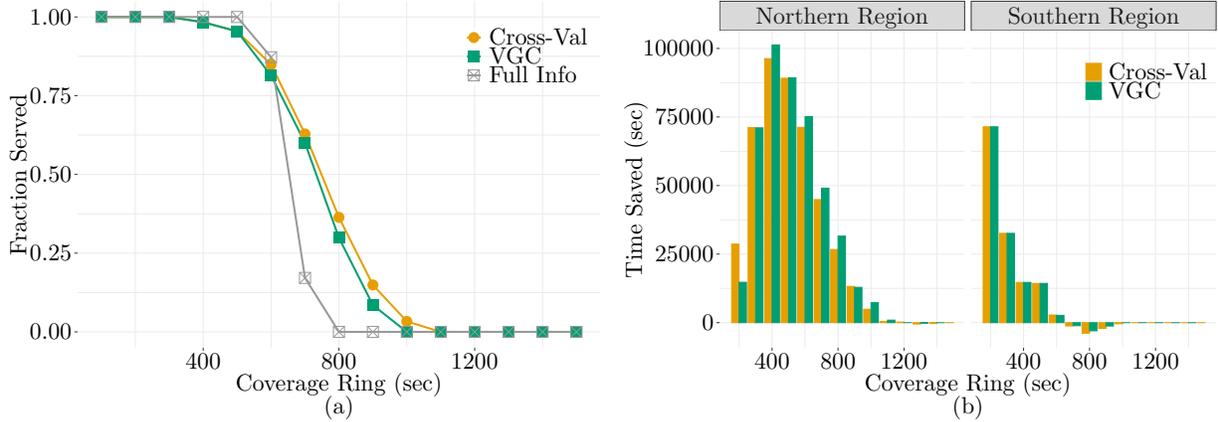
 
\begin{center}
\scalebox{.45}{\input{Figures/plot_frac_served.tex}}
\scalebox{.45}{\input{Figures/plot_time_saved.tex}}
\end{center}
\caption{\label{fig:aed-policy-comp-ii}  \textbf{Estimating Performance across Policy Class.}
Each data point in the graph represents the performance metric of each selected policy for a ring shaped region corresponding to distance in time from a drone depot. 
Graph $(a)$ shows the fraction of patients served in each region for the patients serviced by the Southern depot in \cref{fig:aed-policy-comp}. Graph $(b)$ plots the time saved by each policy. The plots highlight the performance difference in routing decisions between the two policies.
}
\end{figure}

In our third experiment, we highlight the differences in the policy selected by the VGC estimator and the policy selected by cross-validation. In \cref{fig:aed-policy-comp}, we plot the routing decisions of each policy and color code them by the true (oracle) amount of time saved. We highlight two regions (labeled $(a)$ and $(b)$) on the map where drones arrive after the ambulance. We see that in those regions, the cross-validation policy routes to more patients/regions where the drone arrives after the ambulance, thus potentially wasting drone resources and availability. We see in the regions outside of $(a)$ and $(b)$ that routing decisions of the two polices are similar and result in the drone arriving before the the ambulance. Additionally, we note that the drone depots in the southern region of the map are the same for both policies while the drone depots in the northern region are different.

 In \cref{fig:aed-policy-comp-ii}, we plot key performance metrics for regions organized by their distance from a drone depot for these two policies. Specifically, we group OHCA events into ``coverage rings" based on the travel time from the depot to the event.  Each ring is of ``width" 100s. For example, the 800 seconds coverage ring corresponds to all OHCA events that are between 701 and 800 seconds away from the drone depot. In Panel $(a)$ of \cref{fig:aed-policy-comp-ii}, we 
 restrict attention to the the southern region  \cref{fig:aed-policy-comp} where both policies have selected the same drone depot so that we can focus on routing decisions.  We  plot the fraction of patients served for each coverage ring.  We see that the policy chosen by VGC is more conservative with routing in comparison to the policy chosen by cross-validation and  more closely aligns the full information benchmark. 
 
 In panel $(b)$ of \cref{fig:aed-policy-comp-ii}, we compare the time saved between the two policies. We organize the regions into the North and South corresponding to the servicing depots. In the northern region, we see that the VGC policy saves more time in coverage rings further away from the drone depot by sacrificing time saved in closer coverage rings. This difference partially corresponds with the more conservative routing decisions of the VGC, but also can be attributed to the choice of drone depot. We see the VGC policy chooses a depot in less densely populated region
that is more centralized overall, while the cross-validation chooses a depot closer to more densely populated regions in terms of OHCA occurrences. In total, we see that the VGC policy saves $1.43\%$ more time in comparison to the cross-validation policy. However, if we breakdown the time saved with respect to minimum distance from a depot, we see that for patients within 600 seconds of an existing drone depot, the VGC performs less than $1\%$ worse in comparison to the cross-validation policy. However, if we consider patients more than 600 seconds away from existing drone depot, the VGC policy saves $13.8\%$ more time in comparison to cross-validation. We interpret this to mean that both the VGC and cross-validation policies make similar performing depot decisions, but VGC makes significantly better routing decisions, particularly at long distances.  Since these long distances are precisely where the imprecision is most crucial, we argue this is a relevant advantage. 

\blue{\label{note:cross-val-exp-discuss}Finally, we also compare how higher fold cross-validation performs with respect to VGC. In (a) of \cref{fig:cross-val-perf}, we first show how the cross-validation estimators for three different policies performs in estimation error relative to VGC as we vary the number of folds. For all the policies (where increasing $\tau$ corresponds to lower variance policies), the plot shows the root mean-squared error (MSE) of cross-validation is uniformly larger than VGC over the folds considered. Furthermore, the plot highlights a drawback of cross-validation, which is that it is not even clear how to select the optimal number of folds in order to minimize the bias variance trade-off of the cross-validation estimator. In comparison, the VGC with minimal guidance on the choice of the $h$ parameter out-performs common choices of folds for cross-validation such as 5-fold cross-validation and leave-one-out cross-validation. In (b) of \cref{fig:cross-val-perf}, we show how the performance of policy evaluation translates to policy learning in regions further away from the depot, where drone decisions are most crucial.  As expected, VGC out-performs cross-validation across all choices of folds and the fold that performs most similarly to VGC, $20$-fold, also corresponds to the cross-validation fold with the lowest MSE. This observation suggests with an ``ideal" number of folds, cross-validation can perform well in this example, but identifying the right number of folds is non-trivial.}

\begin{figure}
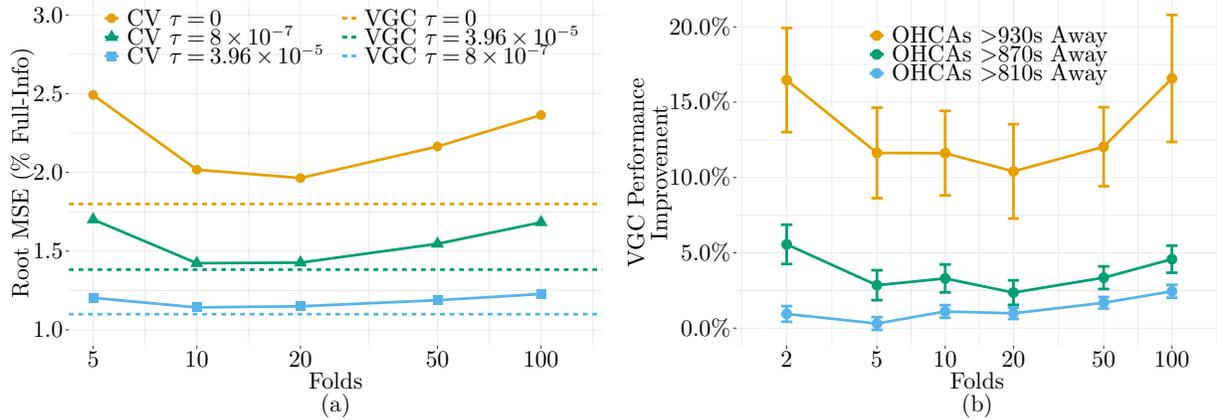
 
\begin{center}
\scalebox{.45}{\input{Figures/cv_folds_mse.tex}}
\scalebox{.45}{\input{Figures/vgc_perf_improve.tex}}
\end{center}
\caption{\label{fig:cross-val-perf}  \textbf{Varying Cross-Validation Folds.} We plot the policy evaluation and learning performance of cross-validation with different folds across 500 simulations. In each simulation there are 100 samples of $\bZ^i$. In (a), we plot the mean squared error of cross-validation for three different policies and compare them with the respective VGC estimates represented by the dotted lines. In (b), we plot the percent improvement VGC has over cross-validation, so larger bars indicate lower cross-validation performance.}
\end{figure}

\section{Conclusion}
Motivated by the poor performance of cross-validation in data-driven optimization problems where data are scarce, we propose a new estimator of the out-of-sample performance
\blue{of an affine plug-in policy}.  Unlike \mbox{cross-validation}, our estimator avoids sacrificing data and uses all the data when training, making it well-suited to settings with scarce data. We prove that our estimator is nearly unbiased, and for ``stable'' optimization problems -- problems whose optimal solutions do not change too much when the coefficient of a single random component changes -- the estimator's variance vanishes. Our notion of stability leads us to consider two special classes of weakly-coupled optimization problems: weakly-coupled-by-variables and weakly-coupled-by-constraints. For each class of problems, we prove an even stronger result and provide high-probability bounds on the error of our estimator that holds uniformly over a policy class. Additionally, in our analysis of optimization problems weakly-coupled-by-constraints, we provide new insight on the stability of the dual solutions. This new insight may provide further insight in problems that leverage the dual solution such as online linear programming.

Our work offers many exciting directions for future research.  Our solution approach to the weakly-coupled problems exploits the decomposability of the underlying optimization problems. We believe that such an approach can be generalized to other settings. \blue{Finally, 
our analysis strongly leverages the linearity of the affine plug-in policy class; it is an open question if similar debiasing techniques might be developed to handle nonlinear objective functions as well. \label{nonlinear}}

\ACKNOWLEDGMENT{
The first two author were partially supported by the National Science Foundation under Grant No. CMMI-1661732.
All three authors thank Prof. Tim Chan and Prof. Justin Boutilier for sharing simulation results and details pertaining to their paper \citep{boutilier2019response}, and the two anonymous reviewers who provided feedback on an earlier revision.
}

{
\bibliographystyle{ormsv080} 
\setlength{\bibsep}{-3pt plus .3ex}
{	\footnotesize
\bibliography{references.bib} 
}
}

%

\newpage

\newpage

\ECSwitch
\ECDisclaimer
\ECHead{
\begin{center}
Online Appendix:
\vspace{8pt} Debiasing In-Sample Policy Performance for Small-Data,  Large-Scale Optimization
\end{center}
}

\begin{APPENDICES}
\section{Background Results on Empirical Processes} 
\label{sec:background_results_on_empirical_processes}

In this appendix we collect some results on the suprema of empirical processes that we will require in our proofs.  All results are either known or easily derived from known results.  Our summary is necessarily brief and we refer the reader to \cite{pollard1990empirical} for a self-contained exposition.  

Let $\Psi (t) = \frac{1}{5} \exp ( t^2 )$. For any real-valued random variable $Z$, we define the $\Psi$-norm $\left\Vert Z \right\Vert_{\Psi}$ to be $\left\Vert Z \right\Vert_{\Psi} \equiv \inf \left\{C > 0 : \mathbb{E}[ \Psi ( \abs{Z} / C )] \le 1 \right\}$. 
Random variables with finite $\Psi$-norm are sub-Gaussian random variables.
We first recall a classical result on the suprema of sub-Gaussian processes over finite sets.

\begin{thm}[Suprema of Stochastic Processes over Finite Sets] \label{thm:Pollard} 
Let
 \[ \bm{f}(\btheta)=(f_{1}(\btheta),\dots,f_{K}(\btheta))\in\mathbb{R}^{K} \]
be a vector of $K$ independent stochastic processes indexed by $\btheta \in \Theta$.
Let 
$\bm{F} \in \mathbb{R}^K_{+}$ be a random vector such that $\abs{f_{k}(\btheta)} \le F_{k}$ for all $\btheta \in \Theta$, $k=1,\dots ,K$, and suppose there exists a constant $M < \infty$ such that  $\abs{\left\{ \bm{f}(\btheta) : \btheta \in \Theta \right\}} \le M$ almost surely. Then, for any $ R > 1$, there exists an absolute constant $C$ such that with probability $1-e^{-R}$
\[
	\sup_{\btheta \in \Theta} \abs{\sum_{k=1}^K f_{k}(\btheta) - \Eb{ \sum_{k=1}^K f_{k}(\btheta) } }
	\le 
	C \cdot  R \cdot 
	\left\Vert \left\Vert \bm{F} \right\Vert_{2} \right\Vert_{\Psi}
	\sqrt{ \log M}. 
\]
\end{thm}
\begin{proof}{\hspace{-12pt}Proof:}
The result follows from the discussion leading up to Eq. (7.4) of \cite{pollard1990empirical} after noting that the entropy integral ($J_n(\omega)$ in the notation of \cite{pollard1990empirical}) is at most $ 9 \| \bm F \|_2 \sqrt{ \log M }$ given the conditions of the theorem.  
\hfill \Halmos \end{proof}

When considering the suprema over potentially infinite sets, we must characterize the ``size" of $\{ \bm f(\btheta) : \btheta \in \Theta \}$ more carefully.  
Recall for any set $\mathcal F$, the $\epsilon$-packing number of $\mathcal F$ is
the largest number of points we can select in $\mathcal F$ such that no two points are within $\ell_2$ distance $\epsilon$.  We denote this packing number by $D(\epsilon, \mathcal F)$.  We restrict attention to sets whose packing numbers do not grow too fast.  
\begin{defn}[Euclidean Sets]
We say a set $\mathcal F$ is \emph{Euclidean} if 
there exists constants $A$ and $W$ such that 
\[
	D(\epsilon \delta, \mathcal F) \ \leq A \epsilon^{-W} \quad \forall 0 < \epsilon < 1,
\] 
where $\delta \equiv \sup_{\bm f \in \mathcal F} \| \bm f\|$.
\end{defn}

Furthermore, note that in the special case that $F_k \leq U$, \cref{thm:Pollard} bounds the suprema by a term that scales like $U \sqrt{K}$.  This bound can be quite loose since $f_k(\btheta)$ typically takes values much smaller than $U$ and is only occasionally large.  Our next result provides a more refined bound on the suprema when the pointwise variance of the process is relatively small and the relevant (random) set is Euclidean. We stress the parameters $A$ and $W$ below must be deterministic.

\begin{thm}[Suprema of Stochastic Processes with Small Variance]\label{thm:SmallVariancePollard}
Suppose that the set $\{ \bm f(\btheta) : \btheta \in \Theta \} \subseteq \R^K$ is Euclidean with parameters $A$ and $W$ almost surely.
Suppose also that 
\begin{enumerate}[label=\roman*)]
	\item There exists a constant $U$ such that $\sup_{\btheta \in \Theta} \| \bm f(\btheta)\|_\infty \ \leq U$, almost surely, and
	\item There exists a constant $\sigma^2$ such that $\sup_{\btheta \in \Theta} \Eb{ \| \bm f(\btheta)\|_2^2} \ \leq \ K \sigma^2$.
\end{enumerate}
Then, there exists an absolute constant $C$ such that for any $R > 1$, with probability at least $1- e^{-R}$, 
\[
	\sup_{\btheta \in \Theta} \abs{ 
	\sum_{k=1}^K f_k(\btheta) - \Eb{\sum_{k=1}^K f_k(\btheta)}
	} \ \leq  \ C R \cdot V(A,W) \sqrt{K}
	\left(\sigma +  \frac{U V(A,W)}{ \sqrt{K}} \right),
\]
where $V(A,W) \equiv \frac{ \log A + W}{\sqrt{\log A}}$.
\end{thm}
\begin{rem} \rm
Notice that when $K$ is sufficiently large, the term in the parenthesis is dominated by $2 \sigma$, and hence the bound does not depend on $U$.
\Cref{thm:SmallVariancePollard} is not tight in its dependence on $R$.  See, for example, Talagrand's Inequality for the suprema of the empirical process (\cite{wainwright2019high}).  We prefer \cref{thm:SmallVariancePollard} to Talagrand's inequality in what follows because it is somewhat easier to apply and is sufficient for our purposes.  
\end{rem}

\proof{Proof of \cref{thm:SmallVariancePollard}.}
For convenience in what follows, let $\mathcal F$ be the (random) set $ \left\{\bm f(\btheta) : \btheta \in \Theta \right\}$.  Let $\delta \equiv \sup_{\bm f \in \mathcal F} \| \bm f\|_2$.  

Our goal will be to apply Theorem A.2 of \citeGR.  That theorem shows that there exists an absolute constant $C_1$ such that 
\begin{equation}\label{eq:SupremaBoundPsi2}
	\sup_{\bm f \in \mathcal F} \abs{ 
	\sum_{k=1}^K f_k - \Eb{\sum_{k=1}^K f_k}
	} \ \leq C_1 R V(A, W) \| \delta \|_\Psi.	
\end{equation}
Thus, the remainder of the proof focuses on bounding $\| \delta \|_\Psi$. As an aside, a naive bound $ \| \delta \|_\Psi \ \leq U \sqrt K$, so we know that this value is finite.  In what follows, we seek a stronger bound.  

Write
\begin{align*}
\delta^2 
&\ = \ 
	\sup_{\bm f \in \mathcal F} \| \bm f \|_2^2
\\ 
& \ \leq \ 
	\sup_{\bm f \in \mathcal F} \| \bm f\|_2^2 - \Eb{ \| \bm f \|_2^2} + \sup_{\bm f \in \mathcal F} \Eb{ \| \bm f\|_2^2} 
\\
& \ \leq \ 
	\sup_{\bm f \in \mathcal F} \| \bm f\|_2^2 - \Eb{ \| \bm f \|_2^2} + K \sigma^2 
\end{align*}

Let $C_2 > 0$ be a constant to be determined later.  Dividing by $C_2$ and taking expectations above shows
\begin{equation} \label{eq:BoundingPsi2Norm_delta}
	\Eb{ e^{\frac{\delta^2}{ C_2 } } } 
	\ \leq \ 
	e^{\frac{K \sigma^2}{C_2}} \cdot \Eb{ e^\frac{\bar Z}{C_2}},
\end{equation}
where
\[
	\bar Z \ \equiv \ \sup_{\bm f \in \mathcal F} \left\{ \| \bm f \|^2_2 - \Eb{ \| \bm f \|_2^2} \right\}
\ = \ 
	\sup_{\bm f \in \mathcal F} \left\{ \sum_{k=1}^K f_j^2 - \Eb{f_j^2} \right\}
\]
Importantly, $\bar Z$ is again the suprema of an empirical process, namely for the ``squared'' elements.  \cite{pollard1990empirical} provides bounds on $\bar Z$ in terms of the entropy integral of the process.  

Specifically, let
$\bm f^2$ denote the vector whose $j^\text{th}$ element is $f_j^2$.  Let $\mathcal F^2 = \{ \bm f^2 : \bm f \in \mathcal F\}$.  Then the entropy integral of the squared process is defined to be
\[
	\overline J 
	\ \equiv \ 9 \int_0^{\overline{\delta}} \sqrt{ \log D(x, \mathcal F^2) } dx,
\]
where $\overline \delta \equiv \sup_{\bm f \in \mathcal F} \| \bm f^2 \|_2$.

Then, in the discussion just prior to Eq. (7.4) of \cite{pollard1990empirical}, it is proven that 
\begin{equation} \label{eq:Psi1Norm_Zbar}
	\Eb{ e^{ \bar Z / \| \bar J\|_\Psi}} \ \leq 25.
\end{equation}
Hence, to bound the right side of \cref{eq:BoundingPsi2Norm_delta}, we will next bound $\| \bar J\|_\Psi$.  This in turn will allow us to bound $\| \delta \|_\Psi$ and invoke Theorem A.2 of \citeGR.

To this end, observe that for any $\bm f, \bm g \in \mathcal F$, we have 
\[
\| \bm f^2 - \bm g^2 \|^2 \ = \ \sum_{k=1}^K (f_j^2 - g_j^2)^2 \ = \sum_{k=1}^K (f_j + g_j)^2 (f_j - g_j)^2 
\ \leq \ 
4 U^2 \| \bm f - \bm g \|^2.
\]
Hence, 
\(
D(\epsilon, \mathcal F^2) \ \leq \ D\left( \frac{\epsilon}{2 U}, \mathcal F \right).
\)
Write
\begin{align*}
\overline J 
& 
	\ \equiv \ 9 \int_0^{\overline{\delta}} \sqrt{ \log D(x, \mathcal F^2) } dx
\\
& 
	\ \leq \  9 \int_0^{\overline{\delta}} \sqrt{ \log D\left( \frac{x}{2U}, \mathcal F \right) } dx.
\\
& 
	\ = \  2 U \cdot 9 \int_0^{\frac{\overline \delta}{2U}} \sqrt{ \log D\left( x , \mathcal F \right) } dx.
\end{align*}
where the last equality is a change of variables.  We now claim we can upper bound this last expression by replacing the upper limit of integration with $\delta$.   Indeed, if $\frac{\overline \delta}{2U} \leq \delta$, then because the integrand is nonnegative, we only increase the integral.  If, $\frac{\overline \delta}{2U} > \delta$, then note that 
\[
\int_\delta^{\frac{\overline \delta}{2U}} \sqrt{ \log D(x, \mathcal F)}dx = 0, 
\]
since $D(x, \mathcal F) = 1$ for all $x \geq \delta$.  Thus, in either case we can replace the upper limit of integration, yielding
\[
	\bar J \ \leq \ 18 U \int_0^\delta \sqrt{ \log D(x, \mathcal F)} dx.
\]

Recall the entropy integral of the original process is given by 
\[
	J \equiv 9 \int_0^\delta \sqrt{ \log D(x, \mathcal F)} dx.
\]
Hence, 
\[
\overline J \leq 2 U J.
\]
Moreover, Theorem A.2 of \citeGR~ shows that $\| J \|_\Psi \leq C_3 \| \delta \|_\Psi  V(A,W)$ for some absolute constant $C_3$. Thus we have successfully bounded \[
	\| \bar J \|_\Psi \ \leq \ 2 U C_3 \| \delta \|_\Psi V(A, W).
\]

Substituting back into \cref{eq:Psi1Norm_Zbar} shows that 
\[
	\Eb{ e^{ \frac{\bar Z}{ 2 U C_3 \| \delta \|_\Psi V(A, W)}} } \ \leq \ 25.
\]

Now choose $C_2$ in \cref{eq:BoundingPsi2Norm_delta} to be $C_2 = \alpha  2 U C_3 \| \delta \|_\Psi V(A, W)$ for some $\alpha > 0$ to be determined later.  Substituting our bound on $\bar Z$ into \cref{eq:BoundingPsi2Norm_delta} shows 

\begin{align} \notag
	\Eb{ \exp\left(\frac{\delta^2}{ \alpha 2 U C_3 \| \delta \|_\Psi V(A, W)} \right)} 
&\ \leq \ 
	e^{\frac{K \sigma^2}{\alpha  2U C_3 \| \delta \|_\Psi V(A, W)}} \cdot \Eb{ e^\frac{\bar Z}{\alpha  2 U C_3 \| \delta \|_\Psi V(A, W)}}
\\ \label{eq:UpperBoundDelta_PsiNorm2}
&	\ \leq \ 
	\exp\left(\frac{K \sigma^2}{ \alpha  2 U C_3 \| \delta \|_\Psi V(A, W) }\right) \cdot 	25^{1/\alpha} ,
\end{align}
where we have used $\alpha > 0$ and Jensen's Inequality to simplify.

We now to choose $\alpha$ large enough that the right side is at most $5$.  Taking logs, it suffices to choose $\alpha$ such that 
\begin{align*}
\log(5) \ \geq \ \frac{1}{\alpha}  \left(\log(25) +  \frac{K \sigma^2}{  2U C_3 \| \delta \|_\Psi V(A, W) }  \right)
\\
\iff
\alpha \ \geq \ 2 + \frac{K \sigma^2}{  2U C_3 \log(5) \cdot \| \delta \|_\Psi V(A, W) }.
\end{align*}
Substituting into \cref{eq:UpperBoundDelta_PsiNorm2} shows 
\[
	\|\delta\|^2_\Psi \ \leq \ \left(2 + \frac{K \sigma^2}{  2U C_3 \log(5) \cdot \| \delta \|_\Psi V(A, W) } \right) 2 U C_3 \| \delta \|_\Psi V(A, W)
	\ = \ 
	4 C_3 U V(A,W) \| \delta\|_\Psi + 
	\frac{K \sigma^2}{ \log(5) }
\]

In summary, $\| \delta \|_\Psi$ is at most the largest solution to the quadratic inequality
\[
	y^2 - b y - c \leq 0,
\]
where
\begin{align*}
	b = 4 C_3 U V(A,W) \qquad \text{ and } \qquad 
	c = \frac{K \sigma^2}{ \log(5)}.
\end{align*}
Bounding the largest root shows
\begin{align*}
y &\ \leq \ \frac{b}{2} + \frac{ \sqrt{b^2 + 4 c}}{2}
\\
& \ \leq \ 
\frac{b}{2} + \frac{ b + 2\sqrt{c}}{2} && (\text{Triangle-Inequality})
\\
& \ = \ 
b + \sqrt{c}.
\end{align*}
Or in other words, 
\[
	\| \delta\|_\Psi \ \leq \ 4 C_3 U V(A,W) + \sigma \sqrt K,
\]
where we upper bounded $(\sqrt{\log(5)})^{-1} \leq 1$.

Now simply substitute into \cref{eq:SupremaBoundPsi2} and collect constants to complete the proof. 
\hfill \Halmos \endproof
\subsection{Method of Bounded Differences Excluding an Exceptional ``Bad" Set}  

In our analysis, we utilize an extension of McDiarmid's inequality
due to \cite{combes2015extension}. Recall, McDiarmid's
inequality shows that for a random vector $\bZ\in\mathcal{X}$ with independent components and
function $f:\mathcal{X}\rightarrow\mathbb{R}$ such that 
\begin{equation} \label{eq:McDiarmidBoundedDiffs}
\abs{f(\bx)-f(\by)} \le \sum_{i=1}^{n}c_{i}\mathbb{I}\left\{ x_{i}\ne y_{i}\right\} \quad 
\forall (\bx,\by) \in\mathcal{X}^{2},
\end{equation}
for some $\bm c \in \R^n$, we have that 
\[
\mathbb{P}\left\{ \abs{ f(\bZ) - \Eb{ f(\bZ) } } \ge t \right\} \ \le \ 2\exp\left(-\frac{2t^{2}}{\sum_{i=1}^{n}c_{i}^{2}}\right).
\]

The next result extends McDiarmid's inequality to a setting where \cref{eq:McDiarmidBoundedDiffs} only holds for all $(\bx, \by) \in \mathcal Y^2$ where $\mathcal Y \subseteq \mathcal X$ is a certain ``good" set:
\begin{thm}
[\cite{combes2015extension}] \label{thm:Combes-McDiarmid}
Let $\bZ \in \mathcal X$ be a random vector with independent components and $f: \mathcal X \mapsto \R$ be a function such that 
\[
	\abs{f(\bx) - f(\by)} \ \leq \ \sum_{j=1}^n c_j \Ib{ x_j \neq y_j} \quad \forall (\bx, \by) \in \mathcal Y,
\]
for some vector $\bm c \in R^n$,  
where $\mathcal Y \subseteq \mathcal X$.  
Let $\overline{c}=\sum_{i=1}^{n}c_{i}$,
and $p = \mathbb{P}\left\{ X\not\in\mathcal{Y}\right\} $.  Then, for any $t > 0$
\[
\Pb{ \abs{f(\bZ)- \Eb{f(\bZ) \mid \bZ \in \mathcal Y}}\ge t + p \bar c} 
\ \le \ 
2\left( p + \exp\left\{ -\frac{2t^2}{\sum_{i=1}^{n}c_{i}^{2}}\right\} \right)
\]
In particular, this implies that for any $\epsilon > 2p$, with probability at least $1-\epsilon$, 
\[
\abs{f(\bZ) - \Eb{f(\bZ) \mid \bZ \in \mathcal Y}} \ \leq p \bar c + 
\| \bm c \|_2 \sqrt{\log\left(\frac{2}{\epsilon - 2p}\right)}.
\]

\end{thm}
\begin{rem}
In the special case that $\mathcal Y = \mathcal X$, then $p = 0$, the theorem recovers McDiarmid's inequality.  
\end{rem}

\section{Properties of the Variance Gradient Correction (VGC)} \label[app]{sec:DanskinStatement}

\blue{\label{note:Danskin-Theorem} First, we state the relevant portion of Danskin's Theorem for reference. See \cite[Section B.5]{bertsekas1997nonlinear} for a proof of a more general result:  
\begin{theorem}[Derivative Result of Danskin's Theorem] \label{thm:DanskinsTheorem}
Let $Z \subseteq \mathbb{R}^m$ be a compact set, and let $\phi : \mathbb{R}^n \times Z \mapsto \mathbb{R}$ be continuous and such that $\phi(\cdot,z) : \mathbb{R}^n \mapsto \mathbb{R}$ is convex for each $z \in Z$. Additionally, define 
\[
    Z(x)=\left\{ \bar{z} \ : \ \phi(x,\bar{z}) = \max_{z \in Z} \phi(x,z) \right\}.
\]
Consider the function $f : \mathbb{R}^n \mapsto \mathbb{R}$ given by 
\[
    f(x) = \max_{z\in Z} \phi(x,z).
\]
If Z(x) consists of a unique point $\bar{z}$ and $\phi(\cdot,\bar{z})$ is differentiable at $x$, then $f$ is differentiable at $x$, and $\nabla f(x) = \nabla_{x} \phi(x, \bar{z})$, where $\nabla_{x} \phi(x, \bar{z})$ is the vector with coordinates 
\[
    \frac{\partial \phi(x, \bar{z})}{\partial x_i},\qquad i=1,\dots,n.
\]
\end{theorem}}

The remainder of the section contains proofs of the results in  \cref{sec:EstimatingInSampleBias}.

\subsection{Proof of \cref{thm:equiv-in-sample}}
\label[app]{sec:appendix_FindingBestInClassPolicy}
This section contains the omitted proofs leading to the proof of \cref{thm:equiv-in-sample}.  We first relate finite difference approximations of the subgradients of $V(\bz + t \be_j)$ to their true values.

\begin{lem}[Subgradients Bound Finite Difference Approximation] \label[lem]{lem:prop-j-comp}
For any $\bz \in \R^n$ and $t \in \R$, we have 
\[
a_j x_j( \bz + t\be_j) t \ \leq V(\bz + t \be_j ) - V(\bz) \ \leq \ a_j x_j(\bz) t.
\]

\end{lem}
\begin{proof}{\hspace{-12pt}Proof:}
Let $f(t) = V(\bz + t\be_j)$.  Recall that $f(t)$ is concave and $f'(t) = a_j x_j (\bz + t \be_j)$ by Danskin's Theorem. Hence, by the subgradient inequality for concave functions, $f(t) \leq f(0) + f'(0) t$ and $f(0) \leq f(t) - t f'(t)$, and thus, $t f'(t) \leq f(t) - f(0) \leq t f'(0)$. This is equivalent to \mbox{$a_j x_j( \bz + t\be_j) t \ \leq V(\bz + t \be_j ) - V(\bz) \ \leq \ a_j x_j(\bz) t$}, which is the desired result.
\hfill \Halmos \end{proof}
\smallskip
Equipped with \Cref{lem:prop-j-comp}, the proof of \Cref{lem:Monotonicity} is nearly immediate.
\begin{proof}{Proof of \Cref{lem:Monotonicity}:}
The bounds in \Cref{lem:prop-j-comp} show that $a_j t ( x_j(\bz) - x_j(\bz + t \be_j) ) \geq 0$.  If $a_j \geq 0$, it follows that $t ( x_j(\bz) - x_j(\bz + t \be_j) ) \geq 0$ for all $t$, which shows that $t \mapsto x_j(\bz + t \be_j)$ is \mbox{non-increasing}.   Similarly, if $a_j < 0$, it follows that $t ( x_j(\bz) - x_j(\bz + t \be_j) ) \leq 0$ for all $t$, which shows that $t \mapsto x_j(\bz + t \be_j)$ is non-decreasing.
\hfill \Halmos
\end{proof}
\smallskip
Before proving \cref{thm:equiv-in-sample}, we establish the following intermediary result on the error of a non-randomized forward step, finite difference.  
\begin{lem}[Forward Step Finite Difference Error] \label[lem]{lem:expected-finite-diff-error} 
Fix some $j$ such that $a_j \neq 0$ and \mbox{$0 < h < 1/e$}.  Then, 
\[
	\abs{
		~\Eb{\xi_{j}x_j(\bZ)
				- \frac{
					V(\bZ+h\sqrt{\nu_{j}}\xi_{j}\be_{j} )
					- V(\bZ )
					}
				{h\sqrt{\nu_{j}}\bm{a}_{j}(\btheta)}
			}
~}
	\ \le \ 
	4  h
		\log \left( \frac{1}{h\sqrt{\nu_{\min}}} \right)
\]
In other words, the forward finite step difference introduces a bias of order $\tilde O(h)$.  
\end{lem}
\begin{proof}{\hspace{-12pt}Proof:}
From \Cref{lem:prop-j-comp}, we see that the term inside the expectation 
can be upper-bounded by the non-negative term
$\xi_{j}\left[x_{j}(\bZ)-x_{j}(\bZ+h\sqrt{\nu_{j}}\xi_{j}\be_{j})\right]$.
Hence, 
\[
\Eb{\xi_{j}x_{j}(\bZ)-\frac{V(\bZ+h\sqrt{\nu_{j}}\xi_{j}\be_{j} )-V(\bZ )}{h\sqrt{\nu_{j}}\bm{a}_{j}(\btheta)}}
\ \leq \ 
\Eb{\xi_{j}\left(x_{j}(\bZ)-x_{j}(\bZ+h\sqrt{\nu_{j}}\xi_{j}\be_{j})\right)} .
\]
To simplify notation, 
let $g(t) = \Eb{ x_{j}\left(\bZ_{-j}+\xi_{j}\be_{j}\right)|\xi_{j}=t }$
where $\bZ_{-j}$ is identical to $\bZ$ but has a $0$ at the $j^{\text{th}}$ component.  Then, 
\[
\abs{ \Eb{\xi_{j}\left(x_{j}(\bZ)-x_{j}(\bZ+h\sqrt{\nu_{j}}\xi_{j}\be_{j})\right) } }
 = \abs{ \int_{-\infty}^{\infty}t\left[g(t)-g(t+h\sqrt{\nu_{j}}t)\right]\phi_{j}(t)\,dt}
\]
where $\phi_{j}(t)$ is the density for $\mathcal{N}(0,1/\nu_{j})$.

To bound the integral, choose a constant $U > 0$ (which we optimize later) and break the integral into
three regions, $(-\infty,-U),(-U,U),(U,\infty)$.  This yields the upper bound
\begin{align*}
\Biggr| \int_{-\infty}^{\infty}t&\left(g(t) - g(t+th\sqrt{\nu_{j}})\right)\phi_{j}(t)\,dt \Biggr|
\\
\ \le\  & \underbrace{\int_{-U}^{U}U\left|g(t)-g(t+th\sqrt{\nu_{j}})\right|\phi_{j}(t)\,dt}_{(a)}+\underbrace{\int_{-\infty}^{-U}\left|t\right|\phi_{j}(t)\,dt+\int_{U}^{\infty}\left|t\right|\phi_{j}(t)\,dt}_{(b)}.
\end{align*}

We first bound $(a)$. As the fist step, we attempt to remove the absolute value. From \Cref{lem:prop-j-comp}, $g(\cdot)$ is a monotone function.
We claim that for $\abs{t} < U$, 
\begin{equation} \label{eq:symmetric-upperbound}
\abs{ g(t) - g(t + h \sqrt{ \nu_{j} } t)} 
\ \leq \ 
\abs{g\left(t-Uh\sqrt{\nu_{j}}\right)-g\left(t + Uh\sqrt{\nu_{j}}\right)},
\end{equation}
since $(t-Uh\sqrt{\nu_{j}},t+Uh\sqrt{\nu_{j}})$ always contains the interval $(t, t + h\sqrt{ \nu_j } t )$.
Let 
\[
b = \begin{cases} U h\sqrt{\nu_j} & \text{ if } a_j > 0,
	\\ 		   - Uh \sqrt{\nu_j} & \text{ otherwise,} 
	\end{cases}
\]
so that 
\(
\abs{g(t - Uh\sqrt{\nu_j}) - g(t + Uh\sqrt{\nu_j}) } 
\ = \ 
g(t-b) - g(t+b).
\)
Then, 
\begin{align*}
\int_{-\infty}^{\infty} \big(g(t-b) - g(t+b)\big) \phi_j(t)\,dt 
& =\int_{-\infty}^{\infty} g\big(t + b\big) \left(\phi_j (t + 2b) - \phi_j(t)\right)\,dt, 
	\quad \text{(Change of variables)}
\\
 & =\int_{-\infty}^{\infty}
 g\left(t+b\right)\int_{t}^{t + 2b}-\nu_j z \phi_{j}(z)\,dz\,dt, \quad (\text{since } \phi_{j}'(z) = -\nu_j z \phi_{j}(z))
 \\
 & \le \nu_j \int_{-\infty}^{\infty} \abs{z} \phi_{j}(z)\int_{z-2b}^{z}g\left(t+b\right)\,dt\,dz, 
	 \quad 	\text{(Fubini's Theorem)}
 \\
 & \leq 2 \abs{b} \nu_{j}\int_{-\infty}^{\infty} \abs{z} \phi_{j}(z)dz, 
 	~~~~~~~~~~~~~~~~\quad \text{(Since $\abs{g(t + b)} \leq 1$ for all $t$)}
 \\
 & = 4b\sqrt{\frac{\nu_{j}}{2\pi}} 
 \\
 & \leq 
  2\sqrt{\frac{2}{\pi}}Uh\nu_{j}.
\end{align*}
To summarize, we have shown that the term $(a)$ satisfies
\[
\int_{-U}^{U}U\left|g(t)-g(t+th\sqrt{\nu_{j}})\right|\phi_{j}(t)\,dt
\ \le \ 2 \sqrt{\frac{2}{\pi}} U^{2}h\nu_{j}
\]

To bound $(b)$, we see that first see that
\[
    \int_{-\infty}^{-U}\left|t\right|\phi_{j}(t)\,dt
    +
    \int_{U}^{\infty}\left|t\right|\phi_{j}(t)\,dt
    =
    2\int_{U}^{\infty}t\phi_{j}(t)\,dt
    =
    2\sqrt{\frac{1}{2\pi\nu_{j}}}\exp\left\{ \frac{-U^{2}\nu_{j}}{2}\right\} 
\]
where the first equality holds by symmetry.

Putting the bounds of $(a)$ and $(b)$ together, we have
\[
    \abs{ \int_{-\infty}^{\infty}t\left[g(t)-g(t+th\sqrt{\nu_{j}})\right]\phi_{j}(t)\,dt }
    \le
    2 \sqrt{\frac{2}{\pi}}U^{2}h\nu_{j} +2\sqrt{\frac{1}{2\pi\nu_{j}}}\exp\left\{ \frac{-U^{2}\nu_{j}}{2}\right\} 
\]
We approximately balance the two terms by letting $U^{2}=\frac{2}{\nu_{j}}\log\left( \frac{1}{h\sqrt{\nu_{j}}}\right)$.  Substituting and simplifying yields
\[
4 \sqrt{\frac{2}{\pi}} h \log\left( \frac{1}{h \sqrt{\nu_j}} \right) + h \sqrt{\frac{2}{\pi}}
\ \leq \ 
4 \sqrt{\frac{2}{\pi}} h \log\left( \frac{1}{h \sqrt{\vmin}} \right) + h \sqrt{\frac{2}{\pi}}.
\]
To simplify, note that $h < 1/e$ and $\vmin \leq 1$ implies that $\log\left( \frac{1}{h \sqrt{\nu_j}} \right) \geq 1$.  Hence, combining the two terms and simplifying provides a bound of 
\(
10\sqrt{\frac{1}{2\pi}}h \log\left( \frac{1}{h\sqrt{\nu_{\min}}} \right).
\)
Note that $10/\sqrt{2\pi} \leq 4$ to complete the proof. 
\hfill \Halmos \end{proof}

\smallskip
We can now prove \cref{thm:equiv-in-sample}.
\begin{proof}{Proof of \cref{thm:equiv-in-sample}:}
Notice that if $a_j(\btheta) = 0$, then the $j^{\rm th}$ term contribute nothing to the bias because $\bx(\bZ)$ is independent of $Z_j$, so  $\Eb{ \xi_j x_j(\bZ) } = 0 = \Eb{D_j(\bZ)}$.
Hence, we focus on terms $j$ where $a_j(\btheta) \neq 0$.  

Decompose the $j^\text{th}$ term as 
\begin{align*} 
 \mathbb{E}\left[
 	\xi_{j}x_{j}(\bZ) 
 	- D_{j}\left(\bZ \right)
 \right] 
 ~=~ & 
 \underbrace{ \mathbb{E}\left[
 	\xi_{j} x_j (\bZ)
 	- \frac{V(\bZ+h\sqrt{\nu_{j}}\xi_{j}\be_{j} )
 		- V(\bZ )}
 		{h\sqrt{\nu_{j}} a_{j}(\btheta)}
 	\right] }_{(a)}\\
 & \qquad + \underbrace{
 \mathbb{E} \left[ \mathbb{E}\left[ \left.
 	\frac{
 		V(\bZ+h\sqrt{\nu_{j}}\xi_{j}\be_{j} )
 		- V(\bZ )
 	}{h\sqrt{\nu_{j}} a_{j}(\btheta)}
 - \frac{
 	V(\bZ+\delta_{j}\be_{j} )
 	- V(\bZ )}
 	{h\sqrt{\nu_{j}} a_{j}(\btheta)}
 \right| \bZ \right] \right]}_{(b)}
\end{align*}
We first bound $(b)$.  Canceling out the $V(\bZ )$ yields
\[
    \frac{1}{h\sqrt{\nu_{j}}\bm{a}_{j}(\btheta)}
    \mathbb{E}\left[
        V(\bZ + h\sqrt{\nu_{j}}\xi_{j}\be_{j} ) 
 	    - V(\bZ+\delta_{j}\be_{j} )
    \right].
\]
From our previous discussion,
\(
    V(\bZ + h\sqrt{\nu_{j}}\xi_{j}\be_{j} ) 
 	    \sim_{d} V(\bZ+\delta_{j}\be_{j} ),
\)
whereby
\mbox{$\mathbb{E}\left[V(\bZ+h\sqrt{\nu_{j}}\xi_{j}\be_{j} ) \,-\, V(\bZ+\delta_{j}\be_{j} )\right]=0$}. 

\Cref{lem:expected-finite-diff-error} bounds $(a)$ by
%
$
    4 h  \log\left(\frac{1}{h\sqrt{\nu_{\min}}}\right)
$. 
Summing over $j$ gives us our intended bound.
\hfill \Halmos \end{proof}

\smallskip
\subsection{Properties of VGC}\label[app]{app:properties-VGC}
We next establish smoothness properties of the VGC.
\proof{Proof of \Cref{lem:D-is-Lipschitz}.}
We begin with i). We first claim that $\btheta \mapsto V(\bz, \btheta)$ is Lipschitz continuous with parameter $Ln (1 + \| \bz \|_\infty)$.  To this end, write
\begin{align*}
V(\bz, \bthetabar) - V(\bz, \btheta)
& = \left(\br\left(\bz, \bthetabar\right)-\br\left(\bz, \btheta\right)\right)^{\top}\bx(\bz, \bthetabar)-\underbrace{\br\left(\bz, \btheta\right)^{\top}\left(\bx(\bz, \btheta)-\bx(\bz, \bthetabar)\right)}_{\leq 0 \text{ by optimality of $\bx(\bz, \btheta)$}}
\\
&\le  \abs{ \left(\br\left(\bz, \bthetabar\right)-\br\left(\bz, \btheta\right)\right)^{\top}\bx(\bz, \bthetabar)}
\\
& \leq 
\| \br(\bz, \bthetabar) - \br(\bz, \btheta) \|_1 \| \bx(\bz, \bthetabar) \|_\infty
\\
& \leq 
\| \br(\bz, \bthetabar) - \br(\bz, \btheta) \|_1
\qquad \qquad \qquad \qquad   (\text{since } \bx(\bz, \bthetabar) \in \mathcal X \subseteq [0, 1]^n)
\\
& \leq
\sum_{j=1}^n \abs{ a_j(\bthetabar) - a_j(\btheta)}\abs{z_j} + \abs{b_j(\bthetabar) - b_j(\btheta)}
\\
& \leq 
\sum_{j=1}^n \big( L \| \bz \|_\infty  \| \bthetabar - \btheta \| + L \| \bthetabar - \btheta \| \big)
\\  
& = 
L n \cdot (1 + \| \bz \|_\infty) \| \bthetabar - \btheta \|.
\end{align*}
Reversing the roles of $\btheta$ and $\bthetabar$ yields an analogous bound, and, hence, 
\[
\left| \br \left(\bz, \bthetabar\right)^{\top}\bx(\bz, \bthetabar) - \br\left(\bz, \btheta\right)^{\top}\bx(\bz, \btheta)\right| \leq
L n \left(1+\left\Vert \bz\right\Vert_{\infty}\right) \left\Vert \bthetabar-\btheta\right\Vert 
\]
This proves the first statement.  

Next, we claim for any $\bz$, 
\begin{equation} \label{eq:ScaledPlugInLipschitz}
\abs{ \frac{1}{a_j(\bthetabar)} V(\bz, \bthetabar) - \frac{1}{a_j(\btheta)} V(\bz, \btheta) } 
\ \leq \ 
\frac{2nL}{a_{\min}} \left( \frac{a_{\max}}{a_{\min}}  \| \bz \|_\infty + \frac{a_{\max}+b_{\max}} {a_{\min}} \right)  \| \bthetabar - \btheta \|. 
\end{equation}
Write
\begin{align*}
\abs{ \frac{1}{a_j(\bthetabar)} V(\bz, \bthetabar) - \frac{1}{a_j(\btheta)} V(\bz, \btheta) } 
& \ = \ 
\abs{ 
\frac{ a_j(\btheta) V(\bz, \bthetabar) - a_j(\bthetabar) V(\bz, \btheta) }{a_j(\btheta) a_j(\bthetabar)}
}
\\
& \ \leq \ 
\abs{ \frac{ V(\bz, \bthetabar) - V(z,\btheta)}{a_j(\bthetabar)}} + 
	\abs{ \frac{ V(z,\btheta) ( a_j(\btheta) - a_j(\bthetabar))}{a_j(\btheta) a_j(\bthetabar)}},
\\
& \ \leq \
\frac{ Ln(1 + \| \bz \|_\infty) \| \bthetabar - \btheta \|}{a_{\min}} + 
\frac{ \abs{V(z,\btheta)} L \| \bthetabar - \btheta \|}{a_{\min}^2},	
\end{align*}
where the first inequality follows
by adding and subtracting $a_j(\btheta)V(\btheta)$ in the numerator, and the second inequality follows from the Lipschitz continuity of $a_j(\btheta)$ and $V(\bz, \btheta)$ (\cref{asn:SmoothPlugIn}).   Next note that 
\[
\abs{V(\bz, \btheta)} 
\ \leq \ 
\| \br(\bz, \btheta) \|_1 \| \bx(\bz, \btheta ) \|_\infty 
\ \leq \ 
\| \ba(\btheta) \circ \bz \|_1 + \| \bm b(\btheta) \|_1
\ \leq \ 
n \| \bz \|_\infty a_{\max} + n b_{\max}.
\]
Substituting above and simplifying proves \cref{eq:ScaledPlugInLipschitz}

We can now prove the lemma.  Fix a component $j$.  Then, 
\begin{align*}
D_j(\bz, \bthetabar) - D_j(\bz, \btheta)
&\ = \ 
\Eb{  \frac{1}{h\sqrt{\nu_j} a_j(\bthetabar)} \left( V(\bZ + \delta_j \be_j, \bthetabar) - V(\bZ, \bthetabar)\right) \ \mid \bZ = \bz} 
\\ & \qquad  - \Eb{ \frac{1}{h\sqrt{\nu_j} a_j(\btheta)} \left( V(\bZ + \delta_j \be_j, \btheta) - V(\bZ, \btheta)\right) \ \mid   \bZ = \bz }
\\ & \ = \ 
\frac{1}{h \sqrt{\nu_j}} \Eb{ \frac{1}{a_j(\bthetabar)} V(\bZ + \delta_j \be_j, \bthetabar) 
					- \frac{1}{a_j(\btheta)} V(\bZ + \delta_j \be_j, \btheta) \mid \bZ = \bz}
\\ & \qquad 
+ 
\frac{1}{h \sqrt{\nu_j}} \Eb{ \frac{1}{a_j(\btheta)} V(\bZ, \btheta) 
					- \frac{1}{a_j(\bthetabar)} V(\bZ, \bthetabar) \mid \bZ = \bz}.
\end{align*}			
Hence, by taking absolute values and applying \cref{eq:ScaledPlugInLipschitz} twice we obtain
\begin{align*} 
\abs{ D_j(\bz, \bthetabar) - D_j(\bz, \btheta) }
& \ \leq \ 
\frac{2nL}{h a_{\min}} \left( \frac{a_{\max}}{a_{\min}}  \Eb{ \| \bZ + \delta_j \be_j \|_\infty \mid \bZ = \bz} + \frac{a_{\max}+b_{\max}} {a_{\min}} \right)  \| \bthetabar - \btheta \| 
\\ & \qquad + 
\frac{2nL}{h a_{\min}} \left( \frac{a_{\max}}{a_{\min}}  \| \bz  \|_\infty + \frac{a_{\max}+b_{\max}} {a_{\min}} \right)  \| \bthetabar - \btheta \|,
\end{align*}
where we have passed through the conditional expectation.  Finally, note that $\Eb{ \| \bZ + \delta_j \be_j \|_\infty \mid \bZ= \bz } \leq \| \bz \|_\infty + \Eb{ \abs{\delta_j}} \ \leq \| \bz \|_\infty + \sqrt{ h^2 + 2h /\nu_j}$ by Jensen's inequality.  We simplify this last expression by noting for $h  < 1/e$, $h^2 < h$,  so that 
\[
\sqrt{ h^2 + 2h /\nu_j} \  \leq  \ \sqrt{ h} \sqrt{ 1 + 2/\vmin} \  \leq  \ 2 \sqrt{\frac{h}{\vmin}}, 
\]
using $\vmin \leq 1$.  Thus, 
$\Eb{ \| \bZ + \delta_j \be_j \|_\infty \mid \bZ = \bz} \leq \| \bz \|_\infty +2 \sqrt{\frac{h}{\vmin}}$.  Substituting above and collecting terms yields
\begin{equation}\label{eq:UpperBoundLip}
\frac{4nL}{h a_{\min}} \left( \frac{a_{\max}}{a_{\min}}  \| \bz  \|_\infty + \frac{a_{\max}+b_{\max}} {a_{\min}}  + \frac{a_{\max}}{a_{\min}} \sqrt{\frac{h}{\vmin}} \right)  \| \bthetabar - \btheta \|.
\end{equation}
We can simplify this expression by letting 
\[
C_3 \geq  \frac{4}{a_{\min}} \cdot \max\left(\frac{a_{\max}}{a_{\min}}, \ 
			\frac{a_{\max} + b_{\max}}{a_{min}}, \ 
			\frac{a_{\max}}{a_{\min}}
			\right) .
\]
Then \cref{eq:UpperBoundLip} is at most 
\[
\frac{C_3 n L}{h}\left( \| \bz \|_\infty + 1 + \sqrt{\frac{h}{\vmin}} \right) \| \bthetabar - \btheta \|_2 
\ \leq \ 
\frac{C_3 n L}{h}\left( \| \bz \|_\infty + \frac{2}{\sqrt \vmin} \right) \| \bthetabar - \btheta \|_2 
\ \leq \ 
\frac{2 C_3 n L}{h \sqrt {\vmin}} \left( \| \bz \|_\infty + 1 \right) \| \bthetabar - \btheta \|_2, 
\]
where we have used the bounds on the precisions (\cref{asn:Parameters}) and $h < 1/e$ to simplify.  
Letting $C_1 = 2 C_3$ proves the first part of the theorem.  

To complete the proof, we require a high-probability bound on $\| \bZ \|_\infty$.  Since $\bZ- \bmu$ is sub-Gaussian, such bounds are well-known \citep{wainwright2019high}, and we have with probability $1-e^{-R}$, 
\[
\| \bZ \|_\infty \ \leq \ C_\mu + \| \bZ - \bmu \|_\infty \leq C_\mu + \frac{C_4}{\sqrt \vmin} \sqrt{ \log n } \sqrt{ R}, 
\]
for some universal constant $C_4$.  Substitute this bound into our earlier Lipschitz bound for an arbitrary $\bz$, and use the \cref{asn:Parameters}, $h < 1/e$, and $R>1$  to collect terms and simplify. \blue{We then sum over the $n$ terms of $D(\bZ,\btheta)$ to complete the proof for i).} \\
\blue{
We now bound ii). Focusing on the $j^{th}$ component of $D(\bZ,(\btheta,h))$
and writing 
\[
D_{j}\left(\bZ,(\btheta,h)\right)\equiv D_{j}\left(\bZ,h,\delta_{j}^{h},\btheta \right)=\mathbb{E}\left[\left.\frac{V(\bZ+\delta_{j}^{h}\be_{j},\btheta)-V(\bZ,\btheta)}{a_{j}(\btheta)h\sqrt{\nu_{j}}}\right|\bZ\right],
\]
we see that 
\begin{align*}
D_{j}\left(\bZ,h,\delta_{j}^{h},\btheta\right)-D_{j}\left(\bZ,\overline{h},\delta_{j}^{\overline{h}},\btheta\right) & =\underbrace{D_{j}\left(\bZ,h,\delta_{j}^{h},\btheta\right)-D_{j}\left(\bZ,\overline{h},\delta_{j}^{h},\btheta\right)}_{(a)}\\
 & \qquad+\underbrace{D_{j}\left(\bZ,\overline{h},\delta_{j}^{h},\btheta\right)-D_{j}\left(\bZ,\overline{h},\delta_{j}^{\overline{h}},\btheta\right)}_{(b)}.
\end{align*}
To bound $(a)$ and $(b)$, we see from the proof of \cref{lem:VGCBounded} that,
\begin{equation}
\left|\frac{V(\bZ,\btheta)-V(\bZ+Y\be_{j},\btheta)}{a_{j}(\btheta)}\right|\le\left|Y\right|.\label{eq:change-in-obj-j-pert}
\end{equation}
We first bound $(a)$. We see 
\begin{align*}
\left|D_{j}\left(\bZ,h,\delta_{j}^{h},\btheta\right)-D_{j}\left(\bZ,\overline{h},\delta_{j}^{h},\btheta\right)\right| & =\left|\frac{\overline{h}-h}{h\overline{h}}\right|\left|\mathbb{E}\left[\left.\frac{V(\bZ+\delta_{j}^{h}\be_{j},\btheta)-V(\bZ,\btheta)}{\sqrt{\nu_{j}}}\right|\bZ\right]\right|\\
 & \le\frac{\left|\overline{h}-h\right|}{h_{\min}^{2}}\left|\mathbb{E}\left[\left.\frac{|\delta_{j}^{h}|}{\sqrt{\nu_{j}}}\right|\bZ\right]\right|,\text{ by Eq. (\ref{eq:change-in-obj-j-pert})}\\
 & \le\frac{\left|\overline{h}-h\right|}{h_{\min}^{2}}\frac{1}{\sqrt{\nu_{\min}}}\sqrt{\frac{3h}{\sqrt{\nu_{\min}}}}\le\frac{\sqrt{3}\left|\overline{h}-h\right|}{h_{\min}^{2}\nu_{\min}^{3/4}},
\end{align*}
where the second to last inequality applies the inequality $\mathbb{E}\left[|\delta_{j}^{h}|\right]=\mathbb{E}\left[\sqrt{|\delta_{j}^{h}|^{2}}\right]\le\sqrt{\mathbb{E}\left[|\delta_{j}^{h}|^{2}\right]}\le\sqrt{\frac{3h}{\sqrt{\nu_{\min}}}}$.
We then bound $(b)$. We see 
\begin{align*}
\left|D_{j}\left(\bZ,\overline{h},\delta_{j}^{h},\btheta\right)-D_{j}\left(\bZ,\overline{h},\delta_{j}^{\overline{h}},\btheta\right)\right| & =\left|\mathbb{E}\left[\left.\frac{V(\bZ+\delta_{j}^{h}\be_{j},\btheta)-V(\bZ+\delta_{j}^{\overline{h}}\be_{j},\btheta)}{a_{j}(\btheta)\overline{h}\sqrt{\nu_{j}}}\right|\bZ\right]\right|\\
 & =\frac{1}{\overline{h}\sqrt{\nu_{j}}}\left|\mathbb{E}\left[\left.f\left(\delta_{j}^{h}\right)-f\left(\delta_{j}^{\overline{h}}\right)\right|\bZ\right]\right|\\
 & \le\frac{1}{\overline{h}\sqrt{\nu_{j}}}W_{2}\left(\delta_{j}^{h},\delta_{j}^{\overline{h}}\right),\text{ by Wainwright (2019, pg. 76)}
\end{align*}
The Wasserstein distance between two mean-zero Gaussians is known
in closed form:
\[
W_{2}\left(\delta_{j}^{h},\delta_{j}^{\overline{h}}\right)=\left|\sqrt{h^{2}+\frac{2h}{\sqrt{\nu_{j}}}}-\sqrt{\overline{h}^{2}+\frac{2\overline{h}}{\sqrt{\nu_{j}}}}\right|\le\sqrt{\left|h^{2}-\overline{h}^{2}+\frac{2\left(h-\overline{h}\right)}{\sqrt{\nu_{j}}}\right|}\le\sqrt{\left(2+\frac{2}{\sqrt{\nu_{j}}}\right)\left|h-\overline{h}\right|},
\]
where the first inequality comes from the common inequality $\left| \sqrt{a} - \sqrt{b} \right| \le \sqrt{\left| a - b \right|} $.
Thus,
\[
\left|D_{j}\left(\bZ,\overline{h},\delta_{j}^{h},\btheta\right)-D_{j}\left(\bZ,\overline{h},\delta_{j}^{\overline{h}},\btheta\right)\right|\le\frac{1}{h_{\min}\sqrt{\nu_{\min}}}\sqrt{\left(2+\frac{2}{\sqrt{\nu_{\min}}}\right)\left|h-\overline{h}\right|}.
\]
Collecting constants of the bounds of $(a)$ and $(b)$, we obtain
our result.
}
\hfill \Halmos \endproof

We now show the that the components of VGC is bounded. 

\blue{
\proof{Proof of \ref{lem:VGCBounded}:}
We see
\begin{align*}
&\frac{V(\bZ+\delta_{j}\be_{j}, \btheta)-V(\bZ, \btheta)}{a_{j}(\btheta)h\sqrt{\nu_{j}}}
 \\
& \qquad =
	\frac{1}{a_j(\btheta) h \sqrt{\nu_j}} 
	\Big( \underbrace{
		\br(\bZ, \btheta)^\top \left( \bx(\bZ + \delta_j \be_j, \btheta) - \bx(\bZ, \btheta) \right)
	}_{ \leq 0 \text{ by optimality of } \bx(\bZ, \btheta) } \; 
	+ \; a_j(\btheta) \delta_j x_j(\bZ + \delta_j\be_j, \btheta)  
	\Big)
\\
&\qquad \leq \  \frac{\delta_{j}x_{j}(\bZ+\delta_{j}\be_{j}, \btheta )}{h\sqrt{\nu_{j}}}
 \leq \  \frac{ \abs{ \delta_{j}} }{h\sqrt{\nu_{\min}}}.
\end{align*}
By an analogous argument,
\[
\frac{V(\bZ, \btheta)-V(\bZ+\delta_{j}\be_{j}, \btheta)}{a_{j}(\btheta)h\sqrt{\nu_{j}}}\le\frac{\abs{ \delta_{j}} }{h\sqrt{\nu_{\min}}}.
\]
Taking the conditional expectation, we see
\[
	\abs{D_j(\bz)} \le 
	\mathbb{E} \left[ \left| \frac{V(\bZ+\delta_{j}\be_{j}, \btheta)-V(\bZ, \btheta)}{a_{j}(\btheta)h\sqrt{\nu_{j}}} \right| \Big|  \bZ \right] 
	\leq \mathbb{E}\left[ \frac{\abs{ \delta_{j}} }{h\sqrt{\nu_{\min}}} \right] 
	\leq \frac{\sqrt 3}{\vmin^{3/4} \sqrt h }.
\]
where the first inequality holds by Jensen's inequality and the last inequality holds by Jensen's inequality as
$\mathbb{E}\left[ \abs{ \delta_{j}} \right] 
\le \mathbb{E}\left[ \sqrt{\delta_{j}^2} \right] 
\le \sqrt{\mathbb{E}\left[ \delta_{j}^2 \right] } 
= \sqrt{h^2 + 2h/\sqrt{\nu_j}} 
\le \sqrt{ \frac{3h}{\sqrt{\vmin}}}$.
\hfill \Halmos \endproof
}
\subsection{Bias Under Violations of \cref{asn:Gaussian}} \label[app]{sec:ViolatingGaussian}

In cases where the precisions $\nu_j$ are not known but estimated by a quantity $\tilde{\nu}_j$, we can construct the VGC in the same fashion, but replacing instances of $\nu_j$ with $\tilde{\nu}_j$, giving us,
\[
\sum_{j: a_j \neq 0}
\frac{1}{ h \sqrt{\tilde{\nu}_j} a_j(\btheta) }  
\Eb{ \biggr( \Obj(\bmu  +  \bxi + \tilde{\delta}_j \be_j) - \Obj(\bmu  +  \bxi) \biggr) \Biggr | \bZ }
\]
where $\tilde{\delta}_j\sim \mathcal{N}(0,h^2 + 2h/\sqrt{\tilde{\nu}_j} )$.  The bias of this VGC is similar to \cref{thm:equiv-in-sample}, except that it picks up an additional bias term due to the approximation error incurred from $\tilde{\nu}_j$, which we quantify in the following lemma.

\begin{lem}[Bias of VGC with Estimated Precisions]\label[lem]{lem:noisy-precision} 
Suppose \cref{asn:Parameters} holds.
Let $\tilde{\nu}_{j}$ be an estimate of $\nu_{j}$
and let $\tilde{\delta}_{j}\sim\mathcal{N}(0,h^{2}+2h/\sqrt{\tilde{\nu}_{j}})$ and assume $\vmin \leq \min_j \tilde \nu_j$.
For any $0<h<1/e$, there exists a constant $C$ dependent on $\nu_{\min}$
such that
\begin{align*}
&\left|\mathbb{E}\left[\sum_{j=1}^{n}\xi_{j}x_{j}(\bZ) - \sum_{j: a_j \neq 0}^{n}
\Eb{ \frac{V(\bZ+\tilde{\delta}_{j}\be_{j})-V(\bZ)}{a_{j}(\btheta)h\sqrt{\tilde{\nu}_{j}}} \Biggr| \bZ }  \right]\right|
\\
& \qquad \leq C\cdot nh\log\left(\frac{1}{h}\right)+\frac{C}{\sqrt{h}} \sum_{j:a_j\ne 0}\left(\left|\nu_{j}^{1/2}-\tilde{\nu}_{j}^{1/2}\right|+\sqrt{\left|\nu_{j}^{1/2}-\tilde{\nu}_{j}^{1/2}\right|}\right)
\end{align*}
\end{lem}

\proof{Proof of \cref{lem:noisy-precision}}
Move the inner conditional expectation outwards and consider a sample path with a fixed $\bZ$.  
Let $D_{j}(t)\equiv\frac{V(\bZ+t\be_{j})-V(\bZ)}{a_{j}h\sqrt{\tilde{\nu}_{j}}}$ if $a_j \ne 0$ and 0 otherwise,
so that $\Eb{ D_{j}(\tilde{\delta}_{j}) \mid \bZ} $ is the $j^{\text{th}}$ component
of the VGC with the estimated precisions and $\sqrt{\frac{\tilde{\nu}_{j}}{\nu_{j}}}\Eb{ D_{j}(\delta_{j}) \mid \bZ } $
is the $j^{\rm th}$ component of the VGC with the correct $\nu_{j}$. Fix some
 $j^{\rm th}$ where $a_j \ne 0$.  Note, 
\[
	\xi_{j}x_{j}(\bZ)-D_{j}(\tilde{\delta}_{j})
	=
	\left(\xi_{j}x_{j}(\bZ)-\frac{V(\bZ+\delta_{j}\be_{j})-V(\bZ)}{a_j h\sqrt{\nu_{j}}}\right)
	+
	\left(\sqrt{\frac{\tilde{\nu}_{j}}{\nu_{j}}}D_{j}(\delta_{j})-D_{j}(\tilde{\delta}_{j})\right)
\]
The expectation of the first term was bounded in \cref{thm:equiv-in-sample}, so we focus
on the expectation of the second. We see
\[
    \sqrt{\frac{\tilde{\nu}_{j}}{\nu_{j}}}D_{j}(\delta_{j})-D_{j}(\tilde{\delta}_{j})
    =
    \underbrace{
        \left(\sqrt{\frac{\tilde{\nu}_{j}}{\nu_{j}}}\,\right)\left(D_{j}(\delta_{j})-D_{j}(\tilde{\delta}_{j})\right)
    }_{(a)}
    + \underbrace{
        \left(1-\sqrt{\frac{\tilde{\nu}_{j}}{\nu_{j}}}\,\right)D_{j}(\tilde{\delta}_{j})
    }_{(b)}.
\]
To bound the expectation of $(a)$, we see first see $t \mapsto  h\sqrt{\tilde{\nu}_{j}}D_{j}(t)$
is $1-$Lipschitz because
\[
\frac{\partial}{\partial t}h\sqrt{\tilde{\nu}_{j}}D_{j}(t)=\frac{\partial}{\partial t}\frac{V(\bZ+t\be_{j})-V(\bZ)}{a_j}= \frac{1}{a_j} a_j x_{j}(\bZ) = x_{j}(\bZ)
\]
by Danskin's theorem and because $x_{j}(\bZ)$ is between 0 and 1. Thus,
by \cite[pg. 76]{wainwright2019high}
\[
\abs{ \mathbb{E}\left[\left.  h\sqrt{\tilde{\nu}_{j}}\left(D_{j}(\delta_{j})-D_{j}(\tilde{\delta}_{j})\right)\right|\bZ\right] }
\le W_{2}\left(\delta_{j},\tilde{\delta}_{j}\right),
\]
where $W_{2}\left(\delta_{j},\tilde{\delta}_{j}\right)$ is the Wasserstein distance between two mean-zero Gaussians $\delta_j$ and $\tilde{\delta}_j$, which is known in closed form:
\begin{align*}
W_{2}\left(\delta_{j},\tilde{\delta}_{j}\right) 
 & = \abs{ \sqrt{h^{2}+\frac{2h}{\tilde{\nu}_{j}^{1/2}}}-\sqrt{h^{2}+\frac{2h}{\nu_{j}^{1/2}} }}
 \le\sqrt{\left|\frac{2h}{\tilde{\nu}_{j}^{1/2}}-\frac{2h}{\nu_{j}^{1/2}}\right|}\le\sqrt{2h\left|\frac{\nu_{j}^{1/2}-\tilde{\nu}_{j}^{1/2}}{\nu_{\min}}\right|}
\end{align*}
where $\nu_{\min} \ \leq \min_{j}\left\{ \min\left\{ \nu_{j},\tilde{\nu}_{j}\right\} \right\} $.

To bound the expectation of $(b)$, we see
\[
	\mathbb{E}\left[\left.\left|D_{j}(\tilde{\delta}_{j})\right|\right|\bZ\right]
	\le
	\mathbb{E}\left[\left.\frac{ \left|a_j \tilde{\delta}_{j}\right|}{|a_j| \tilde{\nu}_j^{1/2} h}\,\right|\bZ\right]
	=
	\frac{1}{\tilde{\nu}_j^{1/2} h}\sqrt{\frac{2}{\pi}
	\left(
		h^{2}+\frac{2h}{\tilde{\nu}_{j}^{1/2}}
	\right)}
	\le
	\sqrt{\frac{2}{\pi}\left(\frac{3}{\nu_{\min}h}\right)}
\]
 where the first equality holds by directly evaluating the expectation
and the last inequality holds because $h < 1/e$ and $\vmin \leq 1$.

Putting it all together, we see
\begin{align*}
 & 
 \left|
 	\mathbb{E}\left[
 		\sum_{j=1}^{n}
 			\xi_{j}x_{j}(\bZ)
 		-
 		\sum_{j = 1}^n
 			\frac{V(\bZ+\tilde{\delta}_{j}\be_{j})
 			-
 			V(\bZ)}{a_{j} h\sqrt{\tilde{\nu}_{j}}}\right]\right|
 \\
 \le 
 & 
 \left|
 	\mathbb{E} \left[ 
 		\sum_{j=1}^{n} 
 		\xi_{j}x_{j}(\bZ)  
 		-\sqrt{\frac{\tilde{\nu}_{j}}{\nu_{j}}}D_{j}(\delta_{j}) 
 	\right] 
 \right|
 +
 \left|
 	\sum_{j=1}^n
 	\mathbb{E}\left[
 		\mathbb{E}\left[\left.\sqrt{\frac{\tilde{\nu}_{j}}{\nu_{j}}} D_{j}(\delta_{j}) 
 			- D_{j}(\tilde{\delta}_{j})\right|\bZ
 		\right]
 	\right]
 \right|
 \\
 \le & 
 \sum_{j=1}^{n}C\cdot h\log\left(\frac{1}{h}\right) 
 + 
 \sum_{j:a_j \ne 0}\frac{1}{h\sqrt{\nu_{\min}}}\sqrt{2h\left|\frac{\nu_{j}^{1/2}-\tilde{\nu}_{j}^{1/2}}{\nu_{\min}}\right|}
 +
 \left|
 	\frac{\nu_{j}^{1/2}-\tilde{\nu}_{j}^{1/2}}{\nu_{\min}}
 \right|
 \sqrt{\frac{2}{\pi}
 \left(\frac{3}{h\nu_{\min}}\right)}
 \\
 \le & 
 \sum_{j=1}^{n} C\cdot h\log\left(\frac{1}{h}\right)
 +
 \sum_{j: a_j \ne 0} \frac{\sqrt{2\left|\nu_{j}^{1/2}-\tilde{\nu}_{j}^{1/2}\right|}}{\sqrt{h}\nu_{\min}^{3/2}}
 +
 \frac{\left|\nu_{j}^{1/2}-\tilde{\nu}_{j}^{1/2}\right|\sqrt{\frac{6}{\pi}}}{\sqrt{h}\nu_{\min}^{3/2}}
\end{align*}
where the first inequality follows from triangle inequality and the second from applying \cref{thm:equiv-in-sample} and our bounds on $(a)$ and $(b)$.
Collecting constants we obtain our intended result.
\hfill \Halmos \endproof

We now highlight when $\xi_j$ are not Gaussian but only sub-Gaussian. Let
\begin{lem}[Bias VGC with Gaussian assumption violated]\label[lem]{lem:NonGaussian} Suppose \cref{asn:Parameters} holds. Let $\xi_j$ be a mean-zero, sub-Gaussian random variable with variance proxy at most $\sigma^2$ and admits a density density $\phi(\cdot)$. Additionally, let $\bar{\xi}_j \sim \mathcal{N}(0,1/\sqrt{\nu_j}$ with density $\bar{\phi}(\cdot)$.
Then, there exists a dimension independent constant $C$, such that
\begin{align*}
&\left|\mathbb{E}\left[\sum_{j=1}^{n}\xi_{j}x_{j}(\bZ) - \sum_{j: a_j \neq 0}^{n}
\Eb{ \frac{V(\bZ+\delta_{j}\be_{j})-V(\bZ)}{a_{j}(\btheta)h\sqrt{\nu_{j}}} \Biggr| \bZ }  \right]\right|
\\
& \qquad \leq 
n\left(\sigma\sqrt{2\pi}-\log\left(\frac{\left\Vert \phi-\bar{\phi}\right\Vert _{1}}{4}\right)\right)\left\Vert \phi-\bar{\phi}\right\Vert _{1}
    + Cnh\log\left(\frac{1}{h}\right)
    + \sum_{j: a_j \neq 0}^{n} W_{2}(\xi_{j},\bar{\xi}_{j})
\end{align*}
\end{lem}
\proof{Proof of \cref{lem:NonGaussian}}
Let $\bar{\bZ}$ be $\bZ$, but with the $j^{{\rm th}}$ component
replaced by $\bar{Z}_{j}=\mu_{j}+\bar{\xi}_{j}$ and let 
\[
D_{j}(t)=\frac{1}{a_{j}h\sqrt{\nu_{j}}}\mathbb{E}\left[\left.V\left(\bZ+\left(\delta_{j}+t-\xi_{j}\right)\be_{j}\right)-V\left(\bZ+\left(t-\xi_{j}\right)\be_{j}\right)\right|\bZ\right].
\]
We see
\begin{align*}
\left|\mathbb{E}\left[\xi_{j}x_{j}(\bZ)-D_{j}(\xi_{j})\right]\right| & \le\left|\mathbb{E}\left[\mathbb{E}\left[\left.\xi_{j}x_{j}(\bZ)-\bar{\xi}_{j}x_{j}(\bar{\bZ})\right|\bZ^{-j}\right]\right]\right|+\left|\mathbb{E}\left[\bar{\xi}_{j}x_{j}(\bar{\bZ})-D_{j}(\bar{\xi}_{j})\right]\right|\\
 & \qquad+\left|\mathbb{E}\left[\mathbb{E}\left[\left.D_{j}(\bar{\xi}_{j})-D_{j}(\xi_{j})\right|\bZ^{-j}\right]\right]\right|
\end{align*}
By Lemma C.2 \citeGR, we see
\[
\left|\mathbb{E}\left[\left.\xi_{j}x_{j}(\bZ)-\bar{\xi}_{j}x_{j}(\bar{\bZ})\right|\bZ^{-j}\right]\right|\le T\left\Vert \phi-\bar{\phi}\right\Vert _{1}+4\exp\left(-\frac{T^{2}}{2\sigma^{2}}\right)\left(T+\sigma\sqrt{2\pi}\right).
\]
To optimize $T$, we first upperbound the latter term as follows
\begin{align*}
4\exp\left(-\frac{T^{2}}{2\sigma^{2}}\right)\left(T+\sigma\sqrt{2\pi}\right) & =4\exp\left(-\frac{T^{2}}{2\sigma^{2}}+\log\left(T+\sigma\sqrt{2\pi}\right)\right)\\
 & \le4\exp\left(-\frac{T^{2}}{2\sigma^{2}}+\left(T+\sigma\sqrt{2\pi}-1\right)\right),\qquad\text{since }\log t<t-1\\
 & =4\exp\left(-\frac{T^{2}}{2\sigma^{2}}+2T+\sigma\sqrt{2\pi}-1-T\right)\\
 & \le4\exp\left(\sigma\sqrt{2\pi}-1-T\right)
\end{align*}
where the last inequality used the fact that the quadratic $-\frac{T^{2}}{2\sigma^{2}}+2T$
is maximized at $T^{*}=4\sigma^{2}$. Substituting the upperbound,
we see
\[
T\left\Vert \phi-\bar{\phi}\right\Vert _{1}+4\exp\left(-\frac{T^{2}}{2\sigma^{2}}\right)\left(T+\sigma\sqrt{2\pi}\right)\le T\left\Vert \phi-\bar{\phi}\right\Vert _{1}+4\exp\left(\sigma\sqrt{2\pi}-1-T\right)
\]
We see the right hand side is minimized at $T^{*}=\sigma\sqrt{2\pi}-1-\log\left(\frac{\left\Vert \phi-\bar{\phi}\right\Vert _{1}}{4}\right)$.
Thus, we see
\[
    \left|\mathbb{E}\left[\left.\xi_{j}x_{j}(\bZ)-\bar{\xi}_{j}x_{j}(\bar{\bZ})\right|\bZ^{-j}\right]\right|
    \le
    \left(\sigma\sqrt{2\pi}-\log\left(\frac{\left\Vert \phi-\bar{\phi}\right\Vert _{1}}{4}\right)\right)\left\Vert \phi-\bar{\phi}\right\Vert _{1}
\]

By \cref{thm:equiv-in-sample}, we see
\[
\left|\mathbb{E}\left[\left.\bar{\xi}_{j}x_{j}(\bar{\bZ})-D_{j}(\bar{\bZ})\right|\bZ^{-j}\right]\right|\le Ch\log\left(\frac{1}{h}\right).
\]
Finally, since $t\mapsto h\sqrt{\nu_{j}}D_{j}(t)$ is $1-$Lipschitz from \cref{lem:noisy-precision},
we see that
\[
\left|\mathbb{E}\left[\left.D_{j}(\bar{\bZ})-D_{j}(\bZ)\right|\bZ^{-j}\right]\right|\le W_{2}(\xi_{j},\bar{\xi}_{j}).
\]
Putting it all together, we see 
\[
    \left|\mathbb{E}\left[\xi_{j}x_{j}(\bZ)-D_{j}(\xi_{j})\right]\right|
    \le 
    \left(\sigma\sqrt{2\pi}-\log\left(\frac{\left\Vert \phi-\bar{\phi}\right\Vert _{1}}{4}\right)\right)\left\Vert \phi-\bar{\phi}\right\Vert _{1}
    + Ch\log\left(\frac{1}{h}\right)
    + W_{2}(\xi_{j},\bar{\xi}_{j}).
\]
Summing over the $j$ terms, we obtain our result
 \hfill \Halmos \endproof

\subsection{Proof of \cref{thm:EfronStein}.} \label[app]{sec:ProofOfEfronStein}
Before proving the theorem, we require the following lemma.  
\begin{lem}[A $\chi^2$-Tail Bound]\label[lem]{lem:ChiSqBound}
	Consider $\bm \delta = \left( \delta_1, \ldots, \delta_n\right)^\top$ where $\delta_j$ is defined as in the definition of the \Danskin~(\cref{eq:rand-finite-diff-D})
	and $0 < h < 1/e$.  Suppose \cref{asn:Parameters} holds.  Then, 
\[
\Eb{ \| \bm \delta \|^2_2 \Ib{ \| \bm \delta \|_2^2 > \frac{18h n}{\sqrt{\vmin}} } } \ \leq \ \frac{36 h n}{\sqrt{\vmin}}e^{-n}.
\]
\end{lem}
\proof{Proof of \cref{lem:ChiSqBound}.}
Let $Y_1, \ldots, Y_n$ be independent standard normals.  

Observe that since $h < 1/e$, the variance of $\delta_j$ is at most $h^2  + \frac{h}{\sqrt{\nu_j}} \ \leq 2h/\sqrt{\nu_{\min}}$.
Then, for $t > 1$
\begin{align} \notag
\Pb{\| \bm \delta \|_2^2 > \frac{2hn}{\sqrt{\vmin}}( 1 + t) }
&\ \leq \ 
\Pb{ \frac{2h}{\sqrt{\vmin}} \sum_{j=1}^n Y_j^2 > \frac{2hn}{\sqrt{\vmin}} (1+t) }
\\ \notag
& \ = \ 
\Pb{ \frac{1}{n} \sum_{j=1}^n Y_j^2 > 1 + t }
\\ \label{eq:ChiSqTail}
& \ \leq 
e^{-nt/8} ,
\end{align}
where the last inequality follows from \cite[pg. 29]{wainwright2019high}.

Next, by the tail formula for expectation, 
\begin{align*}
\Eb{ \| \bm \delta \|_2^2 \Ib{ \| \bm \delta \|_2^2 > \frac{18hn}{\sqrt{\vmin}}} } 
& \ = \ \int_0^\infty \Pb{ \| \bm \delta \|_2^2 \Ib{ \| \bm \delta \|_2^2 > \frac{18hn}{\sqrt{\vmin}} } > t } dt
\\
& = 
\int_0^{ \frac{18hn}{\sqrt{\vmin}}} \Pb{ \| \bm \delta \|_2^2 >  \frac{18hn}{\sqrt{\vmin}}} dt 
	+ \int_{ \frac{18hn}{\sqrt{\vmin}}}^\infty \Pb{ \| \bm \delta \|_2^2 >  t} dt
\\
&\  \leq \
 \frac{18hn}{\sqrt{\vmin}} e^{-n} + \int_{ \frac{18hn}{\sqrt{\vmin}}}^\infty \Pb{ \| \bm \delta \|_2^2 >  t} dt
 &&(\text{Applying \cref{eq:ChiSqTail}})
\\
&\  \leq \
 \frac{18hn}{\sqrt{\vmin}} e^{-n} + 
 \frac{2hn}{\sqrt{\vmin}}  \int_{8}^\infty \Pb{ \| \bm \delta \|_2^2 >  \frac{2hn}{\sqrt{\vmin}} (1 + s) } ds
\\
&\  \leq \
 \frac{9hn}{\sqrt{\vmin}} e^{-n} + 
 \frac{2hn}{\sqrt{\vmin}}  \int_{8}^\infty e^{-ns/8} ds
&&(\text{Applying \cref{eq:ChiSqTail}})
 \\
 & \ = \ 
 \frac{18hn}{\sqrt{\vmin}} e^{-n} + 
 \frac{16h}{\sqrt{\vmin}} e^{-n}
\end{align*} 
Rounding up and combining proves the theorem.  
\hfill \Halmos \endproof

\smallskip
We can now prove the theorem. 
\proof{Proof of \cref{thm:EfronStein}.} 
Proceeding as in the main body, we bound each of the three terms of the out-of-sample estimator error (\cref{eq:EfronSteinOverall}).  
Before beginning, note that under \cref{asn:Parameters}, $\Var(\delta_j) \leq \frac{3h}{\sqrt{\vmin}}$.  We use this upper bound frequently. 

We start with \cref{eq:EfronStein_delta_terms}.  Consider the $k^\text{th}$ non-zero element of the sum.  By definition of $\DR$, 
\begin{align*}
	\abs{ \DR(\bZ, \bm \delta, \tilde{\bm U}) - \DR(\bZ, \bm \delta^k, \tilde{\bm U}) }
&\ = \ 
\abs{\frac{\delta_k}{h \sqrt{\nu_k} a_k} x_k(\bZ + \delta_k \tilde U_k \be_k)
- 
\frac{\overline \delta_k}{h \sqrt{\nu_k} a_k} x_k(\bZ + \overline \delta_k \tilde U_k \be_k)	}
\\ &  \ \leq \ 
\frac{1}{h\sqrt{\nu_k} a_k} ( \abs{\delta_k} + \abs{\overline \delta_k} )
\end{align*}
Hence, squaring and taking expectations, 
\begin{align*}
	\Eb{ \left( \DR(\bZ, \bm \delta, \tilde{ \bm U} ) - \DR(\bZ, \bm \delta^k, \tilde{\bm U} ) \right)^2}
&\  \leq  \ 
\frac{2}{h^2 \nu_k a^2_k} \left( \Eb{ \delta_k^2} + \Eb{\overline \delta_k^2} \right)
\\
&\  \leq  \ 
\frac{12}{h \vmin^{3/2} a_{\min}}.
\end{align*}
Summing over $k$ shows
\[
\cref{eq:EfronStein_delta_terms} \ \leq \ \frac{6n}{h \vmin^{3/2} a_{\min}}.
\]

We now bound \cref{eq:EfronStein_U_terms}.  Again, consider the $k^\text{th}$ non-zero element. By definition, 
\begin{align*}
	\abs{ \DR(\bZ, \bm \delta, \tilde{\bm U}) - \DR(\bZ, \bm \delta, \tilde{\bm U}^k) }
&\ = \ 
\frac{\abs{\delta_k}}{h \sqrt{\nu_k} a_k} \abs{ x_k(\bZ + \delta_k U_k \be_k) - x_k(\bZ + \delta_k \overline U_k \be_k) }
\\ &\ \leq \ 
\frac{2\abs{\delta_k}}{h \sqrt{\vmin} a_{\min}} 
\end{align*}
Hence, 
\begin{align*}
\Eb{ \left( \DR(\bZ, \bm \delta, \tilde{\bm U}) - \DR(\bZ, \bm \delta, \tilde{\bm U}^k) \right)^2 }
& \ \leq \
\frac{4 }{h^2 \vmin a^2_{\min}}  \Eb{\delta_k^2 }
\ \leq \
\frac{12}{h \vmin^{3/2} a^2_{\min}} 
\end{align*}
Summing over $k$ shows 
\[
\cref{eq:EfronStein_U_terms}  \ \leq \ \frac{6n }{h \vmin^{3/2} a^2_{\min}}.
\]

Finally, we bound \cref{eq:EfronStein}.
For convenience, let $\bm W_k \in \R^n$ be the random vector with components $W_{kj} = x_j(\bZ + \delta_j \tilde U \be_j) - x_j(\bZ^k + \delta_j \tilde U \be_j )$.  Then, 
proceeding as in the main text, we have
\begin{align*}
\left( \DR(\bZ) - \DR(\bZ^k) \right)^2  
& \ \leq \ 
\frac{\| \bm \delta \|^2_2 }{h^2 a^2_{\min} \vmin}  \cdot   \sum_{j=1}^n \left( x_j(\bZ + \delta_j \tilde U \be_j) - x_j(\bZ^k + \delta_j \tilde U \be_j )\right)^2 
\\
& \ \leq \ 
\frac{ \| \bm \delta \|^2_2 }{h^2 a^2_{\min} \vmin}  \cdot   \| \bm W_k \|^2_2
\end{align*}

Notice that $\Eb{ \| \bm \delta \|_2^2} ~=~ O \left( n h / \vmin \right)$.  We upper bound this expression by splitting on cases where $\| \bm \delta \|_2^2 > \frac{18hn}{\sqrt{\vmin}}$ or not.  Note this quantity is much larger than the mean, so we expect contributions when $\| \bm \delta \|_2^2$ is large to be small.  

Splitting the expression yields
\begin{align*}
\left( \DR(\bZ) - \DR(\bZ^k) \right)^2  
& \ \leq \ 
\frac{ \| \bm \delta \|^2_2 }{h^2 a^2_{\min} \vmin}  \cdot   \| \bm W_k \|^2_2  \Ib{ \| \bm \delta \|_2^2 > \frac{18hn}{\sqrt{\vmin}} }
+ 
\frac{ \| \bm \delta \|^2_2 }{h^2 a^2_{\min} \vmin}  \cdot   \| \bm W_k \|^2_2 \Ib{ \| \bm \delta \|_2^2 \leq \frac{18hn}{\sqrt{\vmin}} }
\\
& \ \leq \ 
\frac{n}{h^2 a^2_{\min} \vmin}  \| \bm \delta \|^2_2    \Ib{ \| \bm \delta \|_2^2 > \frac{18hn}{\sqrt{\vmin}} }
\ + \
\frac{ 18n }{h a^2_{\min} \vmin^{3/2}}  \cdot   \| \bm W_k \|^2_2
\end{align*}

Next take an expectation and apply \cref{lem:ChiSqBound} to obtain 
\begin{align*}
\Eb{\left( \DR(\bZ) - \DR(\bZ^k) \right)^2}  
& \ \leq \ 
\frac{36 n^2}{h a^2_{\min} \vmin^{3/2}}  e^{-n}
\ + \
\frac{ 18n }{h a^2_{\min} \vmin^{3/2}}  \cdot   \Eb{\| \bm W_k \|^2_2}
\end{align*}
Finally summing over $k$ shows 
\[
\cref{eq:EfronStein} \ \leq \ 
\frac{36 n^3}{h a^2_{\min} \vmin^{3/2}}  e^{-n}
\ + \
\frac{ 18n^3 }{h a^2_{\min} \vmin^{3/2}}  \cdot   \frac{1}{n^2} \sum_{k=1}^n \Eb{\| \bm W_k \|^2_2}
\]

Finally, we combine all three terms in \cref{eq:EfronSteinOverall} yielding 
\begin{align*}
\Var(D(\bZ)) 
&\ \leq \ 
\frac{6n}{h \vmin^{3/2} a_{\min}}
\ + \ 
\frac{6n }{h \vmin^{3/2} a^2_{\min}} 
\ + \ 
\frac{36 n^3}{2h a^2_{\min} \vmin^{3/2}}  e^{-n}
\ + \
\frac{ 18n^3 }{2h a^2_{\min} \vmin^{3/2}}  \cdot   \frac{1}{n^2} \sum_{k=1}^n \Eb{\| \bm W_k \|^2_2}
\\
& \ \leq \ 
\frac{C_2}{h} \max( n^{3-\alpha}, n),
\end{align*}
by collecting the dominant terms.  
\hfill \Halmos \endproof


\subsection{Implementation Details}
\label[app]{app:ImplementationDetails}
As mentioned, in our experiments we utilize a second-order forward finite difference approximation.  Namely, instead of approximating the derivative using a first-order approximation as in \cref{eq:FiniteDiffFirstOrderRemainder}, we approximate 
\[
\left. \frac{\partial}{\partial \lambda} \Obj(\bZ + \lambda \xi_j \be_j)\right|_{\lambda = 0} 
\ = \ 
\frac{1}{2h\sqrt{\nu_j}a_j} \big( 4 \Obj(\bZ+ h\sqrt{\nu_j} \xi_j \be_j) - \Obj(\bZ + 2h\sqrt{\nu_j} \xi_j \be_j)- 3\Obj(\bZ) \big) + O(h^2).
\]
The coefficients in this expansion can be derived directly from a Taylor Series.  We  then use randomization to replace the unknown $h\xi_j$ and $2h\xi_j$ term as before. 
The $j^\text{th}$ element of our estimator becomes
\[
D_j(\bZ) \equiv \mathbb{E} \left[ \left. \frac{1}{ 2h \sqrt{\nu_j} a_j }  
\Big( 4 V(\bZ + \delta^h_j \be_j) - V(\bZ + \delta^{2h}_j \be_j) - 3V(\bZ) \Big) \right| \bZ \right].
\]
where 
$\delta^h_j \sim \mathcal N \left(0,h^{2}+\frac{2h}{\sqrt{\nu_j}} \right)$
and
$\delta^{2h}_j \sim \mathcal N \left(0,4h^{2}+\frac{4h}{\sqrt{\nu_j}} \right)$.

As mentioned, one can always compute the above conditional expectation by Monte Carlo simulation.  In special cases, a more computationally efficient method is to utilize a parametric programming algorithm to determine the values of $\bx(\bZ + t \be_j)$ as $t$ ranges over $\mathbb{R}$.  Importantly, for many classes of optimization problems, including, e.g., linear optimization and mixed-binary linear optimization,  $\bx(\bZ  + t \be_j)$ is piecewise constant on the intervals $(c_i, c_{i+1})$, taking value $\bx^i$, for $i = 1, \ldots, I$, with 
 $c_0 = -\infty$ and $c_I = \infty$.
In this case, 
\begin{align}\label{ex:simple-closed-form}
	\mathbb{E}\left[\left.V(\bZ+\delta_{j}\be_{j})\right|\bZ\right] 
	= 
	\sum_{i=0}^{I-1} r(\bZ)^{\top}\bx^{i} \cdot 
	\int_{c_{i}}^{c_{i+1}}\phi_{\delta_j}(t)\,dt
	+
	\sqrt{\frac{\sigma^{2}}{2\pi}}
	\exp\left(\frac{ \left(c_{j}-Z_j \right)^{2}}{2\sigma^{2}}\right)
\end{align}
where $\phi_{\delta_j}(\cdot)$ is the pdf of $\delta_j$. These integrals can then be evaluated in closed-form in terms of the standard normal CDF.  A similar argument holds for 
\(
	\mathbb{E}\left[\left.V(\bZ+\delta^{2h}_{j}\be_{j})\right|\bZ\right] .
\)
We follow this strategy in our case study in \cref{sec:experiments}.
\section{Problems that are Weakly Coupled by Variables}

\subsection{Convergence of In-Sample Optimism}

In this section we provide a high-probability bound on
\[
\sup_{\bx^0 \in \Y} \sup_{\btheta \in \Theta} \abs{ \bxi^\top \bx(\bZ, \btheta, \bx^0) - \Eb{\bxi^\top\bx(\bZ, \btheta, \bx^0) } }.
\]

As a first step, we bound the cardinality of 
\[
\mathcal{X}^{\Theta, \Y}(\bZ)\equiv\left\{ 
\left( \bx^0, \bx^1(\bZ, \btheta, \bx^0), \ldots, \bx^K(\bZ, \btheta, \bx^0) \right)
\ : \ \btheta\in\Theta, \ \bx^0 \in \Y
\right\}  \subseteq \R^n.
\]

\begin{lem}[Cardinality of Lifted, Decoupled Policy Class] \label[lem]{thm:cardinality-fully-decoupled}
Under the assumptions of \cref{thm:unif-oos-est-wc-v}, there exists an absolute constant $C$
such that
\[
\log \abs{ \mathcal{X}^{\Theta, \Y}(\bZ) }  \le {\dim(\phi)} \cdot \log \left(C K \abs{\Y}  \mathcal{X}_{\max}^2  \right)
\]
\end{lem}
\begin{proof}{\hspace{-12pt}Proof:}
We adapt a hyperplane arrangement argument from \cite{gupta2019shrunk}.  We summarize the pertinent details briefly.   For any $\bx^0 \in \Y$, 
$\bx^k(\bZ, \btheta, \bx^0)$ is fully determined by the relative ordering of the values
\[
\left \{ \phi(\btheta)^{\top} g_k(\bZ^k, \bx^k, \bx^0) \ : \  \bx^k \in \mathcal X^k(\bx^0) \right \}.
\]
This observation motivates us to consider the hyperplanes in $\R^{\dim(\phi)}$
\begin{equation} \label{eq:defHyperplanes}
	H_{k, \bZ, \bx^0} ( \bx^{k},\bar{\bx}^{k})  =\left\{ \phi(\btheta) \ :\  \phi(\btheta)^{\top}\left(g_k(\bZ^k,\bx^k, \bx^0)-g_k(\bZ^k,\bar{\bx}^k, \bx^0 )\right) =  0\right\}  
\end{equation}
for all $\bx^{k},\bar{\bx}^{k} \in \text{Ext}\left( \mathcal X^k(\bx^0)\right)$, $k=1, \ldots, K$, and $\bx^0 \in \Y$.

On one side of 		$H_{k, \bZ, \bx^0} ( \bx^{k},\bar{\bx}^{k}) $, 
 $\bx^k$ is preferred to $\bar{\bx}^k$ in the policy problem \cref{eq:DefXPolicy_Variables}, on the other side $\bar{\bx}^k$ is preferred, and on the hyperplane we are indifferent. Thus, if we consider drawing all such hyperplanes in $\R^{\dim(\phi)}$, then the vector $\bx(\bZ, \btheta, \bx^0)$ is constant  for all $\phi(\btheta)$ in the relative interior of each induced region.  Hence, $\abs{\mathcal X^{\Theta, \Y}(\bZ) }$ is at most the number of such regions.  
\cite{gupta2019shrunk} prove that the number of such regions is at most $(1+2m)^{\text{dim}(\phi)}$
where $m$ is number of hyperplanes in \cref{eq:defHyperplanes}.  By assumption, $m \leq K \mathcal X_{\max}^2 \abs{ \Y }$, and, hence, 
\(
\abs{\mathcal X^\Theta(\bZ, \bx^0)} \ \leq \left( 1 + 2K \abs{\Y} \mathcal X_{max}^2\right)^{\text{dim}(\phi)}.
\)
Collecting constants yields the bound.
\hfill \Halmos \end{proof}

We can now prove our high-probability tail bound.  
\begin{lem}
[Bounding In-sample Optimism]
\label[lem]{thm:unif-in-samp-opt-fd} 
Under the assumptions of  \cref{thm:unif-oos-est-wc-v}, 
there exists a constant $C$ (depending on $\nu_{\min}$) such
that, for any $R > 1$, with probability $1-e^{-R} $
\[
\sup_{\btheta\in\Theta, \bx^0 \in \Y}
\abs{ \bxi^\top\bx(\bZ, \btheta, \bx^0) - \Eb{ \bxi^\top \bx(\bZ, \btheta, \bx^0)} } \ \leq \ 
C S_{\max} R\sqrt{  K {\rm{dim}}(\phi) \log(K \abs{\Y} \mathcal{X}_{\max} ) }
\]
%
\end{lem}
\begin{proof}{\hspace{-12pt}Proof:}
By triangle inequality, 
\begin{align*}
&\sup_{\btheta\in\Theta, \bx^0 \in \Y}
\abs{ \bxi^\top\bx(\bZ, \btheta, \bx^0) - \Eb{ \bxi^\top \bx(\bZ, \btheta, \bx^0)} } 
\\ \quad & \ \leq \ 
\sup_{\bx^0 \in \Y} \abs{ (\bxi^0)^\top \bx^0 - \Eb{ (\bxi^0)^\top \bx^0} }  + 
\sup_{\bx^0 \in \Y, \btheta \in \Theta} \abs{ \sum_{k=1}^K\sum_{j\in S_k} \xi_j x_j(\bZ,\btheta,\bx^0) - \mathbb{E}\left[ \xi_j x_j(\bZ,\btheta,\bx^0) \right]} 
\end{align*}

Consider the first term.  For a fixed $\bx_0$, this is a sum of sub-Gaussian random variables each with parameter at most $1/\sqrt{\vmin}$.  We apply Hoeffding's inequality to obtain a pointwise bound for a fixed $\bx_0$ and then take a a union bound over $\mathcal{X}^0$.  This yields for some absolute constant $c$,
\[
	\mathbb{P}\left\{ 
		\sup_{\bx^0 \in \mathcal{X}^0} 
		\left| \sum_{j\in S_0} \xi_j x^0_j 
		- 
		\mathbb{E}\left[\xi_j x^0_j\right]\right| 
		> t  \right\} 
	\le 2 |\mathcal{X}^0| \exp \left( - \frac{c\nu_{\min} t^2}{|S_0|} \right),
\]
Rearranging shows that, with probability at least $1-\exp\left\{-R\right\}$,
\[
	\sup_{\bx^0 \in \mathcal{X}^0} 
		\abs{  \sum_{j\in S_0} \xi_j x^0_j  
		- 
		\mathbb{E}\left[\xi_j x^0_j \right] }
	\le \sqrt{ \frac{ |S_0| R } {c \nu_{\min}} \log \left( 2|\mathcal{X}^0| \right)}
\]

For the second component, we apply \cref{thm:Pollard}. We bounded the cardinality $\mathcal{X}^{\Theta, \Y}(\bZ)$
in \cref{thm:cardinality-fully-decoupled}.  Consider the vector
\[
\left(\sum_{j\in S_{1}}\xi_{j}x_{j}(\bZ;\btheta,\bx^0),\dots,\sum_{j\in S_{K}}\xi_{j}x_{j}(\bZ;\btheta,\bx^0)\right).
\]
This vector has independent components. We next 
construct an envelope $\bm{F}(\bZ)$ for it and bound the $\Psi$-norm of $\bm F(\bZ)$. 

Since $0\le x_{j}(\bZ, \btheta, \bx^0)\le1$, we take
\[
\bm{F}(\bZ)=\left(\sum_{j\in S_{1}}\left|\xi_{j}\right|,\dots,\sum_{j\in S_{K}}\left|\xi_{j}\right|\right).
\]
Note that
\[
\left(\sum_{j\in S_{k}}\left|\xi_{j}\right|\right)^{2}\overset{(a)}{\le}|S_{k}|\sum_{j\in S_{k}}\left(\frac{\zeta_{j}}{\sqrt{\nu_{j}}}\right)^{2}\le\frac{S_{\max}}{\nu_{\min}}\sum_{j\in S_{k}}\zeta_{j}^{2}
\]
where $\zeta_{j}\sim\mathcal{N}(0,1)$.
Inequality $(a)$ uses $\left\Vert \bm t \right\Vert _{1}\le\sqrt{d}\left\Vert \bm t \right\Vert _{2}$
for $\bm{t}\in\mathbb{R}^{d}$. Plugging into $\left\Vert \left\Vert \bm{F}(\bxi)\right\Vert _{2}\right\Vert _{\Psi}$
we have 
\begin{align*}
\left\Vert \left\Vert \bm{F}(\bZ)\right\Vert _{2}\right\Vert _{\Psi} 
 &\  \le \ \sqrt{\frac{|S_{\max}|}{\nu_{\min}}}  \cdot \left\Vert \sqrt{\sum_{k=1}^{K}\sum_{j\in S_{k}}\zeta_{j}^{2}}\right\Vert _{\Psi}
\  \overset{(b)}{\le} \ \sqrt{\frac{|S_{\max}|}{\nu_{\min}}\cdot2|S_{\max}|K}.
\end{align*}
Inequality $(b)$ follows from Lemma A.1 iv) of \citeGR. 

Applying \cref{thm:Pollard}, combining the bounds of our two components, and collecting constants, we obtain our result.
\hfill \Halmos \end{proof}

\subsection{Convergence of VGC}

Define $V(\bZ, \btheta, \bx^0)$ and $D(\bZ, \btheta, \bx^0)$ analogously to $V(\bZ, \btheta)$ and $D(\bZ, \btheta)$ with $\bx(\bZ, \btheta)$ replaced by $\bx(\bZ, \btheta, \bx^0)$ throughout.  The next step in our proof provides a high-probability bound on 
\[
\sup_{\btheta \in \Theta, \bx^0 \in \Y} \abs{ D(\bZ, \btheta, \bx^0) - \Eb{D(\bZ, \btheta, \bx^0) }}.
\]

We first establish a pointwise bound for a fixed $\btheta, \bx^0$.  
\begin{lem}
[Pointwise Convergence of \Danskin~for a fixed $\bx^0$]\label[lem]{lem:p-wise-decoupled-Danskin} 
Under the assumptions of \cref{thm:unif-oos-est-wc-v}, 
for fixed $\btheta, \bx^0$,  there exists a constant $C$ (that depends
on $\nu_{\min}$) such that, for any $ R > 1$, we have with probability $1-2\exp(-R)$,
\[
\left|D(\bZ, \btheta)-\mathbb{E}\left[D(\bZ, \btheta)\right]\right|\le C\left|S_{\max}\right|\sqrt{\frac{KR}{h}}
\]
\end{lem}
\begin{proof}{\hspace{-12pt}Proof:}
By definition, 
\begin{equation} \label{eq:HoeffdingForDanskin}
D(\bZ, \btheta, \bx^0) 
=
\sum_{\substack{j\in S_{0} \\  a_j(\btheta) \neq 0}} D_j(\bZ, \btheta, \bx^0) \ + \ 
\sum_{k=1}^{K}\mathbb{E}\left[\left.\sum_{\substack{j\in S_{k} \\  a_j(\btheta) \neq 0}} \frac{V(\bZ+\delta_{j}\be_{j}, \btheta, \bx^0)-V(\bZ, \btheta, \bx^0)}{a_{j}(\btheta)h\sqrt{\nu_{j}}}\right|\bZ\right].
\end{equation}
Consider the first term.  Since $\bx^0$ is fixed (deterministic),
\[
D_j(\bZ, \btheta, \bx^0) \ = \ \frac{1}{a_j(\btheta) h \sqrt{\nu_j} } \Eb{  V(\bZ+\delta_{j}\be_{j}, \btheta, \bx^0)-V(\bZ, \btheta, \bx^0) \Big| \bZ }
\ = \ \frac{1}{h \sqrt{ \nu_j} }\Eb{ \delta_j x^0_j} \ = 0.
\]
This equality holds almost surely.  Hence, it suffices to focus on the second term in \cref{eq:HoeffdingForDanskin}.  

Importantly, the second term  is a sum of $K$ \emph{independent} random variables for a fixed $\bx^0$.  We next claim that each of these random variables is bounded.  For any $j$ such that $a_j(\btheta) \neq 0$, let $S_i$ be the subproblem such that $j \in S_i$.  Then, write 
\begin{align*}
&\frac{V(\bZ+\delta_{j}\be_{j}, \btheta, \bx^0)-V(\bZ, \btheta, \bx^0 )}{a_{j}(\btheta)h\sqrt{\nu_{j}}}
 \\
&\qquad =
\frac{1}{a_j(\btheta) h \sqrt{\nu_j}} \Big( \underbrace{\br^k(\bZ, \btheta, \bx^0)^\top \left( \bx^k(\bZ + \delta_j \be_j, \btheta, \bx^0) - \bx^k(\bZ, \btheta, \bx^0) \right)}_{ \leq 0 \text{ by optimality of } \bx^k(\bZ, \btheta, \bx^0) } \; + \; a_j(\btheta) \delta_j x_j(\bZ + \delta_j, \be_j, \btheta, \bx^0)  \Big)
%
\\
&\qquad \leq \  \frac{\delta_{j}x_{j}(\bZ+\delta_{j}\be_{j}, \btheta, \bx^0 )}{h\sqrt{\nu_{j}}}
\\
&\qquad \leq \  \frac{ \abs{ \delta_{j}} }{h\sqrt{\nu_{j}}}.
\end{align*}
By an analogous argument, 
\[
\frac{V(\bZ, \btheta, \bx^0)-V(\bZ+\delta_{j}\be_{j}, \btheta, \bx^0 )}{a_{j}(\btheta)h\sqrt{\nu_{j}}}\le\frac{\abs{ \delta_{j}} }{h\sqrt{\nu_{j}}}.
\]
Hence,
\begin{equation}
\abs{ D_j(\bZ, \btheta, \bx^0)} 
\ \leq \ 
\Eb{\abs{ \frac{V(\bZ+\delta_{j}\be_{j}, \btheta, \bx^0)-V(\bZ, \btheta, \bx^0 )}{a_{j}(\btheta)h\sqrt{\nu_{j}}} } \Big| \bZ } 
 \ \leq \frac{1}{h \sqrt{\nu_j}} \Eb{ \delta_j } \ \leq \ \frac{1}{h \sqrt{\nu_j}}  \cdot \frac{\sqrt{h}}{\nu_j^{1/4}}  \ \leq \frac{1}{\sqrt h \vmin^{3/4}}.
\label{eq:fully-decoup-in-samp-1}
\end{equation}

%
Applying Hoeffding's inequality to \cref{eq:HoeffdingForDanskin} and collecting constants shows 
\begin{align*}
\text{\ensuremath{\mathbb{P}}}\left\{ \left|D(\bZ;\btheta)-\mathbb{E}\left[D(\bZ;\btheta)\right]\right|\ge t\right\}  
\le 2\exp\left(-\frac{h\cdot t^{2}}{C_{0}K\left|S_{\max}\right|^{2}}\right)
\end{align*}
for some constant $C_0$ (depending on $\vmin$). Thus, with probability
$1-\epsilon$, we see
\[
\left|D(\bZ;\btheta)-\mathbb{E}\left[D(\bZ;\btheta)\right]\right|\le\sqrt{\frac{C_{0}\cdot K\left|S_{\max}\right|^{2}}{h}\log\left(\frac{2}{\epsilon}\right)}
\]
Combining constants completes our proof. \hfill \Halmos 
\end{proof}

\smallskip 
We now bound 
\[
\sup_{\btheta \in \Theta} \abs{ D(\bZ, \btheta, \bx^0) - \Eb{ D(\bZ, \btheta, \bx^0)} }
\]
for a fixed $\bx^0$.  The key idea is to use the Lipschitz continuity of $\btheta \mapsto D(\bZ, \btheta, \bx^0)$ to cover the set $\Theta$.  
\blue{
\begin{lem}[Uniform Convergence of \Danskin~for a fixed $\bx^0$]
\label[lem]{lem:unif-danskin-decoupled}
Under the assumptions of \cref{thm:unif-oos-est-wc-v} and for $\mathcal{H} \equiv [h_{\min},h_{\max}]$, 
there exists a constant $C$ (that depends
on $\nu_{\min}$, $C_{\mu}$, L) such that for any $R> 1$, we have with probability $1-2e^{-R}$,
\[
\sup_{\substack{\btheta\in \bar{\Theta}}}\left|D(\bZ, \btheta)-\mathbb{E}\left[D(\bZ, \btheta)\right]\right|
\le 
C S_{\max} \sqrt{\frac{KR}{h_{\min}}} \sqrt{ \log n \cdot \log N\left( \sqrt{\frac{h_{\min}}{Kn^2}}, \Theta\right)N\left( \frac{h_{\min}}{K}, \mathcal{H}\right)}.
\]
\end{lem}
}

\proof{Proof.}
\blue{
Using the full notation $D(\bZ,(\btheta,h))$, we first write the supremum to be over $\btheta \in \Theta$ and $h \in \mathcal{H}$.
Let $\Theta_0$ be a $\varepsilon_1$-covering of $\Theta$ and let $\overline{\mathcal{H}}$ be a $\varepsilon_2$-covering of $\mathcal{H} \equiv [h_{\min}, h_{\max}]$.  Then, 
\begin{subequations} \label{eq:OneStepDiscretization}
\begin{align} \label{eq:CoveringPortion}
\sup_{\substack{\btheta \in \Theta \\ h \in \mathcal{H}}} 
& \abs{D(\bZ, (\btheta, h)) - \Eb{D(\bZ, (\btheta, h))}}  \ \leq \ 
\sup_{\substack{\btheta \in \Theta_0 \\ h \in \overline{\mathcal{H}}}} \abs{D(\bZ, (\btheta, h)) - \Eb{D(\bZ, (\btheta, h))}}  
\\  \label{eq:LipschitzPortion}
& \ + \ 
\sup_{\substack{\btheta, \overline{\btheta} : \| \btheta - \overline \btheta\| \leq \varepsilon_1 \\ h \in \mathcal{H}} } \abs{D(\bZ, (\btheta, h)) - D(\bZ, (\overline{\btheta}, h)) } 
+
\sup_{ \substack{\btheta, \overline{\btheta} : \| \btheta - \overline \btheta\| \leq \varepsilon_1 \\ h \in \mathcal{H}} } \abs{\Eb{ D(\bZ, (\btheta, h)) - D(\bZ, (\overline{\btheta}, h)) } } 
\\ \label{eq:hLipschitzPortion}
& \ + \
\sup_{\substack{\overline{\btheta} \in \Theta_0 \\ h,\overline{h} :  \| h - \overline h \| \leq \varepsilon_2}} \abs{D(\bZ, (\overline{\btheta}, h)) - D(\bZ, (\overline{\btheta}, \overline{h})) } 
+
\sup_{ \substack{\overline{\btheta} \in \Theta_0 \\ h,\overline{h} :  \| h - \overline h \| \leq \varepsilon_2} } \abs{\Eb{ D(\bZ, (\overline{\btheta}, \overline{h})) - D(\bZ, (\overline{\btheta}, h)) } } 
\end{align}
\end{subequations}
%
We bound \cref{eq:LipschitzPortion} and \cref{eq:hLipschitzPortion} using \cref{lem:D-is-Lipschitz}.  
For \cref{eq:LipschitzPortion}, there exists a constant $C_1$ such that with probability at least $1-e^{-R}$
\[
\abs{D(\bZ, (\btheta, h)) - D(\bZ, (\overline{\btheta}, h)) }
\ \leq \ 
	\frac{C_1 \varepsilon_1 n^2 }{h}\sqrt{R \log n}.
\] 
Similarly, there exists $C_2$,  $C_3$ and $C_4$ (depending on $\vmin, L, C_{\mu}, a_{\min}, a_{\max}, b_{\max}$) such that 
\begin{align*}
\abs{\Eb{D(\bZ, (\btheta, h)) - D(\bZ, (\overline{\btheta}, h))} }
\ \leq \ \frac{C_2 \varepsilon_1 n^2 }{h } \left( \Eb{ \| \bz \|_\infty}  + 1 \right)
\ \leq \ \frac{C_3 \varepsilon_1 n^2 }{h } \left( \sqrt{ \log n}   + C_\mu \right)
\ \leq \ \frac{C_4 \varepsilon_1 n^2 }{h } \sqrt{ \log n}, 
\end{align*}
where the second inequality uses a standard bound on the maximum of $n$ sub-Gaussian random variables, and we have used \cref{asn:Parameters} to simplify.  Combining and taking the supremum over $h$ shows 
\[
\cref{eq:LipschitzPortion} \ \leq \ \frac{C_5 \varepsilon_1 n^2}{h_{\min}} \sqrt{ R \log n },
\]
for some constant $C_5$.  \\
For \cref{eq:hLipschitzPortion}, there exists a constant $C_6$ such that 
\[
\abs{D(\bZ, (\btheta, \overline{h})) - D(\bZ, (\btheta, h)) }
\ \leq \ 
	\frac{C_6 \varepsilon_2^{1/2} n }{h_{\min}}.
\] 
Since the same holds for the expectation of the same quantity, we see
\[
\cref{eq:hLipschitzPortion} \ \leq \ \frac{2C_6 \varepsilon_2^{1/2} n }{h_{\min}},
\]
}
\blue{
Similarly, using \cref{lem:p-wise-decoupled-Danskin}
and applying a union bound over elements in $\overline{\Theta}$ and $\overline{\mathcal{H}}$ shows with probability at least $1-e^{-R}$, 
\[
\cref{eq:CoveringPortion} \ \leq C_7 S_{\max} \sqrt{\frac{KR}{h_{\min}}\log\left(N(\varepsilon_1,\Theta)N(\varepsilon_2,\mathcal{H})\right)}, 
\]
for some constant $C_7$. Choosing $\varepsilon_1 = \frac{S_{\max} \sqrt{Kh_{\min}}}{n^2}$, $\varepsilon_2 = \frac{S_{\max}^2 K h_{\min}}{n^2}$, we see 
\[
	\cref{eq:CoveringPortion} + \cref{eq:LipschitzPortion} + \cref{eq:hLipschitzPortion} \le C_8 S_{\max} \sqrt{\frac{KR}{h_{\min}}\log\left(N(\varepsilon_1,\Theta)N(\varepsilon_2,\mathcal{H})\right)},
\]
Finally, we obtain our result by simplifying the above bound slightly since $n = \sum_{k=0}^K \abs{S_k} \leq K S_{\max}$, and hence, 
\begin{align*}
\varepsilon_1 = \frac{KS_{\max}\sqrt{Kh_{\min} }}{Kn^2}
\ \geq \ \sqrt{\frac{h_{\min} }{Kn^2}}
\\
\varepsilon_2 = \frac{S_{\max}^2K^2h_{\min}}{Kn^2}
\ \geq \ 
\frac{h_{\min}}{K}.
\end{align*}
Substituting the lower-bounds, we obtain our intended result.
}

%
\hfill \Halmos \endproof

\smallskip
We can now prove \cref{thm:unif-oos-est-wc-v}.

{\blockedit
\begin{proof}{Proof of \cref{thm:unif-oos-est-wc-v}:}
We proceed to bound each term on the right side of 
\cref{eq:ErrorExpansion}.  

To bound \cref{eq:ErrorExpansion_Optimism}, observe by definition of our lifted policy class and \cref{thm:unif-in-samp-opt-fd}, we have, with probability at least $1-e^{-R}$, that
\begin{align*}
\sup_{\btheta\in \bar{\Theta}}
\abs{ \bxi^\top\bx(\bZ, \btheta) - \Eb{ \bxi^\top \bx(\bZ, \btheta)} } 
& \ \leq \ 
\sup_{\btheta\in \bar{\Theta}, \bx^0 \in \Y}
\abs{ \bxi^\top\bx(\bZ, \btheta, \bx^0) - \Eb{ \bxi^\top \bx(\bZ, \btheta, \bx^0)} } 
\\
&\ \leq \ 
C S_{\max} R\sqrt{  K {\rm{dim}}(\phi) \log(K \abs{\Y} \mathcal{X}_{\max} ) }.
\end{align*}

To bound \cref{eq:ErrorExpanion_VGC}, let $\mathcal{H} \equiv [h_{\min},h_{\max}]$. Then, by applying the union bound to \cref{lem:unif-danskin-decoupled} with \mbox{$R \leftarrow R + \log(1 + \abs{\Y})$} we have that with probability at least $1-2e^{-R}$, 
\begin{align*}
\sup_{\substack{\btheta \in \Theta \\ \bx^0 \in \Y}} \left|D(\bZ, \btheta)-\mathbb{E}\left[D(\bZ, \btheta)\right]\right|
& \ \le \ 
C_1 S_{\max} \sqrt{\frac{K(R  + \log(\abs{1 + \Y})}{h_{\min}}} \sqrt{ \log n \cdot \log \left( N\left( \varepsilon_1, \Theta\right)N\left( \varepsilon_2, \mathcal{H}\right) \right)}
\\
&\ \leq \ 
C_2 S_{\max} \sqrt{\frac{KR \log(\abs{1 + \Y}) }{h_{\min}}} \sqrt{ \log n \cdot \log \left( N\left( \varepsilon_1, \Theta\right)N\left( \varepsilon_2, \mathcal{H}\right) \right)}, 
\end{align*}
for some constants $C_1$ and $C_2$ and where $\varepsilon_1 = \sqrt{\frac{h_{\min}}{Kn^2}}$ and $\varepsilon_2 = \frac{h_{\min}}{K}$.  


Finally, to bound \cref{eq:Bias}, use \cref{thm:equiv-in-sample} and take the supremum over $h\in \mathcal{H}$ to obtain
\[
	\cref{eq:Bias} \ \leq \ C_5 h_{\max} K S_{\max} \log(1/h_{\min}).
\]

Substituting these three bounds  into \cref{eq:ErrorExpansion} and collecting constants proves the theorem.  
\hfill \Halmos \end{proof}
}

\section{Problems that are Weakly Coupled by Constraints}
\subsection{Properties of the Dual Optimization Problem} 
\label{sec:fund_dual_prop}
Throughout the section, we use the notation $\langle \ell, u \rangle$ to denote the interval $[\min(\ell, u), \max(\ell, u)]$.    

Recall our dual formulation
\[
\blambda(\bz, \btheta) \in \arg \min_{\blambda \geq \bm 0} \ \L(\blambda, \bz, \btheta), \ 
\text{ where } \ 
\L(\blambda, \bz, \btheta) \ \equiv \ \bm b^\top \blambda + \frac{1}{n} \sum_{j=1}^n \left[ r_j(z_j, \btheta) - \bA_j^\top \blambda\right]^+.
\]
Since $\L(\blambda)$ is non-differentiable, its (partial) subgradient is not-unique.  We identify a particular subgradient by 
\[
\nabla_{\blambda} \L(\blambda, \bz, \btheta) \ = \ \bm{b} - \frac{1}{n} \sum_{j=1}^n \Ib{ r_j(z_j, \btheta) > \bA_j^\top \blambda } \bA_j.
\]

The following identity that characterizes the remainder term in a first order Taylor-series expansion of $\L(\blambda)$ with this subgradient.  
\begin{lem}[A Taylor Series for $\L(\blambda)$] \label[lem]{lem:TaylorIdentity}
For any $\blambda \in \R^m_+$, $\bz \in \R^n$, $\btheta \in \Theta$, 
\[
\L(\blambda_2, \bz, \btheta ) - \L(\blambda_1, \bz, \btheta) = \nabla_{\blambda} \L(\blambda_1, \bz, \btheta)^\top (\blambda_2 - \blambda_1) + 
\frac{1}{n} \sum_{j=1}^n \Ib{ r_j(z_j, \btheta)  \in \langle \bA_j^\top \blambda_1, \bA_j^\top \blambda_2 \rangle } \abs{ r_j(\bz, \btheta) - \bA_j^\top \blambda_2}.
\]
\end{lem}
\proof{Proof of \cref{lem:TaylorIdentity}:}
Since $\bz$ and $\btheta$ are fixed, drop them from the notation.  Using the definitions of $\L$ and $\nabla_{\blambda}\L$, we see 
it is sufficient to prove that for each $j$, 
\begin{equation} \label{eq:TaylorIdentity_j}
[r_j - \bA_j^\top \blambda_2]^+ - [r_j - \bA_j^\top \blambda_1]^+ + \Ib{ r_j > \bA_j^\top \blambda_1} \bA_j^\top(\blambda_2 - \blambda_1) 
\ = \ 
\Ib{ r_j \in \langle \bA_j^\top \blambda_1, \bA_j^\top \blambda_2 \rangle} \abs{r_j - \bA_j^\top \blambda_2}
\end{equation}

First notice that if $\bA_j^\top \blambda_1 = \bA_j^\top \blambda_2$, then both sides of \cref{eq:TaylorIdentity_j} equal zero.
Further, if \mbox{$r_j \not\in \langle \bA_j^\top \blambda_1, \bA_j^\top \blambda_2 \rangle$}, then both sides are again zero.  
Thus, we need only considering the case where \mbox{$r_j \in \langle \bA_j^\top \blambda_1, \bA_j^\top \blambda_2 \rangle$}.  
We can confirm the identity directly by considering  the cases where \mbox{$\bA^\top\blambda_1 < \bA^\top \blambda_2$} and \mbox{$\bA^\top\blambda_1 > \bA^\top \blambda_2$} separately.
\hfill \Halmos \endproof

\smallskip
The following result is proven in Lemma D.3 of \citeGR.  We reproduce it here for completeness.  
\begin{lem}[Dual Solutions Bounded by Plug-in]\label[lem]{lem:BoundedDualSol} If $\mathcal{X}$ is $s_0$-strictly feasible, then

i) $\|\blambda(\bz,\btheta)\|_1 \le \frac{2}{ns_0} \| \bm{r}(\bz,\btheta) \|_1$

ii) $\mathbb{E}\left[ \|\blambda(\bz,\btheta)\|_1 \right] \le \frac{2}{ns_0} \mathbb{E}\left[ \| \bm{r}(\bz,\btheta) \|_1 \right]$
\end{lem}
\proof{Proof of \cref{lem:BoundedDualSol}:} By optimality, 
$\mathcal{L}(\blambda(\bz,\btheta),\bz,\btheta) 
\le \mathcal{L}(\bm{0},\bz,\btheta)
\le \frac{1}{n}\| \bm{r}(\bz,\btheta) \|_1.$
Since \mbox{$\blambda(\bz,\btheta)\ge \bm{0}$},  it follows that $\| \blambda(\bz,\btheta) \|_1 = \be^{\top}\blambda(\bz,\btheta)$. Thus, 
\begin{align*}
	\| \blambda(\bz,\btheta) \|_1 & \  \leq  \ \max_{\blambda \ge 0} \be^\top \blambda \\
	& \text{s.t.} \quad \bm{b}^{\top}\blambda 
	+ \frac{1}{n} \sum_{j=1}^n 
	\left(r_j(z_j,\btheta) - \bA_j^\top\blambda\right)^+
	\ \leq \  \frac{1}{n} \| \bm{r}(\bz,\btheta) \|_1.
\end{align*}
We upper bound this optimization by relaxing the constraint with penalty $1/s_0 > 0$ to see that
\begin{align*}
	\| \blambda(\bz,\btheta) \|_1 & 
	\leq \max_{\blambda \ge 0} \quad \be^\top \blambda 
	+ \frac{1}{s_0} \left( \frac{1}{n} \| \bm{r}(\bz,\btheta) \|_1
	- \bm{b}^{\top}\blambda 
	- \frac{1}{n} \sum_{j=1}^n  \left(r_j(z_j,\btheta) 
	- \bA_j^\top\blambda\right)^+  \right)
	\\
	& =
	\max_{\blambda \ge 0} \quad \be^\top \blambda 
	+ \frac{1}{s_0} \left( \frac{1}{n} \| \bm{r}(\bz,\btheta) \|_1
	- \bm{b}^{\top}\blambda 
	- \frac{1}{n} \sum_{j=1}^n \max_{x_j\in[0,1]} x_j\left(r_j(z_j,\btheta) 
	- \bA_j^\top\blambda\right) \right)
	\\
	& \leq 
	\max_{\blambda \ge 0} \quad \be^\top \blambda 
	+ \frac{1}{s_0} \left( \frac{1}{n} \| \bm{r}(\bz,\btheta) \|_1
	- \bm{b}^{\top}\blambda 
	- \frac{1}{n} \sum_{j=1}^n x_j^0\left(r_j(z_j,\btheta) 
	- \bA_j^\top\blambda\right) \right)
	\\
	&  = 
	\max_{\blambda \ge 0} \quad 
	\left(\be - \frac{1}{s_0}\bm{b} + \frac{1}{ns_0} \bA\bx^0 \right)^{\top} \blambda
	+ \frac{1}{ns_0}\left(
		 \| \bm{r}(\bz,\btheta) \|_1 
		-  \bm{r}(\bz,\btheta)^{\top} \bx^0
	\right).
\end{align*}
By $s_0$-strict feasibility, $\frac{1}{n}\bA \bx^0 + s_0 \be \le \bm{b} \longleftrightarrow \be - \frac{1}{s_0}\bm{b} + \frac{1}{ns_0} \bA \bx^0 \le 0$.  Hence, $\blambda = \bm{0}$ is optimal for this optimization problem.  Thus, for all $\btheta \in \Theta$,
\begin{align*}
	\| \blambda(\bz,\btheta) \|_1 
\  \le \  
\frac{1}{ns_0}\left(
		\| \bm{r}(\bz,\btheta) \|_1
		- \bm{r}(\bz,\btheta)^{\top} \bx^0 \right)
\ \leq \ 
	\frac{2}{ns_0} \| \bm{r}(\bz,\btheta) \|_1. 
\end{align*}
This proves i). Applying the expectation to both sides completes the proof.
\hfill \Halmos \endproof

\subsection{Constructing the Good Set} 
\label{sec:GoodSet}
To construct the set of $\bZ$ where approximate strong convexity holds or the ``good'' set, we first define the following constants:
\begin{subequations}
\begin{align} \label{eq:DefLambdaMax}
\lambda_{\max} &\equiv \frac{2}{s_0} \left( a_{\max} \left(C_\mu  + \frac{4}{\sqrt \vmin} \right) + b_{\max} \right),
\\\label{eq:DefPhiMin}
\phi_{\min} &\equiv 		
\frac{\sqrt{\vmin}} {a_{\max}\sqrt{2 \pi}}
	\exp \left(-\frac{\vmax (a_{\max} C_{\mu} + b_{\max} + C_A \lambda_{\max})^2 }{2 a_{\min}^2}\right).
	\\ \label{eq:DefBigLambda}
\Lambda_n &= \big\{ (\blambda_1, \blambda_2) \in \R^m_+ \times \R^m_+ \ : \ \| \blambda_1 \|_1 \ \leq \lambda_{\max},  \ \| \blambda_2 \|_1 \leq \lambda_{\max}, \  \| \blambda_1 - \blambda_2 \|_2 \geq 4/n
\big\}, \text{ and } 
\\ \label{eq:DefTofTheta}
T_n & = \left\{ (\blambda, \btheta, \Gamma) \in 
\R^m_+ \times \Theta \times R \ : \ \| \blambda \|_1 \leq \lambda_{\max}, \; \Gamma \geq \frac{1}{n} \right\}.
\end{align}
\end{subequations}
These values depend on the constants defined in \cref{asn:strict-feasibility} and \cref{asn:MatrixAConditions}.

We now define the ``good'' set,
\begin{align*}
\mathcal E_n \ \equiv \ \Biggr\{ \bz \ : \ &  
\big( \nabla_{\blambda} \L(\blambda_1, \bz, \btheta) - \nabla_{\blambda} \L(\blambda_2, \bz, \btheta) \big)^\top (\blambda_1 - \blambda_2)
\\
& \qquad \ \geq \ 
\phi_{\min} \beta \| \blambda_1 - \blambda_2\|_2^2 
- \| \blambda_1 - \blambda_2\|_2^{3/2} V^2 \log(V) \frac{\log^2 n }{\sqrt{n}}
 \quad \forall (\blambda_1, \blambda_2) \in \Lambda_n, \; \forall \btheta \in \Theta,
\\
&\frac{1}{n} \sum_{j=1}^n \Ib{ \abs{r_j(Z_j, \btheta) - \bA_j^\top \blambda} \ \leq \Gamma}   
 \ \leq \ 
\Gamma \sqrt{\vmax}  + 
 \Gamma^{1/2}V \log(V) \frac{\log^2 n }{\sqrt n}, 
\quad \forall (\blambda, \btheta, \Gamma) \in T_n
\\
& \| \bz \|_1 \leq n C_\mu + \frac{2n}{\sqrt{\vmin}}, 
\\
& \| \bz \|_\infty  \leq \log n
\Biggr\},
\end{align*} 

For clarity, we stress that $\phi_{\min} > 0$ and $\lambda_{\max} > 0$ are dimension independent constants.  

We show in the next section that $\Pb{ \bZ \not\in \mathcal E_n } = \tilde{O}(1/n)$.  Thus, the event $\{ \bZ \in \mathcal E_n \}$ happens with high-probability, and we will perform our subsequent probabilistic analysis conditional on this ``good" set.  

\subsection{Bounding the ``Bad" Set} 
 \label{sub:bounding_the_bad_set}
The purpose of this section is to bound $\Pb{\bZ \notin \mathcal E_n}$.  Since, $\mathcal E_n$ consists of four conditions, we treat each separately.  The last two conditions on $\| \bZ\|_1$ and $\| \bZ \|_\infty$ can be analyzed using standard techniques for sub-Gaussian random variables.
\begin{lem}[Bounding $\| \bZ \|_1$] \label[lem]{lem:Bounding1NormZ}  Under \cref{asn:Parameters},
\[
  \Pb{ \| \bZ \|_1  > n C_{\mu} + \frac{2n}{\sqrt{\vmin}} }
  \ \leq \   e^{-n/32}.
\]
\end{lem}
\proof{Proof of \cref{lem:Bounding1NormZ}:}
Note that $\Eb{ \abs{ Z_j - \mu_j}} \ \leq \frac{1}{\sqrt{\vmin}}$ by Jensen's inequality. Furthermore, because each $Z_j$ is sub-Gaussian with variance proxy $\frac{1}{\vmin}$, we have by Lemma A.1 of \citeGR~ that 
$\| \abs{ Z_j - \mu_j}\|_\Psi \ = \| Z_j - \mu_j\|_\Psi \ \leq \ \frac{2}{\sqrt{\vmin}}$. 
Thus, $\abs{Z_j - \mu_j} - \Eb{\abs{Z_j - \mu_j} }$ is a mean-zero sub-Gaussian random variable with variance proxy at most $\frac{16}{\vmin}$.  Finally, observe 
 $\| \bZ \|_1 \ \leq \ n C_{\mu} + \sum_{j=1}^n \abs{Z_j - \mu_j}$.  Hence, 
\begin{align*}
 \Pb{ \| \bZ \|_1  > n C_{\mu} + \frac{2n}{\sqrt{\vmin}} }
    & \ \leq \ 
 \Pb{ \frac{1}{n} \sum_{j=1}^n \abs{Z_j - \mu_j}  > \frac{2}{\sqrt{\vmin}} }
 \\
 & \ \leq \ 
  \Pb{ \frac{1}{n} \sum_{j=1}^n \abs{Z_j - \mu_j} - \Eb{\abs{Z_j - \mu_j} }  > \frac{1}{\sqrt{\vmin}} }
\\
& \ \leq \ 
e^{-\frac{n}{32}},
\end{align*}
by the usual bound on the sum of independent sub-Gaussian random variables. 

\hfill \Halmos \endproof

\begin{lem}[Bounding $\| \bZ \|_\infty$] \label[lem]{lem:BoundingInfNormZ} Under \cref{asn:Parameters}, there exists a dimension independent constant $n_0$ such that for all $n \geq n_0$, 
 \[
   \Pb{ \| \bZ \|_\infty  > \log n}
   \ \leq \   \frac{1}{n^2}.
 \]
 \end{lem}
\proof{Proof of \cref{lem:BoundingInfNormZ}:}
By \cite{wainwright2019high}, 
 \(
 \Eb{\| \bZ \|_\infty \|} \ \leq \ C_1 \sqrt{\log n},
 \)
 for some dimension independent constant $C_1$.  Moreover, by \cite[Example 2.29]{wainwright2019high}, 
 \( 
 \| \bZ \|_\infty - \Eb{ \| \bZ\|_\infty}
 \)
 is sub-Gaussian with variance proxy at most $1/\vmin$.  Hence, 
 \begin{align*}
 \Pb{ \| \bZ\|_\infty > \log n}
 & \ \leq \ 
 \Pb{ \| \bZ \|_\infty - \Eb{\| \bZ \|_\infty }
  > \log n - C_1 \sqrt{\log n}}.
\end{align*}
For $n$ sufficiently large, $\log n - C_1 \sqrt{\log n} \geq \frac{1}{2} \log n$.  Hence, for $n$ sufficiently large, this last probability is at most 
 \begin{align*}
 \Pb{ \| \bZ \|_\infty - \Eb{\|\bZ\|_\infty}
 > \frac{1}{2} \log n 
 }
 &\ \leq \ \exp\left(-\frac{ \vmin \log^2 n}{8}\right)
 \\
 & \ \leq \ n^{-\frac{\vmin \log n}{8}}.
 \end{align*}
 For $n$ sufficiently large, the exponent is at most $-2$, proving the lemma. 
\hfill \Halmos \endproof

\smallskip
We next establish that the inequality bounding the behavior over $T_n$ hold with high probability.  As a preparation, we first bound the supremum of a particular stochastic process over this set.
\begin{lem}[Suprema over $T_n$] \label[lem]{lem:BoundingSupremaTn}  Recall the definition of $T_n$ in \cref{eq:DefTofTheta}.  Under \cref{asn:Parameters,asn:LiftedAffinePlugInII,asn:SmoothPlugIn}, there exist dimension independent constants $C$ and $n_0$ such for all $n \geq n_0$, we have that for any $R > 1$, with probability at least $1-e^{-R}$, 
\begin{align*}
&\sup_{(\blambda, \btheta, \Gamma) \in T_n}
\abs{\sum_{j=1}^n \Ib{\abs{ r_j(Z_j, \btheta) - \bA_j^\top \blambda} \leq \Gamma} - \Pb{\abs{ r_j(Z_j, \btheta) - \bA_j^\top \blambda } \leq \Gamma } } \Gamma^{-1/2} \ \leq \ C R V \log V \sqrt n.
\end{align*}
\end{lem}
\proof{Proof of \cref{lem:BoundingSupremaTn}:}
Our goal will be to apply \cref{thm:SmallVariancePollard}.  As a first step, we claim that for a fixed $\blambda, \btheta, \Gamma$, there exists a dimension independent constant $C_1$ such that 
\[
	\Pb{\abs{ r_j(Z_j, \btheta) - \bA_j^\top \blambda }  \ \leq \ \Gamma } \ \leq 
	\sqrt{\nu_j} C_1 \Gamma.
\]
To prove the claim, notice that this quantity is the probability that a Gaussian random variables lives in an interval of length $2\Gamma$.  Upper bounding the density of the Gaussian by its value at its mean shows the probability is at most 
\(
\sqrt{\frac{\nu_j}{2 \pi}}.
\)
Thus, 
\begin{equation} \label{eq:Meanj_ThetaAnalysis}
	\Pb{\abs{ r_j(Z_j, \btheta) - \bA_j^\top \blambda }  \ \leq \ \Gamma } \ \leq 
	2 \Gamma \sqrt{\frac{\nu_j}{2 \pi}} \ \leq 
	\Gamma \sqrt{\nu_j}.
\end{equation}

This upperbound further implies that there exists a dimension independent constant $C_2$ such that the parameter ``$\sigma^2$'' in \cref{thm:SmallVariancePollard} is at most $C_2$, because the indicator squared equals the indicator.  We also take the parameter ``$U$'' to be $\sqrt n$ since $\Gamma \ \geq \frac{1}{n}$.  

Thus, to apply the theorem it remains to show the set 
\[
	\mathcal F \equiv \left\{\left(\Ib{ \abs{r_j(Z_j, \btheta) - \bA_j^\top \blambda} \leq \Gamma} \right)_{j=1}^n \ : \ (\blambda, \btheta, \Gamma) \in T_n\right\}
\]
is Euclidean and compute its parameters.  

Consider the set 
\[
	\mathcal F_1 \equiv \left\{\left(\Ib{ r_j(Z_j, \btheta) - \bA_j^\top \blambda \leq \Gamma} \right)_{j=1}^n \ : \ (\blambda, \btheta, \Gamma) \in T_n\right\}.
\]
By \cref{asn:LiftedAffinePlugInII}, this set has VC-dimension at most $V$, and, hence, also has pseudo-dimension at most $V$.  The same is true of the set 
\[
	\mathcal F_2 \equiv \left\{\left(\Ib{ \bA_j^\top \blambda - r_j(Z_j, \btheta) \leq \Gamma} \right)_{j=1}^n \ : \ (\blambda, \btheta, \Gamma) \in T_n\right\}.
\]
Since $\mathcal F = \mathcal F_1 \wedge \mathcal F_2$, by \cite[Lemma 5.1]{pollard1990empirical} there exists an absolute constant $C_2$ such that $\mathcal F$ has pseudo-dimension at most $C_2V$.  By Theorem A.3 of \citeGR, $\mathcal F$ is Euclidean with parameters $A = (C_2 V)^{6C_2 V}$ and $W = 4C_2 V$.  The relevant complexity parameter ``$V(A,W)$'' is then at most 
\[
	\frac{6C_2 V \log(C_2 V)  + 4C_2 V}{\sqrt{6C_2 V \log(C_2 V)}} \ \leq \ 
	C_3 \sqrt{ V \log V},
\]
for some dimension independent constant $C_3$. 

\Cref{thm:SmallVariancePollard} now bounds the suprema by 
\(
	C_4 R V \log(V) \sqrt n,
\)
completing the proof. 
\hfill \Halmos \endproof

\smallskip
Equipped with \cref{lem:BoundingSupremaTn}, we can now show the relevant condition holds with high-probability.
\begin{lem}[Bounding Away from Degeneracy]\label[lem]{lem:BoundAwayDegeneracy}
Recall the definition of $T_n$ in \cref{eq:DefTofTheta}. Under \cref{asn:Parameters,asn:LiftedAffinePlugInII,asn:SmoothPlugIn} there exists a dimension independent constant $n_0$ such that for all $n \geq n_0$, 
with probability at least $1-1/n$ we have that
\begin{align*}
&\frac{1}{n} \sum_{j=1}^n \Ib{ \abs{r_j(Z_j, \btheta) - \bA_j^\top \blambda} \ \leq \Gamma}   
 \ \leq \ 
\Gamma \sqrt{\vmax} \ + \
 \Gamma^{1/2} V \log(V) \frac{\log^2 n }{\sqrt n}, 
\quad \forall (\blambda, \btheta, \Gamma) \in T_n.
\end{align*}
\end{lem}
\proof{Proof of \cref{lem:BoundAwayDegeneracy}}
Apply \cref{lem:BoundingSupremaTn} with $R =\log n$ to conclude that with probability at least $1-1/n$, for all $(\blambda, \btheta, \Gamma) \in T_n$ simultaneously, we have 
\begin{align*}
&\abs{\sum_{j=1}^n \Ib{\abs{ r_j(Z_j, \btheta) - \bA_j^\top \blambda} \leq \Gamma} - \Pb{\abs{ r_j(Z_j, \btheta) - \bA_j^\top \blambda } \leq \Gamma } } \ \leq \ C \Gamma^{1/2}  V \log(V) \log(n) \sqrt n.
\end{align*}
Then observe that as in the proof of \cref{eq:Meanj_ThetaAnalysis}, 
\(
\Pb{\abs{ r_j(Z_j, \btheta) - \bA_j^\top \blambda } \leq \Gamma } \ \leq \Gamma \sqrt{\vmax}.
\)
Finally, for $n$ sufficiently large, $C \leq \log n$.  Rearranging then completes the proof.  
\hfill \Halmos \endproof
\begin{rem} We describe the condition in \cref{lem:BoundAwayDegeneracy} as ``Bounding Away from Degeneracy" because $r_j(\bZ_j, \btheta)- \bA_j^\top \blambda$ is the reduced cost of the $j^\text{th}$ component at the dual solution $\blambda$.  Hence, the lemma asserts that there are not too many reduced costs that are less than $1/n$.  
\end{rem}

It remains to establish that the approximate strong convexity condition over $\Lambda_n$ holds with high probability.  As preparation, we again bound the suprema of a particular stochastic process.
\begin{lem}[Suprema over $\Lambda_n$]
\label[lem]{lem:SupremaForApproxSC}
Under \cref{asn:Parameters,asn:LiftedAffinePlugInII,asn:SmoothPlugIn}, there exists a dimension independent constant $C$ such that for any $R>1$, with probability at least $1-e^{-R}$, we have 
\begin{align*}
&	\sup_{(\blambda_1, \blambda_2) \in \Lambda_n, \btheta \in \Theta}
	\abs{\frac{1}{n} \sum_{j=1}^n 
	\Big(\Ib{r_j(Z_j, \btheta) \in \langle \bA_j^\top \blambda_1, \bA_j^\top \blambda_2 \rangle} 
	 	 	- 
	 	 	\Pb{r_j(Z_j, \btheta) \in \langle \bA_j^\top \blambda_1, \bA_j^\top \blambda_2 \rangle}\Big) 
	 \frac{\abs{\bA_j^\top(\blambda_1-\blambda_2)}}{\| \blambda_1 - \blambda_2 \|_{2}^{3/2}}
	}
\\
& \qquad 	\ \leq \ 
C RV^2 \log(V) \sqrt n.	
\end{align*}
\end{lem}
\proof{Proof of \cref{lem:SupremaForApproxSC}:}
Our strategy will be to apply \cref{thm:SmallVariancePollard}.  To this end, we first claim that there exists a dimension independent constant $\phi_{\max}$
such that for any fixed $(\blambda_1, \blambda_2) \in \Lambda_n$ and $\btheta \in \Theta$
\[
	\Pb{r_j(Z_j, \btheta) \in \langle \bA_j^\top \blambda_1, \bA_j^\top \blambda_2 \rangle}
	\ \leq \ 
	\phi_{\max} \| \blambda_1 - \blambda_2 \|_1.
\]
To prove the claim, notice that this is the probability that a Gaussian random variable lives in an interval of length at most $\abs{\bA_j^\top(\blambda_1 - \blambda_2)}
\ \leq C_A \| \blambda_1 - \blambda_2\|_1$.  Upper bounding the Gaussian density by the square root of its precision proves the claim.  

We next argue that this claim implies that there exists a dimension independent constant $C_1$ such that the parameter ``$\sigma^2$''
in \cref{thm:SmallVariancePollard} is at most $C_1$.  Indeed, an indicator squared is still the same indicator.  
Scaling by 
\[
\frac{ \abs{\bA_j^\top(\blambda_1 - \blambda_2)}^2}{\| \blambda_1 - \blambda_2\|_2^{3}} \ \leq \	
\frac{C_A^2 m }{\| \blambda_1 - \blambda_2\|_2},
\] 
and then averaging over $j$ proves that $\sigma^2$ is at most $C_1 m$.  

We can take the parameter ``$U$'' to be $C_A \sqrt n$ since
\[
	\frac{\abs{\bA_j^\top(\blambda_1 - \blambda_2)}}{\| \blambda_1 - \blambda_2\|_{2}^{3/2}} \ \leq \ 
	C_A \sqrt{m} \| \blambda_1 - \blambda_2 \|_{2}^{-1/2} \leq C_A \sqrt{mn},
\]
because $(\blambda_1, \blambda_2) \in \Lambda_n$.

Thus, to apply \cref{thm:SmallVariancePollard} we need only show that the set 
\[
	\mathcal F \equiv \left\{\left(\Ib{r_j(Z_j, \btheta) \in \langle \bA_j^\top \blambda_1, \bA_j^\top \blambda_2 \rangle} \frac{\abs{\bA_j^\top(\blambda_1 -\blambda_2)}}{\| \blambda_1 - \blambda_2	\|_{2}^{3/2}} \right)_{j=1}^n \ : \ 
			(\blambda_1, \blambda_2) \in \Lambda_n, \ \btheta \in \Theta\right\}
\]
is Euclidean and determine its parameters.  To this end, first consider the sets
\begin{align*}
\mathcal F_1 &\equiv 
	\left\{\left(\Ib{r_j(Z_j, \btheta) \geq \bA_j^\top \blambda_1}\right)_{j=1}^n \ : \ \blambda_1 \in \R^m_+, \ \btheta \in \Theta\right\}
\\
\mathcal F_2 &\equiv 
	\left\{\left(\Ib{r_j(Z_j, \btheta) \leq \bA_j^\top \blambda_2}\right)_{j=1}^n \ : \ \blambda_2 \in \R^m_+, \ \btheta \in \Theta\right\}
\\
\mathcal F_3 &\equiv 
	\left\{\left(\Ib{r_j(Z_j, \btheta) \geq \bA_j^\top \blambda_2}\right)_{j=1}^n \ : \ \blambda_2 \in \R^m_+, \ \btheta \in \Theta\right\}
\\
\mathcal F_4 &\equiv 
	\left\{\left(\Ib{r_j(Z_j, \btheta) \leq \bA_j^\top \blambda_1}\right)_{j=1}^n \ : \ \blambda_1 \in \R^m_+, \ \btheta \in \Theta\right\}
\end{align*}
By \cref{asn:LiftedAffinePlugInII}, each of these sets has VC-dimension at most $V$ (they are indicator sets for functions with pseudo-dimension at most $V$).  Now define the set 
\[
	\mathcal F_5  \equiv
\left\{	\left(\Ib{r_j(Z_j, \btheta) \in \langle \bA_j^\top \blambda_1, \blambda_2 \rangle}\right)_{j=1}^n \ : \ (\blambda_1, \blambda_2) \in \Lambda_n\right\},
\]  
and notice that 
\[
	\mathcal F_5 \subseteq 
	\left(\mathcal F_1 \wedge \mathcal F_2\right) \vee \left(\mathcal F_3 \wedge \mathcal F_4 \right).
\]
Hence, by \cite[Lemma 5.1]{pollard1990empirical}, there exists an absolute constant $C_2 > 1$ such that $\mathcal F_5$ has pseudodimension at most $C_2 V$.  By Theorem A.3 of \citeGR, $\mathcal F_5$ is thus Euclidean with parameters 
$A = (C_2V)^{6 C_2V}$ and $W = 4C_2 V$.  

Now consider the set
\[
	\mathcal F_6 = 
	\left\{\left(\frac{\bA_j^\top (\blambda_1 - \blambda_2)}{\| \blambda_1 - \blambda_2\|_{2}^{3/2}}\right)_{j=1}^n  \ : \ 
			(\blambda_1, \blambda_1) \in \Lambda_n \right\},
\]
and notice 
\[
\mathcal F_6 \subseteq \{ \left(\bA_j^\top \blambda\right)_{j=1}^n \ : \blambda \in \R^m\}.
\]
This latter set belongs to a vector space of dimension at most $m$, and hence has pseudo-dimension at most $m \leq V$. Thus, by Theorem A.3 of \citeGR, it is Euclidean with parameters at most $A = V^{6V}$ and $W = 4V$.  

To conclude, notice that $\mathcal F$ is the pointwise product of $\mathcal F_5$ and $\mathcal F_6$.  Hence, by \cite[Lemma 5.3]{pollard1990empirical}, we have that 
$\mathcal F$ is Euclidean with parameters 
$A= (C_3V)^{C_3V} \cdot C_3^{C_3 V }$ and 
 $W= C_3 V$ for some absolute constant $C_3$.  In particular, the relevant complexity parameter ``$V(A,W)$'' for $\mathcal F$ is at most 
\(
	C_4\sqrt{ V \log(V)}
\)
for some dimension independent parameter $C_4$. 

Applying \cref{thm:SmallVariancePollard} now shows that suprema of the lemma is at most 
\(
	C_5 R (\sqrt{V \log V})^2 m\sqrt{mn},
\)
for some dimension independent $C_5$. Since $m \leq V$, this completes the lemma. 
\hfill \Halmos \endproof

\smallskip
Equipped with \cref{lem:SupremaForApproxSC}, we can prove the approximate strong convexity condition holds with high probability. 
\begin{lem}[Approximate Strong Convexity with High Probability]
\label[lem]{lem:ApproxStrongConvexCondition}
Under \cref{asn:Parameters,asn:LiftedAffinePlugInII,asn:SmoothPlugIn}, there exists a dimension independent constant $n_0$ such that for all $n \geq n_0$, we have with probability at least $1 - \frac{1}{n}$ that the following inequality holds simultaneously for all $(\blambda_1, \blambda_2) \in \Lambda_n$ and $\btheta \in \Theta$:
\begin{align*}
\big( \nabla_{\blambda} \L(\blambda_1, \bz, \btheta) - \nabla_{\blambda} \L(\blambda_2, \bz, \btheta) \big)^\top (\blambda_1 - \blambda_2) \ \geq \ 
\phi_{\min} \beta \| \blambda_1 - \blambda_2\|_2^2 
- \| \blambda_1 - \blambda_2\|_2^{3/2} V^2 \log(V) \frac{\log^2 n}{\sqrt{n}}.
\end{align*}
\end{lem}

\proof{Proof of \cref{lem:ApproxStrongConvexCondition}:}
By choosing $R = \log n$, \cref{lem:SupremaForApproxSC} shows that there exists a dimension independent constant $C_1$ with probability at least $1-1/n$
\begin{align*}
&	\sup_{(\blambda_1, \blambda_2) \in \Lambda_n, \btheta \in \Theta}
	\frac{1}{n} \sum_{j=1}^n 
	\Big(\Ib{r_j(Z_j, \btheta) \in \langle \bA_j^\top \blambda_1, \bA_j^\top \blambda_2 \rangle} 
	 	 	- 
	 	 	\Pb{r_j(Z_j, \btheta) \in \langle \bA_j^\top \blambda_1, \bA_j^\top \blambda_2 \rangle}\Big) 
	 \frac{\abs{\bA_j^\top(\blambda_1-\blambda_2)}}{\| \blambda_1 - \blambda_2 \|^{3/2}}
\\
& \qquad 	\ \geq \ 
-C V^2 \log(V) \sqrt{n} \log n.	
\end{align*}
This inequality implies that for any $(\blambda_1, \blambda_2) \in \Lambda_n$ and $\btheta \in \Theta$,
\begin{align*}
&\frac{1}{n} \sum_{j=1}^n 
	\Ib{r_j(Z_j, \btheta) \in \langle \bA_j^\top \blambda_1, \bA_j^\top \blambda_2 \rangle}  
	 \abs{\bA_j^\top(\blambda_1-\blambda_2)}
\\
& \qquad 	\ \geq \ 
\frac{1}{n} \sum_{j=1}^n \Pb{r_j(Z_j, \btheta) \in \langle \bA_j^\top \blambda_1, \bA_j^\top \blambda_2 \rangle}\abs{\bA_j^\top(\blambda_1-\blambda_2)}
-C \| \blambda_1 - \blambda_2 \|^{3/2} V^2 \log(V) \frac{\log n}{\sqrt{n}}. 
\end{align*}

Thus our first goal will be to bound the summation on the right side. Isolate the $j^\text{th}$ term.  The probability 
\(
\Pb{r_j(Z_j, \btheta) \in \langle \bA_j^\top \blambda_1, \bA_j^\top \blambda_2\rangle }
\)
is the probability that a Gaussian random variable lives in an interval of length at most $\abs{\bA_j^\top(\blambda_1 - \blambda_2)}$.  Moreover, the endpoints of this interval are most $\abs{ \bA_j^\top \blambda_i} \leq C_A \blambda_{\max}$ for $i =1, 2$, since $(\blambda_1, \blambda_2) \in \Lambda_n$.  It follows that these endpoints are no further than $a_j(\btheta) \mu_j + b_j(\btheta) + C_A \lambda_{\max}$ from the mean of the relevant Gaussian.  Thus, we can lower bound the density of the Gaussian on this interval.  This reasoning proves
\begin{align*}
	& \Pb{r_j(Z_j, \btheta) \in \langle \bA_j^\top \blambda_1, \bA_j^\top \blambda_2\rangle } \abs{\bA_j^\top (\blambda_1 - \blambda_2)}
	\\
	&\ \geq \ 
	\left(\bA_j^\top(\blambda_1 - \blambda_2)\right)^2
	\cdot 
	\frac{\sqrt{\nu_j}} {a_j(\btheta)\sqrt{2 \pi}}
	\exp \left(-\frac{\nu_j (a_j(\btheta) \mu_j + b_j(\btheta) + C_A \lambda_{\max})^2 }{2 a_j(\btheta)^2}\right)
	\\& \ \geq \ 
	\phi_{\min}\left(\bA_j^\top(\blambda_1 - \blambda_2)\right)^2.
\end{align*}

Averaging over $j$ shows
\begin{align*}
\frac{1}{n} \sum_{j=1}^n \Pb{r_j(Z_j, \btheta) \in \langle \bA_j^\top \blambda_1, \bA_j^\top \blambda_2 \rangle}\abs{\bA_j^\top(\blambda_1-\blambda_2)}
& \ \geq \ 
\phi_{\min} (\blambda_1 - \blambda_2)^\top \frac{1}{n} \sum_{j=1}^n \bA_j \bA_j^\top (\blambda_1 - \blambda_2)
\\ & \ \geq  \ 
\phi_{\min}\beta \| \blambda_1 - \blambda_2 \|_2^2,
\end{align*}
by \cref{asn:MatrixAConditions}.

Substitute above, and notice if $n_0 = e^C$, then $\log n \geq C$ to complete the proof.  
\hfill \Halmos \endproof
\smallskip
Finally, a simple union bound gives
\begin{lem}[$\bZ \in \mathcal{E}_n$ with High Probability]
\label[lem]{lem:ZinGoodSetHighProb}  Under \cref{asn:Parameters,asn:LiftedAffinePlugInII,asn:SmoothPlugIn} there exists a dimension independent constant $n_0$ such that for all $n\geq n_0$, $\Pb{\bZ \in \mathcal E_n} \geq 1-\frac{4}{n}$.
\end{lem}
\proof{Proof.} Combine \cref{lem:Bounding1NormZ,lem:BoundingInfNormZ,lem:ApproxStrongConvexCondition,lem:BoundAwayDegeneracy} and apply a union bound.  
\subsection{Properties of the Good Set} 
In this section, we argue that for data realizations $\bz \in \mathcal{E}_n$, our optimization problems satisfy a number of properties, and, in particular, the dual solutions and VGC satisfy a bounded differences condition.  We start by showing that small perturbations to the data $\bz$ still yield dual solutions that are bounded.  Note, any $\bz \in \mathcal E_n$ satisfies the assumptions of the next lemma.  

\begin{lem}[Bounded Duals] \label[lem]{lem:DualsBoundedOnE}
Suppose \cref{asn:SmoothPlugIn,asn:Parameters} hold and 
$\|\bm t \|_\infty \leq  \frac{3\sqrt n}{\sqrt{\vmin}}$ and $\bz$ satisfies $\|\bz\|_{1} \leq nC_{\mu} + \frac{2n}{\sqrt{\vmin}}$.
Then, for all $j =1, \ldots, n$, 
\[
\sup_{\btheta \in \Theta} \ \| \blambda(\bz+ t_j \be_j, \btheta) \|_1 \ \leq \ \lambda_{\max}.
\]
\end{lem}
\proof{Proof of \cref{lem:DualsBoundedOnE}:}
Write
\begin{align*}
\| \blambda(\bz + t_j \be_j, \btheta) \|_1 
& \ \leq \ \frac{2}{n s_0} \| \bm r(\bz + t_j \be_j, \btheta) \|_1 
&&(\text{\cref{lem:BoundedDualSol}})
\\
&\ \leq \ 
\frac{2}{n s_0} \left( a_{\max} \| \bz \|_1 + a_{\max}  \abs{t_j} + b_{\max} n  \right) 
&&(\text{Definition of $\br(\cdot, \btheta)$})
\\ & \ \leq \ 
\frac{2}{s_0} \left( a_{\max} \left( C_\mu + \frac{2}{\sqrt{\vmin}} +  \frac{3}{\sqrt{n \vmin}}  \right)  + b_{\max}  \right) 
&&(\text{by assumptions on } \|\bz\|_1 \text{ and } \| \bm t \|_\infty) 
\\ & \leq \ \lambda_{\max}, 
\end{align*}
since $3/\sqrt{n} \leq \sqrt{ 3} \leq 2$.  
Taking the supremum of both sides over $\btheta \in \Theta$ completes the proof.  
\hfill \Halmos \endproof
\smallskip

We next establish a bounded differences condition for the dual solution $\blambda(\bz, \btheta)$.  
\begin{lem}[Bounded Differences for the Dual] \label[lem]{lem:DualStability}
Suppose \cref{asn:SmoothPlugIn,asn:Parameters} hold and that 
$\bz \in \mathcal E_n$ and $\| \blambda(\bzbar, \btheta) \|  \ \leq \lambda_{\max}$.  Then, there exists a dimension independent constant $C$  such that 
\[ 
\| \blambda(\bz, \btheta) - \blambda(\bar{\bz}, \btheta) \|_2 \ \leq \ 
C V^3\log^2(V) \frac{\log^4 n}{n} \sum_{j=1}^n \Ib{z_j \neq \bar z_j}. 
\]
\end{lem}
\proof{Proof of \cref{lem:DualStability}:} 
To declutter the notation, define
\begin{align*}
f_1(\blambda) \equiv \L(\blambda, \bz, \btheta), \quad \blambda_1 \equiv \blambda(\bz, \btheta), 
\\
f_2(\blambda) \equiv \L(\blambda, \bzbar, \btheta), \quad \blambda_2 \equiv \blambda(\bzbar, \btheta).
\end{align*}
Furthermore, let $I_j = \langle \bA_j^\top\blambda_1, \bA_j^\top \blambda_2 \rangle$.

Notice if $ \| \blambda_1 - \blambda_2 \|_2 \leq 4/n$, the inequality is immediate for $C = 4$ since $m \geq 1$.  Hence, we assume throughout that $\| \blambda_1 - \blambda_2 \|_2 > 4/n$.  

Using \cref{lem:TaylorIdentity} we have that 
\begin{align} \notag
f_1(\blambda_2) - f_1(\blambda_1) 
& \ = \ 
\nabla f_1(\blambda_1)^\top (\blambda_2 - \blambda_1) + \frac{1}{n} \sum_{j=1}^n \Ib{ r_j(z_j) \in I_j} \abs{r_j(z_j) - \bA_j^\top \blambda_2} 
\\ \label{eq:TaylorSeries1}
& \ \geq \ 
\frac{1}{n} \sum_{j=1}^n \Ib{ r_j(z_j) \in I_j} \abs{r_j(z_j) - \bA_j^\top \blambda_2}, 
\end{align}
where the inequality uses $f_1(\cdot)$ is convex and $\blambda_1$ is an optimizer.   Analogously, we have that 
\begin{equation} \label{eq:TaylorSeries2}
f_2(\blambda_1) - f_2(\blambda_2) \ \geq \ 
\frac{1}{n} \sum_{j=1}^n \Ib{ r_j(\bar{z}_j) \in I_j} \abs{ r_j(\bar{z}_j) - \bA_j^\top \blambda_1 }.
\end{equation}
Adding \cref{eq:TaylorSeries1,eq:TaylorSeries2} yields 
\begin{align*}
&f_1(\blambda_2) - f_1(\blambda_1) + f_2(\blambda_1) - f_2(\blambda_2) 
\\ & \quad \ \geq \ 
\frac{1}{n} \sum_{j=1}^n \left( \Ib{ r_j(\bar{z}_j) \in I_j } \abs{ r_j(\bar{z}_j) - \bA_j^\top \blambda_1 }
+ 
 \Ib{ r_j(z_j) \in I_j} \abs{r_j(z_j) - \bA_j^\top \blambda_2} \right)
\end{align*}

Isolate the $j^\text{th}$ term on the right.  To lower bound this term,  note that when $z_j \neq \bar{z}_j$, 
\begin{align*}
&\Big| \Ib{ r_j(\bar{z}_j) \in I_j} \abs{ r_j(\bar{z}_j) - \bA_j^\top \blambda_1 }
- 
\Ib{ r_j(z_j) \in I_j} \abs{ r_j(z_j) - \bA_j^\top \blambda_1 }  \Big| 
\\
& \qquad  \ \leq \ 
\Ib{ r_j(\bar{z}_j) \in I_j } \abs{ r_j(\bar{z}_j) - \bA_j^\top \blambda_1 }
+ 
\Ib{ r_j(z_j) \in I_j} \abs{ r_j(z_j) - \bA_j^\top \blambda_1 }  
\\
& \qquad  \ \overset{\text{(a)}}{\leq} \ 
2 \abs{ \bA_j^\top (\blambda_1 - \blambda_2) }
\\
& \qquad \ \leq \
2 C_A \| \blambda_1 - \blambda_2 \|_2,
\end{align*}
where inequality (a) follows because each indicator is non-zero only when the corresponding $r$ is in the interval $\langle \bA_j^\top \blambda_1, \bA_j^\top \blambda_2 \rangle$.  
Hence, substituting above and rearranging yields
\begin{align} \notag
&f_1(\blambda_2) - f_1(\blambda_1) + f_2(\blambda_1) - f_2(\blambda_2) 
+ 2 C_A \| \blambda_1 - \blambda_2 \|_2 \cdot \frac{1}{n} \sum_{j=1}^n \Ib{ z_j \neq \bar{z}_j} 
\\ \notag
& \quad \ \geq \ 
\frac{1}{n} \sum_{j=1}^n  \Ib{ r_j({z}_j) \in I_j} \left(  \abs{ r_j(z_j) - \bA_j^\top \blambda_2 }  +  \abs{ r_j(z_j) - \bA_j^\top \blambda_1 }   \right)
\\ \label{eq:RInTheMiddle}
& \quad \ = \  
\frac{1}{n} \sum_{j=1}^n  \Ib{ r_j({z}_j) \in I_j} \abs{\bA_j^\top(\blambda_1 - \blambda_2)},
\\  \notag
& \quad = 
\big( \nabla_{\blambda} \L(\blambda_1, \bz, \btheta) - \nabla_{\blambda} \L(\blambda_2, \bz, \btheta) \big)^\top (\blambda_1 - \blambda_2)
\\ \notag
& \quad \geq 
\phi_{\min} \beta \| \blambda_1 - \blambda_2 \|_2^2 -  V^2 \log(V)\cdot \frac{ \log^2 n }{\sqrt{n}}\cdot \| \blambda_1 - \blambda_2 \|_2^{3/2}, \qquad \text{(since $\bz \in \mathcal E_n, (\blambda_1, \blambda_2) \in \Lambda_n)$)}
\end{align}
where \cref{eq:RInTheMiddle} follows because when the indicator is non-zero, $r_j(z_j)$ is between $\bA_j^\top \blambda_1$ and $\bA_j^\top \blambda_2$.     

To summarize the argument so far, we have shown that 
\begin{align} \label{eq:MasterEQ}
f_1(\blambda_2) - &f_1(\blambda_1) + f_2(\blambda_1) - f_2(\blambda_2) 
\; + \; 2 C_A \| \blambda_1 - \blambda_2 \|_2 \; \frac{1}{n} \sum_{j=1}^n \Ib{ z_j \neq \bar{z}_j} 
 \\\notag
 &\ \geq \ 
 \phi_{\min} \beta \| \blambda_1 - \blambda_2 \|_2^2 -  V^2 \log(V) \frac{ \log^2 n }{\sqrt{n}}\| \blambda_1 - \blambda_2 \|_2^{3/2}.
\end{align}
The next step of the proof upper bounds the left side.  
By definition of $f_1(\cdot), f_2(\cdot)$, 
\begin{align*}
&f_1(\blambda_2) - f_1(\blambda_1) + f_2(\blambda_1) - f_2(\blambda_2)  
\\ & \quad \ = \ 
\frac{1}{n} \sum_{j=1}^n \left( \left[ r_j(z_j) - \bA_j^\top \blambda_2\right]^+ 
				-\left[ r_j(z_j) - \bA_j^\top \blambda_1\right]^+  
				+ \left[ r_j(\bar{z}_j) - \bA_j^\top \blambda_1\right]^+
				- \left[ r_j(\bar z_j) - \bA_j^\top \blambda_2\right]^+ \right).
\end{align*}
The $j^\text{th}$ term is non-zero only if $z_j \neq \bar{z_j}$.  In that case, 
\begin{align*}
& \left[ r_j(z_j) - \bA_j^\top \blambda_2\right]^+ 
				-\left[ r_j(z_j) - \bA_j^\top \blambda_1\right]^+  
				+ \left[ r_j(\bar{z}_j) - \bA_j^\top \blambda_1\right]^+
				- \left[ r_j(\bar z_j) - \bA_j^\top \blambda_2\right]^+
\\ & \qquad \leq \ 
2 \abs{ \bA_j^\top (\blambda_2 - \blambda_1) }
\\ & \qquad \leq \ 
2 C_A \| \blambda_1 - \blambda_1 \|_2
\\ & \qquad \leq \
2 C_A \| \blambda_1 - \blambda_1 \|_2.
\end{align*}
Summing over $j$ for which $z_j \neq \bar z_j$ and substituting into the left side of \cref{eq:MasterEQ} yields, 
\begin{align} \label{eq:MasterEq_UB}
4 C_A \| \blambda_1 - \blambda_1 \|_2 \cdot \frac{1}{n} \sum_{j=1}^n \Ib{ z_j \neq \bar z_j}
\ \geq \ 
 \phi_{\min} \beta \| \blambda_1 - \blambda_2 \|_2^2 
 -  V^2 \log(V) \frac{ \log^2 n }{\sqrt{n}}\| \blambda_1 - \blambda_2 \|_2^{3/2}.
 \end{align} 

To simplify this expression, recall that by assumption 
\[
	\| \blambda_1 - \blambda_2 \|_2  \geq \frac{4}{n} 
	\implies 
	\sqrt n \| \blambda_1 - \blambda_2 \|_2^{1/2} \geq 1.
\]
Hence we can inflate the left side of \cref{eq:MasterEq_UB} by multiplying by $\sqrt n \| \blambda_1 - \blambda_2 \|_2^{1/2}$ and then rearranging to obtain
\begin{align*}
	\phi_{\min}\beta \| \blambda_1 - \blambda_2\|_2^2 
&\ \leq \ 
4C_A  \| \blambda_1 - \blambda_2\|^{3/2} \cdot
	\frac{1}{\sqrt n} \sum_{j=1}^n \Ib{z_j \neq \bar z_j} + 
	V^2\log(V) \frac{\log^2 n}{\sqrt{n}} \| \blambda_1 - \blambda_2\|_2^{3/2}
\\
& \ \leq \ 
C_1  V^2 \log(V) \frac{\log^2 n}{\sqrt n} \sum_{j=1}^n \Ib{z_j \neq \bar z_j} \| \blambda_1 - \blambda_2\|_2^{3/2},
\end{align*}
for some dimension independent constant $C_1$.
Dividing both sides by $\| \blambda_1 - \blambda_2\|_2^{3/2}$ and combining constants yields
\begin{align} \label{eq:MasterEq_UB_1}
\| \blambda_1 - \blambda_2 \|_2^{1/2} \ \leq \ C_2V^2 \log(V) \frac{\log^2 n}{\sqrt{n}} \sum_{j=1}^n \Ib{z_j \neq \bar z_j},
 \end{align} 
for some dimension independent constant $C_2$.  
Multiply \cref{eq:MasterEq_UB_1} by $V \log(V) \frac{\log^2 n}{\sqrt n} \|\blambda_1 - \blambda_2\|_2 $ to see that 
\[
	 V\log(V)\frac{ \log^2 n }{\sqrt{n}}\| \blambda_1 - \blambda_2 \|_2^{3/2}
	\leq 
	C_2 V^3\log^2(V) \frac{\log^4 n}{n} \sum_{j=1}^n \Ib{z_j \neq \bar z_j} \; \| \blambda_1 - \blambda_2 \|_2. 
\]
Substitute this upper-bound to \cref{eq:MasterEq_UB},  yielding
\begin{align*} 
4 C_A \| \blambda_1 - \blambda_1 \|_2 \cdot \frac{1}{n} \sum_{j=1}^n \Ib{ z_j \neq \bar z_j}
\ \geq \ 
 \phi_{\min} \beta \| \blambda_1 - \blambda_2 \|_2^2 
 -  C_2 V^3\log^2(V) \frac{\log^4 n}{n} \sum_{j=1}^n \Ib{z_j \neq \bar z_j} \; \| \blambda_1 - \blambda_2 \|_2.
 \end{align*} 

Now divide by $\| \blambda_1 - \blambda_2\|_2 $ and rearrange to complete the proof.  
\hfill \Halmos \endproof

 We now use this result to show the VGC is also Lipschitz in the Hamming distance. The key to the following proof is that that strong-duality shows $V(\bz, \btheta) = n \L(\blambda(\bz, \btheta), \bz, \btheta)$.
\begin{lem}
[Bounded Differences for VGC]\label[lem]{lem:VGC-Dual-Solution-Bnd} Let
$\bz,\overline{\bz}\in\mathcal{E}_{n}$. 
Suppose \cref{asn:SmoothPlugIn,asn:Parameters} hold.
Then, there exists
a dimension independent constant $C$, such that for any $n$ such that $\frac{\log n}{nh}\le 1$, we have that 
\[
	\left|D(\bz,\btheta)-D(\overline{\bz},\btheta)\right|
	\le
	\frac{C}{h}~ 
	V^3\log^2(V) \log^4(n) ~ \sum_{i=1}^{n}\Ib{z_{i}\ne\overline{z}_{i}}\]
\end{lem}

\proof{Proof of \cref{lem:VGC-Dual-Solution-Bnd}:}
Notice if $\bz = \bar{\bz}$ the lemma is trivially true.  Hence, throughout, we assume $\bz \neq \bar{\bz}$.  Since $\btheta$ is fixed throughout, we also drop it from the notation.  

As a first step, we will prove the two inequalities
\begin{subequations} \label{eq:TwoInequalities}
\begin{align}  \label{eq:TwoInequalities_lb}
V(\bz + \delta_{j}\be_{j})-V(\bz)
& \ \geq \ 
\left(r_j(\bz) + a_j \delta_j - \bA_j^\top \blambda(\bz + \delta_j \be_j)\right)^+ - 
\left(r_j(\bz_j) - \bA_j^\top \blambda(\bz + \delta_j \be_j)\right)^+,
\\ \label{eq:TWoInequalities_ub}
V(\bz + \delta_j \be_j) - V(\bz) 
& \ \leq \ 
\left(r_j(\bz) + a_j \delta_j - \bA_j^\top \blambda(\bz)\right)^+ - 
\left(r_j(\bz_j) - \bA_j^\top \blambda(\bz )\right)^+.
\end{align}
\end{subequations}
To prove the first inequality, write 
\begin{align} \label{eq:VGC-dual-sol-bnd-ii}
 V(\bz+\delta_{j}\be_{j})-V(\bz)
= \ &   n\L \left(\blambda\left(\bz+\delta_{j}\be_{j}\right), \bz+\delta_{j}\be_{j} \right)
	- n \L \left(\blambda\left(\bz\right), \bz\right)
\\ \notag
\le \ & n\L \left(\blambda\left(\bz\right), \bz+\delta_{j}\be_{j} \right)
	-n\L \left(\blambda\left(\bz\right), \bz\right)
 \\ \notag
= \ & \left(r_j(z_{j})+a_{j}\delta_{j}-\bA_{j}^{\top}\blambda(\bz)\right)^{+}-\left(r_j(z_{j})-\bA_{j}^{\top}\blambda(\bz)\right)^{+} 
\end{align}
where the inequality holds by the sub-optimality of $\blambda(\bz)$
for $\L\left(\blambda, \bz+\delta_{j}\be_{j}\right)$,
and the last equality holds since all terms except the $j^{\text{th}}$
in the summation of the Lagrangian cancel out.  A similar argument using the sub-optimality of $\blambda(\bz + \delta_j \be_j)$ for $\L (\blambda, \bz + \delta_j \be_j)$ proves the lower bound.  

The next step of the proof establishes that there exists a dimension independent constant $C_1$ such that
\begin{align} \notag
 & \Eb{ V(\bz+\delta_{j}\be_{j})-V(\bz)-\left(V(\overline{\bz}+\delta_{j}\be_{j})-V(\overline{\bz})\right) }
\\\label{eq:VGC-dual-sol-bnd-v}
& \qquad \ \le \   C_1 \left(
V^3\log^2(V) \frac{\log^4 n}{n}\sum_{i=1}^{n}\Ib{z_{i}\ne\overline{z}_{i}} \right) \Ib{z_{j} = \overline{z}_{j}} 
+
C_1 \sqrt{h} ~ \Ib{z_{j}\ne\overline{z}_{j}}
\end{align}

As suggested by the bound, we will consider two cases depending on whether $z_j = \bar z_j$.

\noindent \textbf{Case 1:} $z_j \neq \bar z_j$.
Notice the inequalities \cref{eq:TwoInequalities} apply as well when $\bz$ is replaced by $\bar{\bz}$.  Hence, applying the upper bound for the first term and the lower bound for the second term shows 
\begin{align*}
 & V(\bz + \delta_{j}\be_{j})
  -V(\bz) - \left(V(\overline{\bz}+\delta_{j}\be_{j})-V(\overline{\bz})\right)
\\
 & \quad \le 
 \left(
    \left(
        r_j(z_j) + a_{j}\delta_{j} - \bA_{j}^{\top}\blambda (\bz)
    \right)^{+}
    -\left(r_j(z_j)-\bA_{j}^{\top}\blambda(\bz)\right)^{+}
\right)
\\
 & \qquad-\left(\left(r_j(\bar z_j) +a_{j}\delta_{j}-\bA_{j}^{\top}\blambda(\overline{\bz}+\delta_{j}\be_{j})\right)^{+}-\left(r_j(\bar z_j) -\bA_{j}^{\top}\blambda(\overline{\bz}+\delta_{j}\be_{j})\right)^{+}\right)
 \\
 & \quad \leq \ 
 2\abs{a_j} \abs{\delta_j}, 
 \end{align*}
because $t \mapsto t^+$ is a $1$-Lipschitz function.
Take expectations of both sides, using Jensen's inequality and upper bounding the variance of $\delta_j$ shows
\begin{align*}
\Eb{ V(\bz + \delta_j \be_j) - V(\bz) - (V(\bzbar + \delta_j \be_j) - V(\bzbar) )}
\ \leq \ 
	2 a_{\max}\Eb{ \abs{\delta_j}} 
\ \leq 2 a_{\max} \frac{\sqrt{3h}}{\vmin^{1/4}}.
\end{align*}
Collecting constants proves the inequality when $z_j \neq \bar z_j$.

\noindent \textbf{Case 2:} $z_j = \bar z_j$. 
Proceeding as in Case 1, 
\begin{align*}
 & V(\bz+\delta_{j}\be_{j})-V(\bz)-\left(V(\overline{\bz}+\delta_{j}\be_{j})-V(\overline{\bz})\right)
\\
& \quad \le \left(r_j(z_j)+a_{j}\delta_{j}-\bA_{j}^{\top}\blambda(\bz) \right)^{+}-\left(r_j(z_j)-\bA_{j}^{\top}\blambda(\bz)\right)^{+}
\\
 & \qquad-\left(\left(r_j(\bar z_j) +a_{j}\delta_{j}-\bA_{j}^{\top}\blambda(\overline{\bz}+\delta_{j}\be_{j})\right)^{+}-\left(r_j(\bar z_j) -\bA_{j}^{\top}\blambda(\overline{\bz}+\delta_{j}\be_{j}))\right)^{+}\right)
 \\
 & \quad \leq 
 2\abs{\bA_j^\top (\blambda(\bz) - \blambda(\bar{\bz} + \delta_j \be_j) )}
 \\
 & \quad \leq 
 2 C_A \| \blambda(\bz) - \blambda(\bar{\bz})\| + 
 2 C_A \| \blambda(\bzbar) - \blambda(\bzbar + \delta_j \be_j)\|,
\end{align*}
where the second inequality follows again because $t \mapsto t^+$ is a contraction, but we group the terms in a different order, and the last inequality follows from the triangle-inequality and the Cauchy-Schwarz inequality.  We can bound the first term by invoking \cref{lem:DualStability} yielding
\begin{align} \label{eq:UnscaledVGC_Diff}
  V(\bz+\delta_{j}\be_{j})-V(\bz) & -\left(V(\overline{\bz}+\delta_{j}\be_{j})-V(\overline{\bz})\right) \\ \notag
\ \leq \ &
C_2 V^3\log^2(V) \frac{\log^4 n}{n} \sum_{i=1}^n \Ib{z_i \neq \bar z_i}
+ 
 2 C_A \| \blambda(\bzbar) - \blambda(\bzbar + \delta_j \be_j)\|,
\end{align}
for some dimension independent constant $C_2$.  
Taking expectations shows that to prove a bound, it will suffice to bound $\Eb{ \| \blambda(\bzbar) - \blambda(\bzbar + \delta_j \be_j)\|}$. 
To this end, consider splitting the expectation based on whether $\abs{\delta_j} \geq 3 \sqrt{\frac{n}{\vmin}}$.

If $\abs{\delta_j} \leq 3 \sqrt{\frac{n}{\vmin}}$, then by \cref{lem:DualsBoundedOnE}, $\| \blambda(\bzbar + \delta_j \be_j)\|_1 \ \leq \lambda_{\max}$.  Hence we can invoke \cref{lem:DualStability} again yielding
\begin{align*}
\Eb{ 
\| \blambda(\bzbar) - \blambda(\bzbar + \delta_j \be_j)\| \Ib{\abs{\delta_j} \leq 3 \sqrt{\frac{n}{\vmin}} }} 
& \ \leq \ 
C_3 V^3 \log^2(V) \frac{\log^4 n}{n} \Pb{\abs{\delta_j} \leq 3 \sqrt{\frac{n}{\vmin}}  }
\\ & \ \leq \ 
C_3 V^3 \log^2(V) \frac{\log^4 n}{n}.
\end{align*}

Next, assume $\abs{\delta_j} \geq 3 \sqrt{\frac{n}{\vmin}}$.  Write
\begin{align*}
\| \blambda(\bzbar) - \blambda(\bzbar + \delta_j \be_j)\| 
&\ \leq \ 
\lambda_{\max} + \frac{2}{n s_0}\| \br(\bzbar + \delta_j \be_j)\|_1 
	&&(\text{by \cref{lem:BoundedDualSol}})
\\& \ \leq \ 
\lambda_{\max} + \frac{2}{n s_0}\| \br(\bzbar)\|_1 + \frac{2 \abs{a_j}}{n s_0}\| \abs{\delta_j}
&&(\text{by def. of $\br(\cdot)$})
\\ & \ \leq \ 
\lambda_{\max} + C_4 + \frac{2 a_{\max}}{n s_0} \abs{\delta_j},
\end{align*}
for some dimension independent constant $C_4$, because $\bzbar \in \mathcal E$ implies that $\| \br(\bz)\|/n$ is bounded by a (dimension-independent) constant.

Thus, for some dimension independent constant $C_5$ we have
\begin{align*}
\Eb{ \| \blambda(\bzbar) - \blambda(\bzbar + \delta_j \be_j)\| \Ib{ \abs{\delta_j} > 3 \sqrt{\frac{n}{\vmin}}}}
& \ \leq \ 
\Eb{\left(C_5 +  \frac{C_5}{n}\abs{\delta_j}\right) \Ib{ \abs{\delta_j} > 3 \sqrt{\frac{n}{\vmin}}}}
\\ 
&\  \leq \ 
\Eb{\left(C_5 +  \frac{C_5}{n}\right) \abs{\delta_j} \Ib{ \abs{\delta_j} > 3 \sqrt{\frac{n}{\vmin}}}} 
\\ & \ = \ 
\Eb{ 2C_5 \abs{\delta_j} \Ib{ \abs{\delta_j} > 3 \sqrt{\frac{n}{\vmin}}} },
\end{align*}
where the final inequality uses $3\sqrt{\frac{n}{\vmin}} \geq 1$ since $n \geq 3$.
Integration by parts with the Gaussian density shows there exists a dimension independent constant $C_6$ such that 
\begin{align*}
\Eb{ \abs{\delta_j} \Ib{\abs{\delta_j} > 3 \sqrt{\frac{n}{\vmin}}}} \ \leq C_6 \sqrt{ h} e^{-nh/\sqrt{\vmin}} \ \leq \ 
C_6 \sqrt{h} e^{-nh} \ \leq \  \frac{C_6}{n},
\end{align*}
because $nh > \log n$ and $h < 1$ by assumption.  

Combining the two cases shows that 
\[
	\Eb{ \| \blambda(\bzbar) - \blambda(\bzbar + \delta_j \be_j)\|} \ \leq \ C_7 V^3 \log^2(V) \frac{\log^4 n}{n}
\]
for some dimension independent constant $C_7$.

Taking the expectation of \cref{eq:UnscaledVGC_Diff}, substituting this bound and collecting constants proves 
\begin{align*} 
 & V(\bz+\delta_{j}\be_{j})-V(\bz)-\left(V(\overline{\bz}+\delta_{j}\be_{j})-V(\overline{\bz})\right)
\ \leq \
C_8 V^3\log^2(V) \frac{\log^4 n}{n} \sum_{j=1}^n \Ib{z_j \neq \bar z_j}
\end{align*}
for some constant $C_8$.  Combining with Case 1 establishes \cref{eq:VGC-dual-sol-bnd-v}.

Now, by symmetry, \cref{eq:VGC-dual-sol-bnd-v} holds with the roles of $\bz$ and $\bzbar$ reversed.  Hence, \cref{eq:VGC-dual-sol-bnd-v} also holds after taking the absolute values of both sides.  

We can now write, 
\begin{align*}
\abs{ D(\bz) - D(\bzbar) }
& \ \leq \ \sum_{j=1}^n \frac{1}{h \sqrt{\nu_j} \abs{a_j}}
\abs{ \Eb{V(\bz + \delta_j\be_j) - V(\bz) - \big(V(\bzbar + \delta_j \be_j) - V(\bzbar)\big)} }
\\
& \ \leq \ 
\frac{C_1}{h \sqrt{\vmin} a_{\min}} 
V^3\log^2(V) \frac{\log^4 n}{n}\sum_{i=1}^{n}\Ib{z_{i}\ne\overline{z}_{i}}  \sum_{j=1}^n\Ib{z_{j} = \overline{z}_{j}} 
+
\frac{C_1}{\sqrt{h \vmin} a_{\min}} \sum_{j=1}^n ~ \Ib{z_{j}\ne\overline{z}_{j}}
\\
&\ \leq \ 
\frac{C_9}{h}~ 
V^3\log^2(V) \log^4(n) ~ \sum_{i=1}^{n}\Ib{z_{i}\ne\overline{z}_{i}},  
\end{align*}
for some constant $C_9$.  This completes the proof. 
\hfill \Halmos \endproof

Finally, we show that $\btheta \mapsto \blambda(\bz, \btheta)$ is also smooth on $\mathcal{E}_n$, at least locally.
\begin{lem}[Local Smoothness of Dual Solution in $\btheta$] \label[lem]{lem:DualSmoothness}
 Suppose $\bm{z}\in\mathcal{E}_{n}$ and that \cref{asn:Parameters,asn:SmoothPlugIn,asn:LiftedAffinePlugInII} hold.
Then, there exist dimension independent constants $C$ and $n_0$ such that for any $n \geq n_0$ and any $\bthetabar$ such that $\| \bthetabar - \btheta\| \leq \frac{1}{n}$, we have that
\[
    \left\Vert \blambda(\bz, \btheta) -\blambda(\bz, \bthetabar) \right\Vert _{2}
    \ \le \
    C V^2 \log V \frac{ \log^{5/4} n }{\sqrt{n}}
\]
\end{lem}

\proof{Proof of \cref{lem:DualSmoothness}:}
The proof is similar to that of \cref{lem:DualStability}.
To declutter the notation, define 
\[
    f_{1}(\blambda) \equiv
    \mathcal{L}(\blambda,\bm{z},\btheta),
    \quad \blambda_{1} \equiv \blambda(\bz,\btheta)
\]
\[
    f_{2}(\blambda) \equiv 
    \mathcal{L}(\blambda,\bm{z},\bthetabar),
    \quad \blambda_{2} \equiv \blambda(\bz,\bthetabar)
\]
Furthermore, let 
$I_{j} = 
  \left\langle 
    \bm{A}_{j}^{\top}\blambda_{1},\bm{A}_{j}^{\top}\blambda_{2}
  \right\rangle 
$.

If $\| \blambda_1 - \blambda_2 \|_2 \leq 4 V^2 \log V \frac{\log n}{\sqrt n}$, then the lemma holds trivially for $C = 4$.  Hence, for the remainder, we assume $\| \blambda_1 - \blambda_2 \|_2 > 4 V^2 \log V \frac{\log n}{\sqrt n}$.  In particular, by \cref{lem:DualsBoundedOnE}, this implies $(\blambda_1, \blambda_2) \in \Lambda_n$.

Using \cref{lem:TaylorIdentity}, we have that
\begin{align*}
    f_{1}(\blambda_{2}) - f_{1}(\blambda_{1}) & 
    = \nabla f_{1}(\blambda_{1})^{\top}(\blambda_{2}-\blambda_{1})
    + \frac{1}{n}\sum_{j=1}^{n}\mathbb{I}\left\{ r_j(z_j,\btheta)\in I_{j}\right\} 
    \left|r_j(z_j,\btheta)-\bA_j^\top\blambda_{2}\right| 
    \\
    & \ge \frac{1}{n}\sum_{j=1}^{n}\mathbb{I}\left\{ r_j(z_j,\btheta)\in I_{j}\right\} 
    \left| r_j(z_j, \btheta) - \bA_j^\top\blambda_{2} \right|
\end{align*}
where $f_{1}(\cdot)$ is convex and $\blambda_{1}$ is an optimizer.
Similarly, we have that
\[
    f_{2}(\blambda_{1}) - f_{2}(\blambda_{2}) 
    \ge 
    \frac{1}{n} \sum_{j=1}^{n}
        \mathbb{I}\left\{ r_j(z_j,\bthetabar)\in I_{j}\right\} 
        \left|r_j(z_j,\bthetabar)-\bA_j^\top\blambda_{1}\right|.
\]
Adding yields
\begin{align} \label{eq:MasterEQ_Theta}
    f_{1}(\blambda_{2}) & 
    - f_{1}(\blambda_{1}) + f_{2}(\blambda_{1}) - f_{2}(\blambda_{2})
    \\ \notag
    &  \ \ge  \ 
    \frac{1}{n} \sum_{j=1}^{n} 
    \mathbb{I}\left\{ r_j(z_j,\btheta)\in I_{j}\right\} 
    \left|r_j(z_j,\btheta) - \bA_j^\top\blambda_{2}\right| 
    + \mathbb{I}\left\{ r_j(z_j,\bthetabar)\in I_{j}\right\} 
    \left|r_j(z_j,\bthetabar) - \bA_j^\top\blambda_{1}\right|.
\end{align}

We would like to combine the $j^{\text th}$ summand to simplify. To this end, adding and subtracting \\
\(
\mathbb{I} \big\{r_j(z_j, \btheta) \in I_j\big\} \abs{r_j(z_j, \btheta) - \bA_j^\top \blambda_2}
\) 
yields, 
\begin{subequations} \label{eq:DifferencesInRemainders}
\begin{align} \notag
& \Ib{r_j(z_j,\btheta)\in I_{j}} 
    \left| r_j(z_j,\btheta) - \bA_j^\top\blambda_{2} \right|
    + \Ib{r_j(z_j,\bthetabar)\in I_{j}} 
    \left| r_j(z_j,\bthetabar) - \bA_j^\top\blambda_{1} \right|
\\ \label{eq:GradTerm}
& \quad = \ 
\Ib{r_j(z_j, \btheta) \in I_j} \Big( \abs{r_j(z_j, \btheta) - \bA_j^\top \blambda_2} + \abs{r_j(z_j, \btheta) - \bA_j^\top \blambda_1 }\Big) 
\\ \label{eq:DiffTerm}
& \qquad +
\Ib{r_j(z_j, \bthetabar) \in I_j} \abs{r_j(z_j, \bthetabar) - \bA_j^\top \blambda_1} - 
\Ib{r_j(z_j, \btheta) \in I_j} \abs{r_j(z_j, \btheta) - \bA_j^\top \blambda_1}
\end{align}
\end{subequations}

We simplify \cref{eq:GradTerm} by noting that when the indicator is non-zero, 
\[
\abs{r_j(z_j, \btheta) - \bA_j^\top \blambda_2 } + 
\abs{r_j(z_j, \btheta) - \bA_j^\top \blambda_1 }
\  = \ 
\abs{\bA_j^\top (\blambda_2 - \blambda_1)} 
\]
Hence, 
\[
    {\rm\cref{eq:GradTerm}} \ = \ \Ib{r_j(z_j, \btheta) \in I_j} \abs{\bA_j^\top (\blambda_2 - \blambda_1)}.
\]
We rewrite \cref{eq:DiffTerm} as 
\begin{align*}
{\rm\cref{eq:DiffTerm} }
& \ = \ 
\Ib{r_j(z_j, \bthetabar) \in I_j} \Big( 
    \abs{r_j(z_j, \bthetabar) - \bA_j^\top \blambda_1 } 
   - \abs{r_j(z_j, \btheta) - \bA_j^\top \blambda_1 } 
\Big)
\\
& \qquad + 
\abs{r_j(z_j, \btheta) - \bA_j^\top \blambda_1 }
\Big(
\Ib{r_j(z_j, \bthetabar) \in I_j} - \Ib{r_j(z_j, \btheta) \in I_j}
\Big)
\\
& \ \overset{(a)}{\geq} \ 
- \Ib{r_j(z_j, \bthetabar) \in I_j} \abs{ r_j(z_j, \bthetabar) - r_j(z_j, \btheta)}
\\
& \qquad + 
\abs{r_j(z_j, \btheta) - \bA_j^\top \blambda_1 }
\Big(
\Ib{r_j(z_j, \bthetabar) \in I_j} - \Ib{r_j(z_j, \btheta) \in I_j}
\Big),
\\  
& \ \overset{(b)}{\geq} \ 
- \abs{ r_j(z_j, \bthetabar) - r_j(z_j, \btheta)}
- 
\abs{r_j(z_j, \btheta) - \bA_j^\top \blambda_1 }
\Ib{ r_j(z_j, \bthetabar) \not \in I_j, \ r_j(z_j, \btheta) \in I_j},
\\
& \ \overset{(c)}{\geq} \ 
- L \|\btheta - \bthetabar\|_2 (\abs{z_j} + 1)
- \abs{r_j(z_j, \btheta) - \bA_j^\top \blambda_1 }
\Ib{ r_j(z_j, \bthetabar) \not \in I_j, \ r_j(z_j, \btheta) \in I_j},
\\
& \ \overset{(d)}{\geq} \ 
- \frac{L}{n} (\abs{z_j} + 1)
- \abs{r_j(z_j, \btheta) - \bA_j^\top \blambda_1 }
\Ib{ r_j(z_j, \bthetabar) \not \in I_j, \ r_j(z_j, \btheta) \in I_j},
\end{align*}
where inequality (a) is the triangle inequality, inequality (b) rounds the indicators, inequality (c) follows from the Lipschitz assumptions on $a_j(\btheta)$ and $b_j(\btheta)$, and inequality (d) uses $\| \btheta - \bthetabar\| \ \leq \frac{1}{n}$.
Finally note that when the last indicator is non-zero, 
\[
  \abs{r_j(z_j, \btheta) - \bA_j^\top \blambda_1} 
  \ \leq \
  \abs{ \bA_j^\top(\blambda_2 - \blambda_1)} \ \leq \ 
  2C_A \sqrt{m} \lambda_{\max}.
\] 
Substituting this bound above and the resulting lower bound on \cref{eq:DiffTerm} into \cref{eq:DifferencesInRemainders} proves
\begin{align} \label{eq:GradDiffTerm_w_Remainders} 
& \Ib{r_j(z_j,\btheta)\in I_{j}} 
    \left| r_j(z_j,\btheta) - \bA_j^\top\blambda_{2} \right|
    + \Ib{r_j(z_j,\bthetabar)\in I_{j}} 
    \left| r_j(z_j,\bthetabar) - \bA_j^\top\blambda_{1} \right|
\\ \notag
& \quad \geq 
\Ib{r_j(z_j, \btheta) \in I_j} \abs{\bA_j^\top (\blambda_2 - \blambda_1)}
- \frac{L}{n} (\abs{z_j} + 1)
- 
2C_A \sqrt{m} \lambda_{\max}
\Ib{ r_j(z_j, \bthetabar) \not \in I_j, \ r_j(z_j, \btheta) \in I_j}.
\end{align}

We can further clean up the last indicator by noting that 
\begin{align*}
&    \Ib{r_j(z_j, \bthetabar) \not \in I_j, \ r_j(z_j, \btheta) \in I_j} 
\\& \quad    \implies 
    \text{ Either } 
    \abs{r_j(z_j, \btheta) - \bA_j^\top \blambda_1} \leq
    \abs{r_j(\bz, \btheta) - r_j(\bz, \bthetabar)}
    \text{ or } 
    \abs{r_j(z_j, \btheta) - \bA_j^\top \blambda_2} \leq
    \abs{r_j(\bz, \btheta) - r_j(\bz, \bthetabar)}.
\end{align*}
Moreover, because $\bz \in \mathcal E$, we can use the Lipschitz assumptions on $a_j(\btheta)$ and $b_j(\btheta)$ and the fact that $\| \btheta - \bthetabar\|\leq \frac{1}{n}$ to write
\[
	\abs{r_j(\bz, \btheta) - r_j(\bz, \btheta)} \ 
	\leq \ 2 L \| \btheta - \bthetabar\| \log n
    \ \leq \  \frac{2L \log n}{n}.
\]
Thus, 
\begin{align*}
\Ib{r_j(z_j, \bthetabar) \notin I_j, r_j(z_j, \btheta) \in I_j} \ &\leq \ \Ib{ \abs{r_j(z_j, \btheta) - \bA_j^\top \blambda_1} \leq \frac{2L \log n}{n}} 
 \; + \; \Ib{ \abs{r_j(z_j, \btheta) - \bA_j^\top \blambda_2} \leq \frac{2L \log n}{n}}.
\end{align*}

Making this substitution into \cref{eq:GradDiffTerm_w_Remainders}, averaging over $j$, 
and substituting this bound into \cref{eq:MasterEQ_Theta} shows
\begin{align*} 
    f_{1}(\blambda_{2}) & 
    - f_{1}(\blambda_{1}) + f_{2}(\blambda_{1}) - f_{2}(\blambda_{2})
\\ 
&  \quad \geq
    \frac{1}{n} \sum_{j=1}^{n}
        \Ib{r_j(z_j, \btheta) \in I_j} \abs{\bA_j^\top (\blambda_2 - \blambda_1)}
        - \frac{L}{n^2} \sum_{j=1}^{n} (\abs{z_j} + 1)
\\ 
& \qquad         
        - 2C_A \sqrt{m} \lambda_{\max}
        \frac{1}{n} \sum_{j=1}^{n} 
        \Ib{ \abs{r_j(z_j, \btheta) - \bA_j^\top \blambda_1} \leq \frac{2L \log n}{n}} + \Ib{ \abs{r_j(z_j, \btheta) - \bA_j^\top \blambda_2} \leq \frac{2L \log n}{n}}
\\ 
&  \quad =
    \left(\nabla_{\blambda} \L(\blambda_1, \bz, \btheta) - \nabla_{\blambda} \L(\blambda_2, \bz, \btheta)
    \right)^\top 
    (\blambda_1 - \blambda_2) 
    - \frac{L}{n} (\| \bz \|_1 / n + 1) 
\\ 
& \qquad 
        - 2C_A \sqrt{m} \lambda_{\max}
        \frac{1}{n} \sum_{j=1}^{n} 
        \Ib{ \abs{r_j(z_j, \btheta) - \bA_j^\top \blambda_1} \leq \frac{2L \log n}{n}} + \Ib{ \abs{r_j(z_j, \btheta) - \bA_j^\top \blambda_2} \leq \frac{2L \log n}{n}}
\\ 
& \quad \geq 
    \phi_{\min}\beta \| \blambda_1 - \blambda_2\|_2^2 -  V^2 \log(V) \frac{\log^2 n }{\sqrt{n}}\| \blambda_1 - \blambda_2\|_2^{3/2}
    - \frac{2L}{n} \left(C_\mu + 2/\sqrt{\vmin} \right)
\\
& \qquad    - 8 C_A \sqrt{m} \lambda_{\max} \left( L\sqrt{\vmax}
\frac{\log n}{n}   + 
 \sqrt{2 L} 
 V \log(V) \frac{\log^{5/2} n }{n} \right),
\end{align*}
because $\frac{2L \log n}{n} \geq \frac{1}{n}$ by \cref{asn:Parameters} and 
$\bz \in \mathcal E_n$.  Using \cref{asn:Parameters} to further simplify, we have thus far shown that  for some dimension independent constant $C_2$, 
\begin{align} \label{eq:MasterEQ2_Theta}
    f_1(\blambda_2) - f_1(\blambda_1) + 
    f_2(\blambda_1) - f_2(\blambda_2) 
    &\ \geq \ 
    \phi_{\min} \beta \| \blambda_1 - \blambda_2\|_2^2 -  V^2 \log(V) \frac{\log^2 n }{\sqrt{n}} \| \blambda_1 - \blambda_2\|_2^{3/2}
\\ \notag
& \qquad  
    - C_2 V^2 \log V \frac{\log^{5/2} n}{n}
\end{align}

We next proceed to upper bound left side of this inequality. 
By definition of $f_{1}(\cdot),f_{2}(\cdot)$,
\begin{align*}
    f_{1}(\blambda_{2}) & 
    - f_{1}(\blambda_{1}) + f_{2}(\blambda_{1}) - f_{2}(\blambda_{2})
    \\
    = & 
    \frac{1}{n} \sum_{j=1}^{n}\left( \left[r_j(z_j,\btheta) - \bA_j^\top\blambda_{2} \right]^{+}
    - \left[r_j(z_j,\btheta) - \bA_j^\top\blambda_{1}\right]^{+}
    + \left[r_j(z_j,\bthetabar) - \bA_j^\top\blambda_{1}\right]^{+}
    - \left[r_j(z_j,\bthetabar) - \bA_j^\top\blambda_{2}\right]^{+}
    \right).
\end{align*}
Focusing on the $j^{\sf th}$ term, we see
\begin{align*}
 & \left[r_j(z_j,\btheta)-\bA_j^\top\blambda_{2}\right]^{+}-\left[r_j(z_j,\btheta)-\bA_j^\top\blambda_{1}\right]^{+}+\left[r_j(z_j,\bthetabar)-\bA_j^\top\blambda_{1}\right]^{+}-\left[r_j(z_j,\bthetabar)-\bA_j^\top\blambda_{2}\right]^{+}
\\ & 
\quad \le  2\left|r_j(z_j,\btheta)-r_j(z_j,\bthetabar)\right|
\\ & \quad \leq 
2 L \| \btheta - \bthetabar\|_2 ( \abs{z_j} + 1),
\\ & \quad \leq 
\frac{2L}{n}( \abs{z_j} + 1),
\end{align*}
where the penultimate inequality uses \cref{asn:SmoothPlugIn}.
Averaging over $j$, we see
\[
f_{1}(\blambda_{2})-f_{1}(\blambda_{1})+f_{2}(\blambda_{1})-f_{2}(\blambda_{2})
\ \le \ 
\frac{2L}{n} \left(\frac{\| \bz\|_1}{n} + 1\right) 
\ \leq \ 
\frac{C_3}{n}
\]
for some constant $C_3$, since $\bz \in \mathcal E_n$.

Substitute into \cref{eq:MasterEQ2_Theta} to see that 
\begin{align*}
    \frac{C_3}{n} 
    &\ \geq \ 
    \phi_{\min} \beta \| \blambda_1 - \blambda_2\|_2^2 - \| \blambda_1 - \blambda_2\|_2^{3/2} V^2 \log(V) \frac{\log^2 n }{\sqrt{n}}
    - C_2 V^2 \log V  \frac{\log^{5/2} n}{n}
\end{align*}
Rearranging and collecting constants shows 
\begin{align} \notag
\phi_{\min} \beta \| \blambda_1 - \blambda_2\|_2^2 - \| \blambda_1 - \blambda_2\|_2^{3/2} V^2 \log(V) \frac{\log^2 n }{\sqrt{n}}
& \ \leq \ 
    C_4 V^2 \log V  \frac{\log^{5/2} n}{n},
\end{align}
for some dimension-independent constant $C_4$. 

We can also lower bound the left side by recalling 
\[
\| \blambda_1 - \blambda_2 \|_2 > V^2 \log V \frac{\log n}{\sqrt n}	
\implies 
\frac{n^{1/4}\| \blambda_1 - \blambda_2 \|^{1/2}_2}{\log^{1/2}(n) \; V^2 \log(V)  } > 1.	
\]
Hence, inflating the second term on the left yields 
\begin{equation} \notag
	\phi_{\min} \beta \| \blambda_1 - \blambda_2\|_2^2 - 
	\frac{\log^{3/2}n}{n^{1/4}}\| \blambda_1 - \blambda_2 \|^2_2 
 \leq C_5 V^2 \log V \frac{\log^{5/2} n}{n}.
\end{equation}

For $n$ sufficiently large, the first term on the left is at least twice the second.  
Rearranging and taking square roots shows 
\[
    \| \blambda_1 - \blambda_2 \|_2 \ \leq \ 
    C_6 \sqrt{ V^2 \log V}  \frac{\log^{5/4} n}{\sqrt n}.
\]
Recalling that $V \geq 2$ proves the theorem.

\hfill \Halmos \endproof
\subsection{Pointwise Convergence Results} 
\label[app]{sec:appendix_pointwise_results}
To prove our theorem, we require the uniform convergence of the in-sample optimism to expectation and the uniform convergence of VGC to its expectation.  In this section, we first establish several pointwise convergence results to assist with this task.  Our main workhorse will be \cref{thm:Combes-McDiarmid} where $\mathcal E_n$ defines the good set on which our random variables satisfy a bounded differences condition.

As a first step, we will show that the dual solutions converge (for a fixed $\btheta$) to their expectations.  In preparation, we first bound the behavior of the dual on the bad set.
\begin{lem}[Dual Solution Conditional Expectation Bound]\label[lem]{lem:cond-dual-sol-bnd} 
Suppose \cref{asn:SmoothPlugIn,asn:Parameters} both hold.  Let 
\[
	\mathcal{E}_{1,n} 
	\equiv 
	\left\{\bz :  \| \bz \|_1 \leq n C_\mu + \frac{2n}{\sqrt{\vmin}}\right\}.
\] Then, there exists a dimension independent constant $C$, such that
	\[
		\mathbb{E}\left[ 
			\lambda_i(\bZ,\btheta) 
			\mathbb{I} \left\{ \bZ \in \mathcal{E}_{1,n}^{c} \right \}
		\right]
		\le 
		C \exp\left( -\frac{n}{C} \right)
	\]
\end{lem} 

\proof{Proof of \cref{lem:cond-dual-sol-bnd}:}
We first bound
\begin{align*}
	\mathbb{E}\left[
		\left\Vert \bZ\right\Vert _{1}
		\mathbb{I} \left\{\bZ \in \mathcal{E}_{1,n}^{c} \right\} \right]
		& = 
		\int_{0}^{\infty} \Pb{ \left\Vert \bZ\right\Vert _{1}
		\mathbb{I} \left\{ \bZ \in \mathcal{E}_{1,n}^{c} \right\}
		\ge t } dt
	\\
 	& =
 	\int_0^{n C_\mu + \frac{2n}{\sqrt \vmin}}
 	\Pb{ \| \bZ\|_1 > n C_\mu + \frac{2n}{\sqrt \vmin}} dt
 	\ + \ 
 	\int_{n C_\mu + \frac{2n}{\sqrt \vmin}}^{\infty}
 	\Pb{ \left\Vert \bZ\right\Vert _{1} \ge t } dt
 	\\
 	& \ \leq \ 
	\left(n C_\mu + \frac{2n}{\sqrt \vmin}\right) e^{-n/32}	
	\ + \ 
\int_{n C_\mu + \frac{2n}{\sqrt \vmin}}^{\infty}
 	\Pb{ \left\Vert \bZ\right\Vert _{1} \ge t } dt
\end{align*}
By inspection, there exists a dimension independent constant $C_1$ such that the first term is at most $C_1 e^{-n/C_1}$.

To analyze the second term, recall $\| \bZ\|_1 \ \leq n C_\mu + \sum_{j=1}^n \abs{Z_j - \mu_j}$.  Hence, 
\begin{align*}
\int^\infty_{nC_\mu + \frac{2n}{\sqrt{\vmin}}} \Pb{ \abs{\bZ}_1 \geq t} dt
& \ = \ 
\int^\infty_{nC_\mu + \frac{2n}{\sqrt{\vmin}}} \Pb{ \frac{1}{n}\sum_{j=1}^n \abs{Z_j - \mu_j} \geq \frac{1}{n} t - C_\mu} dt
\\
& \ = \ 
\int^\infty_{nC_\mu + \frac{2n}{\sqrt{\vmin}}} \Pb{ \frac{1}{n}\sum_{j=1}^n \abs{Z_j - \mu_j}- \Eb{\abs{Z_j - \mu_j}} \geq \frac{1}{n} t - C_\mu - \frac{1}{\vmin}} dt,
\end{align*}
since $\Eb{\abs{Z_j - \mu_j}} \ \leq \frac{1}{\sqrt{\vmin}}$ by Jensen's inequality.  Now make the change of variable $s = \frac{t}{n} - C_\mu - \frac{1}{\sqrt{\vmin}}$ to obtain
\begin{align*}
n\int_{\vmin^{-1/2}}^\infty \Pb{ \frac{1}{n}\sum_{j=1}^n \abs{Z_j - \mu_j}- \Eb{\abs{Z_j - \mu_j}} \geq s} ds
& \ \leq \ 
n \int_{\vmin^{-1/2}}^\infty e^{-\frac{s^2\vmin n}{32}} ds
\end{align*}
because $\abs{Z_j - \mu_j} - \Eb{\abs{Z_j - \mu_j}}$ is a mean-zero, sub-Gaussian random variable with variance proxy at most $\frac{16}{\vmin}$.  (See \cref{lem:Bounding1NormZ} for clarification.)
Making another change of variables proves this last integral is equal to
\begin{align*}
\frac{4}{\sqrt{\vmin n}} \int_{\sqrt n /4}^\infty e^{-t^2/2} dt, 
& \ \leq \ 
\frac{16}{n\sqrt{\vmin}} e^{-\frac{n}{32}}
\end{align*}
This value is also at most $C_2 e^{-n/C_2}$ for some constant $C_2$.

In summary, we have shown that there exists a dimension independent constant $C_3$ such that
\[
	\mathbb{E}\left[
		\left\Vert \bZ\right\Vert _{1}
		\mathbb{I} \left\{\bZ \in \mathcal{E}_{1,n}^{c} \right\} \right] \ \leq C_3 e^{-n / C_3}.
\]	

Now to prove the lemma, recall by \cref{lem:BoundedDualSol}, 
\begin{equation}
	\lambda_{i}(\bZ,\btheta) 
	\le \left\Vert \blambda(\bZ,\btheta)\right\Vert _{1} 
	\le \frac{2}{s_{0}n} \left\Vert \bm{r}(\bZ,\btheta) \right\Vert_1
	\le \frac{2}{s_{0}n} \left( a_{\max}\left\Vert \bZ\right\Vert _{1} + b_{\max}n \right)	 
\end{equation}
where the second inequality holds by \cref{asn:SmoothPlugIn}.  Multiplying by $\Ib{ \bZ  \in \mathcal E^c_{1,n}}$ and taking expectations hows,
\[
	\Eb{ \lambda_i(\bZ, \btheta) \Ib{\bZ \in \mathcal E^c_{1,n}} } \ \leq \ 
	\frac{C_4}{n} e^{-n/C_4} + C_4 \Pb{\bZ \in \mathcal E^c_{1,n}}
	\ \leq \ \frac{C_4}{n} e^{-n/C_4} + C_4 e^{-n/32},
\]
by \cref{lem:Bounding1NormZ}.  Collecting constants proves the lemma.  
%
%
%
\hfill \Halmos \endproof

We now use \cref{thm:Combes-McDiarmid} to prove that the dual solution concentrates at its expectation for any fixed $\btheta \in \Theta$.
 \begin{lem}
[Pointwise Convergence Dual Solution]\label[lem]{lem:dual-sol-pointwise}
Fix some $\btheta \in \Theta$ and $i = 1, \ldots, m$.
Under \cref{asn:Parameters,asn:SmoothPlugIn,asn:LiftedAffinePlugInII}
There exists dimension independent constants $C$ and $n_0$, such that for all $n\ge n_0 e^R$, the
following holds with probability $1-\exp\left(-R\right)$
\[
\left|\lambda_{i}(\bZ,\btheta)-\mathbb{E}\left[\lambda_{i}(\bZ,\btheta)\right]\right|
\le 
C V^3 \log^2 V \frac{\log^4 n}{\sqrt n }\sqrt{R}.\]
\end{lem}

\proof{Proof of \cref{lem:dual-sol-pointwise}:}
The proof will use the dual stability condition (\cref{lem:DualStability}) to apply \cref{thm:Combes-McDiarmid}.  Since $\btheta$ is fixed throughout, we drop it from the notation.  


By triangle inequality, 
\begin{equation} \label{eq:SplitLambdaIntoConditional}
  \left|\lambda_{i}(\bZ)-\mathbb{E}\left[\lambda_{i}(\bZ)\right]\right| 
  \ \le \ 
  \underbrace{\left|\lambda_{i}(\bZ)-\Eb{\left.\lambda_{i}(\bZ)\right|\bZ\in\mathcal E_n }\right|}_{(a)}
  +
  \underbrace{\left|\Eb{\left.\lambda_{i}(\bZ)\right|\bZ\in\mathcal E_n } -\mathbb{E}\left[\lambda_{i}(\bZ)\right]\right|}_{(b)}.
\end{equation}

We first bound $(b)$ by a term that is $O(1/n)$. We see that 
\begin{align*}
 	\mathbb{E}\left[\left.\lambda_{i}(\bZ)\right| \bZ \in \mathcal E_n  \right]
 	- \mathbb{E}\left[\lambda_{i}(\bZ)\right]
	& \ =  \ 
	\Big(\mathbb{E}\left[
		\left.\lambda_{i}(\bZ)\right| \bZ \in \mathcal E_n 
	\right]
	- \mathbb{E}\left[
		\left.\lambda_{i}(\bZ)\right| \bZ \notin \mathcal E_n 
	\right]
	\Big)
	\mathbb{P}\left\{ \bZ \notin \mathcal E_n  \right\} 
	\\
	& \ \le  \ 
	\frac{C_1}{n}  
	+ \mathbb{E}\left[
		\lambda_{i}(\bZ)
		\mathbb{I}\left\{ \bZ \notin \mathcal E_n \right\} 
	\right],
\end{align*}
where we used \cref{lem:DualsBoundedOnE,lem:ZinGoodSetHighProb} to bound the first term and $C_1$ is a dimension independent constant.
To bound the second term, define the set
\[
	\mathcal{E}_{0}
	\equiv
	\left\{ \bz : \|\bz\|_{1} \leq nC_{\mu} + \frac{2n}{\sqrt{\vmin}}\right\}.
\]
Notice, $\mathcal E_n \subseteq \mathcal E_{0}$.
Then write, 
\begin{align*}
	\mathbb{E}\left[
		\lambda_{i}(\bZ)
		\mathbb{I}\left\{ \bZ \notin \mathcal E_n \right\} 
	\right] 
 	& =
 	\mathbb{E}\left[
 		\lambda_{i}(\bZ)
 		\left(\mathbb{I}\left\{ \bZ\in\mathcal{E}_{0}^{c}\right\} 
 		+ \mathbb{I}\left\{ \bZ\in\mathcal{E}_{0}\cap\mathcal{E}_{n}^{c}\right\} 
 		\right)
 	\right]
 	\\
 	& \le
 	C_2 \exp\left(-\frac{n}{C_2} \right)
 	+ \lambda_{\max} \mathbb{P}\left\{ \bZ\in\mathcal{E}_{n}^c\right\} 
 	&&(\text{\cref{lem:cond-dual-sol-bnd,lem:DualsBoundedOnE}})
 	\\
 	& \le
 	C_3 \exp\left(-\frac{n}{C_3} \right)
 	+ \frac{C_2}{n}
 	&& (\text{\cref{lem:ZinGoodSetHighProb}}), 
\end{align*}
for dimension independent constants $C_2$ and $C_3$.

Collecting terms shows that there exists a dimension independent constant $C_3$ such that for $n$ sufficiently large, 
\[
	\text{Term (b) of \cref{eq:SplitLambdaIntoConditional}}  \ = \ \left| \Eb{\lambda_i(\bZ) \mid \bZ \in \mathcal E_n } -\mathbb{E}\left[\lambda_{i}(\bZ)\right]\right|\le \frac{C_{3}}{n}.
\]

We now bound Term $(a)$ by leveraging \cref{thm:Combes-McDiarmid}.  First note that for any $\bZ,\bar{\bZ}\in\mathcal E_n$, we have
\[
\left|\lambda_{i}(\bZ)-\lambda_{i}(\bar{\bZ})\right|=\sqrt{\left|\lambda_{i}(\bZ)-\lambda_{i}(\bar{\bZ})\right|^{2}}\le\sqrt{\sum_{i=1}^{m}\left|\lambda_{i}(\bZ)-\lambda_{i}(\bar{\bZ})\right|^{2}}=\left\Vert \blambda(\bZ)-\blambda(\bar{\bZ})\right\Vert _{2}.
\]
Thus, by \cref{lem:DualStability}, we see that
\[
	\left|\lambda_{i}(\bZ)-\lambda_{i}(\bar{\bZ})\right|
	\le
	C_{4}V^3 \log^2 V\cdot  \frac{\log^4(n)}{n}\sum_{j=1}^{n}\mathbb{I}\left\{ Z_{j}\ne\bar{Z}_{j}\right\}, 
\]
and, hence, $\lambda_i(\cdot)$ satisfies the bounded differences condition on $\mathcal E_n$.  
By \cref{lem:ZinGoodSetHighProb},
\mbox{\(
\mathbb{P}\left\{ \bZ \not \in\mathcal E_n \right\} \le \frac{C_{5}}{n}.
\)}
By the assumptions, $n > 4C_5e^R \implies e^{-R}> \frac{2C_5}{n}$.  
\Cref{thm:Combes-McDiarmid} then shows that with probability at least $1-e^{-R}$, 
\begin{align*}
\lambda_i(\bZ) - \Eb{\lambda_i(\bZ) \mid \bZ \in \mathcal E_n} 
&\ \leq \ 
C_5C_4 V^3 \log^2(V) \frac{\log^4 n}{n} 
+ 
C_4 V^3 \log^2(V) \frac{\log^4 n}{\sqrt n} 
\sqrt{\log\left(\frac{2}{e^{-R} - 2C_5/n}\right)}
\\
&\ \leq \ 
C_6 V^3 \log^2(V) \frac{\log^4 n}{\sqrt n} 
\sqrt{\log\left(\frac{2}{e^{-R} - 2C_5/n}\right)}
\\
&\ \leq \ 
C_6 V^3 \log^2(V) \frac{\log^4 n}{\sqrt n} 
\sqrt{\log\left(\frac{4}{e^{-R}}\right)}
\\
&\ \leq \ 
C_7 V^3 \log^2(V) \frac{\log^4 n}{\sqrt n} 
\sqrt{R},
\end{align*}
where the third inequality again uses $n > 4C_5e^R$, and the remaining inequalities simply collect constants and dominant terms.  



To summarize, by substituting the two bounds into the upperbound of $\left|\lambda_{i}(\bZ)-\mathbb{E}\left[\lambda_{i}(\bZ)\right]\right|$ in \cref{eq:SplitLambdaIntoConditional}, we obtain that with probability at least $1-e^{-R}$
\[
	\left| \lambda_i(\bZ) - \mathbb{E}\left[ \lambda_i(\bZ) \right] \right|
	\leq 
C_8 V^3 \log^2(V) \frac{\log^4 n}{\sqrt n} 
\sqrt{R}
	+ \frac{C_8}{n}.
\]
Collecting terms completes the proof.

\hfill \Halmos \endproof

\proof{Proof of \cref{lem:VGC-pw-conv}:}
Since $\btheta$ is fixed, we drop it from the notation.
By triangle inequality, 
\begin{equation}\label{eq:vgc-point-conv-i}
	\left|D(\bZ)-\mathbb{E}\left[D(\bZ)\right]\right| \le 
	\left|D(\bZ)-\mathbb{E}\left[\left.D(\bZ)\right|\bZ\in\mathcal E_n \right]\right|+\left|\mathbb{E}\left[\left.D(\bZ)\right|\bZ\in\mathcal E_n \right]-\mathbb{E}\left[D(\bZ)\right]\right|.
\end{equation}
We bound the latter term first.
Since $D(\bZ)$ is bounded by \cref{lem:VGCBounded}, we see
\begin{align*} 
	\left|\mathbb{E}\left[\left.D(\bZ)\right|\bZ\in\mathcal{E}_{n}\right]
	- \mathbb{E}\left[D(\bZ)\right]\right|
	= & \Bigl|\mathbb{E}\left[\left.D(\bZ)\right|\bZ\in\mathcal{E}_{n}\right]
	- \mathbb{E}\left[\left.D(\bZ)\right|\bZ \notin \mathcal{E}_{n} \right]\Bigr|
	\mathbb{P}\left\{ \bZ\notin\mathcal{E}_{n}\right\} 
	\\
	\le & \frac{C_{1}}{n\sqrt{h}}
\end{align*}
for some dimension independent constant $C_1$ by using \cref{lem:ZinGoodSetHighProb}.

We now bound the first term.  We use \cref{thm:Combes-McDiarmid}. Recall from \cref{lem:VGC-Dual-Solution-Bnd} that
\[
\left|D(\bZ)-D(\overline{\bZ})\right|\le C_{2} \log^4 n \cdot  V^3 \log^2 V \cdot \frac{1}{h}\sum_{j=1}^{n}\mathbb{I}\left\{ Z_{j}\ne\overline{Z}_{j}\right\} 
\]
for $\bZ,\overline{\bZ}\in\mathcal{E}_{n}$ and from \cref{lem:ZinGoodSetHighProb}
$\mathbb{P}\left\{ \bZ \not \in\mathcal{E}_{n}\right\} \le \frac{C_{1}}{n}$.  
Finally if $n > 4C_1 e^R$, then $2C_1/n < \frac{1}{2} e^{-R}$, and we have by \cref{thm:Combes-McDiarmid} 
that
\begin{align*}
\abs{D(\bZ) - \Eb{D(\bZ) \mid \bZ \in \mathcal E_n}} 
&\ \leq \ 
C_3 V^3 \log^2(V) \frac{\log^4(n)}{h} + 
C_3 V^3 \log^2(V) \frac{\log^4(n)\sqrt n}{h}
\sqrt{ \log\left(\frac{2}{e^{-R} - 2C_1/n}\right)}
\\
& \ \leq \ 
C_4 V^3 \log^2(V) \frac{\log^4(n)\sqrt n}{h}
\sqrt{ \log\left(\frac{2}{e^{-R} - 2C_1/n}\right)}
\\
& \ \leq \ 
C_5 V^3 \log^2(V) \frac{\log^4(n)\sqrt n}{h}
\sqrt{R},
\end{align*}
where the last line again uses $n > 4C_1 e^R$.

Returning to the initial upper bound \cref{eq:vgc-point-conv-i},  we apply our two bounds to see
\[
	\abs{D(\bZ) - \Eb{D(\bZ)}} 
	\ \leq \ 
	C_6 V^3 \log^2(V) \frac{\log^4(n)\sqrt n}{h} \sqrt{R} 
	+ \frac{C_{6}}{n\sqrt{h}}
\]
By \cref{asn:Parameters}, $h < 1 <  n$ implies that $\frac{\sqrt{n}}{h} \ \geq \frac{1}{n \sqrt h}$.  Hence, collecting dominant terms completes the proof.  
\hfill \Halmos \endproof
\subsection{Uniform Convergence of Dual Solutions} 
The goal of this section is to extend our previous pointwise results to uniform results over all $\btheta \in \Theta$. Let $\overline{\Theta}$
be a minimal $\frac{1}{n}$-covering of $\Theta$.  Then, for every $\btheta\in\Theta$
there exists $\bthetabar\in\overline{\Theta}$ such that $\| \btheta-\bthetabar\|_{2}\le\frac{1}{n}$. 

\begin{lem}
[Uniform Convergence Dual Solution]\label[lem]{lem:Uniform-Dual-sol-conv}
Under the assumptions of \cref{thm:WC-Constraints-bound},
there exists dimension independent constants $C$ and $n_0$ such that for any $R >1$ and any 
$n \geq n_0e^R$,  the following holds with probability $1-2 e^{-R}$:
\[
	\sup_{\btheta\in\Theta}
	\|\blambda(\bZ,\btheta) - \mathbb{E}\left[\blambda(\bZ,\btheta)\right]\|_{\infty}
	\le 
	C V^2 \log^2 V \log m \sqrt{R\log N\left(\frac{1}{n},\Theta\right)} \frac{\log^4 n}{\sqrt n}
\]
\end{lem}

\proof{Proof of \cref{lem:Uniform-Dual-sol-conv}:}
By triangle inequality, 
\begin{align*}
\sup_{\btheta\in\Theta}\| \blambda(\bZ,\btheta)-\mathbb{E}\left[\blambda(\bZ,\btheta)\right]\|_\infty & \le\underbrace{\sup_{\btheta\in\Theta}\|\blambda(\bZ,\btheta)-\blambda(\bZ,\bthetabar)\|_\infty}_{(a)}+\underbrace{\sup_{\btheta\in\Theta}\|\mathbb{E}\left[\blambda(\bZ,\btheta)-\blambda(\bZ,\bthetabar)\right]\|_\infty}_{(b)}\\
 & \qquad+\underbrace{\sup_{\bthetabar\in\bar{\Theta}}\|\blambda(\bZ,\bthetabar)-\mathbb{E}\left[\blambda(\bZ,\bthetabar)\right]\|_\infty}_{(c)}
\end{align*}
We bound each term separately. 

First we bound Term $(a)$. If $\bZ \in \mathcal E_n$, then, from \cref{lem:DualSmoothness}, and bounding the $\ell_\infty$-norm by the $\ell_2$-norm,
\[
\sup_{\btheta\in\Theta}\| \blambda(\bZ,\btheta)-\blambda(\bZ,\bthetabar)\|_\infty
\le
 C_1 V^3 \log^2 V \frac{ \log^{5/4} n }{\sqrt{n}}
\]
for some dimension independent constant $C_1$.  
By \cref{lem:ZinGoodSetHighProb}, this occurs with probability at least $1-4/n$.

Next, we bound $(b)$. Telescoping the expectation as before, we have for any $i =1, \ldots, m$ that 
\begin{align*}
	\mathbb{E}\left[\lambda_{i}(\bZ,\btheta) - \lambda_{i}(\bZ,\bthetabar)\right] 
	& =
	\mathbb{E}\left[\lambda_{i}(\bZ,\btheta) - \lambda_{i}(\bZ,\bthetabar) \Bigr| 
	\bZ \in \mathcal{E}_{n} \right]
	\mathbb{P}\left\{ \bZ \in \mathcal{E}_{n}\right\} 
	\\
 	& \qquad 
 	+ \mathbb{E}\left[\lambda_{i}(\bZ,\btheta) - \lambda_{i}(\bZ,\bthetabar)\Bigr|
 	\bZ \notin \mathcal{E}_{n} \right]
 	\mathbb{P}\left\{ \bZ \notin \mathcal{E}_{n} \right\} 
\end{align*}
We can bound the first term using \cref{lem:DualSmoothness,lem:ZinGoodSetHighProb}.
To bound the second term, define the set 
\[
	\mathcal E_{1,n} \equiv 
	\left\{ \bz \ : \ \| \bz \|_1 \ \leq n C_\mu + \frac{2n}{\sqrt\vmin}\right\},
\]
and recall that $\mathcal E_n \subseteq \mathcal E_{1,n}$.  Observe that
\begin{align*}
	\mathbb{E}\left[\lambda_{i}(\bZ,\btheta) \Bigr| \bZ \notin \mathcal{E}_{n} \right]
	\Pb{\bZ \notin \mathcal{E}_{n}}
	& =
	\mathbb{E}\left[\lambda_{i}(\bZ,\btheta)\mathbb{I}\left\{\bZ \notin \mathcal{E}_{n}\right\} \right]
	\\
	& \leq 
	\mathbb{E}\left[
		\lambda_{i}(\bZ,\btheta) 
		\left( \mathbb{I}\left\{\bZ \notin \mathcal{E}_{1,n}\right\}
		+
		\mathbb{I}\left\{\bZ \in \mathcal{E}_{1,n},  \bZ \notin \mathcal{E}_{n}\right\} 
		\right) 
	\right]
	\\
	& \leq
	\mathbb{E}\left[ \lambda_i(\bZ,\btheta)\mathbb{I}\left\{\bZ \notin \mathcal{E}_{1,n}\right\} \right]
	+ \lambda_{\max} \Pb{\bZ \notin \mathcal{E}_n } 
	&&(\text{\cref{lem:DualsBoundedOnE}})
	\\
	& \leq
	C_2 \exp \left( - \frac{n}{C_2} \right) +  \frac{C_2 \lambda_{\max} }{n}
	&&(\text{\cref{lem:ZinGoodSetHighProb,lem:cond-dual-sol-bnd}}),
\end{align*}
for some dimension independent constant $C_2$.
Combining these observations shows that
\begin{align*}
	\mathbb{E}\left[\lambda_{i}(\bZ,\btheta) - \lambda_{i}(\bZ,\bthetabar)\right] 
	& \le  
	C_3 V^3 \log^2 V \frac{ \log^{5/4} n }{\sqrt{n}} +
	C_3 \exp \left( - \frac{n}{C_3} \right) +  \frac{C_3 \lambda_{\max} }{n} \\
	& \le
	C_4 V^3 \log^2 V \frac{ \log^{5/4} n }{\sqrt{n}}
\end{align*}
where $C_3$ and $C_4$ are dimension-independent constants.  Taking the supremum over $\btheta \in \Theta$ and over $i=1, \ldots, m$, bounds Term $(b)$.

Finally, we bound Term $(c)$. We see that
\[
	\sup_{\bthetabar\in\bar{\Theta},0\le i \le m}
	\left|\lambda_{i}(\bZ,\bthetabar)-\mathbb{E}\left[\lambda_{i}(\bZ,\bthetabar)\right]\right|
	\le 
	C_{5}V^3 \log^2(V) \frac{\log^4 n}{\sqrt n} \sqrt{R \log \left(m\cdot N\left(\frac{1}{n},\Theta\right)\right)}
\]
by applying Lemma \ref{lem:dual-sol-pointwise} and taking the union
bound over the $\left|\overline{\Theta}\right|\le N\left(\frac{1}{n},\Theta\right)$
elements in $\overline{\Theta}$ by the $m$ choices of $i$.

Taking a union bound over the probabilities that bounds hold on Terms $(a)$ and $(c)$ and adding term $(b)$ shows that there exists a dimension independent constant $C$ such that with probability $ 1 - e^{-R} - 4/n$
\[
	\sup_{\btheta\in\Theta}
	\| \blambda(\bZ,\btheta) - \mathbb{E}\left[\blambda(\bZ,\btheta)\right] \|_\infty 
	\ \leq \ 
	C V^3 \log^2(V) \log m \sqrt{R \log N\left(\frac{1}{n},\Theta\right)} \frac{\log^4 n}{\sqrt n}.
\]
Finally, note that if $n > 4e^{-R}$, this last probability is at least $1-2e^{-R}$ to complete the proof.
\hfill \Halmos \endproof

\subsection{Uniform Convergence of In-Sample Optimism}
In this section, we construct a high-probability bound for
\[ 
	\sup_{\btheta\in\Theta}
	\left|
		\frac{1}{n}
		\sum_{j=1}^{n}
		\blue{\xi_{j}}x_{j}(\bZ, \btheta)
		-
		\mathbb{E}\left[\blue{\xi_{j}}x_{j}(\bZ, \btheta)\right]
	\right|
\]
where we recall that $\bZ = \bmu + \bxi $. Note for convenience we have scaled this by $\frac{1}{n}$. 

Constructing the bound requires  decomposing the in-sample optimism into several sub-components.  We outline the subcomponents by providing the proof to \Cref{lem:unif-conv-in-samp-wide-lp} first.   For convenience, in this section \emph{only}, we use the notation $\blambda(\btheta) \equiv \Eb{\blambda(\bZ, \btheta)}$ as shorthand.
\begin{lem}
[Uniform In-sample Optimism for Coupling Constraints] \label[lem]{lem:unif-conv-in-samp-wide-lp} Let $N(\varepsilon, \Theta)$ be the $\varepsilon-$covering number of $\Theta$. Under the assumptions of \cref{thm:WC-Constraints-bound}, there exists dimension independent constants $C$ and $n_0$ such that for any $R > 1$ and $n\ge n_0e^R$, the following holds with probability $1-6\exp(-R)$.
\[
	\sup_{\btheta\in\Theta}
	\left|\frac{1}{n}\sum_{j=1}^{n} \blue{\xi_{j}}x_{j}(\bZ;\btheta)
	-\mathbb{E}\left[\blue{\xi_{j}} x_{j}(\bZ;\btheta)\right]\right|
	\le 
	C V^3 \log^3 V
\sqrt{\log N\left(\frac{1}{n},\Theta\right) } \cdot \frac{R \log^4(n)}{\sqrt n}
\]
\end{lem}
\smallskip 
\proof{Proof of \cref{lem:unif-conv-in-samp-wide-lp}.} 
By triangle inequality, we see,
\begin{align*}
	& \sup_{\btheta\in\Theta}
 	\left|
 		\frac{1}{n}\sum_{j=1}^{n} \blue{\xi_{j}} x_{j}(\bZ, \btheta) - \mathbb{E}\left[ \blue{\xi_{j}} x_{j}(\bZ, \btheta)\right]
 	\right|
 	\\
 	& \ \le \
 	\underbrace{
 		\sup_{\btheta\in\Theta}
 		\left|
 			\frac{1}{n}\sum_{j=1}^{n}\blue{\xi_{j}} \left(x_{j}(\bZ, \btheta) - 
 			\mathbb{I}\left\{ r_{j}(Z_{j}, \btheta)\ge\bA_{j}^{\top}\blambda(\bZ, \btheta)\right\} \right)\right|
 		}_{\text{Rounding Error}}
 	\\
 	&  \qquad + 
 	\underbrace{
 		\sup_{\btheta\in\Theta}
 		\left|
 			\frac{1}{n}\sum_{j=1}^{n} \blue{\xi_{j}} \left(	
 				\mathbb{I}\left\{ r_{j}(Z_{j}, \btheta)\ge\bA_{j}^{\top}\blambda(\bZ, \btheta)\right\} 
 				-
 				\mathbb{I}\left\{ r_{j}(Z_{j}, \btheta)\ge\bA_{j}^{\top}\blambda(\btheta)\right\} 
 			\right)
 		\right|}_{\text{Dual Approximation Error}}
 	\\
 	& \qquad +
 	\underbrace{
 		\sup_{\btheta\in\Theta}
 		\left|
 			\frac{1}{n}\sum_{j=1}^{n} \blue{\xi_{j}}
 			\mathbb{I}\left\{ r_{j}(Z_{j}, \btheta) \ge \bA_{j}^{\top}\blambda(\btheta)\right\} 
 			-
 			\mathbb{E}\left[ \blue{\xi_{j}} \mathbb{I}\left\{ r_{j}(Z_{j}, \btheta) \ge \bA_{j}^{\top}\blambda(\btheta)\right\} \right]
 		\right|
 	}_{\text{ULLN for Dual Approximation}}
 	\\
 	& \qquad +
 	\underbrace{
 		\sup_{\btheta\in\Theta}
 		\left|
 			\frac{1}{n}\sum_{j=1}^{n}
 			\mathbb{E}\left[ \blue{\xi_{j}} \left(
 				\mathbb{I}\left\{ r_{j}(Z_{j}, \btheta)\ge\bA_{j}^{\top}\blambda(\btheta)\right\} 
 			- 
 				\mathbb{I}\left\{ r_{j}(Z_{j}, \btheta)\ge\bA_{j}^{\top}\blambda(\bZ, \btheta)\right\} 
 			\right)\right]
 		\right|
 	}_{\text{Expected Dual Approximation Error} }
 	\\
 	& \qquad + 
 	\underbrace{
 		\sup_{\btheta\in\Theta}
 			\left|
 				\frac{1}{n}\sum_{j=1}^{n}
 				\mathbb{E}\left[
 					\blue{\xi_{j}} \left(
 						\mathbb{I}\left\{ r_{j}(Z_{j}, \btheta)
 						\ge
 						\bA_{j}^{\top}\blambda(\bZ, \btheta)\right\} -x_{j}(\bZ, \btheta)
 					\right)
 				\right]
 			\right|
 		}_{\text{Expected Rounding Error}}.
\end{align*}
For Rounding and Expected Rounding Error, we have
\begin{align*}
\left|
	\sum_{j=1}^{n}\blue{\xi_{j}} \left(x_{j}(\bZ, \btheta)-\mathbb{I}\left\{ r_{j}(Z_{j}, \btheta)
	\ge
	\bA_{j}^{\top}\blambda(\bZ, \btheta)\right\} \right)
\right| 
& \le \left\Vert \bxi \right\Vert _{\infty}\left\Vert x_{j}(\bZ, \btheta)-\mathbb{I}\left\{ r_{j}(Z_{j}, \btheta)\ge\bA_{j}^{\top}\blambda(\bZ, \btheta)\right\} \right\Vert _{1}\\
& \le \left\Vert \bxi \right\Vert _{\infty} m
\end{align*}
where the first inequality follows Holder's inequality and the second inequality holds by complementary slackness.
Note,
\[
\text{\ensuremath{\mathbb{P}}}\left\{ \left\Vert \bxi\right\Vert _{\infty}\ge t\right\} \le\sum_{j=1}^{n}\text{\ensuremath{\mathbb{P}}}\left\{ \left|\xi_{j}\right|\ge t\right\} \le\sum_{j=1}^{n}2\exp\left\{ -\frac{\nu_{\min}t^{2}}{2}\right\}. 
\]
Moreover, $\Eb{ \| \bxi \|_\infty } \ \leq \ C_1 \sqrt{\log n}$ for some dimension independent constant $C_1$.
Thus, with probability at least $1-e^{-R}$, we have
\blue{ \begin{align*}
    \text{Rounding Error} + \text{Expected Rounding Error}
	& \ \le \
	\left\Vert \bxi\right\Vert _{\infty} \frac{m}{n} + \mathbb{E}\left\Vert \bxi\right\Vert _{\infty}\frac{m}{n}\\
 	& \ = \
 	\frac{m}{n}\left(
 		\left\Vert \bxi\right\Vert _{\infty} +\mathbb{E}\left\Vert \bxi\right\Vert _{\infty}
 	\right)
 	\\
 	& \ \le \
 	\frac{C_2 m}{n}\sqrt{R\log n},
\end{align*}}
for some dimension independent constant $C_2$. 

We bound the Dual Approximation Error terms \blue{in \Cref{lem:err-approx-dual-sol}} with our uniform bounds on the dual solutions from \Cref{lem:Uniform-Dual-sol-conv} below, proving that 
with probability at least $1-4e^{-R}$,
\begin{align}\label{eq:unif-in-samp-opt-i}
& \text{Dual Approximation Error} + \text{Expected Dual Approximation Error} \\
& \ \le \ C_3R\sqrt{\frac{V}{n}}+C_3 V^2 \log^2 V
\log(m) \cdot 
\sqrt{R\log N\left(\frac{1}{n},\Theta\right)} \cdot \frac{\log^4(n)}{\sqrt n}
\\
& \ \leq \ 
C_4 V^2 \log^2 V
\log(m) \frac{R}{\sqrt n} \cdot 
\sqrt{\log N\left(\frac{1}{n},\Theta\right) }\frac{\log^4(n)}{\sqrt n}
\end{align}
for some dimension independent constants $C_3$ and $C_4$.


We bound the ULLN for Dual Approximation term in \cref{lem:ULLN-dual-approx} below to prove that
\[
\text{ULLN for Dual Approximation} \ \le \ C_{5}R\sqrt{\frac{V}{n}}
\]
with probability $1-\exp(-R)$. 


Taking a union bound over all probabilities and summing all bounds yields the result.
\hfill \Halmos \endproof

\begin{lem}
[ULLN for Dual Approximation] \label[lem]{lem:ULLN-dual-approx}
Under \cref{asn:LiftedAffinePlugInII,asn:Parameters}, 
there exists a dimension independent constant $C$ such that for any $R > 1$, with probability at least
$1-e^{-R}$,
\[
	\sup_{\btheta\in\Theta}
	\left|
		\sum_{j=1}^{n} \blue{\xi_{j}} 
		\mathbb{I}\left\{ r_{j}(Z_{j}, \btheta)\ge\bA_{j}^{\top}\blambda(\btheta)\right\} 
		- \mathbb{E}\left[\blue{\xi_{j}} \mathbb{I}\left\{ r_{j}(Z_{j}, \btheta)\ge\bA_{j}^{\top}\blambda(\btheta)\right\} \right]\right|\le C\cdot R\sqrt{Vn}
\]
\end{lem}
\proof{\hspace{-12pt}Proof of \cref{lem:ULLN-dual-approx}:}
We first note that 
\begin{align*}
 & \sup_{\btheta\in\Theta}\left|\sum_{j=1}^{n} \blue{\xi_{j}} \mathbb{I}\left\{ r_{j}(Z_{j}, \btheta)\ge\bA_{j}^{\top}\blambda(\btheta)\right\} -\mathbb{E}\left[\blue{\xi_{j}} \mathbb{I}\left\{ r_{j}(Z_{j}, \btheta)\ge\bA_{j}^{\top}\blambda(\btheta)\right\} \right]\right|\\
 & \ \le \ \sup_{\btheta\in\Theta,\blambda\in\mathbb{R}^{m}}\left|\sum_{j=1}^{n} \blue{\xi_{j}}  \mathbb{I}\left\{ r_{j}(Z_{j}, \btheta)\ge\bA_{j}^{\top}\blambda\right\} -\mathbb{E}\left[\blue{\xi_{j}} \mathbb{I}\left\{ r_{j}(Z_{j}, \btheta)\ge\bA_{j}^{\top}\blambda\right\} \right]\right|
\end{align*}
and the summation is a sum of centered independent random variables.
We will apply \cref{thm:Pollard} to the last expression.  Specifically, we consider the envelope $\bm{F}(\bZ)=(\abs{Z_j})_{j=1}^n$. Then, 
we have
\[
	\left\Vert \left\Vert \bm{F}(\bZ)\right\Vert _{2}\right\Vert _{\Psi}
	\overset{(a)}{\le}
	\left\Vert \frac{\left\Vert \bm{\zeta}\right\Vert _{2}}{\sqrt{\nu_{\min}}} \right\Vert _{\Psi}
	\overset{(b)}{\le}
	\sqrt{\frac{2n}{\nu_{\min}}}
	= C_{1}\sqrt{n}
\]
for some dimension independent constant $C_1$. We see $(a)$ holds by letting $\zeta_{j}=\sqrt{\nu_{j}}\xi_{j}$
and $(b)$ holds by Lemma A.1 iv) from GR2020. 

Next,
\[
\abs{ \left\{ \left( \blue{\xi_{j}}  \mathbb{I}\left\{ r_{j}(Z_{j}, \btheta)\ge\bA_{j}^{\top}\blambda\right\} \right)_{j=1}^n \ : \ \btheta\in\Theta,\; \blambda\in\mathbb{R}^{m}\right\} }
\ \leq \ 
\abs{ \left\{ \left(\mathbb{I}\left\{ r_{j}(Z_{j}, \btheta)\ge\bA_{j}^{\top}\blambda\right\} \right)_{j=1}^n \ : \ \btheta\in\Theta,\; \blambda\in\mathbb{R}^{m}\right\} },
\]
and by \cref{asn:LiftedAffinePlugInII}, the latter set has VC-dimension $V$ and hence cardinality at most $2^V$.  

Thus, we see that with probability $1-e^{-R}$, that
\[
\sup_{\btheta\in\Theta,\blambda\in\mathbb{R}^{m}}\left|\sum_{j=1}^{n} \blue{\xi_{j}}  \mathbb{I}\left\{ r_{j}(Z_{j}, \btheta)\ge\bA_{j}^{\top}\blambda\right\} -\mathbb{E}\left[ \blue{\xi_{j}} \mathbb{I}\left\{ r_{j}(Z_{j}, \btheta)\ge\bA_{j}^{\top}\blambda\right\} \right]\right|\le C_2 R\sqrt{V n}
\]
for some absolute constant $C_2$.
\hfill \Halmos \endproof

We next provide bounds for the Dual Approximation Error terms in the proof of \cref{lem:unif-conv-in-samp-wide-lp}.

\begin{lem}
[Dual Approximation Error] \label[lem]{lem:err-approx-dual-sol}
Assume \cref{asn:Parameters,asn:LiftedAffinePlugInII} hold.  Then, there exists dimension independent constants $C$ and $n_0$ such that for any $R>1$ and $n > n_0e^R$, we have 
with probability at least $1-4e^{-R}$, the following two inequalities hold simultaneously:
\begin{align*}
& \sup_{\btheta\in\Theta}
  \left|
  	\sum_{j=1}^{n}\blue{\xi_{j}} \left(
  		\mathbb{I}\left\{ r_{j}(Z_{j}, \btheta)\ge\bA_{j}^{\top}\blambda(\bZ, \btheta)\right\} 
  		-\mathbb{I}\left\{ r_{j}(Z_{j}, \btheta)\ge\bA_{j}^{\top}\blambda(\btheta)\right\} 
  	\right)
  \right|
  \le CR\sqrt{Vn},
\\
& \sup_{\btheta\in\Theta}
  \left|
  	\sum_{j=1}^{n}
  	\mathbb{E}\left[\blue{\xi_{j}} 
	  	\left(
	  		\mathbb{I}\left\{ r_{j}(Z_{j}, \btheta)\ge\bA_{j}^{\top}\blambda(\bZ, \btheta)\right\} 
	  		- \mathbb{I}\left\{ r_{j}(Z_{j}, \btheta)\ge\bA_{j}^{\top}\blambda(\btheta)\right\} 
	  	\right)
  	\right]
  \right|
\\
& \qquad \le 
C V^2 \log^3(V) \cdot \sqrt{R\log N\left(\frac{1}{n},\Theta\right)} \cdot \log^4(n) \sqrt n.
\end{align*}
\end{lem}
\proof{\hspace{-12pt}Proof of \cref{lem:err-approx-dual-sol}:}
First observe that under the conditions of the theorem, \cref{lem:Uniform-Dual-sol-conv} implies that
for some dimension independent constant $C_1$,
 with probability at least $1-2e^{-R}$, 
\[
	\sup_{\btheta \in \Theta} \| \blambda(\bZ, \btheta) - \blambda(\btheta) \|_2 \ \leq \ 
	\underbrace{C_1 V^2 \log^2(V) \log(m) \cdot \sqrt{R\log N\left(\frac{1}{n},\Theta\right)} \cdot \frac{\log^4(n)}{\sqrt n} }_{\equiv \delta},
\]
where we have by bounding the $\ell_2$-norm by $\sqrt m$ times the $\ell_\infty$ norm.  Define the right side to be the constant $\delta$ as indicated.

We will restrict attention to the events where both the above inequality holds and also $\bZ \in \mathcal E_n$.  By the union bound and \cref{lem:ZinGoodSetHighProb}, this event happens with probability at least $1-2e^{-R} - 4/n$. For $n > 4e^R$, this probability is at least $1-3e^{-R}$.  


Now write
\begin{align}\label{eq:err-approx-dual-sol-i}
 & \left|\sum_{j=1}^{n} \blue{\xi_{j}} \left(\mathbb{I}\left\{ r_{j}(Z_{j}, \btheta)\ge\bA_{j}^{\top}\blambda(\bZ, \btheta)\right\} -\mathbb{I}\left\{ r_{j}(Z_{j}, \btheta)\ge\bA_{j}^{\top}\blambda(\btheta)\right\} \right)\right|
 \\ \nonumber 
& \qquad \le  \sum_{j=1}^{n}\left|\blue{\xi_{j}} \left(\mathbb{I}\left\{ r_{j}(Z_{j}, \btheta)\ge\bA_{j}^{\top}\blambda(\bZ, \btheta)\right\} -\mathbb{I}\left\{ r_{j}(Z_{j}, \btheta)\ge\bA_{j}^{\top}\blambda(\btheta)\right\} \right)\right|
\\\nonumber
& \qquad \leq  \sum_{j=1}^{n}\left|\blue{\xi_{j}} \right|\mathbb{I}\left\{ r_{j}(Z_{j}, \btheta)-\bA_{j}^{\top}\blambda(\btheta)\in\left[-C_{A}\delta,C_{A}\delta\right]\right\},
\end{align}
because $\| \blambda(\bZ, \btheta) - \blambda(\btheta)\|_2 \ \leq \delta$.  Furthermore, when the indicator is non-zero, we can bound 
\begin{equation} \label{eq:BoundAbsZInInterval}
\abs{ \blue{\xi_{j}} } \ \leq \ 
\frac{1}{a_{\min}}\left(C_{A}\lambda_{\max}+C_{A}\delta+b_{\max}\right) + C_{\mu}
\ \leq \ C_2(1+\delta),
\end{equation}
for some dimension independent $C_2$.


By \cref{lem:BoundedDualSol}, $\|\blambda(\btheta)\|_1 \leq \frac{2}{n s_0}\mathbb{E}\| \bm{r}(\bZ,\btheta) \|_1 \leq \lambda_{\max}$ and thus we can upper bound, 
\begin{align*}
 & \sup_{\btheta\in\Theta}
 \left|\sum_{j=1}^{n} \blue{\xi_j} \left(\mathbb{I}\left\{ r_{j}(Z_{j}, \btheta)\ge\bA_{j}^{\top}\blambda(\bZ, \btheta)\right\} -\mathbb{I}\left\{ r_{j}(Z_{j}, \btheta)\ge\bA_{j}^{\top}\blambda(\btheta)\right\} \right)\right|
 \\
\le & 
\underbrace{C_2(1+\delta)\sup_{\btheta\in\Theta,\left\Vert \blambda\right\Vert _{1} \le\lambda_{\max}}
\left|\sum_{j=1}^{n} \mathbb{I}\left\{ r_{j}(Z_{j}, \btheta)-\bA_{j}^{\top}\blambda\in\left[-C_{A}\delta,C_{A}\delta\right]\right\} -\mathbb{P}\left\{ r_{j}(Z_{j}, \btheta)-\bA_{j}^{\top}\blambda\in\left[-C_{A}\delta,C_{A}\delta\right]\right\} \right|}_{(i)}
\\
 & \qquad\qquad
 + \underbrace{C_2 (1+\delta) \sup_{\btheta\in\Theta}\mathbb{P}\left\{ r_{j}(Z_{j}, \btheta)-\bA_{j}^{\top}\blambda(\btheta)\in\left[-C_{A}\delta,C_{A}\delta\right]\right\} }_{(ii)}
\end{align*}
To bound the supremum $(i)$, we will apply \cref{thm:Pollard}.  Not that the vector $\be$ is a valid envelope.  To bound the cardinality of 
\[
	\mathcal F \equiv \left\{\left(\Ib{r_j(Z_j, \btheta) - \bA_j^\top \blambda \in [-C_A \delta, C_A\delta]}\right)_{j=1}^n \ : \ 
			\blambda \in \R^m_+, \; \ \btheta \in \Theta\right\},
\]
consider the two sets 
\begin{align*}
	\mathcal F_1 &\equiv \left\{\left(\Ib{r_j(Z_j, \btheta) - \bA_j^\top \blambda \geq -C_A \delta}\right)_{j=1}^n \ : \ 
			\blambda \in \R^m_+, \; \ \btheta \in \Theta\right\},
\\
	\mathcal F_2 &\equiv \left\{\left(\Ib{r_j(Z_j, \btheta) - \bA_j^\top \blambda \leq C_A \delta}\right)_{j=1}^n \ : \ 
			\blambda \in \R^m_+, \; \ \btheta \in \Theta\right\}.			
\end{align*}
Under \cref{asn:LiftedAffinePlugInII}, both sets have pseudo-dimension at most $V$.  Furthermore, $\mathcal F = \mathcal F_1 \wedge \mathcal F_2$.  Hence, by \cite{pollard1990empirical}, there exists an absolute constant $C_3$ such that the pseudo-dimension of $\mathcal F$ is at most $C_3 V$, and hence its cardinality is at most $n^{C_3 V}$.  

Thus, applying \cref{thm:Pollard} shows that there exists a constant $C_4$ such that with probability at least $1-e^{-R}$, 
\[
	\text{Term (i)}  \ \leq \ C_4(1+\delta) R \sqrt{V n \log n}.
\]

To evaluate Term $(ii)$, we recognize it as the probability as the probability that a Gaussian random variable lives in an interval of length $2C_A \delta$.  Upper bounding the Gaussian density by its value at the mean shows 
\begin{align}
\mathbb{P}\left\{ r_{j}(Z_{j}, \btheta)-\bA_{j}^{\top}\blambda\in\left[-C_{A}\delta,C_{A}\delta\right]\right\} \label{eq:err-approx-dual-sol-ii}
 & \le2C_{A} \sqrt{\frac{\vmax}{2 \pi}}\delta
 \ \leq C_5 \delta.
\end{align}
Thus,
\[
	\text{Term } (ii) \ \leq C_6(1+\delta)\delta.
\]


Combining our bounds, we see that with probability at least $1-4e^{-R}$,
\begin{align*}
\sup_{\btheta\in\Theta} & \left|\sum_{j=1}^{n} \blue{\xi_j} \left(\mathbb{I}\left\{ r_{j}(Z_{j}, \btheta)\ge\bA_{j}^{\top}\blambda(\bZ, \btheta)\right\} -\mathbb{I}\left\{ r_{j}(Z_{j}, \btheta)\ge\bA_{j}^{\top}\blambda(\btheta)\right\} \right)\right|
\\
& \ \le  \ 
C_7(1+\delta) R \sqrt{V n} + C_7(1+\delta)\delta
\\
& \ \leq \ 
C_7 R \sqrt{V n}
,
\end{align*}
by substituting the value of $\delta$ and only retaining the dominant terms. 
This proves the first result of the lemma.

To prove the second result of the lemma, note that
\begin{align*}
 & \left|\sum_{j=1}^{n}\E\left[\blue{\xi_j}\left(\mathbb{I}\left\{ r_{j}(Z_{j}, \btheta)\ge\bA_{j}^{\top}\blambda(\bZ, \btheta)\right\} -\mathbb{I}\left\{ r_{j}(Z_{j}, \btheta)\ge\bA_{j}^{\top}\blambda(\btheta)\right\} \right)\right]\right|
 \\
&  \quad \le \sum_{j=1}^{n}\Eb{ \abs{\blue{\xi_j}\left(\mathbb{I}\left\{ r_{j}(Z_{j}, \btheta)\ge\bA_{j}^{\top}\blambda(\bZ, \btheta)\right\} -\mathbb{I}\left\{ r_{j}(Z_{j}, \btheta)\ge\bA_{j}^{\top}\blambda(\btheta)\right\} \right)}}
\\
& \quad \le  C_2(1+\delta) \sum_{j=1}^{n}  \Pb{r_{j}(Z_{j}, \btheta)-\bA_{j}^{\top}\blambda(\btheta)\in\left[-C_{A}\delta,C_{A}\delta\right]}
\\
& \quad \le  nC_8(1+\delta) \delta,
\end{align*}
where the second inequality uses the bound on $\abs{\blue{\xi_j}}$ (\cref{eq:BoundAbsZInInterval}) and the last inequality follows from argument leading to
\cref{eq:err-approx-dual-sol-ii} above.  
Substituting the value of $\delta$, using the assumption that $V\ge m$, and retaining only the dominant terms completes the proof.  
\hfill \Halmos \endproof

\subsection{Uniform Convergence of VGC} 
\begin{lem}[Uniform \Danskin~for Coupling Constraints] \label[lem]{thm:unif-conv-D} Let $N(\varepsilon, \Theta)$ be the $\varepsilon-$covering number of $\Theta$ and the assumptions under \cref{thm:WC-Constraints-bound} hold. There exists dimension independent constants
$C$ and $n_0$ such that for  $n\ge n_0 e^R$ the following holds with probability $1-C e^{-R}$,
\blue{
\[
	\sup_{\btheta\in \bar\Theta}
	\left|
		D(\bZ, \btheta)
		-
		\mathbb{E}\left[D(\bZ, \btheta)\right]
	\right|
	\le
	C \cdot V^2 \log^2 V  \cdot \sqrt{R\cdot n \log \left( n \cdot N(n^{-3/2}, \Theta) \right)} \cdot \frac{\log^4 n}{h_{\min}}
\] } 
\end{lem}

\proof{Proof of \cref{thm:unif-conv-D}}
\blue{ We follow a similar strategy to \cref{lem:unif-danskin-decoupled} and again consider the full notation version of the VGC, $D(\bZ,(\btheta,h))$, and take the supremum over $\btheta \in \Theta$ and $h \in \mathcal{H}$ where $\mathcal{H} \equiv [h_{\min}, h_{\max}]$.  
Let $\Theta_0$ be a minimal $n^{-3/2}$-covering of $\Theta$. In particular, for any $\btheta \in \Theta$ there exists a $\bthetabar \in \Theta_0$ such that 
$\left\Vert \bthetabar-\btheta\right\Vert _{2}\le n^{-3/2}$. Similarly, let $\bar{\mathcal{H}}$ be the $n^{-1}$-covering of $\mathcal{H}$.
By telescoping, 
\begin{align}\label{eq:Danskin-sup-1st-decomp}
	\sup_{\btheta\in\Theta} &
	\left|
		D(\bZ, (\btheta, \overline{h}))-\mathbb{E}\left[D(\bZ, (\btheta, \overline{h}))\right]\right|
 \le \underbrace{
 		\sup_{\substack{\bthetabar \in \Theta_0 \\ \overline{h} \in \overline{\mathcal{H}}}}
 		\left| D(\bZ, (\bthetabar, \overline{h}))-\mathbb{E}\left[D(\bZ, (\bthetabar, \overline{h}))\right]\right|}_{(i)}
	 \\
 & \qquad + \underbrace{
		\sup_{ \substack{\btheta, \overline{\btheta} : \| \btheta - \overline \btheta\| \leq n^{-3/2} \\ h \in \mathcal{H}} } 
		\left| D(\bZ, (\btheta, h))- D(\bZ, (\bthetabar, h))\right| 
		+ \sup_{\substack{\btheta, \overline{\btheta} : \| \btheta - \overline \btheta\| \leq n^{-3/2} \\ h \in \mathcal{H}}} 
		\left|\mathbb{E}\left[D(\bZ, (\bthetabar, h))\right]-\mathbb{E}\left[D(\bZ, (\btheta, h))\right]\right|
	}_{(ii)}
 	\nonumber \\
 & \qquad  + \underbrace{
 		\sup_{ \substack{\bthetabar \in \Theta_0 \\ h,\overline{h} :  \| h - \overline h \| \leq n^{-1}} } 
		\left| D(\bZ, (\bthetabar, h)) - D(\bZ, (\bthetabar, \overline{h}))\right| 
		+ \sup_{\substack{\bthetabar \in \Theta_0 \\ h,\overline{h} :  \| h - \overline h \| \leq n^{-1}}} 
		\left|\mathbb{E}\left[D(\bZ, (\bthetabar, h))\right]-\mathbb{E}\left[D(\bZ, (\bthetabar, \overline{h})) \right]\right|
	}_{(iii)}\nonumber 
\end{align}
We bound Term $(i)$ by taking a union bound over the 
$ N(n^{-3/2}, \Theta)$ elements of $\bar \Theta$ and $ N(n^{-1}, \mathcal{H})$ elements of $\overline{\mathcal{H}}$ in combination with the pointwise bound from \cref{lem:VGC-pw-conv}.  
This shows that with probability at least $1-4e^{-R}$
\[
	\sup_{\substack{\bthetabar \in \Theta_0 \\ \overline{h} \in \overline{\mathcal{H}}}}
	\left|
		D(\bZ, \bthetabar)
		-
		\mathbb{E}\left[D(\bZ, \bthetabar)\right]
	\right|
	\le
C_1 V^2 \log^2(V) \sqrt{R \log \left( N\left(n^{-3/2}, \Theta\right) N(n^{-1}, \mathcal{H}) \right)} \frac{\sqrt{n} \log^4(n)}{h_{\min}} . 
\]
Terms $(ii)$ and $(iii)$ of \cref{eq:Danskin-sup-1st-decomp} can be bounded as follows. 
}

\blue{
First, for $(ii)$, we see by \Cref{lem:D-is-Lipschitz} that for $\| \btheta - \overline{\btheta} \| \leq n^{-3/2}$ there exists a constant $C_1$ such that
\[
		\left| D(\bZ, (\btheta, h))- D(\bZ, (\bthetabar, h))\right| 
	\le \frac{C_1 L}{h} \sqrt{\frac{R}{\nu_{\min}}} \cdot n^{1/2} \sqrt{\log n}
\]
with probability $1 - \exp(-R)$. Similarly, there exists $C_2$,  $C_3$ and $C_4$ (depending on $\vmin, L, C_{\mu}, a_{\min}, a_{\max}, b_{\max}$) such that 
\begin{align*}
\abs{\Eb{D(\bZ, (\btheta, h)) - D(\bZ, (\overline{\btheta}, h))} }
\ \leq \ \frac{C_2 n^{1/2} }{h } \left( \Eb{ \| \bz \|_\infty}  + 1 \right)
\ \leq \ \frac{C_3 n^{1/2} }{h } \left( \sqrt{ \log n}   + C_\mu \right)
\ \leq \ \frac{C_4 n^{1/2} }{h } \sqrt{ \log n}, 
\end{align*}
where the second inequality uses a standard bound on the maximum of $n$ sub-Gaussian random variables, and we have used \cref{asn:Parameters} to simplify. Combining the two terms and taking the supremum over $h \in [h_{\min},h_{\max}]$, we see there exists a constant $C_5$ (depending on $C_1$ and $C_4$), such that 
\[
	(ii) \le \frac{C_5 \sqrt{R n \log n}}{h_{\min}}
\]
Finally, for $(iii)$ we see by \Cref{lem:D-is-Lipschitz} that for $\| h - \overline h \| \leq n^{-1}$, there exists an absolute constant $C_6$ such that,
\[
	(iii) \le \frac{C_6 \sqrt{n}}{h_{\min}\nu_{\min}}
\]
Combining, we see there exists dimension independent constants $C_7$ and $C_8$ such that with probability $1-5e^{-R}$,
\begin{align*}
	\sup_{\substack{\btheta \in \Theta \\ h \in \mathcal{H}}}
	\abs{
		D(\bZ, \btheta)-
		\mathbb{E}\left[D(\bZ, \btheta)\right]
	} 
	& 
	\le
	C_7 V^2 \log^2(V) \sqrt{R n \log \left( N\left(n^{-3/2}, \Theta\right)N(n^{-1}, \mathcal{H}) \right)} \frac{\log^4(n)}{h_{\min}} 
	+
	C_7
	\frac{\sqrt{R n \log n}}{h_{\min}} \\
	& 
	\le
	C_8 \cdot V^2 \log^2 V  \cdot \sqrt{R\cdot n \log \left(n \cdot N(n^{-3/2}, \Theta) \right) } \cdot \frac{\log^4 n}{h_{\min}}
\end{align*}
where the last inequality holds as $N(n^{-1}, \mathcal{H}) \le n$.  This completes the proof.  
}
\hfill \Halmos \endproof

To obtain uniform bounds, we characterize the complexity of the policy class through the \blue{$n^{-3/2}$} covering number of the parameter space $\Theta$. As an example to demonstrate the size of the covering number, consider the case where $\Theta$ is a compact subset of $\mathbb{R}^p$ with a finite diameter $\Gamma$. Applying Lemma 4.1 of \cite{pollard1990empirical} and that the $\epsilon-$packing number bounds the $\epsilon-$covering number, we see that \blue{$N(n^{-3/2},\Theta) \le  \left( 3n^{3/2}\Gamma \right)^p$}.  Combining this bound with \cref{eq:Danskin-sup-1st-decomp}, we obtain the following corollary.


\begin{cor}[Uniform Convergence for Finite Policy Class]\label[cor]{cor:unif-conv-fin-D}
Let $\Theta$ be a compact subset of $\mathbb{R}^p$ with a finite diameter $\Gamma$.   There exists dimension independent constants
$C,n_0$ such that for $n\ge n_0 $ the following holds with probability $1-C\exp\left\{ -R\right\} $,
\[
	\sup_{\btheta \in \Theta}
	\frac{1}{n}\left|
		D(\bZ, \btheta)
		-
		\mathbb{E}\left[D(\bZ, \btheta)\right]
	\right|
	\le C\log^4 n \cdot V^2 \log V
	\cdot \frac{1}{h_{\min}}\sqrt{\frac{R}{n}\cdot 
	\blue{ p \log \left( 3n^{3/2}\Gamma \right) } }
\]
\end{cor}
This corollary shows that the complexity of the policy class depends on the number of parameters of the plug-in policies. We see from \cref{sec:plug-in-policy} that $p$ for many common policy classes does not depend on $n$, implying that the convergence of the VGC estimator to its expectation follows the rate from \cref{cor:unif-conv-fin-D} up to log terms. For example, $p$ for mixed effect policies depends on the dimension of $\bW_j$ which reflects the information available, such as features, for each $\mu_j$. This is typically fixed even as the number of observations $n$ may increase. This implies that for many policy classes, the estimation error converges to 0 as $n\rightarrow \infty$.

\subsection{Proof of \cref{thm:WC-Constraints-bound}}
We can now prove \cref{thm:WC-Constraints-bound}.

\begin{proof}{Proof of \cref{thm:WC-Constraints-bound}:}
\blue{
We proceed to bound each term on the right side of 
\cref{eq:ErrorExpansion}.  
}
\blue{
To bound \cref{eq:ErrorExpansion_Optimism}, we have by \cref{lem:unif-conv-in-samp-wide-lp}, that with probability at least $1-e^{-R}$, that
\begin{align*}
\sup_{\btheta\in \bar{\Theta}}
\abs{ \bxi^\top\bx(\bZ, \btheta) - \Eb{ \bxi^\top \bx(\bZ, \btheta)} } 
& \ \leq \ 
C V^3 \log^3 V
\sqrt{ n \cdot \log N\left(\frac{1}{n},\Theta\right) } \cdot R \log^4(n).
\end{align*}
}
\blue{
To bound \cref{eq:ErrorExpanion_VGC}, let $\mathcal{H} \equiv [h_{\min},h_{\max}]$. Then, by \cref{thm:unif-conv-D}, we have that for some dimension independent constant $C_1$ that with probability at least $1-C_1 e^{-R}$, 
\begin{align*}
   \sup_{\btheta\in \bar\Theta}
	\left|
		D(\bZ, \btheta)
		-
		\mathbb{E}\left[D(\bZ, \btheta)\right]
	\right|
	\le
	C_1 \cdot V^2 \log^2 V  \cdot \sqrt{R\cdot n \log \left( n \cdot N(n^{-3/2}, \Theta) \right)} \cdot \frac{\log^4 n}{h_{\min}}.
\end{align*}
}
\blue{
Finally, to bound \cref{eq:Bias}, use \cref{thm:equiv-in-sample} and take the supremum over $h\in \mathcal{H}$ to obtain
\[
	\cref{eq:Bias} \ \leq \ C_2 h_{\max} n \log(1/h_{\min}).
\]
Substituting these three bounds  into \cref{eq:ErrorExpansion} and collecting constants proves the theorem.  }
\hfill \Halmos \end{proof}

\end{APPENDICES}

\end{document}